\documentclass[12pt,reqno,psamsfonts]{amsart}

\usepackage[alphabetic]{amsrefs}
\usepackage{feynmp}

\usepackage{latexsym}
\usepackage{amsfonts}
\usepackage{amssymb}
\usepackage{amscd}
\usepackage{amsbsy}
\usepackage{amsmath}
\usepackage{bbm}
\usepackage{euscript}%[mathscr]
\usepackage{mathrsfs}

\usepackage{yhmath}
\usepackage[dvips]{graphicx}
\usepackage{epsfig}    %this was commented
\usepackage[all]{xy}
\usepackage{amsthm}

\xyoption{knot}
\xyoption{graph}
\xyoption{tile}
\xyoption{arc}

\def\ie{{\it i.e.\ }}
\def\eg{{\it e.g.\ }}
\def\cf{{\it cf.\ }}
\def\rhs{{\it r.h.s.\ }}
\def\lhs{{\it l.h.s.\ }}

\def\End{\mathop{{\rm End}}\nolimits}

\def\Hom{\mathop{{\rm Hom}}\nolimits}
\def\Ext{\mathop{{\rm Ext}}\nolimits}
\def\Tor{\mathop{{\rm Tor}}\nolimits}
\def\Ob{\mathop{{\rm Ob}}\nolimits}

\def\Tr{\mathop{{\rm Tr}}\nolimits}

\def\deg{ \mathop{{\rm deg}}\nolimits }
\def\p{^{\prime}}
\def\pp{^{\prime\prime}}

\def\del{ \partial }

\def\max{\mathop{{\rm max}}\nolimits}

\def\End{\mathop{{\rm End}}\nolimits}

\def\pr#1#2{ \noindent{\em Proof of #1~\ref{#2}.} }
\def\myqed{ \hfill $\Box$ }

\def\lrbc#1{ \left( #1 \right) }
\def\lrbs#1{ \left[ #1 \right] }

\def\corr#1{ \langle #1 \rangle }

\def\xmapta#1{ \mathop{{\longrightarrow}}^{#1} }

\def\xumap#1#2#3{ \begin{CD} #1 @>{#2}>> #3 \end{CD} }

\parskip=\medskipamount                 % space between paragraphs
\thicklines                     % thick straight lines for pictures
\def\inbar{\vrule height1.5ex width.4pt depth0pt}
\def\IC{\relax\,\hbox{$\inbar\kern-.3em{\rm C}$}}
\def\IN{\relax{\rm I\kern-.18em N}}
\def\IQ{\relax\,\hbox{$\inbar\kern-.3em{\rm Q}$}}
\def\IR{\relax{\rm I\kern-.18em R}}
\def\ZZ{\relax{\sf Z\kern-.4em Z}}

\def\cA{{\cal A}}  \def\cC{{\cal C}} 
\def\cE{{\cal E}}
  \def\cH{{\cal H}} 
   \def\cM{{\cal M}}
 \def\cO{{\cal O}} \def\cP{{\cal P}} 
   
\newtheorem{theorem}{Theorem}[section]
\newtheorem{proposition}[theorem]{Proposition}
\newtheorem{corollary}[theorem]{Corollary}
\newtheorem{conjecture}[theorem]{Conjecture}
\newtheorem{lemma}[theorem]{Lemma}
\newtheorem{definition}[theorem]{Definition}
\newtheorem{remark}[theorem]{Remark}

\catcode`\@=11

\newif\if@fewtab\@fewtabtrue

\catcode`\@=11

\newif\if@fewtab\@fewtabtrue

{\count255=\time\divide\count255 by 60
\xdef\hourmin{\number\count255} \multiply\count255
by-60\advance\count255 by\time
\xdef\hourmin{\hourmin:\ifnum\count255<10 0\fi\the\count255}}
\def\ps@draft{\let\@mkboth\@gobbletwo
    \def\@oddhead{}
    \def\@oddfoot
      {\hbox to 7 cm{\footnotesize {\em Draft of \jobname:} \draftdate
       \hfil}\hskip -7cm\hfil\rm\thepage \hfil}
    \def\@evenhead{}\let\@evenfoot\@oddfoot}

\def\ceqno{\global\@fewtabfalse
    \ifcase\@eqcnt \def\@tempa{& & &}\or \def\@tempa{& &}
      \or \def\@tempa{&}
      \or\def\@tempa{}\fi\@tempa
{\rm(\theequation)}}

\def\aeqno#1{\global\@fewtabfalse
    \ifcase\@eqcnt \def\@tempa{& & &}\or \def\@tempa{& &}
      \or \def\@tempa{&}
      \or\def\@tempa{}\fi\@tempa
{\rm(\theequation,#1)}}

\def\label#1{\ifnum\draftcontrol=1
 \global\def\draftnote{$\scriptstyle #1$}\fi
 \@bsphack\if@filesw {\let\thepage\relax
   \def\protect{\noexpand\noexpand\noexpand}%
\xdef\@gtempa{\write\@auxout{\string
      \newlabel{#1}{{\@currentlabel}{\thepage}}}}}\@gtempa
   \if@nobreak \ifvmode\nobreak\fi\fi\fi
  \@esphack}

\def\alabel#1#2{\label{#1}\global\@fewtabfalse
    \ifcase\@eqcnt \def\@tempa{& & &}\or \def\@tempa{& &}
      \or \def\@tempa{&}
      \or\def\@tempa{}\fi\@tempa
{\hbox to 3cm{\phantom{\rm(\theequation,#2)} \draftnote
\hfil}\hskip -3cm {\rm(\theequation,#2)}}}

\def\clabel#1{\label{#1}\global\@fewtabfalse
    \ifcase\@eqcnt \def\@tempa{& & &}\or \def\@tempa{& &}
      \or \def\@tempa{&}
      \or\def\@tempa{}\fi\@tempa
{\hbox to 3cm{\phantom{\rm(\theequation)} \draftnote \hfil}\hskip
-3cm{\rm(\theequation)}}}

\def\eqnarray{\def\draftnote{{}}\global\@fewtabtrue
\stepcounter{equation}\let\@currentlabel=\theequation
\global\@eqnswtrue
\global\@eqcnt\z@\tabskip\@centering\let\\=\@eqncr
$$\halign to \displaywidth\bgroup\@eqnsel\hskip\@centering\@eqcnt\z@
  $\displaystyle\tabskip\z@{##}$&\global\@eqcnt\@ne
  \hskip 1\arraycolsep \hfil$\displaystyle{##}$\hfil
  &\global\@eqcnt\tw@ \hskip 1\arraycolsep
$\displaystyle\tabskip\z@{##}$ \hfil
\tabskip\@centering&\global\@eqcnt\thr@@\llap{##}\tabskip\z@ \cr}

\def\endeqnarray{\@@eqncr\egroup
      \global\advance\c@equation\m@ne$$\global\@ignoretrue}

\def\@eqnnum{\hbox to 3cm{\phantom{\rm(\theequation)} \draftnote
                         \hfil}\hskip -3cm {\rm(\theequation)}}

\def\@@eqncr{\let\@tempa\relax
    \ifcase\@eqcnt \def\@tempa{& & &}\or \def\@tempa{& &}
      \or \def\@tempa{&}
      \or\def\@tempa{}
\fi\@tempa \if@eqnsw \if@fewtab\@eqnnum\fi
\stepcounter{equation}\fi\global
\@eqnswtrue\global\@eqcnt\z@\global\@fewtabtrue\cr}

\def\draftcite#1{\ifnum\draftcontrol=1#1\else{}\fi}

\def\@lbibitem[#1]#2{\item{}\hskip -3cm \hbox to 2cm
{\hfil$\scriptstyle\draftcite{#2}$}\hskip
1cm[\@biblabel{#1}]\if@filesw
     {\def\protect##1{\string ##1\space}\immediate
      \write\@auxout{\string\bibcite{#2}{#1}}}\fi\ignorespaces}

\def\@bibitem#1{\item\hskip -3cm \hbox to 2cm
{\hfil $\scriptstyle\draftcite{#1}$}\hskip 1cm \if@filesw
\immediate\write\@auxout
       {\string\bibcite{#1}{\the\value{\@listctr}}}\fi\ignorespaces}

\def\draftdate{\number\month/\number\day/\number\year\ \ \ \hourmin }
 \global\def\draftcontrol{0}
\catcode`\@=12

\def\theequation{{\thesection.\arabic{equation}}}
\def\qq{\begin{eqnarray}}
\def\qqq{\end{eqnarray}}
\def\ee{\begin{eqnarray}}
\def\eee{\end{eqnarray}}
\def\rx#1{~(\ref{#1})}
\def\rxw#1{(\ref{#1})}
\def\ex#1{eq.\hspace*{-3pt}\rx{#1}}
\def\eex#1{eqs.\hspace*{-3pt}\rx{#1}}
\def\cx#1{~\cite{#1}}
\def\cxw#1{\cite{#1}}
\def\rw#1{~\ref{#1}}

\def\xlee#1{ \begin{eqnarray} \label{#1} }
\def\xeee{ \end{eqnarray} }
\def\ylee#1{ \begin{eqnarray}\nonumber }
\def\yeee{ \end{eqnarray} }
\def\zlee#1{ \begin{displaymath} }
\def\zleee{ \end{displaymath} }
\def\wlee#1{ $ }
\def\weee{ $ }

\newlength{\shiftwidth}
\def\shift#1{&&\hbox to \shiftwidth{\hfill $\displaystyle#1$}}
\newlength{\sshiftwidth}
\def\sshift#1{\lefteqn{\hbox to
\sshiftwidth{\hfill$\displaystyle#1$}}}
\def\qbezier{\bezier{120}}
\def\DottedCircle{
\bezier{4}(0.966,-0.259)(1.04,0)(0.966,0.259)
\bezier{4}(0.966,0.259)(0.897,0.518)(0.707,0.707)
\bezier{4}(0.707,0.707)(0.518,0.897)(0.259,0.966)
\bezier{4}(0.259,0.966)(0,1.04)(-0.259,0.966)
\bezier{4}(-0.259,0.966)(-0.518,0.897)(-0.707,0.707)
\bezier{4}(-0.707,0.707)(-0.897,0.518)(-0.966,0.259)
\bezier{4}(-0.966,0.259)(-1.04,0)(-0.966,-0.259)
\bezier{4}(-0.966,-0.259)(-0.897,-0.518)(-0.707,-0.707)
\bezier{4}(-0.707,-0.707)(-0.518,-0.897)(-0.259,-0.966)
\bezier{4}(-0.259,-0.966)(0,-1.04)(0.259,-0.966)
\bezier{4}(0.259,-0.966)(0.518,-0.897)(0.707,-0.707)
\bezier{4}(0.707,-0.707)(0.897,-0.518)(0.966,-0.259) }
\def\Endpoint[#1]{
\ifcase#1 \put(1,0){\circle*{0.15}}
\or\put(0.866,0.5){\circle*{0.15}}
\or\put(0.5,0.866){\circle*{0.15}} \or\put(0,1){\circle*{0.15}}
\or\put(-0.5,0.866){\circle*{0.15}}
\or\put(-0.866,0.5){\circle*{0.15}} \or\put(-1,0){\circle*{0.15}}
\or\put(-0.866,-0.5){\circle*{0.15}}
\or\put(-0.5,-0.866){\circle*{0.15}} \or\put(0,-1){\circle*{0.15}}
\or\put(0.5,-0.866){\circle*{0.15}}
\or\put(0.866,-0.5){\circle*{0.15}} \fi}
\def\Arc[#1]{
\thicklines         % this can be changed!
\ifcase#1 \bezier{25}(0.966,-0.259)(1.04,0)(0.966,0.259) \or
\bezier{25}(0.966,0.259)(0.897,0.518)(0.707,0.707) \or
\bezier{25}(0.707,0.707)(0.518,0.897)(0.259,0.966) \or
\bezier{25}(0.259,0.966)(0,1.04)(-0.259,0.966) \or
\bezier{25}(-0.259,0.966)(-0.518,0.897)(-0.707,0.707) \or
\bezier{25}(-0.707,0.707)(-0.897,0.518)(-0.966,0.259) \or
\bezier{25}(-0.966,0.259)(-1.04,0)(-0.966,-0.259) \or
\bezier{25}(-0.966,-0.259)(-0.897,-0.518)(-0.707,-0.707) \or
\bezier{25}(-0.707,-0.707)(-0.518,-0.897)(-0.259,-0.966) \or
\bezier{25}(-0.259,-0.966)(0,-1.04)(0.259,-0.966) \or
\bezier{25}(0.259,-0.966)(0.518,-0.897)(0.707,-0.707) \or
\bezier{25}(0.707,-0.707)(0.897,-0.518)(0.966,-0.259) \fi}
\def\DottedArc[#1]{
\ifcase#1 \bezier{4}(0.966,-0.259)(1.04,0)(0.966,0.259) \or
\bezier{4}(0.966,0.259)(0.897,0.518)(0.707,0.707) \or
\bezier{4}(0.707,0.707)(0.518,0.897)(0.259,0.966) \or
\bezier{4}(0.259,0.966)(0,1.04)(-0.259,0.966) \or
\bezier{4}(-0.259,0.966)(-0.518,0.897)(-0.707,0.707) \or
\bezier{4}(-0.707,0.707)(-0.897,0.518)(-0.966,0.259) \or
\bezier{4}(-0.966,0.259)(-1.04,0)(-0.966,-0.259) \or
\bezier{4}(-0.966,-0.259)(-0.897,-0.518)(-0.707,-0.707) \or
\bezier{4}(-0.707,-0.707)(-0.518,-0.897)(-0.259,-0.966) \or
\bezier{4}(-0.259,-0.966)(0,-1.04)(0.259,-0.966) \or
\bezier{4}(0.259,-0.966)(0.518,-0.897)(0.707,-0.707) \or
\bezier{4}(0.707,-0.707)(0.897,-0.518)(0.966,-0.259) \fi}
\def\Chord[#1,#2]{
\thinlines \ifnum#1>#2\Chord[#2,#1] \else\ifnum#1<#2 \ifcase#1
\ifcase#2 \or\qbezier(1,0)(0.516,0.138)(0.866,0.5)
\or\qbezier(1,0)(0.45,0.26)(0.5,0.866)
\or\qbezier(1,0)(0.327,0.327)(0,1)
\or\qbezier(1,0)(0.179,0.311)(-0.5,0.866)
\or\qbezier(1,0)(0.0536,0.2)(-0.866,0.5) \or\put(1, 0){\line(-2,
0){2}} \or\qbezier(1,0)(0.0536,-0.2)(-0.866,-0.5)
\or\qbezier(1,0)(0.179,-0.311)(-0.5,-0.866)
\or\qbezier(1,0)(0.327,-0.327)(0,-1)
\or\qbezier(1,0)(0.45,-0.26)(0.5,-0.866)
\or\qbezier(1,0)(0.516,-0.138)(0.866,-0.5) \fi \or\ifcase#2\or
\or\qbezier(0.866,0.5)(0.378,0.378)(0.5,0.866)
\or\qbezier(0.866,0.5)(0.26,0.45)(0,1)
\or\qbezier(0.866,0.5)(0.12,0.446)(-0.5,0.866)
\or\qbezier(0.866,0.5)(0,0.359)(-0.866,0.5)
\or\qbezier(0.866,0.5)(-0.0536,0.2)(-1,0) \or\put(0.866,
0.5){\line(-5, -3){1.73}}
\or\qbezier(0.866,0.5)(0.146,-0.146)(-0.5,-0.866)
\or\qbezier(0.866,0.5)(0.311,-0.179)(0,-1)
\or\qbezier(0.866,0.5)(0.446,-0.12)(0.5,-0.866)
\or\qbezier(0.866,0.5)(0.52,0)(0.866,-0.5) \fi \or\ifcase#2\or\or
\or\qbezier(0.5,0.866)(0.138,0.516)(0,1)
\or\qbezier(0.5,0.866)(0,0.52)(-0.5,0.866)
\or\qbezier(0.5,0.866)(-0.12,0.446)(-0.866,0.5)
\or\qbezier(0.5,0.866)(-0.179,0.311)(-1,0)
\or\qbezier(0.5,0.866)(-0.146,0.146)(-0.866,-0.5) \or\put(0.5,
0.866){\line(-3, -5){1}} \or\qbezier(0.5,0.866)(0.2,-0.0536)(0,-1)
\or\qbezier(0.5,0.866)(0.359,0)(0.5,-0.866)
\or\qbezier(0.5,0.866)(0.446,0.12)(0.866,-0.5) \fi
\or\ifcase#2\or\or\or \or\qbezier(0,1.)(-0.138,0.516)(-0.5,0.866)
\or\qbezier(0,1.)(-0.26,0.45)(-0.866,0.5)
\or\qbezier(0,1.)(-0.327,0.327)(-1,0)
\or\qbezier(0,1.)(-0.311,0.179)(-0.866,-0.5)
\or\qbezier(0,1.)(-0.2,0.0536)(-0.5,-0.866) \or\put(0, 1){\line(0,
-2){2}} \or\qbezier(0,1.)(0.2,0.0536)(0.5,-0.866)
\or\qbezier(0,1.)(0.311,0.179)(0.866,-0.5) \fi
\or\ifcase#2\or\or\or\or
\or\qbezier(-0.5,0.866)(-0.378,0.378)(-0.866,0.5)
\or\qbezier(-0.5,0.866)(-0.45,0.26)(-1,0)
\or\qbezier(-0.5,0.866)(-0.446,0.12)(-0.866,-0.5)
\or\qbezier(-0.5,0.866)(-0.359,0)(-0.5,-0.866)
\or\qbezier(-0.5,0.866)(-0.2,-0.0536)(0,-1) \or\put(-0.5,
0.866){\line(3, -5){1}}
\or\qbezier(-0.5,0.866)(0.146,0.146)(0.866,-0.5) \fi
\or\ifcase#2\or\or\or\or\or
\or\qbezier(-0.866,0.5)(-0.516,0.138)(-1,0)
\or\qbezier(-0.866,0.5)(-0.52,0)(-0.866,-0.5)
\or\qbezier(-0.866,0.5)(-0.446,-0.12)(-0.5,-0.866)
\or\qbezier(-0.866,0.5)(-0.311,-0.179)(0,-1)
\or\qbezier(-0.866,0.5)(-0.146,-0.146)(0.5,-0.866) \or\put(-0.866,
0.5){\line(5, -3){1.73}} \fi \or\ifcase#2\or\or\or\or\or\or
\or\qbezier(-1,0)(-0.516,-0.138)(-0.866,-0.5)
\or\qbezier(-1,0)(-0.45,-0.26)(-0.5,-0.866)
\or\qbezier(-1,0)(-0.327,-0.327)(0,-1)
\or\qbezier(-1,0)(-0.179,-0.311)(0.5,-0.866)
\or\qbezier(-1,0)(-0.0536,-0.2)(0.866,-0.5) \fi
\or\ifcase#2\or\or\or\or\or\or\or
\or\qbezier(-0.866,-0.5)(-0.378,-0.378)(-0.5,-0.866)
\or\qbezier(-0.866,-0.5)(-0.26,-0.45)(0,-1)
\or\qbezier(-0.866,-0.5)(-0.12,-0.446)(0.5,-0.866)
\or\qbezier(-0.866,-0.5)(0,-0.359)(0.866,-0.5) \fi
\or\ifcase#2\or\or\or\or\or\or\or\or
\or\qbezier(-0.5,-0.866)(-0.138,-0.516)(0,-1)
\or\qbezier(-0.5,-0.866)(0,-0.52)(0.5,-0.866)
\or\qbezier(-0.5,-0.866)(0.12,-0.446)(0.866,-0.5) \fi
\or\ifcase#2\or\or\or\or\or\or\or\or\or
\or\qbezier(0,-1.)(0.138,-0.516)(0.5,-0.866)
\or\qbezier(0,-1.)(0.26,-0.45)(0.866,-0.5) \fi
\or\ifcase#2\or\or\or\or\or\or\or\or\or\or
\or\qbezier(0.5,-0.866)(0.378,-0.378)(0.866,-0.5) \fi\fi\fi\fi}
\def\FullChord[#1,#2]{
\Endpoint[#1] \Endpoint[#2] \Arc[#1] \Arc[#2] \Chord[#1,#2] }
\def\EndChord[#1,#2]{
\Endpoint[#1] \Endpoint[#2] \Chord[#1,#2] }
\def\Picture#1{
\begin{picture}(2,1)(-1,-0.167)
#1
\end{picture}
}
\def\DottedChordDiagram[#1,#2]{
\Picture{\DottedCircle \FullChord[#1,#2]} }
\def\ZZ{ \mathbb{Z} }
\def\IQ{ \mathbb{Q} }
\def\IC{ \mathbb{C} }
\def\IR{ \mathbb{R} }

\def\bfx{ \mathbf{x} }

\def\cA{ {\cal A} }

\def\cA{ \mathcal{A} }

\def\cC{ \mathcal{C} }

\def\cE{ \mathcal{E} }

\def\cH{ \mathcal{H} }

\def\hlf{ {1\over 2} }

\def\sltln{ \sum_{\xlam\in\rTLn } }
\def\sltl{ \sum_{\xlam\in\yTL } }

\def\lmii{ \lim_{i\rightarrow\infty} }

\def\xId{ \mathbbm{1} }

\def\JW{Jones-Wenzl}
\def\JWp{\JW\ projector}

\def\TLb{Temperley-Lieb}
\def\TLba{TL}
\def\TLa{\TLb\ algebra}

\def\cylb{cylindrical}
\def\cbr{\cylb\ braid}

\def\aptg{cap-tangle}
\def\uptg{cup-tangle}

\def\dstt{distinguished triangle}

\def\stI{ \mathbf{I} }
\def\stIp{ \stI\p }

\def\ngbr{negative braid}

\def\nels#1{ |#1| }
\def\nI{ \nels{\stI} }
\def\nIp{ \nels{\stIp} }

\def\nsplcd{negatively spliced}

\def\apdg{cap-degree}
\def\updg{cup-degree}
\def\qdgr{$q$-degree}

\def\TNG{Tng}
\def\FT{FT}
\def\TL{TL}

\def\symalg#1{ \accentset{\smalg}{#1} }

\def\rxv#1{ \mathrm{#1} }
\def\rTNG{ \rxv{\TNG} }
\def\rFT{ \rxv{\FT} }
\def\rTL{ \rxv{\TL} }

\def\xop{ \mathrm{op} }

\def\cTL{ \symalg{\rTL} }
\def\cTLop{ \cTL^{\xop} }
\def\cTLv#1{ \cTL_{#1} }
\def\cTLn{ \cTLv{n} }
\def\cTLm{ \cTLv{m} }
\def\cTLvv#1#2{ \cTL\xivv{#1}{#2} }
\def\cTLmn{ \cTLvv{m}{n} }
\def\cTLkm{ \cTLvv{k}{m} }
\def\cTLkn{ \cTLvv{k}{n} }
\def\cTLnn{ \cTLvv{n}{n} }

\def\dTL{ \symcat{\rTL} }
\def\dTLop{ \dTL^{\xop} }
\def\dTLv#1{ \dTL_{#1} }
\def\dTLn{ \dTLv{n} }

\def\dTLvv#1#2{ \dTL\xivv{#1}{#2} }
\def\dTLmn{ \dTLvv{m}{n} }

\def\dTLtn{ \dTLv{2n} }
\def\dTLtnz{ \dTLvv{2n}{0} }
\def\dTLztn{ \dTLvv{0}{2n} }
\def\dTLtmz{ \dTLvv{2m}{0} }
\def\dTLtmtn{ \dTLvv{2m}{2n} }

\def\Hmtnz{ \Hom_{\dTLtnz} }

\def\Hmmn{ \Hom_{\dTLmn} }
\def\Hmtmtn{ \Hom_{\dTLtmtn} }

\def\QcTL{ \mathrm{Q}\cTL }
\def\QcTLv#1{ \QcTL_{#1} }
\def\QcTLn{ \QcTLv{n} }

\def\QcTLvv#1#2{ \QcTL\xivv{#1}{#2} }
\def\QcTLmn{ \QcTLvv{m}{n} }

\def\cTLpinf{ \cTL^{+} }
\def\cTLminf{ \cTL^{-} }
\def\cTLminfop{ (\cTLminf)^{\xop} }

\def\symcat#1{ \accentset{\circ}{#1} }
\def\symbcat#1{ \lrbc{ #1 }^{\circ} }
\def\spsymbcat#1{ \spcc{\symbcat{#1}} }

\def\dTL{ \symcat{\rTL} }
\def\dTLp{ \dTL^+ }
\def\dTLnp{ \dTLp_n }

\def\xivv#1#2{_{#1,#2} }
\def\xiv#1{_#1}

\def\rTNGvv#1#2{ \rTNG\xivv{#1}{#2} }
\def\rTNGv#1{ \rTNG\xiv{#1} }
\def\rTNGmn{ \rTNGvv{m}{n} }
\def\rTNGn{ \rTNGv{n} }

\def\rTLvv#1#2{ \rTL\xivv{#1}{#2} }

\def\ttngvv#1#2{$(#1,#2)$-tangle}
\def\ttngmn{\ttngvv{m}{n}}
\def\ttngnm{\ttngvv{n}{m}}
\def\ttngnn{\ttngvv{n}{n}}

\def\ttngzz{\ttngvv{0}{0}}

\def\ttngztn{\ttngvv{0}{2n}}
\def\ttngztm{\ttngvv{0}{2m}}
\def\ttngtmtn{\ttngvv{2m}{2n}}

\def\ttngtltm{\ttngvv{2l}{2m}}

\def\Zqqi{ \ZZ[q,q^{-1}] }
\def\Zsqqi{ \ZZ[[q,q^{-1}] }
\def\Zsqiq{ \ZZ[[q^{-1},q] }
\def\Qq{ \mathrm{Q}(q) }

\def\op{ \mathrm{op} }

\def\acapni{ \acapvv{n}{i} }

\def\ccapni{ \ccapvv{n}{i} }

\def\xU{ U }
\def\xUv#1#2{ \xU_{#1,#2} }
\def\xUni{ \xUv{n}{i} }

\def\jwp{ P }
\def\bjwp{ \finv{\jwp} }
\def\jwph{ \hat{\jwp} }
\def\jwpv#1{ \jwp_{#1} }
\def\jwpn{ \jwpv{n} }
\def\jwpk{ \jwpv{k} }
\def\jwpnp{ \jwpn^+ }
\def\jwpnm{ \jwpn^- }

\def\bjwpxvv#1#2{ \bjwp_{#1,#2} }
\def\bjwpxnm{ \bjwpxvv{n}{m} }

\def\jwpxvv#1#2{ \jwp_{#1,#2} }
\def\jwphxvv#1#2{ \jwph_{#1,#2} }

\def\jwpxtntm{ \jwpxvv{2n}{2m} }
\def\jwpxnv#1{ \jwpxvv{n}{#1} }
\def\jwpxnm{ \jwpxnv{m} }
\def\jwpxnk{ \jwpxnv{k} }
\def\jwpxmk{ \jwpxvv{m}{k} }
\def\jwpxnmp{ \jwpxnv{m\p} }
\def\jwpxnn{ \jwpxnv{n} }
\def\jwpxtnv#1{ \jwpxvv{2n}{#1} }

\def\jwpxstm{ \jwpxvv{\ast}{m} }
\def\jwphxstz{ \jwphxvv{\ast}{0} }
\def\jwpxvz#1{ \jwpxvv{#1}{0} }
\def\jwpxtmz{ \jwpxvz{2m} }
\def\jwpxtnz{ \jwpxvz{2n} }
\def\jwpxtnz{ \jwpxtnv{0} }

\def\xtau{ \tau }
\def\xtaup{ \xtau\p }
\def\xlam{ \lambda }
\def\xlamp{ \xlam\p }
\def\ctau{ \symalg{\xtau} }
\def\clam{ \symalg{\xlam} }
\def\dtau{ \symcat{\xtau} }
\def\dlam{ \symcat{\xlam} }

\def\xlamv#1{ \xlam_{#1} }
\def\xlamo{ \xlamv{1} }
\def\xlamtw{ \xlamv{2} }

\def\dtaup{ \symcat{\xtaup} }

\def\dlam{ \symcat{\xlam} }
\def\dlamp{ \symcat{\xlamp} }

\def\Ktau{ \Ksymbim{\xtau} }
\def\Ktauo{ \Ksymbim{\xtauo} }
\def\Ktaut{ \Ksymbim{\xtaut} }
\def\Klam{ \zKsymbim{\xlam} }
\def\Kal{ \zKsymbim{\xal} }
\def\Kbet{ \zKsymbim{\xbet} }

\def\xnu{ \nu }

\def\xnuv#1{ \xnu_{#1} }
\def\xnum{ \xnuv{m} }

\def\xca{ a }
\def\xcav#1{ \xca_{#1} }
\def\xcal{ \xcav{\xlam} }
\def\xcalv#1{ \xcal(#1) }
\def\xcalt{ \xcalv{\xtau} }
\def\xcalw#1{ \xcav{\xlam,#1} }
\def\xcali{ \xcalw{i} }
\def\xcalj{ \xcalw{j} }
\def\xcaliv#1{ \xcali(#1) }
\def\xcalit{ \xcaliv{\xtau} }

\def\xL{ L }
\def\dL{ \symcat{\xL} }

\def\dLi{ \dLv{i} }
\def\dLimo{ \dLv{i-1} }

\def\aL{ \symalg{\xL} }

\def\dgo{ \deg_{\mathrm{h}} }
\def\dgt{ \deg_q }
\def\dgh{ \deg_{2} }

\def\tgrshv#1#2#3{ \lrbs{#2,#1,#3} }
\def\tgrsshv#1#2{ \lrbs{#2,#1} }
\def\tgrsshji{ \tgrsshv{j}{i} }

\def\qshv#1{ \lrbs{ #1 }_{q} }
\def\qshj{ \qshv{j} }
\def\qsho{ \qshv{1} }
\def\qshmo{ \qshv{-1} }
\def\qtshv#1{ \lrbs{#1}_{q,2} }

\def\tgrshteq#1{ \qtshv{#1} }
\def\tgrshnzn{ \tgrshteq{n} }
\def\tgrshmzm{ \tgrshteq{m} }
\def\tgrshmmn{ \tgrshteq{m-n} }
\def\tgrshminmn{ \tgrshteq{-m-n} }

\def\spshn{ \tgrsshv{n}{n-1} }
\def\spshmn{ \spshn^{2m} }

\def\ospshn{ \tgrsshv{-n}{1-n} }
\def\ospshmn{ \ospshn^{2m} }

\def\Kz{\mathrm{K}_0}
\def\Kzg{$\Kz$-group}

\def\Sc{ \mathrm{S} }
\def\Scv#1{ \Sc^{#1} }
\def\Sco{ \Scv{1} }

\def\btcyl{ \beta_{\mathrm{cyl}} }

\def\btcylv#1{ \beta_{\mathrm{cyl},#1} }

\def\btcyln{ \btcylv{n} }

\def\vhf{ \frac{1}{2} }

\def\vthf{ \tfrac{1}{2} }
\def\vthh{ \tfrac{3}{2} }

\def\qpqi{ q + q^{-1} }
\def\qvh{ q^{\vhf} }
\def\qvmh{ q^{-\vhf} }

\def\xmrf{f}

\def\yal{ \alpha }
\def\ybet{ \beta }

\def\Opqv#1{ O_{+}(q^{#1})}

\def\Opqm{ \Opqv{m} }

\def\Oh{ O^{\mathrm{h}} }
\def\Ohp{ \Oh_{+} }

\def\Ohpv#1{ \Ohp(#1) }

\def\Ohpm{ \Ohpv{\ym} }

\def\ym{ m }
\def\ymv#1{ \ym_{#1} }
\def\ymi{ \ymv{i} }
\def\Ohpmi{ \Ohpv{\ymi} }

\def\ctC{ \mathcal{C} }

\def\symfr{ \blacksquare }
\def\symfr{ \circ }

\def\Cone{ \mathop{\mathrm{Cone}} }
\def\Cnv#1{ \Cone(#1) }
\def\CnBv#1{ \Cone\Big(#1\Big) }

\def\Cnbf{ \Cnv{\xbf} }
\def\Cnbfi{ \Cnv{\xbfi} }

\def\Cnbtfz{ \Cnv{\xbtfz} }
\def\Cnchdlbtfz{ \Cnv{\chdlbtfz} }

\def\xmu{ \mu }
\def\xcv#1{ c_{#1} }
\def\cjilam{ \xcv{i,j,\xmu}^{\xlam} }
\def\cjmilam{ \xcv{-i,j,\xmu}^{\xlam} }

\def\qpb{$q^+$\nobreakdash-bounded}

\def\ctjw{ \mathbf{P} }
\def\ctjwv#1{ \ctjw_{#1} }
\def\ctjwn{ \ctjwv{n} }
\def\ctjwp{ \ctjw^{+} }
\def\ctjwm{ \ctjw^{-} }
\def\ctjwpv#1{ \ctjwp_{#1} }
\def\ctjwmv#1{ \ctjwm_{#1} }
\def\ctjwpn{ \ctjwpv{n} }
\def\Kctjwpn{ \Kz(\ctjwpn) }
\def\ctjwmn{ \ctjwmv{n} }

\def\xgd{$\xId$-balanced}

\def\xIdbv#1{ \xId_{#1} }
\def\xIdbn{ \xIdbv{n} }

\def\xsg{ \sigma }
\def\xsgv#1{ \xsg_{#1} }
\def\xsgi{ \xsgv{i} }
\def\xsgiv#1{ \xsgv{i_{#1}} }
\def\xsgikpo{ \xsgiv{k+1} }

\def\brb{ \beta }
\def\brbp{ \brb\p }
\def\cbrb{ \symcat{\brb} }
\def\cbrba{ \symcats{\brb} }
\def\cbrbpa{ \symcats{\brbp} }
\def\spcc#1{ #1_{\sharp} }
\def\cbrbs{ \spcc{\cbrb} }
\def\cbrbas{ \spcc{\cbrba} }
\def\cbrbpas{ \spcc{\cbrbpa} }
\def\cbrbps{ \cbrbs\p }

\def\xC{ C }
\def\xCv#1{ \xC_{#1} }
\def\xCo{ \xCv{1} }
\def\xCt{ \xCv{2} }
\def\xCi{ \xCv{i} }
\def\xCipo{ \xCv{i+1} }
\def\xbC{ \mathbf{C} }
\def\xbCv#1{ \xbC_{#1} }
\def\xbCo{ \xbCv{1} }
\def\xbCt{ \xbCv{2} }
\def\xbCvv#1#2{ \xbC_{#1,#2} }
\def\xbCvn#1{ \xbCvv{#1}{n} }
\def\xbCmn{ \xbCvn{m} }
\def\xbCon{ \xbCvn{1} }
\def\xbCn{ \xbCv{n} }

\def\qmd{`module'}
\def\qmds{`modules'}
\def\qcmd{chain \qmd}
\def\qcmds{chain \qmds}

\def\otbl{(1,2)-balanced}
\def\otblc{(1,2)-balance}
\def\dct{cut}
\def\dctv#1{#1-\dct}
\def\odct{\dctv{1}}

\def\crs{ \# }
\def\crsv#1{ \crs_{#1} }

\def\mrf{ f }
\def\mrf{ \mathbf{f} }
\def\mrfv#1{ \mrf_{#1} }
\def\mrfz{ \mrfv{0} }
\def\mrfo{ \mrfv{1} }
\def\mrfm{ \mrfv{m} }
\def\mrfmmo{ \mrfv{m-1} }
\def\mrfmo{ \mrfv{m+1} }

\def\dmrf{ \dsymv{\mrf} }
\def\dmrfv#1{ \dmrf_{#1} }
\def\dmrfz{ \dmrfv{0} }
\def\dmrfo{ \dmrfv{1} }
\def\dmrfm{ \dmrfv{m} }
\def\dmrfmmo{ \dmrfv{m-1} }
\def\dmrfmo{ \dmrfv{m+1} }

\def\hteqv{ \simeq_{\mathrm{h}} }

\def\cpd{ d }
\def\cpdv#1{ \cpd_{#1} }
\def\cpdlam{ \cpdv{\xlam} }

\def\alm{ a_{\xlam} }
\def\blm{ b_{\xlam} }

\def\xnot{ \xlam_\varnothing }
\def\cnot{ \symcat{\xlam}_\varnothing }

\def\scA{ \mathcal{A} }
\def\scAs{ \scA_{\sharp} }

\def\scAp{ \scA\p }
\def\scAp{ \scAs }

\def\xctA{ \mathsf{A} }

\def\xbA{ \mathbf{A} }
\def\xbAs{ \xbA_{\sharp} }
\def\xA{ A }
\def\xAv#1{ \xA_{#1} }

\def\xAi{ \xAv{i} }
\def\xAio{ \xAv{i+1} }
\def\xAimo{ \xAv{i-1} }

\def\xbAv#1{ \xbA_{#1} }
\def\xbAz{ \xbAv{0} }
\def\xbAo{ \xbAv{1} }
\def\xbAt{ \xbAv{2} }
\def\xbAi{ \xbAv{i} }
\def\xbAio{ \xbAv{i+1} }

\def\xbAp{ \xbA\p }
\def\xbApv#1{ \xbA_{\sharp,#1} }
\def\xbApz{ \xbApv{0} }
\def\xbApo{ \xbApv{1} }
\def\xbApt{ \xbApv{2} }
\def\xbApi{ \xbApv{i} }
\def\xbApio{ \xbApv{i+1} }

\def\xctB{ \mathcal{B} }
\def\xctBv#1{ \xctB_{#1} }
\def\xctBn{ \xctBv{n} }
\def\xctBnd{ \dsymv{\xctBn} }
\def\xctBmn{ \xctBv{m,n} }

\def\xbB{ \mathbf{B} }
\def\xB{ B }
\def\xBv#1{ \xB_{#1} }

\def\xBi{ \xBv{i} }
\def\xBio{ \xBv{i+1} }

\def\xd{ d }
\def\xdv#1{ \xd_{#1} }

\def\xdi{ \xdv{i} }
\def\xdio{ \xdv{i+1} }
\def\xdit{ \xdv{i+2} }
\def\xdimo{ \xdv{i-1} }

\def\xdp{ \xd\p }
\def\xdpv#1{ \xdp_{#1} }

\def\xdpi{ \xdpv{i} }
\def\xdpio{ \xdpv{i+1} }
\def\xdpit{ \xdpv{i+2} }
\def\xdpimo{ \xdpv{i-1} }

\def\xbf{ \mathbf{f} }
\def\xbfv#1{ \xbf_{#1} }
\def\xbfi{ \xbfv{i} }

\def\xbfz{ \xbfv{0} }
\def\xbfo{ \xbfv{1} }
\def\xbft{ \xbfv{2} }

\def\xbh{ \mathbf{h} }
\def\xbhv#1{ \xbh_{#1} }
\def\xbhi{ \xbhv{i} }

\def\xbd{ \mathbf{d} }
\def\xbdv#1{ \xbd_{#1} }
\def\xbdi{ \xbdv{i} }

\def\xbhf{ \hat{\xbf} }
\def\xbhfv#1{ \xbhf_{#1} }
\def\xbhfi{ \xbhfv{i} }
\def\xbhfimo{ \xbhfv{i-1} }

\def\xbhh{ \hat{\xbh} }
\def\xbhhv#1{ \xbhh_{#1} }
\def\xbhhi{ \xbhhv{i} }
\def\xbhhz{ \xbhhv{0} }
\def\xbhho{ \xbhhv{1} }

\def\xbch{ \check{\xbh} }
\def\xbchv#1{ \xbch_{#1} }
\def\xbchi{ \xbchv{i} }

\def\xbhg{ \hat{\ybg} }
\def\xbhgv#1{ \xbhg_{#1} }
\def\xbhgi{ \xbhgv{i} }
\def\xbhgimo{ \xbhgv{i-1} }

\def\xbfAB{ \xbfv{\xbA\xbB} }
\def\xbfBC{ \xbfv{\xbB\xbC} }
\def\xbfAC{ \xbfv{\xbA\xbC} }

\def\xbtf{ \tilde{\xbf} }
\def\xbtfv#1{ \xbtf_{#1} }
\def\xbtfi{ \xbtfv{i} }
\def\xbtfio{ \xbtfv{i+1} }
\def\xbtfz{ \xbtfv{0} }

\def\xbtfm{ \xbtfv{m} }

\def\xbtfp{ \tilde{\xbf\p} }
\def\xbtfpv#1{ \xbtfp_{#1} }
\def\xbtfpi{ \xbtfpv{i} }

\def\xf{ f }
\def\xfv#1{ \xf_{#1} }
\def\xfi{ \xfv{i} }

\def\yf{ f }
\def\yfv#1{ \yf_{#1} }

\def\yfi{ \yfv{i} }
\def\yfio{ \yfv{i+1} }

\def\yg{ g }

\def\ybg{ \mathbf{\yg} }
\def\ybgv#1{ \ybg_{#1} }
\def\ybgA{ \ybgv{\xbA} }
\def\ybgB{ \ybgv{\xbB} }
\def\ybgi{ \ybgv{i} }
\def\ybgio{ \ybgv{i+1} }
\def\ybgimo{ \ybgv{i-1} }
\def\ybgz{ \ybgv{0} }
\def\ybgo{ \ybgv{1} }
\def\ybgt{ \ybgv{2} }
\def\ybgh{ \ybgv{3} }

\def\ybtg{ \tilde{\ybg} }
\def\ybtgv#1{ \ybtg_{#1} }
\def\ybtgi{ \ybtgv{i} }
\def\ybtgio{ \ybtgv{i+1} }

\def\ybtgz{ \ybtgv{0} }

\def\yh{ h }
\def\ybh{ \mathbf{\yh} }
\def\ybhv#1{ \ybh_{#1} }
\def\ybhi{ \ybhv{i} }
\def\ybhz{ \ybhv{0} }
\def\ybho{ \ybhv{1} }
\def\ybht{ \ybhv{2} }

\def\ybhp{ \ybh\p }
\def\ybhpv#1{ \ybhp_{#1} }
\def\ybhpi{ \ybhpv{i} }
\def\ybhpz{ \ybhpv{0} }
\def\ybhpo{ \ybhpv{1} }

\def\ybhtp{ \tilde{\ybh}\p }
\def\ybhtpv#1{ \ybhtp_{#1} }
\def\ybhtpi{ \ybhtpv{i} }

\def\xCh{ \mathop{\mathbf{Ch}} }
\def\xChv#1{ \xCh(#1) }
\def\xChA{ \xChv{\xctA} }

\def\xKh{ \mathop{\mathbf{K}} }
\def\xKhv#1{ \xKh(#1) }
\def\xKhA{ \xKhv{\xctA} }

\def\isor{isomorphism order}

\def\ysiov#1{ \left|\, #1\, \right|_{\cong} }
\def\ysiorv#1{ |\, #1\, |_{\cong} }
\def\ysiobf{ \ysiov{\xbf} }
\def\ysiobfi{ \ysiov{\xbfi} }

\def\stblz{stabilizing}

\def\xtrn{ \mathbf{t} }
\def\xtrnvv#1#2{ \xtrn_{\leq #1} #2 }
\def\xtrniv#1{ \xtrnvv{i}{#1} }
\def\xtrnNv#1{ \xtrnvv{N}{#1} }

\def\chcpl{chain complex}
\def\chmp{chain morphism}
\def\chcpls{\chcpl es}

\def\trnc{truncation}

\def\trncv#1{$#1$-th \trnc}
\def\trnci{\trncv{i}}

\def\chdl{ \delta }
\def\chdlv#1{ \chdl_{#1} }
\def\chdlbf{ \chdlv{\xbf} }
\def\chdlbfi{ \chdlv{\xbfi} }
\def\chdlbfz{ \chdlv{\xbfz} }
\def\chdlbtfz{ \chdlv{\xbtfz} }
\def\chdlbtfm{ \chdlv{\xbtfm} }

\def\idl{ \iota }
\def\idlv#1{ \idl_{#1} }
\def\idlbf{ \idlv{\xbf} }
\def\idlbgi{ \idlv{\ybgi} }
\def\idlbgz{ \idlv{\ybgz} }
\def\idlbgo{ \idlv{\ybgo} }

\def\idlbhi{ \idlv{\ybhi} }

\def\Cch{Cauchy}
\def\xsyst{system}
\def\xdsyst{direct \xsyst}
\def\Csq{\Cch\ \xsyst}

\def\dlm{ \varinjlim }
\def\ilm{ \varprojlim }

\def\qord{$q$-order}

\def\xsupv#1{ \sup\left\{#1 \right\} }
\def\xinfv#1{ \inf\left\{#1 \right\} }

\def\yordhb#1{ \bigl| \;#1 \;\bigr|_{\mathrm{h}} }
\def\yordh#1{ \left|\, #1\, \right|_{\mathrm{h}} }
\def\yordhr#1{ |\, #1\, |_{\mathrm{h} } }
\def\yordch#1{ \yordh{\Cnv{#1}} }
\def\yordq#1{ \left|\, #1\, \right|_q }

\def\zordh#1{ | #1 |_{\mathrm{h}} }
\def\zordch#1{ \zordh{\Cnv{#1}} }

\def\chsq{\xdsyst}

\def\chlm{ \lim\nolimits_{\xCh} }

\def\cmtr#1#2#3#4#5#6{
\xymatrix{
#4 \ar[r]_-{#1} \ar@/^1pc/[rr]^-{#3} &
#5 \ar[r]_-{#2} &
#6
},\qquad #3\hteqv #2\,#1
}

\def\atcmv#1#2{ #1\,#2 + #2\,#1 }

\def\xIdv#1{ \xId_{#1} }
\def\yIdAi{ \xIdv{\xbAi} }
\def\yIdA{ \xIdv{\xbA} }

\def\tchlm{chain limit}

\def\mmp{multi-map}

\def\ybC{ \mathbf{C} }
\def\ybCv#1{ \ybC_{#1} }
\def\ybCz{ \ybCv{0} }
\def\ybCo{ \ybCv{1} }
\def\ybCt{ \ybCv{2} }
\def\ybCi{ \ybCv{i} }
\def\ybCimo{ \ybCv{i-1} }

\def\ycC{ \mathcal{C} }
\def\yctC{ \tilde{\ycC} }

\def\ybtC{ \tilde{\ybC} }
\def\ybtCv#1{ \ybtC_{#1} }
\def\ybtCz{ \ybtCv{0} }
\def\ybtCo{ \ybtCv{1} }

\def\ybtCi{ \ybtCv{i} }

\def\ybtCio{ \ybtCv{i+1} }

\def\mtcn{multi-cone}

\def\wbC{ \tilde{\xbC} }
\def\wbCvv#1#2{ \wbC_{#1,#2} }
\def\wbCnv#1{ \wbCvv{#1}{n} }
\def\wbCmn{ \wbCnv{m} }
\def\wbCzn{ \wbCnv{0} }
\def\wbCmnp{ \wbCmn\p }

\def\xrahv#1{ \xrightarrow{\;\;\;#1\;\;\;} }
\def\xratv#1{ \xrightarrow{\;\;#1\;\;} }
\def\xraov#1{ \xrightarrow{\;#1\;} }

\def\rxratv#1{ \xleftarrow{\;\;#1\;\;} }
\def\rxraov#1{ \xleftarrow{\;#1\;} }

\def\Kh{Khovanov homology}

\def\amap{ \symalg{-} }

\def\clckw{clockwise}
\def\cclckw{counterclockwise}

\def\dsym{ \vee }
\def\dsymv#1{ #1^{\dsym} }
\def\xtaud{ \dsymv{\xtau} }

\def\oltlmn{ \bigoplus_{\xlam\in\yTLmn } }
\def\sltln{ \sum_{\xlam\in\yTLn } }

\def\sabTLztn{ \sum_{\xal,\xbet\in\yTLztn} }

\def\lmii{ \lim_{i\rightarrow\infty} }
\def\lmmi{ \lim_{m\rightarrow\infty} }

\def\xId{ \mathbbm{1} }

\def\JW{Jones-Wenzl}
\def\JWp{\JW\ projector}

\def\TLb{Temperley-Lieb}
\def\TLba{TL}
\def\TLa{\TLb\ algebra}

\def\cylb{torus}
\def\cbr{\cylb\ braid}

\def\aptg{cap-tangle}
\def\uptg{cup-tangle}

\def\dstt{distinguished triangle}

\def\stI{ \mathbf{I} }
\def\stJ{ \mathbf{J} }
\def\stIp{ \stI\p }

\def\stJp{ \stJ\p }

\def\ngbr{negative braid}

\def\nels#1{ |#1| }
\def\nI{ \nels{\stI} }
\def\nIp{ \nels{\stIp} }

\def\nsplcd{negatively spliced}

\def\apdg{cap-degree}
\def\updg{cup-degree}
\def\qdgr{$q$-degree}

\def\TNG{Tng}
\def\FT{FT}
\def\TL{TL}

\def\apr{ \mathrm{pr} }
\def\prsymalg#1{ \symalg{#1}_{\apr} }

\def\rxv#1{ \mathit{#1} }
\def\rTNG{ \rxv{\TNG} }
\def\rFT{ \rxv{\FT} }
\def\rTL{ \rxv{\TL} }

\def\xop{ \mathrm{op} }

\def\cTL{ \symalg{\rTL} }
\def\cTLop{ \cTL^{\xop} }
\def\cTLv#1{ \cTL_{#1} }
\def\cTLopv#1{ \cTLop_{#1} }
\def\cTLz{ \cTLv{0} }
\def\cTLn{ \cTLv{n} }

\def\cTLm{ \cTLv{m} }
\def\cTLmop{ \cTLopv{m} }
\def\cTLvv#1#2{ \cTL\xivv{#1}{#2} }
\def\cTLmn{ \cTLvv{m}{n} }
\def\cTLnm{ \cTLvv{n}{m} }

\def\cTLtmtn{ \cTLvv{2m}{2n} }
\def\cTLprmn{ \cTLmn^{\mathrm{pr}} }
\def\cTLkm{ \cTLvv{k}{m} }
\def\cTLkn{ \cTLvv{k}{n} }

\def\cfQTL{ \fQ\cTL }

\def\cfQTLvv#1#2{ \cfQTL\xivv{#1}{#2} }

\def\cfQTLztn{ \cfQTLvv{0}{2n} }
\def\cfQTLtnz{ \cfQTLvv{2n}{0} }
\def\cfQTLtmz{ \cfQTLvv{2m}{0} }
\def\cfQTLtntn{ \cfQTLvv{2n}{2n} }
\def\cfQTLtmtn{ \cfQTLvv{2n}{2m} }

\def\ftct{ \mathrm{spl} }

\def\dTL{ \symcat{\rTL} }
\def\dTLtl{ \check{\dTL} }
\def\dTLtlv#1{ \dTLtl_{#1} }
\def\dTLtlztn{ \dTLtlv{0,2n} }
\def\dTLtltnz{ \dTLtlv{2n,0} }

\def\dTLtltmz{ \dTLtlv{2m,0} }

\def\dTLtltmtn{ \dTLtlv{2m,2n} }
\def\dTLtlmn{ \dTLtlv{m,n} }
\def\dTLtllm{ \dTLtlv{l,m} }
\def\dTLtlln{ \dTLtlv{l,n} }

\def\dTLop{ \dTL^{\xop} }
\def\dTLv#1{ \dTL_{#1} }
\def\dTLn{ \dTLv{n} }
\def\dTLmn{ \dTLv{m,n} }
\def\dTLnm{ \dTLv{n,m} }
\def\dTLtn{ \dTLv{2n} }

\def\dTLct{ \dTL^{\ftct} }
\def\dTLtlct{ \dTLtl^{\ftct} }
\def\dTLtlctv#1{ \dTLtlct_{#1} }
\def\dTLtlctvv#1#2{ \dTLtlctv{#1,#2} }
\def\dTLtlcttmtn{ \dTLtlctvv{2m}{2n} }

\def\dTLctvv#1#2{ \dTLct_{#1,#2} }

\def\dTLcttmtn{ \dTLctvv{2m}{2n} }

\def\dTLz{ \dTLv{0} }
\def\dTLvv#1#2{ \dTL\xivv{#1}{#2} }
\def\dTLmn{ \dTLvv{m}{n} }

\def\dTLztn{ \dTLvv{0}{2n} }
\def\dTLtntn{ \dTLvv{2n}{2n} }

\def\dTLcp{ \dTL^{\ftct,-} }
\def\dTLcpv#1{ \dTLcp_{#1} }
\def\dTLcpvv#1#2{ \dTLcpv{#1,#2} }

\def\dTLcptntm{ \dTLcpvv{2n}{2m} }

\def\dTLcptn{ \dTLcpv{2n} }

\def\dTLcq{ \dTL^{\ftct,+} }
\def\dTLcqv#1{ \dTLcq_{#1} }
\def\dTLcqtn{ \dTLcqv{2n} }

\def\QcTL{ \mathrm{Q}\cTL }
\def\QcTLv#1{ \QcTL_{#1} }
\def\QcTLn{ \QcTLv{n} }

\def\QcTLvv#1#2{ \QcTL\xivv{#1}{#2} }
\def\QcTLmn{ \QcTLvv{m}{n} }

\def\QcTLprv#1#2{ \QcTL_{#1\,|#2} }
\def\QcTLprnm{ \QcTLprv{n}{m} }
\def\QcTLprnmp{ \QcTLprv{n}{m\p} }

\def\cTLpinf{ \cTL^{+} }
\def\cTLminf{ \cTL^{-} }
\def\cTLminfop{ (\cTLminf)^{\xop} }

\def\symbcat#1{ \lrbc{ #1 }^{\circ} }
\def\spsymbcat#1{ \spcc{\symbcat{#1}} }

\def\dTL{ \symcat{\rTL} }
\def\dTLnp{ \dTLp_n }

\def\dTLp{ \dTL^- }
\def\dTLpq{ \dTL^{-/+} }
\def\dTLcpq{ \dTL^{\ftct,-/+} }
\def\dTLcp{ \dTL^{\ftct,-} }

\def\dTLpv#1{ \dTLp_{#1} }
\def\dTLptn{ \dTLpv{2n} }

\def\dTLpqvv#1#2{ \dTLpq_{#1,#2} }
\def\dTLpqv#1{ \dTLpq_{#1} }

\def\dTLpqtmtn{ \dTLpqvv{2m}{2n} }

\def\dTLpqtn{ \dTLpqv{2n} }

\def\dTLcpvv#1#2{ \dTLcp_{#1,#2} }
\def\dTLcpv#1{ \dTLcp_{#1} }

\def\dTLcptn{ \dTLcpv{2n} }

\def\dTLcpqvv#1#2{ \dTLcpq_{#1,#2} }
\def\dTLcpqv#1{ \dTLcpq_{#1} }

\def\dTLcpqtmtn{ \dTLcpqvv{2m}{2n} }
\def\dTLcpqtntn{ \dTLcpqvv{2n}{2n} }

\def\dTLcpqtn{ \dTLcpqv{2n} }

\def\xivv#1#2{_{#1,#2} }
\def\xiv#1{_{#1}}

\def\rTNGvv#1#2{ \rTNG\xivv{#1}{#2} }
\def\rTNGv#1{ \rTNG\xiv{#1} }

\def\rTNGmn{ \rTNGvv{m}{n} }
\def\rTNGn{ \rTNGv{n} }

\def\rTLvv#1#2{ \rTL\xivv{#1}{#2} }

\def\ttngvv#1#2{$(#1,#2)$-tangle}
\def\ttngmn{\ttngvv{m}{n}}
\def\ttngtmtn{\ttngvv{2m}{2n}}
\def\ttngnm{\ttngvv{n}{m}}
\def\ttngtntm{\ttngvv{2n}{2m}}
\def\ttngnn{\ttngvv{n}{n}}
\def\ttngmm{\ttngvv{m}{m}}

\def\ttngzz{\ttngvv{0}{0}}
\def\ttngztm{\ttngvv{0}{2m}}
\def\ttngtt{\ttngvv{2}{2}}

\def\Zqqi{ \ZZ[q,q^{-1}] }
\def\Zsqqi{ \IQ[[q,q^{-1}] }
\def\ZZsqqi{ \ZZ[[q,q^{-1}] }
\def\Zsqiq{ \ZZ[[q^{-1},q] }

\def\fQ{ Q }
\def\fQ{ \mathbb{Q} }
\def\Qq{ \mathrm{\fQ}(q) }

\def\op{ \mathrm{op} }

\def\acapni{ \acapvv{n}{i} }

\def\ccapni{ \ccapvv{n}{i} }

\def\xU{ U }
\def\xUv#1#2{ \xU_{#1,#2} }
\def\xUni{ \xUv{n}{i} }

\def\jwp{ \mathrm{P} }
\def\jwpv#1{ \jwp_{#1} }
\def\jwpn{ \jwpv{n} }

\def\jwpm{ \jwpv{m} }

\def\jwpnp{ \jwpn^+ }
\def\jwpnm{ \jwpn^- }

\def\xtau{ T }
\def\xtaup{ \xtau\p }
\def\xlam{ G }
\def\xlamp{ \xlam\p }
\def\xlampp{ \xlam\pp }
\def\ctau{ \psymalg{\xtau} }

\def\clam{ \psymalg{\xlam} }
\def\prclam{ \prsymalg{\xlam} }

\def\dtau{ \symcat{\xtau} }

\def\dlam{ \symcat{\xlam} }
\def\dlamo{ \zsymcat{\xlamo} }
\def\dlamtw{ \zsymcat{\xlamtw} }
\def\Klam{ \Ksymbim{\xlam} }
\def\Klamp{ \Ksymbim{\xlamp} }

\def\udlam{ \symcat{\uxlam} }

\def\xnu{ \nu }

\def\xnuv#1{ \xnu_{#1} }
\def\xnum{ \xnuv{m} }

\def\xca{ a }
\def\xcav#1{ \xca_{#1} }
\def\xcal{ \xcav{\xlam} }
\def\xcalv#1{ \xcal(#1) }
\def\xcalt{ \xcalv{\xtau} }
\def\xcalw#1{ \xcav{\xlam,#1} }
\def\xcali{ \xcalw{i} }
\def\xcalj{ \xcalw{j} }
\def\xcaliv#1{ \xcali(#1) }
\def\xcalit{ \xcaliv{\xtau} }
\def\xcaljv#1{ \xcalj(#1) }
\def\xcaljt{ \xcaljv{\xtau} }

\def\xL{ L }
\def\dL{ \symcat{\xL} }

\def\dLi{ \dLv{i} }
\def\dLimo{ \dLv{i-1} }

\def\dgo{ \deg_{\mathrm{h}} }
\def\dgt{ \deg_q }
\def\dgh{ \deg_{2} }

\def\tgrshv#1#2#3{ \lrbs{#2,#1,#3} }
\def\tgrsshv#1#2{ \lrbs{#2,#1} }
\def\tgrsshji{ \tgrsshv{j}{i} }

\def\qshv#1{ \lrbs{ #1 }_{q} }
\def\qshj{ \qshv{j} }
\def\qsho{ \qshv{1} }
\def\qshmo{ \qshv{-1} }
\def\qshtn{ \qshv{2n} }

\def\hqshv#1{\left[#1\right]_{\mathrm{h},q} }

\def\spshn{ \tgrsshv{n}{n-1} }
\def\spshmn{ \spshn^{2m} }

\def\ospshn{ \tgrsshv{-n}{1-n} }
\def\ospshmn{ \ospshn^{2m} }

\def\Kz{\mathrm{K}_0}
\def\Kzg{$\Kz$-group}

\def\Kzp{ \Kz^+ }
\def\QKz{ \Kz^{Q} }

\def\Gz{ \mathrm{G} }

\def\Sc{ \mathrm{S} }
\def\Scv#1{ \Sc^{#1} }
\def\Sco{ \Scv{1} }

\def\btcyl{ \brb_{\mathrm{cyl}} }

\def\btcylv#1{ \brb_{\mathrm{cyl},#1} }

\def\btcyln{ \btcylv{n} }

\def\vhf{ \frac{1}{2} }

\def\vthf{ \tfrac{1}{2} }
\def\vthh{ \tfrac{3}{2} }

\def\qpqi{ q + q^{-1} }
\def\qvh{ q^{\vhf} }
\def\qvmh{ q^{-\vhf} }

\def\xmrf{f}

\def\yal{ \alpha }
\def\ybet{ \beta }

\def\Opqv#1{ O_{+}(q^{#1})}

\def\Opqm{ \Opqv{m} }

\def\Oh{ O^{\mathrm{h}} }
\def\Ohp{ \Oh_{-} }

\def\Ohpv#1{ \Ohp(#1) }

\def\Ohpm{ \Ohpv{\ym} }
\def\ym{ m }
\def\ymv#1{ \ym_{#1} }
\def\ymi{ \ymv{i} }
\def\Ohpmi{ \Ohpv{\ymi} }

\def\ctC{ \mathcal{C} }

\def\xctC{ \ctC }
\def\cctC{ \ctC }

\def\symfr{ \blacksquare }
\def\symfr{ \circ }

\def\Cone{ \mathop{\mathrm{Cone}} }
\def\Cnv#1{ \Cone(#1) }
\def\CnBv#1{ \Cone\Big(#1\Big) }

\def\Cnbf{ \Cnv{\xbf} }
\def\Cnbfi{ \Cnv{\xbfi} }

\def\Cnbtfz{ \Cnv{\xbtfz} }
\def\Cnchdlbtfz{ \Cnv{\chdlbtfz} }

\def\xmu{ \mu }
\def\xcv#1{ m_{#1} }
\def\xav#1{ m_{#1} }

\def\cjilam{ \xcv{i,j,\xmu}^{\xlam} }
\def\cjmilam{ \xcv{-i,j,\xmu}^{\xlam} }

\def\ajilam{ \xav{i,j,\xmu}^{\xlam} }
\def\ajmilam{ \xav{-i,j,\xmu}^{\xlam} }

\def\aija{ \xav{i,j}^{\inoa} }
\def\aja{ \xav{j}^{\inoa} }

\def\qpb{$q^+$\nobreakdash-bounded}

\def\ctjw{ \mathbf{P} }
\def\ctjwv#1{ \ctjw_{#1} }
\def\ctjwn{ \ctjwv{n} }
\def\ctjwp{ \ctjw^{-} }
\def\ctjwm{ \ctjw^{+} }
\def\ctjwpv#1{ \ctjwp_{#1} }
\def\ctjwmv#1{ \ctjwm_{#1} }
\def\ctjwpn{ \ctjwpv{n} }
\def\Kctjwpn{ \Kz(\ctjwpn) }
\def\ctjwmn{ \ctjwmv{n} }

\def\xgd{$\xId$-balanced}

\def\xIdbv#1{ \xId_{#1} }
\def\xIdbn{ \xIdbv{n} }

\def\xsg{ S }
\def\xsgv#1{ \xsg_{#1} }
\def\xsgi{ \xsgv{i} }
\def\xsgiv#1{ \xsgv{i_{#1}} }
\def\xsgikpo{ \xsgiv{k+1} }

\def\brb{ B }
\def\brbp{ \brb\p }
\def\cbrb{ \symcat{\brb} }
\def\cbrba{ \symcats{\brb} }
\def\cbrbpa{ \symcats{\brbp} }
\def\spcc#1{ #1^{\sharp} }
\def\Bspcc#1{ \spcc{ \xlrB{#1} } }
\def\cbrbs{ \spcc{\cbrb} }
\def\cbrbas{ \spcc{\cbrba} }
\def\cbrbpas{ \spcc{\cbrbpa} }
\def\cbrbps{ \cbrbs\p }

\def\xbAc{ \spcc{\xbA} }

\def\xC{ C }
\def\xCv#1{ \xC_{#1} }

\def\xCo{ \xCv{1} }
\def\xCt{ \xCv{2} }
\def\xCi{ \xCv{i} }

\def\xCipo{ \xCv{i+1} }
\def\xbC{ \mathbf{C} }
\def\xbCv#1{ \xbC_{#1} }
\def\xbCo{ \xbCv{1} }
\def\xbCt{ \xbCv{2} }
\def\xbCvv#1#2{ \xbC_{#1,#2} }
\def\xbCvn#1{ \xbCvv{#1}{n} }
\def\xbCmn{ \xbCvn{m} }
\def\xbCon{ \xbCvn{1} }
\def\xbCn{ \xbCv{n} }

\def\prbC{ \xbC }
\def\prbCs{ \spcc{\prbC } }

\def\prbCsn{ \prbCs_n }

\def\qmd{`module'}
\def\qmds{`modules'}
\def\qcmd{chain \qmd}
\def\qcmds{chain \qmds}

\def\otbl{angle-shaped}

\def\otblc{angle shape}
\def\rshp{ray shape}

\def\dct{cut}
\def\dctv#1{#1-\dct}
\def\odct{\dctv{1}}

\def\crs{ \# }
\def\crsv#1{ \crs_{#1} }

\def\mrf{ f }
\def\mrf{ \mathbf{f} }
\def\mrfv#1{ \mrf_{#1} }
\def\mrfz{ \mrfv{0} }
\def\mrfo{ \mrfv{1} }
\def\mrfm{ \mrfv{m} }
\def\mrfmmo{ \mrfv{m-1} }
\def\mrfmo{ \mrfv{m+1} }

\def\dmrf{ \dsymv{\mrf} }
\def\dmrfv#1{ \dmrf_{#1} }
\def\dmrfz{ \dmrfv{0} }
\def\dmrfo{ \dmrfv{1} }
\def\dmrfm{ \dmrfv{m} }
\def\dmrfmmo{ \dmrfv{m-1} }
\def\dmrfmo{ \dmrfv{m+1} }

\def\hteqv{ \sim }

\def\cpd{ d }
\def\cpdv#1{ \cpd_{#1} }
\def\cpdlam{ \cpdv{\xlam} }

\def\alm{ a_{\xlam} }
\def\blm{ b_{\xlam} }

\def\xnot{ \xlam_\varnothing }
\def\cnot{ \symcat{\xlam}_\varnothing }

\def\scA{ \mathcal{A} }
\def\scAs{ \scA_{\sharp} }

\def\scAp{ \scA\p }
\def\scAp{ \scAs }

\def\xctA{ \mathsf{A} }

\def\xbA{ \mathbf{A} }
\def\xbAs{ \xbA_{\sharp} }
\def\xA{ A }
\def\xAv#1{ \xA_{#1} }

\def\xAi{ \xAv{i} }
\def\xAz{ \xAv{0} }
\def\xAo{ \xAv{1} }
\def\xAmi{ \xAv{-i} }
\def\xAio{ \xAv{i+1} }

\def\xAimo{ \xAv{i-1} }

\def\xAmi{ \xAv{-i} }

\def\ina{ \xlam }
\def\inap{ \xlam\p }
\def\inapp{ \xlam\pp }
\def\inA{ \Lambda }
\def\inh{ h }

\def\xbAv#1{ \xbA_{#1} }
\def\xbAz{ \xbAv{0} }
\def\xbAo{ \xbAv{1} }
\def\xbAt{ \xbAv{2} }
\def\xbAi{ \xbAv{i} }
\def\xbAa{ \xbAv{\ina} }
\def\xbAap{ \xbAv{\inap} }
\def\xbAio{ \xbAv{i+1} }

\def\xbAp{ \xbA\p }
\def\xbApv#1{ \xbA_{\sharp,#1} }
\def\xbApz{ \xbApv{0} }
\def\xbApo{ \xbApv{1} }
\def\xbApt{ \xbApv{2} }
\def\xbApi{ \xbApv{i} }
\def\xbApio{ \xbApv{i+1} }

\def\xbAd{ (\xbA,\xbd) }

\def\xctB{ \mathcal{B} }
\def\xctBv#1{ \xctB_{#1} }
\def\xctBn{ \xctBv{n} }
\def\xctBnd{ \dsymv{\xctBn} }
\def\xctBmn{ \xctBv{m,n} }

\def\xbB{ \mathbf{B} }
\def\xB{ B }
\def\xBv#1{ \xB_{#1} }

\def\xBi{ \xBv{i} }
\def\xBio{ \xBv{i+1} }

\def\xd{ d }
\def\xdv#1{ \xd_{#1} }

\def\xdi{ \xdv{i} }

\def\xdio{ \xdv{i+1} }
\def\xdit{ \xdv{i+2} }
\def\xdimo{ \xdv{i-1} }

\def\xdp{ \xd\p }
\def\xdpv#1{ \xdp_{#1} }

\def\xdpi{ \xdpv{i} }
\def\xdpio{ \xdpv{i+1} }
\def\xdpit{ \xdpv{i+2} }
\def\xdpimo{ \xdpv{i-1} }

\def\xbf{ \mathbf{f} }
\def\xbfv#1{ \xbf_{#1} }
\def\xbfi{ \xbfv{i} }

\def\xbfz{ \xbfv{0} }
\def\xbfo{ \xbfv{1} }
\def\xbft{ \xbfv{2} }

\def\xbfvv#1#2{ \xbfv{#1,#2} }
\def\xbfaap{ \xbfv{\ina,\inap} }

\def\xbd{ \mathbf{d} }
\def\xbdv#1{ \xbd_{#1} }
\def\xbdi{ \xbdv{i} }

\def\xbh{ \mathbf{h} }
\def\xbhv#1{ \xbh_{#1} }
\def\xbhi{ \xbhv{i} }

\def\xbd{ \mathbf{d} }
\def\xbdv#1{ \xbd_{#1} }
\def\xbdi{ \xbdv{i} }

\def\xbhf{ \hat{\xbf} }
\def\xbhfv#1{ \xbhf_{#1} }
\def\xbhfi{ \xbhfv{i} }
\def\xbhfimo{ \xbhfv{i-1} }

\def\xbhh{ \hat{\xbh} }
\def\xbhhv#1{ \xbhh_{#1} }
\def\xbhhi{ \xbhhv{i} }
\def\xbhhz{ \xbhhv{0} }
\def\xbhho{ \xbhhv{1} }

\def\xbch{ \check{\xbh} }
\def\xbchv#1{ \xbch_{#1} }
\def\xbchi{ \xbchv{i} }

\def\xbhg{ \hat{\ybg} }
\def\xbhgv#1{ \xbhg_{#1} }
\def\xbhgi{ \xbhgv{i} }
\def\xbhgimo{ \xbhgv{i-1} }

\def\xbfAB{ \xbfv{\xbA\xbB} }
\def\xbfBC{ \xbfv{\xbB\xbC} }
\def\xbfAC{ \xbfv{\xbA\xbC} }

\def\xbtf{ \tilde{\xbf} }
\def\xbtfv#1{ \xbtf_{#1} }
\def\xbtfi{ \xbtfv{i} }
\def\xbtfio{ \xbtfv{i+1} }
\def\xbtfz{ \xbtfv{0} }

\def\xbtfm{ \xbtfv{m} }

\def\xbtfp{ \tilde{\xbf\p} }
\def\xbtfpv#1{ \xbtfp_{#1} }
\def\xbtfpi{ \xbtfpv{i} }

\def\xf{ f }
\def\xfv#1{ \xf_{#1} }
\def\xfi{ \xfv{i} }

\def\yf{ f }
\def\yfv#1{ \yf_{#1} }

\def\yfi{ \yfv{i} }
\def\yfio{ \yfv{i+1} }

\def\yg{ g }

\def\ybg{ \mathbf{\yg} }
\def\ybgv#1{ \ybg_{#1} }
\def\ybgA{ \ybgv{\xbA} }
\def\ybgB{ \ybgv{\xbB} }
\def\ybgi{ \ybgv{i} }
\def\ybgio{ \ybgv{i+1} }
\def\ybgimo{ \ybgv{i-1} }
\def\ybgz{ \ybgv{0} }
\def\ybgo{ \ybgv{1} }
\def\ybgt{ \ybgv{2} }
\def\ybgh{ \ybgv{3} }

\def\ybtg{ \tilde{\ybg} }
\def\ybtgv#1{ \ybtg_{#1} }
\def\ybtgi{ \ybtgv{i} }
\def\ybtgio{ \ybtgv{i+1} }

\def\ybtgz{ \ybtgv{0} }

\def\yh{ h }
\def\ybh{ \mathbf{\yh} }
\def\ybhv#1{ \ybh_{#1} }
\def\ybhi{ \ybhv{i} }
\def\ybhz{ \ybhv{0} }
\def\ybho{ \ybhv{1} }
\def\ybht{ \ybhv{2} }

\def\ybhp{ \ybh\p }
\def\ybhpv#1{ \ybhp_{#1} }
\def\ybhpi{ \ybhpv{i} }
\def\ybhpz{ \ybhpv{0} }
\def\ybhpo{ \ybhpv{1} }

\def\ybhtp{ \tilde{\ybh}\p }
\def\ybhtpv#1{ \ybhtp_{#1} }
\def\ybhtpi{ \ybhtpv{i} }

\def\xCh{ \mathop{\mathbf{Ch}} }
\def\xChv#1{ \xCh(#1) }
\def\xChA{ \xChv{\xctA} }

\def\xChm{ \xCh^- }
\def\xChmv#1{ \xChm(#1) }
\def\xChmA{ \xChmv{\xctA} }

\def\xKh{ \mathop{\mathbf{K}} }
\def\xKhv#1{ \xKh(#1) }
\def\xKhA{ \xKhv{\xctA} }

\def\xKhm{ \xKh^- }
\def\xKhmv#1{ \xKhm(#1) }
\def\xKhmA{ \xKhmv{\xctA} }

\def\isor{isomorphism order}

\def\ysiov#1{ \left|\, #1\, \right|_{\cong} }
\def\ysiorv#1{ |\, #1\, |_{\cong} }
\def\ysiobf{ \ysiov{\xbf} }
\def\ysiobfi{ \ysiov{\xbfi} }

\def\stblz{stabilizing}

\def\xtrn{ \mathbf{t} }
\def\xtrnvv#1#2{ \xtrn^-_{\leq #1} #2 }
\def\xtrniv#1{ \xtrnvv{i}{#1} }
\def\xtrnNv#1{ \xtrnvv{N}{#1} }
\def\xtrntmov#1{ \xtrnvv{2m-1}{#1} }

\def\xtrngvv#1#2{ \xtrn^-_{\geq #1} #2 }
\def\xtrngtmv#1{ \xtrngvv{2m}{#1} }

\def\chcpl{chain complex}
\def\chmp{chain morphism}
\def\chcpls{\chcpl es}

\def\trnc{truncation}

\def\trncv#1{$#1$-th \trnc}
\def\trnci{\trncv{i}}

\def\chdl{ \delta }
\def\chdlv#1{ \chdl_{#1} }
\def\chdlbf{ \chdlv{\xbf} }
\def\chdlbfi{ \chdlv{\xbfi} }
\def\chdlbfz{ \chdlv{\xbfz} }
\def\chdlbtfz{ \chdlv{\xbtfz} }
\def\chdlbtfm{ \chdlv{\xbtfm} }

\def\idl{ \iota }
\def\idlv#1{ \idl_{#1} }
\def\idlbf{ \idlv{\xbf} }
\def\idlbgi{ \idlv{\ybgi} }
\def\idlbgz{ \idlv{\ybgz} }
\def\idlbgo{ \idlv{\ybgo} }

\def\idlbhi{ \idlv{\ybhi} }

\def\Cch{Cauchy}
\def\xsyst{system}
\def\xdsyst{direct \xsyst}
\def\Csq{\Cch\ \xsyst}

\def\dlm{ \lim\limits_{\longrightarrow} }
\def\ilm{ \lim\limits_{\longleftarrow} }

\def\qord{$q$-order}

\def\xsupv#1{ \sup\left\{#1 \right\} }
\def\xinfv#1{ \inf\left\{#1 \right\} }
\def\xminv#1{ \min\left\{#1 \right\} }

\def\yordhb#1{ \bigl| \;#1 \;\bigr|_{\mathrm{h}} }
\def\yordh#1{ \left|\, #1\, \right|_{\mathrm{h}} }
\def\yordhr#1{ |\, #1\, |_{\mathrm{h} } }
\def\yordch#1{ \yordh{\Cnv{#1}} }
\def\yordq#1{ \left|\, #1\, \right|_q }

\def\zordh#1{ | #1 |_{\mathrm{h}} }
\def\zordch#1{ \zordh{\Cnv{#1}} }

\def\chsq{\xdsyst}

\def\chlm{ \lim\nolimits_{\xCh} }

\def\cmtr#1#2#3#4#5#6{
\xymatrix{
#4 \ar[r]_-{#1} \ar@/^1pc/[rr]^-{#3} &
#5 \ar[r]_-{#2} &
#6
},\qquad #3\hteqv #2\,#1
}

\def\atcmv#1#2{ #1\,#2 + #2\,#1 }

\def\xIdv#1{ \xId_{#1} }
\def\yIdAi{ \xIdv{\xbAi} }
\def\yIdA{ \xIdv{\xbA} }

\def\tchlm{chain limit}

\def\mmp{morphism}

\def\ybC{ \mathbf{C} }
\def\ybCv#1{ \ybC_{#1} }
\def\ybCz{ \ybCv{0} }
\def\ybCo{ \ybCv{1} }
\def\ybCt{ \ybCv{2} }
\def\ybCi{ \ybCv{i} }
\def\ybCimo{ \ybCv{i-1} }

\def\ycC{ \mathcal{C} }
\def\yctC{ \tilde{\ycC} }

\def\ybtC{ \tilde{\ybC} }
\def\ybtCv#1{ \ybtC_{#1} }
\def\ybtCz{ \ybtCv{0} }
\def\ybtCo{ \ybtCv{1} }

\def\ybtCi{ \ybtCv{i} }

\def\ybtCio{ \ybtCv{i+1} }

\def\mtcn{multi-cone}

\def\wbC{ \tilde{\xbC} }
\def\wbCvv#1#2{ \wbC_{#1,#2} }
\def\wbCnv#1{ \wbCvv{#1}{n} }
\def\wbCmn{ \wbCnv{m} }
\def\wbCzn{ \wbCnv{0} }
\def\wbCmnp{ \wbCmn\p }

\def\xrahv#1{ \xrightarrow{\;\;\;#1\;\;\;} }
\def\xratv#1{ \xrightarrow{\;\;#1\;\;} }
\def\xraov#1{ \xrightarrow{\;#1\;} }

\def\rxratv#1{ \xleftarrow{\;\;#1\;\;} }
\def\rxraov#1{ \xleftarrow{\;#1\;} }

\def\Kh{Khovanov homology}

\def\amap{ \symalg{-} }

\def\clckw{clockwise}
\def\cclckw{counterclockwise}

\def\dsym{ \bigtriangledown }
\def\dsymv#1{ #1^{\dsym} }
\def\dflpv#1{ #1^{\dflp} }
\def\dflppv#1{ #1^{\dflpp} }
\def\xtaud{ \dsymv{\xtau} }
\def\xtauf{ \dflpv{\xtau} }

\def\dsmd{ \vee }
\def\dsmdv#1{ #1^{\dsmd} }

\def\dflp{ \lozenge }
\def\zdflp{ \dflp }
\def\dflpe{ \dflp_e }
\def\dflpev#1{ #1^{\dflpe} }
\def\dflpp{ \dot{\dflp} }
\def\dflpv#1{ #1^{\dflp} }
\def\xtauf{ \dflpv{\xtau} }

\def\ddul{ \vee }
\def\ddulv#1{ #1^{\ddul} }

\def\ZZt{ \ZZ_2 }

\def\xIn{ Z }
\def\xyO{ O }
\def\xyH{ \mathcal{H} }

\def\xyHij{ \xyH_{i,j} }

\def\Shv#1{ \Sh_{#1} }

\def\Sho{ \Shv{1} }
\def\Shg{ \Shv{\xyg} }

\def\Shvv#1#2{ \Sh_{#1|#2} }

\def\Shgk{ \Shvv{\xyg}{\xyk} }

\def\xyk{ k }
\def\xyg{ g }
\def\Shcv#1{ \Sh_{|#1} }
\def\Shck{ \Shcv{\xyk} }
\def\Shctg{ \Shcv{2\xyg} }

\def\Stv#1{ \St_{#1} }
\def\Stn{ \Stv{n} }
\def\Stm{ \Stv{m} }

\def\toy{toy}

\def\rcr{ \xr_{\mathrm{cr}} }
\def\rcrvv#1#2{ \rcr(#2,#1) }

\def\rcrShgL{ \rcrvv{\Shg}{\xL} }
\def\rcrSotL{ \rcrvv{\Sot}{\xL} }

\def\hbds{handlebodies}

\def\tsswh{3-spheres with handles}
\def\tswgh{3-sphere with $\xyg$ handles}

\def\TQFT{TQFT}

\def\xstbl{stable}
\def\sthm{\xstbl\ homology}
\def\stWRTi{\xstbl\ \WRT\ invariant}
\def\sttoy{\xstbl\ \toy}
\def\sttoyt{\sttoy\ theory}
\def\tWRTt{\toy\ \WRT\ theory}
\def\sttWRTt{\xstbl\ \tWRTt}

\def\stty{\xstbl\ \toy}

\def\tstyWRTt{\stty\ \WRT\ theory}

\def\tpWRTt{topological \WRT\ theory}
\def\rstrd{restricted}
\def\strtWRTt{\xstbl\ \rstrd\ \tpWRTt}
\def\rTQFT{\rstrd\ \TQFT}
\def\orTQFT{oriented \rTQFT}

\def\cA{ \mathcal{A} }
\def\cAop{ \cA^{\xop} }

\def\cC{ \mathcal{C} }

\def\cH{ \mathcal{H} }
\def\cHv{ \cH^{\vee} }
\def\cHStv#1{ \cH(\Stv{#1}) }
\def\cHvStv#1{ \cHv(\Stv{#1}) }
\def\cHStn{ \cHStv{n} }
\def\cHvStn{ \cHvStv{n} }

\def\bmdl{bimodule}

\def\Kbr{Kauffman bracket}
\def\xprng{ \llcorner }

\def\otQq{ \otimes_{\Qq} }

\def\bfn{ \mathbf{n} }

\def\xdummy{ - }

\def\stmp{state map}
\def\obmp{object map}
\def\sphtng{spherical tangle}

\def\xbrb#1{ \bigl( #1 \bigr) }

\def\rF{ \mathbb{F} }
\def\xrS{ \mathcal{S} }
\def\xrM{ \mathcal{M} }

\def\adS{ S }

\def\inb{`in'}
\def\outb{`out'}
\def\xinv#1{ (#1)_{\mathrm{in}} }
\def\xoutv#1{ (#1)_{\mathrm{out}} }

\def\xstmp#1{ \langle #1 \rangle }
\def\xobmp#1{ \langle \!\langle #1 \rangle\!\rangle}

\def\xstmpd{ \xstmp{\xdummy} }
\def\xobmpd{ \xobmp{\xdummy} }

\def\xobmpv#1#2{ \xobmp{#1,#2} }

\def\adcbs{admissible connected boundaries}
\def\adcm{admissible  manifold}

\def\xwk{weak}
\def\xalc{algebraic}
\def\walg{\xwk\ \xalc}
\def\alcat{\xalc\ categorical}

\def\xalg{algebra}

\def\xring{algebra}

\def\idfn{identification}

\def\zzA{ A }
\def\zzB{ B }

\def\finv#1{ \bar{#1} }
\def\wfinv#1{ \overline{#1} }
\def\finvm{ \finv{\;\;} }
\def\xxv{ v }

\def\sprj{semi-projective}

\def\HmTlmnlot{ \Hom_{\dTLtlmn}\xlrb{\dlamo,\dlamtw} }

\def\rsP{ \mathbf{P} }
\def\rhP{ \hat{\rsP} }
\def\rsPv#1{ \rsP_{#1} }
\def\rsPn{ \rsPv{n} }
\def\rsPnc{ \rsPn^{\sharp} }
\def\rsPm{ \rsPv{m} }
\def\rhPv#1{ \rhP_{#1} }

\def\rhPs{ \rhPv{\ast} }

\def\prjPvv#1#2{ P_{#1,#2} }

\def\prjPba{ \prjPvv{\xbet}{\xal} }

\def\xlN{ N }

\def\idia{ a }
\def\idib{ b }

\def\inoa{ a }

\def\ide{ e }
\def\idev#1{ \ide_{#1} }
\def\ideo{ \idev{1} }
\def\ideN{ \idev{\xlN} }
\def\idea{ \idev{\idia} }
\def\ideb{ \idev{\idib} }

\def\idP{ P }
\def\idPv#1{ \idP_{#1} }
\def\idPa{ \idPv{\idia} }
\def\idPb{ \idPv{\idib} }

\def\idS{ S }
\def\idSv#1{ \idS_{#1} }
\def\idSa{ \idSv{\idia} }

\def\irred{irreducible}

\def\obA{ A }
\def\obB{ B }
\def\obC{ C }

\def\obE{ E }
\def\obEv#1{ \obE_{#1} }
\def\obEo{ \obEv{1} }
\def\obEN{ \obEv{\xlN} }
\def\obEa{ \obEv{\inoa} }

\def\smaoN{ \sum_{1\leq\inoa\leq\xlN} }
\def\bpaoN{ \bigoplus_{1\leq\inoa\leq\xlN} }

\def\qpord{$q^+$-order}

\def\upr{universal projective resolution}

\def\Qgmod{\cvgmod{\IQ}}
\def\QgmodQ{ (\Qgmod)^{Q} }
\def\Qgmodpq{ (\Qgmod)^{-/+} }

\def\cHdbM{ \cH(\del\bbM) }
\def\cHdbMp{ \cH(\del\bbM\p) }
\def\cHvdbM{ \cHv(\del\bbM) }

\def\cAdbM{ \cA(\del\bbM) }
\def\cAopdbM{ \cAop(\del\bbM) }

\def\cCdbM{ \cC(\del\bbM) }
\def\cCdbMp{ \cC(\del\bbM\p) }
\def\Loti{ \overset{\mathrm{L}}{\otimes} }

\def\xglu{ \mathrm{gluing} }

\def\chq{ \chi_q }
\def\dmq{ \dim_q }

\def\empc{ \cC(\varnothing) }

\def\HKhbd{ \HKhb(\xdummy) }

\def\xcnv{convenient}

\def\psbaf{ \psymalg{\xbet\tcmp\xalf} }

\def\blpbv#1#2{ \xlrb{#1,#2}_{\Tor} }

\def\xbAf{ \dflpv{\xbA} }

\def\TorR{ \Tor_{\aR} }

\def\fKDbM{ \xbM }

\def\sprbs#1{ \Bigl( #1 \Bigr) }

\def\Ktau{\zKsymbim{\xtau} }
\def\Ktauo{\zKsymbim{\xtauo} }
\def\Ktaut{\zKsymbim{\xtaut} }

\def\funcm#1{ \hat{#1} }
\def\htau{ \funcm{\xtau} }

\def\hef{ f }
\def\hefv#1{ \hef_{#1} }
\def\hefa{ \hefv{\xal} }
\def\hefao{ \hefv{\xalo} }
\def\hefat{ \hefv{\xalt} }

\def\smmuf{ \mu_{\xphaot} }

\def\xcob{ \mathrm{Cob} }
\def\coblot{ \xcob(\xlamo,\xlamtw) }

\def\bldl{ \left\langle\!\left\langle }
\def\bldr{ \right\rangle\!\right\rangle }

\def\sbldl{ \langle\!\langle }
\def\sbldr{ \rangle\!\rangle }

\def\zsymcat#1{ \sbldl #1 \sbldr }

\def\symcat#1{ \bldl #1 \bldr }
\def\symbcat#1{ \bldl #1 \bldr }
\def\Bsymcat#1{ \Bigl<\!\!\Bigl< #1 \Bigr>\!\!\Bigr> }
\def\bsymcat#1{ \bigl<\!\bigl< #1 \bigr>\!\bigr> }

\def\Dsymcat#1{ \symcat{#1}_{\xdrv} }
\def\zDsymcat#1{ \zsymcat{#1}_{\xdrv} }
\def\Ksymcat#1{ \symcat{#1}_{\xhKv} }

\def\Kbsymbim#1{ \bsymcat{#1}_{\xhKv} }

\def\xspsh{ \mathrm{s} }
\def\ctss#1{ #1^{\xspsh } }
\def\ctsos#1{ #1^{-\xspsh } }
\def\ctsh#1{ \spcc{#1} }
\def\symcats#1{ \ctss{\symcat{#1}} }

\def\bsymcats#1{ \ctss{\bsymcat{#1}} }

\def\bsymcatos#1{ \ctsos{\bsymcat{#1}} }

\def\ctpss#1{ \spcc{#1}\p }
\def\symcatps#1{ \ctpss{\symcat{#1}} }

\def\bsymcatps#1{ \ctpss{\bsymcat{#1}} }
\def\bsymcatsh#1{ \ctsh{\bsymcat{#1}} }

\def\Ksymbim#1{ \symcat{#1}_{\xhKv} }
\def\Dsymbim#1{ \symcat{#1}_{\xdrv} }

\def\zKsymbim#1{ \zsymcat{#1}_{\xhKv} }
\def\zDsymbim#1{ \zsymcat{#1}_{\xdrv} }

\def\psymalg#1{ \langle #1 \rangle }
\def\symalg#1{ \left\langle #1 \right\rangle }
\def\Bsymalg#1{ \Bigl< #1 \Bigr> }
\def\bsymalg#1{ \bigl< #1 \bigr> }

\def\sphsymalg#1{ \symalg{#1,\StI} }
\def\Sotsymalg#1{ \symalg{#1,\Sot} }
\def\Btsymalg#1{ \symalg{#1,\bbB^3} }

\def\dTL{ \mathsf{TL} }
\def\cTL{ \mathrm{TL} }
\def\rTL{ \mathit{TL} }

\def\cfQTLct{ \cfQTL^{\ftct} }
\def\cfQTLctvv#1#2{ \cfQTLct_{#1,#2} }

\def\cfQTLcttntn{ \cfQTLctvv{2n}{2n} }
\def\cfQTLctv#1{ \cfQTLct_{#1} }
\def\cfQTLcttn{ \cfQTLctv{2n} }
\def\cfQTLcttmtn{ \cfQTLctvv{2m}{2n} }

\def\cTLctp{ \cTL^{\ftct,+} }
\def\cTLctpvv#1#2{ \cTLctp_{#1,#2} }

\def\cTLctpv#1{ \cTLctp_{#1} }
\def\cTLctptn{ \cTLctpv{2n} }
\def\cTLctptmtn{ \cTLctpvv{2m}{2n} }

\def\acapni{ \acapvv{n}{i} }

\def\ccapni{ \ccapvv{n}{i} }

\def\acupni{ \acupvv{n}{i} }

\def\ccupni{ \ccupvv{n}{i} }

\def\xL{ L }
\def\dL{ \symcat{\xL} }

\def\dLi{ \dLv{i} }
\def\dLimo{ \dLv{i-1} }

\def\xnot{ \xlam_\varnothing }
\def\cnot{ \symcat{\xnot} }

\def\mpalg{ \symalg{\xdummy} }
\def\mpcat{ \symcat{\xdummy} }

\def\mpKbim{ \zKsymbim{\xdummy} }
\def\mpDcat{ \Dsymcat{\xdummy} }

\def\dtaus{ \symcats{\xtau} }

\def\dLi{ \xC_i }
\def\dLimo{ \xC_{i-1} }

\def\lcir{
\xygraph{
!{0;/r1.5pc/:}
!{\vcap-}
!{\vcap}
}
}

\def\xcrsp{
\xygraph{
!{0;/r1.5pc/:}
[u(0.5)]
!{\xoverv}
}
}

\def\xpver{
\xygraph{
!{0;/r1.5pc/:}
[u(0.5)]
!{\xunoverv}
}
}

\def\xphor{
\xygraph{
!{0;/r1.5pc/:}
[u(0.5)]
!{\xunoverh}
}
}

\def\xvfro{
\xygraph{
!{0;/r1.5pc/:}
[u(0.5)]
!{\xcapv@(0)}
[u(1.5)]
[u(0.45)r(0.23)]
[d(1.5)]
*{\symfr\;\scriptstyle{1}}
}
}

\def\xvert{
\xygraph{
!{0;/r1.5pc/:}
[u(0.5)]
!{\xcapv@(0)}
}
}

\def\gbrv#1#2#3{
\xygraph{
!{0;/r1pc/:}
[u(0.5)]
!{\hcross}
!{\hcross}
[l(1)d(0.25)]
*{\scriptstyle{\vdots}}
[u(0.75)]
*{\scriptstyle{#1}}
[d(1)r(#3)]
*{\scriptstyle{#2}}
}
}

\def\gobrv#1#2#3{\,
\xygraph{
!{0;/r1pc/:}
[u(0.5)]
!{\hcrossneg}
!{\hcrossneg}
[l(1)d(0.25)]
*{\scriptstyle{\vdots}}
[u(0.75)]
*{\scriptstyle{#1}}
[d(1)r(#3)]
*{\scriptstyle{#2}}
}
}

\def\gobrlv#1#2#3{
\xygraph{
!{0;/r1pc/:}
[u(0.5)]
!{\hcrossneg}
!{\hcrossneg}
[d(1.5)l(2)]
!{\xcaph[2]@(0)}
[u(1.5)r(1)]
[l(1)d(0.25)]
*{\scriptstyle{\vdots}}
[u(0.75)]
*{\scriptstyle{#1}}
[d(1)r(#3)]
*{\scriptstyle{#2}}
}
}

\def\dtwst{
\xygraph{
!{0;/r1pc/:}
[u(0.5)]
!{\hcrossneg}
!{\hcrossneg}
}
}

\def\dccp{
\xygraph{
!{0;/r1pc/:}
[u(0.5)]
!{\xcapv}
[u(1)r(1)]
!{\xcapv[-1]}
}
}

\def\dpar{
\xygraph{
!{0;/r1pc/:}
[u(0.5)]
!{\xunoverh}
}
}

\def\gbrmv#1#2{ \gbrv{m}{#1}{#2} }
\def\gbrov#1#2{ \gbrv{1}{#1}{#2} }
\def\gbrmn{ \gbrmv{n}{1.5} }
\def\gbrmtn{ \gbrmv{2n}{1.5} }

\def\gbrmnmt{ \gbrmv{n-2}{2} }

\def\gobrmv#1#2{ \gobrv{m}{#1}{#2} }
\def\gobrov#1#2{ \gobrv{1}{#1}{#2} }
\def\gobrmn{ \gobrmv{n}{1.5} }
\def\gobrmtn{ \gobrmv{2n}{1.5} }
\def\gobrotn{ \gobrov{2n}{1.5} }
\def\cobrotn{ \bsymcat{\gobrotn} }

\def\abrmn{ \bsymalg{\;\gbrmn} }
\def\cbrmns{ \bsymcats{\;\gbrmn} }

\def\aobrmn{ \bsymalg{\;\gobrmn} }
\def\cobrmns{ \bsymcatos{\;\gobrmn} }

\def\cbrmnmtps{ \bsymcatps{\;\gbrmnmt} }
\def\cbrmnmts{ \bsymcats{\;\gbrmnmt} }

\def\gbrov#1#2{ \gbrv{1}{#1}{#2} }
\def\gbron{ \gbrov{n}{1.5} }
\def\gbronmt{ \gbrov{n-2}{2} }

\def\gobrov#1#2{ \gobrv{1}{#1}{#2} }
\def\gobron{ \gobrov{n}{1.5} }

\def\gobrotl{ \gobrov{\ztl}{2} }

\def\gobronmo{ \gobrov{n-1}{2} }

\def\gobrolv#1#2{ \gobrlv{1}{#1}{#2} }

\def\gobrolnmo{ \gobrolv{n-1}{2} }

\def\abron{ \bsymalg{\;\gbron} }
\def\cbrons{ \bsymcats{\;\gbron} }

\def\cobrons{ \bsymcatos{\;\gobron} }

\def\cobronsh{ \bsymcatsh{\gobron} }

\def\cobrmtnsh{ \bsymcatsh{\gobrmtn} }
\def\cobrmtn{ \bsymcat{\gobrmtn} }

\def\cobrmotnsh{ \bsymcatsh{\gobrmotn} }

\def\gbrmov#1#2{ \gbrv{m+1}{#1}{#2} }
\def\gbrmon{ \gbrmov{n}{1.5} }

\def\gobrmov#1#2{ \gobrv{m+1}{#1}{#2} }
\def\gobrmon{ \gobrmov{n}{1.5} }
\def\gobrmotn{ \gobrmov{2n}{1.75} }

\def\cbrmons{ \bsymcats{\;\gbrmon} }

\def\cobrmons{ \bsymcatos{\;\gobrmon} }

\def\abrmtn{ \bsymalg{\gbrmtn} }

\def\gcapv#1#2#3{
\xygraph{
!{0;/r1pc/:}
[u(0.5)]
!{\hcap-}
[u(0.5)]
*{\scriptstyle{#2}}
[d(1)r(#3)]
*{\scriptstyle{#1}}
}
}

\def\gccapv#1#2#3{
\xygraph{
!{0;/r1pc/:}
[u(0.5)]
!{\hcap-}
[d(0.25)]
!{\hcap[-0.5]}
[u(0.25)]
[u(0.5)]
*{\scriptstyle{#2}}
[d(1)r(#3)]
*{\scriptstyle{#1}}
}
}

\def\gccaplv#1#2#3{
\xygraph{
!{0;/r1pc/:}
[u(0.5)]
!{\hcap-}
[d(0.25)]
!{\hcap[-0.5]}
[u(0.25)]
[d(1.5)l(0.5)]
!{\xcaph[0.5]@(0)}
[u(1.5)l(0.5)]
[u(0.5)]
*{\scriptstyle{#2}}
[d(1)r(#3)]
*{\scriptstyle{#1}}
}
}

\def\gcapvi#1#2{ \gcapv{#1}{i}{#2} }
\def\gcapni{ \gcapvi{n}{0.5} }

\def\acapni{ \bsymalg{\gcapni} }
\def\ccapni{ \bsymcat{\!\gcapni} }

\def\gcapvI#1#2{ \gccapv{#1}{\stI}{#2} }
\def\gcapvJ#1#2{ \gccapv{#1}{\stJ}{#2} }

\def\gcapvJp#1#2{ \gccapv{#1}{\stJp}{#2} }

\def\gcapnI{ \gcapvI{n}{0.75} }

\def\gcapmJ{ \gcapvJ{m}{0.75} }
\def\gcapnJ{ \gcapvJ{n}{0.75} }
\def\gcapnJp{ \gcapvJp{n}{0.75} }

\def\gcapnmoJ{ \gcapvJ{n-1}{1.25} }

\def\gcaptnJ{ \gcapvJ{2n}{0.75} }

\def\gcaplnmoJ{ \gccaplv{n-1}{\stJ}{1.25} }
\def\ccapnI{ \bsymcat{ \gcapnI } }

\def\acapnJ{ \bsymalg{\gcapnJ} }

\def\gcapfv#1#2#3#4{
\xygraph{
!{0;/r1pc/:}
[u(0.5)]
!{\hcap-}
[u(0.5)]
*{\scriptstyle{#2}}
[l(0.5)d(1)]
*{\circ}
[l(0.75)]
*{\scriptstyle{#3}}
[r(#4)]
*{\scriptstyle{#1}}
}
}

\def\gcapnit{ \gcapfv{n}{i}{2}{2} }
\def\gcapnik{ \gcapfv{n}{i}{k}{2} }

\def\gcupv#1#2#3{
\xygraph{
!{0;/r1pc/:}
[u(0.5)]
!{\hcap}
[u(0.5)]
*{\scriptstyle{#2}}
[d(1)r(#3)]
*{\scriptstyle{#1}}
}
}

\def\gccupv#1#2#3{
\xygraph{
!{0;/r1pc/:}
[u(0.5)]
!{\hcap}
[d(0.25)]
!{\hcap[0.5]}
[u(0.25)]
[u(0.5)]
*{\scriptstyle{#2}}
[d(1)r(#3)]
*{\scriptstyle{#1}}
}
}

\def\gccuplv#1#2#3{
\xygraph{
!{0;/r1pc/:}
[u(0.5)]
!{\hcap}
[d(0.25)]
!{\hcap[0.5]}
[u(0.25)]
[d(1.5)l(0)]
!{\xcaph[0.5]@(0)}
[u(1.5)l(1)]
[u(0.5)]
*{\scriptstyle{#2}}
[d(1)r(#3)]
*{\scriptstyle{#1}}
}
}

\def\gcupvi#1#2{ \gcupv{#1}{i}{#2} }
\def\gcupni{ \gcupvi{n}{-0.5} }

\def\acupni{ \bsymalg{\gcupni} }
\def\ccupni{ \bsymcat{\gcupni\!} }

\def\gcupvI#1#2{ \gccupv{#1}{\stI}{#2} }

\def\gcupnI{ \gcupvI{n}{-0.5} }
\def\gcupnmoI{ \gcupvI{n-1}{-1} }

\def\gcuplnmoI{ \gccuplv{n-1}{\stI}{-1} }
\def\gcuptnI{ \gcupvI{2n}{-0.5} }

\def\ccupnI{ \bsymcat{ \gcupnI } }

\def\gcupvIp#1#2{ \gccupv{#1}{\stI\p}{#2} }
\def\gcupnIp{ \gcupvIp{n}{-0.5} }

\def\ccupnIp{ \bsymcat{\gcupnIp} }

\def\gidbrv#1#2{
\xygraph{
!{0;/r1pc/:(0,-1.75)}
[u(2.5)]
!{\xcaph@(0)}
[d(1.5)l(1)]
!{\xcaph@(0)}
[u(1)l(0.4)]
*{\vdots}
[d(0.3)r(#2)]
*{\scriptstyle{#1}}
}
}

\def\gidbrn{ \gidbrv{n}{1} }
\def\gidbrtn{ \gidbrv{2n}{1.25} }

\def\gidbrto{ \gidbrv{\ztl+1}{1.5} }
\def\cidbrto{ \bsymcat{\,\gidbrto} }
\def\gidbrnmtd{ \gidbrv{n-2d}{2} }
\def\cidbrn{ \bsymcat{\,\gidbrn} }
\def\tstgrt{
\xygraph{
!{0;/r1pc/:}
!{\xcaph@(0)}
[d(1.5)l(1)]
!{\xcaph@(0)}
[u(1)l(0.4)]
*{\vdots}
[d(0.3)r(1)]
*{\scriptstyle{n}}
}
}

\def\testgr{
\underbrace{
\xygraph{
!{0;/r1.5pc/:}
[] *+[F]{\;\;\scriptstyle{n}\;\;}
[l(0.5)d(0.4)]
!{\xcapv[0.5]@(0)}
[r(1)u(1)]
!{\xcapv[0.5]@(0)}
[u(2.3)]
!{\xcapv[0.5]@(0)}
[l(1)u(1)]
!{\xcapv[0.5]@(0)}
}
}_{n}
}

\def\testgr1{
\xygraph{
!{0;/r1.5pc/:}
!{\vtwist}
!{\vtwist}
}
}

\def\testgr2{
\xygraph{
!{0;/r1pc/:}
!{\hcrossneg}
!{\hcrossneg}
}
}

\def\nidm{ \noindent
*************************************************
}

\def\sgrp{set of morphisms}
\def\rgh{rough}
\def\rtng{\rgh\ tangle}

\def\Grtg{Grothendieck group}

\def\Grtn{Grothendieck}

\def\xpro{projected}
\def\xprot{\xpro\ tangle}

\def\crmt{crossingless matching}
\def\xcrmt{\crmt}
\def\xcrmtv#1{crossingless $#1$-matching}
\def\xcrmtn{\xcrmtv{n}}

\def\fcrmt{flipped \crmt}

\def\tcmp{ \circ }

\def\Shcl{$\Sh$-closure}

\def\Zgrdd{$\ZZ$-graded}
\def\Zgrd{$\ZZ$-grading}
\def\Zdgr{$\ZZ$-degree}

\def\Ztgrdd{$\ZZt$-graded}
\def\Ztgrd{$\ZZt$-grading}
\def\Ztdgr{$\ZZt$-degree}

\def\tcomp{total complex}

\def\xcnst{constituent}
\def\xcnstt{\xcnst\ tangle}
\def\xcnsto{\xcnst\ object}
\def\xcnstc{\xcnst\ complex}

\def\flt{flat}
\def\fltc{\flt\ cobordism}
\def\rdcd{reduced}
\def\rfltc{\rdcd\ \fltc}

\def\elmt{elementary}

\def\efltc{\elmt\ \fltc}

\def\elcb{\elmt\ cobordism}

\def\xmult{$x$-multiplication}

\def\xmov{movie}
\def\cmov{cobordism \xmov}
\def\Rmov{Reidemeister \xmov}

\def\qim{quasi-isomorphism}

\def\qtriv{quasi-trivial}

\def\bbS{ \mathbb{S} }
\def\bbB{ \mathbb{B} }
\def\Bt{ \bbB^3 }
\def\bbM{ M }
\def\bbI{ \mathbb{I} }
\def\So{ \bbS^1 }
\def\St{ \bbS^2 }
\def\Sot{ \St\times\So }
\def\Sh{ \bbS^3 }
\def\bbB{ \mathbb{B} }
\def\Bh{ \bbB^3 }

\def\xr{ r }
\def\SU{ \mathrm{SU} }
\def\SUt{ \SU(2) }

\def\xL{ L }
\def\xM{ M }

\def\fWRT{Witten-Reshetikhin-Turaev}

\def\WRT{WRT}
\def\WRTp{\WRT\ polynomial}

\def\isotp{ \approx }
\def\homeom{ \cong }

\def\xmrd{meridian}

\def\Zvvv#1#2#3{ Z_{#1}(#2,#3) }
\def\Zrvv#1#2{ \Zvvv{\xr}{#1}{#2} }
\def\ZrLM{ \Zrvv{\xL}{\xM} }
\def\ZrLSot{ \Zrvv{\xL}{\Sot} }

\def\ZrLShg{ \Zrvv{\xL}{\Shg} }
\def\ZrtSot{ \Zrvv{\xtau}{\Sot} }
\def\Zrxv#1{ Z_\xr(#1) }

\def\ZrShg{ \Zrxv{\Shg} }
\def\ZrSot{ \Zrxv{\Sot} }

\def\Zt{ \tilde{Z} }
\def\Ztvvv#1#2#3{ \Zt_{#1}(#2,#3) }
\def\Ztrvv#1#2{ \Ztvvv{\xr}{#1}{#2} }

\def\ZtrLShg{ \Ztrvv{\xL}{\Shg} }

\def\Jnqvv#1#2{ J(#1,#2) }
\def\JnqShv#1{ \Jnqvv{#1}{\Sh} }
\def\JnqSotv#1{ \Jnqvv{#1}{\Sot} }
\def\JnLq{ \JnqShv{\xL} }
\def\JnLotq{ \JnqSotv{\xL} }

\def\JnSgLq{ \Jnqvv{\xL}{\Shg} }

\def\lSh#1{ #1;\Sh }
\def\lSot#1{ #1;\Sot }

\def\pvv#1#2{ (#1,#2) }
\def\pLM{ \pvv{\xL}{\xM} }

\def\ttngvv#1#2{$(#1,#2)$-tangle}
\def\ttngmn{\ttngvv{m}{n}}
\def\ttngnm{\ttngvv{n}{m}}
\def\ttngnn{\ttngvv{n}{n}}

\def\ttngzz{\ttngvv{0}{0}}

\def\ttngtntn{\ttngvv{2n}{2n}}

\def\shlf{ \tfrac{1}{2} }

\def\xinter{ [0,1] }

\def\clv#1#2{ (#1;#2) }
\def\clShv#1{ \clv{#1}{\Sh} }

\def\clSotv#1{ \clv{#1}{\Sot} }
\def\cltSot{ \clSotv{\xtau} }

\def\clstv#1{ \clv{#1}{\ast} }

\def\xlv#1#2{ #1;#2 }
\def\xlShv#1{ \xlv{#1}{\Sh} }

\def\xtauv#1{ \xtau_{#1} }
\def\xtauo{ \xtauv{1} }
\def\xtaut{ \xtauv{2} }

\def\xgam{ \gamma }
\def\xgamv#1{ \xgam_{#1} }

\def\xgami{ \xgamv{i} }

\def\xgamn{ \xgamv{n} }

\def\flt{planar}

\def\qi{ q^{-1} }

\def\qepr{ q = \exp(\pi i/\xr) }

\def\brb{ B }
\def\xlam{ G }
\def\xtau{ T }

\def\dLio{ C_{i+1} }

\def\xaH{ H }
\def\xaHv#1{ \xaH_{#1} }
\def\xaHn{ \xaHv{n} }
\def\xaHm{ \xaHv{m} }
\def\xaHop{ \xaH^{\xop} }
\def\xaHopv#1{ \xaHop_{#1} }
\def\xaHopm{ \xaHopv{m} }
\def\xaHopn{ \xaHopv{n} }

\def\xaHdv#1#2{ \xaH_{#1,#2} }
\def\xaHdnm{ \xaHdv{n}{m} }
\def\xaHdmn{ \xaHdv{m}{n} }

\def\xaHe{ \xaH^{\mathrm{e}} }
\def\xaHev#1{ \xaHe_{#1} }
\def\xaHne{ \xaHev{n} }

\def\prj{ \mathrm{pr} }

\def\xbnd{ \mathrm{b} }
\def\xdrv{ \mathsf{D} }
\def\xhKv{ \mathsf{K} }
\def\xdrb{ \xdrv^{\mathrm{\xbnd}} }

\def\xhKb{ \xhKv^{\mathrm{\xbnd}} }

\def\xhKmp{ \xhKv^{-}_{\prj} }
\def\xhKmq{ \xhKv^{+}_{\prj} }
\def\xhKmpq{ \xhKpq_{\prj} }
\def\xhKmpv#1{ \xhKmp(#1) }
\def\xhKmpqv#1{ \xhKmpq(#1) }
\def\xhKbp{ \xhKv^{\mathrm{\xbnd}}_{\prj} }
\def\xhKbpv#1{ \xhKbp(#1) }
\def\xhKbsp{ \xhKb_{\mathrm{sp} } }
\def\xhKbv#1{ \xhKb( #1 ) }
\def\xhKm{ \xhKv^- }
\def\xhKpq{ \xhKv^{-/+} }

\def\cvgmod#1{ #1-\mgmod }

\def\cvxpmod#1{ \cvpmod{#1} }
\def\cvxpmod#1{ \text{$#1$-$\mathsf{pr}$} }

\def\Qgmod{ \cvgmod{\IQ} }
\def\Qgmodv#1{ (\Qgmod)^{#1} }
\def\Qgmodp{ \Qgmodv{-} }
\def\Qgmodq{ \Qgmodv{+} }

\def\Dbmd#1{ \xdrb(#1) }

\def\Dbgmd#1{ \xdrb(#1) }
\def\bDbgmd#1{ \xdrb\xlrb{#1} }

\def\xKbgmd#1{ \xhKb(\cvgmod{#1}) }
\def\Kbgmd#1{ \xhKb(#1) }

\def\DbmdR{ \Dbmd{\aR} }
\def\DbmdRot{ \Dbmd{\aRotd} }
\def\DbmdRe{ \Dbmd{\aRe} }
\def\xhKmpR{ \xhKmpv{\aR} }

\def\xhKmpqR{ \xhKmpq(\aR) }
\def\xhKmpRot{ \xhKmpv{\aRotd} }
\def\DbmdHne{ \Dbmd{\xaHne} }
\def\xhKmpRe{ \xhKmpv{\aRe} }
\def\xhKmpHdv#1#2{ \xhKmpv{\xaHdv{#1}{#2}} }
\def\xhKmpHdnm{ \xhKmpHdv{n}{m} }
\def\xhKmpqHdv#1#2{ \xhKmpqv{\xaHdv{#1}{#2}} }
\def\xhKmpqHdnm{ \xhKmpqHdv{n}{m} }

\def\xhKbpR{ \xhKbpv{\aR} }
\def\xhKbpHn{ \xhKbpv{\xaHn} }

\def\xhKbR{ \xhKbv{\aR} }
\def\xhKbRe{ \xhKbv{\aRe} }
\def\xhKbRot{ \xhKbv{\aRotd} }

\def\DbgHdv#1#2{ \Dbgmd{\xaHdv{#1}{#2} } }
\def\DbgHev#1{ \Dbgmd{\xaHev{#1} } }

\def\DbgRe{ \Dbgmd{\aRe} }

\def\DbgHv#1{ \Dbgmd{\xaHv{#1}} }
\def\DbgHn{ \DbgHv{n} }

\def\KbgHdv#1#2{ \Kbgmd{\xaHdv{#1}{#2} } }
\def\KbgHv#1{ \Kbgmd{\xaHv{#1}} }
\def\KbgHn{ \KbgHv{n} }

\def\xKbgHdv#1#2{ \xKbgmd{\xaHdv{#1}{#2} } }

\def\KbspHdv#1#2{ \xhKbsp(\xaHdv{#1}{#2}) }
\def\KbspHdnm{ \KbspHdv{n}{m} }

\def\KbspRot{ \xhKbsp(\aRotd) }

\def\DbgHdnm{ \DbgHdv{n}{m} }
\def\DbgHen{ \DbgHev{n} }

\def\xKbgHdnm{ \xKbgHdv{n}{m} }

\def\KbgHdnm{ \KbgHdv{n}{m} }

\def\KbgHdmn{ \KbgHdv{m}{n} }

\def\spf{ \mathcal{F} }
\def\spfv#1{ \spf_{#1} }
\def\spfn{ \spfv{H} }
\def\spfK{ \spfv{\xhKv} }

\def\spfKD{ \spfv{\xhKv\xdrv} }

\def\spP{ \mathcal{P} }
\def\spPr{ \spP }

\def\spPv#1{ \spP_{#1} }
\def\spPK{ \spPv{\xhKv} }
\def\spPTL{ \spPv{\dTL} }
\def\spPsp{ \spPv{\mathrm{sp} } }

\def\spPpl{ \spP^+ }

\def\xalv#1{ \xal_{#1} }
\def\xalo{ \xalv{1} }
\def\xalt{ \xalv{2} }

\def\xald{ \dsymv{\xal} }
\def\xalf{ \dflpv{\xal} }

\def\dal{ \symcat{\xal} }
\def\dalo{ \symcat{\xalo} }
\def\dalt{ \symcat{\xalt} }
\def\dbet{ \symcat{\xbet} }
\def\dal{ \symcat{\xal} }

\def\Kal{ \Ksymcat{\xal} }
\def\Kbet{ \Ksymcat{\xbet} }

\def\dald{ \zsymcat{\xald} }
\def\dbetd{ \zsymcat{\xbetd} }

\def\xbet{ \beta }
\def\xbetd{ \dsymv{\xbet} }

\def\aR{ R }
\def\aRv#1{ \aR_{#1} }
\def\aRo{ \aRv{1} }
\def\aRt{ \aRv{2} }
\def\aRoop{ \aRo^\xop }
\def\aRtop{ \aRt^\xop }
\def\aRop{ \aR^\xop }
\def\aRd{ \aR\otimes\aRop }
\def\aRotd{ \aRo\otimes\aRtop }

\def\aRrv#1{ \aR_{|#1} }

\def\aRrj{ \aRrv{j} }
\def\aRrz{ \aRrv{0} }
\def\aRrp{ \aRrv{>0} }

\def\aM{ M }

\def\aN{ N }

\def\xbM{ \mathbf{\aM} }
\def\xbN{ \mathbf{\aN} }

\def\aRe{ \aR^{\mathrm{e}} }

\def\Hhom{Hochschild homology}

\def\Hchom{Hochschild cohomology}

\def\xHH{ {\mathrm{HH} } }
\def\xHHv#1{ \xHH_{#1} }

\def\xHHmi{ \xHHv{-i} }
\def\xHHuv#1{ \xHH^{#1} }
\def\xHHui{ \xHHuv{i} }
\def\xHHl{ \xHHv{\bullet} }
\def\xHHu{ \xHHuv{\bullet} }

\def\xHHlv#1{ \xHHl(#1) }

\def\chq{ \chi_q }

\def\dal{ \symcat{\xal} }
\def\xalp{ \xal\p }

\def\xalpz{ \xalp }

\def\xlam{ \lambda }
\def\xtau{ \tau }
\def\xsg{ \sigma }
\def\xbt{ \beta }
\def\brb{ \xbt }
\def\xal{ \alpha }
\def\xgam{ \gamma }

\def\qie{ \simeq }

\def\zd{ d }
\def\zdv#1{ \zd_{#1} }
\def\zdl{ \zdv{\xlam} }
\def\zdlp{ \zdl\p }
\def\zt{ t }
\def\ztv#1{ \zt_{#1} }
\def\ztl{ \ztv{\xlam} }
\def\ztlp{ \ztv{\xlamp} }

\def\thr{through}
\def\thrd{\thr\ degree}

\def\tct{split}

\def\xunt{ \xaHn }

\def\cstI{ \mathcal{I} }
\def\cstIvv#1#2{ \cstI_{#1,#2} }

\def\cstInt{ \cstIvv{n}{\zt} }
\def\cstInm{ \cstIvv{n}{m} }

\def\cstIntl{ \cstIvv{n}{\ztl} }

\def\cstImtl{ \cstIvv{m}{\ztl} }

\def\stA{ \mathcal{A}_{\angle} }
\def\stAs#1#2{ \stA\tgrsshv{#1}{#2} }
\def\stAxv#1#2{ \stA(#1,#2) }
\def\stAxvv#1#2#3{ \stA(#1,#2;#3) }
\def\stAxtn{ \stAxv{\zt}{n} }
\def\stAxtln{ \stAxv{\ztl}{n} }
\def\stAxtlptn{ \stAxv{\ztlp}{2n} }
\def\stAxtlnmo{ \stAxv{\ztl}{n-1} }
\def\stAxtlnm{ \stAxvv{\ztl}{n}{m} }
\def\stAxtnm{ \stAxvv{\zt}{n}{m} }
\def\stAxtltnm{ \stAxvv{\ztl}{2n}{m} }
\def\ryA{ \mathcal{A}_{\multimapdotinv} }
\def\ryA{ \mathcal{A}_{\nearrow} }

\def\mtcn{multi-cone}

\def\mgmod{ \mathsf{gmod} }

\def\HKh{ \mathrm{H}^{\mathrm{Kh}} }
\def\HKhv#1{ \HKh_{#1} }
\def\HKhb{ \HKhv{\bullet} }
\def\Hst{ \mathrm{H}^{\mathrm{st}} }
\def\Hstv#1{ \Hst_{#1} }
\def\Hstb{ \Hstv{\bullet} }
\def\Hsti{ \Hstv{i} }

\def\Hstbd{ \Hstb(\xdummy) }

\def\Hm{ \mathrm{H} }
\def\Hmlv#1{ \Hm_{#1} }
\def\Hmb{ \Hmlv{\bullet} }
\def\Hmli{ \Hmlv{i} }

\def\Hmuv#1{ \Hm^{#1} }
\def\Hmub{ \Hmuv{\bullet} }

\def\xlrb#1{ \bigl( #1 \bigr) }

\def\xlrB#1{ \Bigl( #1 \Bigr) }

\def\xund#1{ \underline{#1} }
\def\uxlam{ \xund{\xlam} }
\def\uxtau{ \xund{\xtau} }

\def\xfhm{ \hzbr }
\def\xfhmv#1{ \xfhm_{#1} }
\def\xfhma{ \xfhmv{\xal} }
\def\xfhmao{ \xfhmv{\xalo} }
\def\xfhmat{ \xfhmv{\xalt} }
\def\xfhmapz{ \xfhmv{\xalpz}\p }

\def\yfhmv#1{ \zbrv{#1} }
\def\yfhma{ \yfhmv{\xal} }
\def\yfhmao{ \yfhmv{\xalo} }
\def\yfhmat{ \yfhmv{\xalt} }

\def\yfhmap{ \yfhmv{\xalp}\p }

\def\xIdapz{ \xIdv{\xalpz} }
\def\xIda{ \xIdv{\xal} }

\def\xph{ \varphi }

\def\xphaot{ \xph }
\def\hphaot{ \hat{\xphaot} }

\def\xIdtau{ \xIdv{\xtau} }

\def\sgmu{ \mu }
\def\sgmuv#1{ \sgmu_{#1} }
\def\sgmua{ \sgmuv{\xal} }
\def\sgmuao{ \sgmuv{\xalo} }
\def\sgmuat{ \sgmuv{\xalt} }

\def\cSg{ \Sigma }
\def\cSgv#1{ \cSg_{#1} }
\def\cSgij{ \cSgv{ij} }

\def\hSg{ \hat{\cSg} }
\def\hSgv#1{ \hSg_{#1} }

\def\xSo{ S^1 }
\def\xSov#1{ \xSo_{(#1)} }
\def\xSoo{ \xSov{1} }
\def\xSok{ \xSov{k} }

\def\zp{ p }
\def\zpv#1{ \zp_{#1} }
\def\zpo{ \zpv{1} }
\def\zpt{ \zpv{2} }
\def\zph{ \zpv{3} }
\def\zpf{ \zpv{4} }
\def\zpi{ \zpv{i} }
\def\zpj{ \zpv{j} }

\def\ftu{4Tu}
\def\ftur{\ftu\ relation}

\def\eleps{ \epsilon }
\def\elepsv#1{ \eleps_{#1} }
\def\elepso{ \elepsv{1} }
\def\elepsk{ \elepsv{k} }

\def\deleps{ \ccbsym{\eleps} }
\def\delepsv#1{ \deleps_{#1} }
\def\delepso{ \delepsv{1} }
\def\delepsk{ \delepsv{k} }

\def\beleps{ \boldsymbol{\eleps} }
\def\dbeleps{ \ccbsym{\beleps} }

\def\ccbsym#1{ \hat{#1} }
\def\zbr{ \boldsymbol{\rho} }
\def\zbrv#1{ \zbr_{#1} }
\def\zbral{ \zbrv{\xal} }

\def\hzbr{ \ccbsym{\zbr} }

\def\ip{ i\p }
\def\np{ n\p }
\def\tstI{ \tilde{\stI} }

\def\mtB{ B }
\def\mtBi{ \mtB^{-1} }
\def\mtBvv#1#2{ \mtB_{#1#2} }
\def\mtBivv#1#2{ \mtBi_{#1#2} }

\def\mtBIJ{ \mtBvv{\stI}{\stJ} }
\def\mtBiIJ{ \mtBivv{\stI}{\stJ} }

\def\mtBab{ \mtBvv{\xal}{\xbet} }
\def\mtBiab{ \mtBivv{\xal}{\xbet} }

\def\qpmv#1{ q^{#1} - q^{-#1} }

\def\smmnev{ \sum_{\substack{0\leq m\leq n \\ n-m\in 2\ZZ}} }

\def\xx{ x }

\def\cxzab{ c_{\xal\xbet} }
\def\ay{ y }

\def\xmltv#1{ \mathrm{m}_{#1} }
\def\xmltx{ \xmltv{\xx} }

\def\xlamxv#1#2{\xlam_{#1|#2} }
\def\xlamIJm{ \xlamxv{\stI\stJ}{m} }

\def\prclamv#1#2{ \prsymalg{\xlamxv{#1}{#2}} }
\def\prclamIJm{ \prclamv{\stI\stJ}{m} }
\def\prclamIJmp{ \prclamv{\stIp\stJp}{m\p} }
\def\prclamIJpm{ \prclamv{\stI\stJp}{m} }

\def\cpmsh{ \xlrb{\lSh{\jwpm}} }

\def\Tbr{ \mathrm{T} }
\def\Tbrvv#1#2{ \Tbr_{#1,#2} }
\def\Tbrnnm{ \Tbrvv{n}{nm} }
\def\Tbrnmnm{ \Tbrvv{n}{-nm} }

\def\ncc{n}
\def\nccv#1{ \ncc_{#1} }

\def\nccab{ \nccv{\xal\xbet} }

\def\uhdv#1{ |#1|^+_{\mathrm{h}} }

\def\xSg{ \Sigma }
\def\SgI{ \xSg\times\bbI }
\def\xIRt{ \IR^2 }
\def\StI{ \St\times\bbI }

\def\yLnk{ \mathfrak{L} }
\def\yLnkv#1{ \yLnk(#1) }
\def\yLnkM{ \yLnkv{M} }
\def\yLnkSh{ \yLnkv{\Sh} }
\def\yLnkSot{ \yLnkv{\Sot} }

\def\tyLnk{ \tilde{\yLnk} }
\def\tyLnkv#1{ \tyLnk(#1) }
\def\tyLnkM{ \tyLnkv{M} }
\def\tyLnkSh{ \tyLnkv{\Sh} }
\def\tyLnkSot{ \tyLnkv{\Sot} }

\def\yTng{ \mathfrak{T} }

\def\yTngvv#1#2{ \yTng_{#1}(#2) }
\def\yTngbnM{ \yTngvv{}{\bbM} }
\def\yTngbnMp{ \yTngvv{}{\bbM\p} }

\def\yTngv#1{ \yTng(#1) }
\def\yTngM{ \yTngv{\bbM} }
\def\yTngSg{ \yTngv{\SgI} }

\def\yTngBtn{ \yTngvv{n}{\Bt} }

\def\yTnglv#1{ \yTng_{#1} }
\def\yTngtntn{ \yTnglv{2n,2n} }
\def\yTngtmtn{ \yTnglv{2m,2n} }
\def\yTngtntm{ \yTnglv{2n,2m} }

\def\yTngtnz{ \yTnglv{2n,0} }
\def\yTngzn{ \yTnglv{0,n} }
\def\yTngnn{ \yTnglv{n,n} }

\def\yTngmn{ \yTnglv{m,n} }
\def\yTngnm{ \yTnglv{n,m} }

\def\yTngStv#1{ \yTngvv{#1}{\StI} }
\def\yTngSttntn{ \yTngStv{2n,2n} }
\def\yTngSttmtn{ \yTngStv{2m,2n} }

\def\tyTng{ \tilde{\yTng} }

\def\yTngop{ \yTng^{\xop} }

\def\tyTngop{ \tyTng^{\xop} }

\def\tyTngopv#1{ \tyTngop(#1) }
\def\tyTngopSg{ \tyTngopv{\SgI} }

\def\tyTngvv#1#2{ \tyTng_{#1}(#2) }
\def\tyTngBtn{ \tyTngvv{n}{\Bt} }
\def\tyTngbnM{ \tyTngvv{\bfn}{\bbM} }
\def\tyTngbnM{ \tyTngvv{}{\bbM} }

\def\tyTngv#1{ \tyTng(#1) }
\def\tyTngM{ \tyTngv{\bbM} }
\def\tyTngSg{ \tyTngv{\SgI} }
\def\tyTngSgv#1{ \tyTng_{#1}(\SgI) }
\def\tyTngSgmn{ \tyTngSgv{m,n} }
\def\tyTngSgnn{ \tyTngSgv{n,n} }
\def\tyTngSgn{ \tyTngSgv{n} }

\def\yTngSt{ \yTngv{\StI} }
\def\tyTngSt{ \tyTngv{\StI} }

\def\yMp{ \mathop{\mathrm{Map}} }
\def\yMpv#1{ \yMp_{#1} }

\def\yMpM{ \yMpv{\bbM} }
\def\yMpSt{ \yMpv{\StI} }
\def\yMpStM{ \yMpv{M} }

\def\IRtint{ \xIRt\times\bbI }
\def\Stint{ \St\times\xinter }

\def\yrlt{relative}
\def\yrlhm{\yrlt\ homeomorphism}
\def\rmcg{\yrlt\ mapping class group}

\def\yBr{ \mathfrak{B} }
\def\yBrSgv#1{ \yBr_{#1}(\xSg) }
\def\yBrSgn{ \yBrSgv{n} }
\def\yBrv#1{ \yBr_{#1} }
\def\yBrn{ \yBrv{n} }

\def\yTL{ \yTng^{\mathrm{TL}} }
\def\yTLv#1{ \yTL_{#1} }
\def\yTLn{ \yTLv{n} }
\def\yTLmn{ \yTLv{m,n} }

\def\yTLztn{ \yCrn }

\def\yCr{ \mathfrak{C} }
\def\yCrv#1{ \yCr_{#1} }
\def\yCrm{ \yCrv{m} }
\def\yCrn{ \yCrv{n} }

\def\hoam{ \mathop{\mathfrak{s}} }
\def\hoho{ \tilde{\hoam} }

\def\xtw{ \mathop{tw} }

\def\xsigm{ \sigma }
\def\xsigmv#1{ \xsigm_{#1} }
\def\xsigmo{ \xsigmv{1} }
\def\xsigmt{ \xsigmv{2} }

\def\yk{ l }

\def\thl{\thinlines}
\def\thlb{\thinlines}

\def\brtwpvv#1#2#3{
\setlength{\unitlength}{0.2mm}
\begin{picture}(#3,30)(-30,0)
\thl
\put(0,0){\oval(10,10)[bl]}
\put(-5,0){\line(0,1){20}}
\put(-10,20){\oval(10,10)[t]}
\put(-20,0){\oval(10,10)[br]}
\put(-2,3){\line(1,0){2}}
\put(-2,17){\line(1,0){2}}
\put(-8,3){\line(-1,0){17}}
\put(-8,17){\line(-1,0){17}}
\put(-20,-5){\line(-1,0){5}}
\put(-27,-10){$\circ$}
\put(-27,-22){$\scriptstyle{#1}$}
\put(8,5){$\scriptstyle{#2}$}
\put(-18,2){$\cdot$}
\put(-18,7){$\cdot$}
\end{picture}
}

\def\yvspv#1{
\xygraph{
!{0;/r1pc/:}
[d(#1)]
*{}
}
}

\def\yvspoh{ \yvspv{1.5} }

\def\brtwpv#1#2{ \brtwpvv{#1}{#2}{50} }

\def\tmrvv#1#2{
\setlength{\unitlength}{0.2mm}
\begin{picture}(#2,30)(-30,0)
\thl
\put(-10,0){\oval(10,10)[br]}
\put(-5,0){\line(0,1){20}}
\put(-10,20){\oval(10,10)[t]}
\put(-10,0){\oval(10,10)[bl]}
\put(-2,3){\line(1,0){2}}
\put(-2,17){\line(1,0){2}}
\put(-8,3){\line(-1,0){17}}
\put(-8,17){\line(-1,0){17}}
\put(8,5){$\scriptstyle{#1}$}
\put(-18,2){$\cdot$}
\put(-18,7){$\cdot$}
\end{picture}
}

\def\tmrv#1{ \tmrvv{#1}{50} }

\def\tmrpvv#1#2#3{
\bsymalg{
\setlength{\unitlength}{0.2mm}
\begin{picture}(#3,30)(-30,0)
\thl
\put(-10,0){\oval(10,10)[br]}
\put(-5,0){\line(0,1){20}}
\put(-10,20){\oval(10,10)[t]}
\put(-10,0){\oval(10,10)[bl]}
\put(-2,3){\line(1,0){2}}
\put(-2,17){\line(1,0){2}}
\put(-8,3){\line(-1,0){17}}
\put(-8,17){\line(-1,0){17}}
\put(8,5){$\scriptstyle{#1}$}
\put(-18,2){$\cdot$}
\put(-18,7){$\cdot$}
\put(-15,-20){$\scriptstyle{#2}$}
\end{picture}
}
}

\def\tmrpv#1#2{ \tmrpvv{#1}{#2}{50} }

\def\brtwpfvv#1#2{
\setlength{\unitlength}{0.2mm}
\begin{picture}(#2,30)(-30,0)
\thl
\put(0,0){\oval(10,10)[bl]}
\put(-5,0){\line(0,1){20}}
\put(-10,20){\oval(10,10)[t]}
\put(-20,0){\oval(10,10)[br]}
\put(-2,3){\line(1,0){2}}
\put(-2,17){\line(1,0){2}}
\put(-8,3){\line(-1,0){17}}
\put(-8,17){\line(-1,0){17}}
\put(-20,-5){\line(-1,0){5}}
\put(8,5){$\scriptstyle{#1}$}
\put(-18,2){$\cdot$}
\put(-18,7){$\cdot$}
\end{picture}
}

\def\brtwpfv#1{ \brtwpfvv{#1}{45} }

\def\brtwpnmoe{
\setlength{\unitlength}{0.2mm}
\begin{picture}(70,30)(-30,0)
\thl
\put(0,0){\oval(10,10)[bl]}
\put(-5,0){\line(0,1){20}}
\put(-10,20){\oval(10,10)[t]}
\put(-20,0){\oval(10,10)[br]}
\put(-2,3){\line(1,0){2}}
\put(-2,17){\line(1,0){2}}
\put(-8,3){\line(-1,0){17}}
\put(-8,17){\line(-1,0){17}}
\put(-20,-5){\line(-1,0){5}}
\put(8,5){$\scriptstyle{n-1}$}
\put(-18,2){$\cdot$}
\put(-18,7){$\cdot$}
\put(0,32.5){\line(-1,0){25}}
\end{picture}
}

\def\brrtpv#1#2{
\setlength{\unitlength}{0.2mm}
\begin{picture}(45,30)(-20,5)
\thl
\put(0,-5){\line(1,0){6}}
\put(0,0){\oval(10,10)[bl]}
\put(-5,0){\line(0,1){20}}
\put(-10,20){\oval(10,10)[tr]}
\put(-16,25){\line(1,0){6}}
\put(-2,3){\line(1,0){8}}
\put(-2,17){\line(1,0){8}}
\put(-8,3){\line(-1,0){8}}
\put(-8,17){\line(-1,0){8}}
\put(-17,20){$\circ$}
\put(-17,33){$\scriptstyle{#1}$}
\put(13,5){$\scriptstyle{#2}$}
\put(-15,2){$\cdot$}
\put(-15,7){$\cdot$}
\end{picture}
}

\def\brlfpv#1#2{
\setlength{\unitlength}{0.2mm}
\begin{picture}(45,30)(-30,0)
\thl
\put(-10,25){\line(1,0){5}}
\put(-10,20){\oval(10,10)[tl]}
\put(-20,0){\oval(10,10)[br]}
\put(-5,3){\line(-1,0){20}}
\put(-5,17){\line(-1,0){20}}
\put(-20,-5){\line(-1,0){5}}
\put(-27,-10){$\circ$}
\put(-27,-22){$\scriptstyle{#1}$}
\put(4,5){$\scriptstyle{#2}$}
\put(-18,2){$\cdot$}
\put(-18,7){$\cdot$}
\end{picture}
}

\def\brcapv#1{
\setlength{\unitlength}{0.2mm}
\begin{picture}(45,30)(-20,5)
\thl
\put(0,-5){\line(1,0){6}}
\put(0,0){\oval(10,10)[bl]}
\put(-5,0){\line(0,1){20}}
\put(0,20){\oval(10,10)[tl]}
\put(0,25){\line(1,0){6}}
\put(-2,3){\line(1,0){8}}
\put(-2,17){\line(1,0){8}}
\put(-8,3){\line(-1,0){8}}
\put(-8,17){\line(-1,0){8}}
\put(13,5){$\scriptstyle{#1}$}
\put(-15,2){$\cdot$}
\put(-15,7){$\cdot$}
\end{picture}
}

\def\brcupvv#1#2#3#4{
\setlength{\unitlength}{0.2mm}
\begin{picture}(#3,30)(#4,5) %45,-40 ||| 90, -80
\thl
\put(-20,25){\line(-1,0){5}}
\put(-20,20){\oval(10,10)[tr]}
\put(-20,0){\oval(10,10)[br]}
\put(-5,3){\line(-1,0){20}}
\put(-5,17){\line(-1,0){20}}
\put(-20,-5){\line(-1,0){5}}
\put(#2,5){$\scriptstyle{#1}$}
\put(-18,2){$\cdot$}
\put(-18,7){$\cdot$}
\end{picture}
}

\def\brcupv#1{ \brcupvv{#1}{-37}{45}{-40} }

\def\brccpv#1{
\setlength{\unitlength}{0.2mm}
\begin{picture}(55,30)(-30,5)
\thl
\put(0,0){\oval(10,10)[bl]}
\put(-5,0){\line(0,1){20}}
\put(-20,20){\oval(10,10)[tr]}
\put(-20,25){\line(-1,0){5}}
\put(0,20){\oval(10,10)[tl]}
\put(-20,0){\oval(10,10)[br]}
\put(-2,3){\line(1,0){7}}
\put(-2,17){\line(1,0){7}}
\put(-8,3){\line(-1,0){17}}
\put(-8,17){\line(-1,0){17}}
\put(-20,-5){\line(-1,0){5}}
\put(0,-5){\line(1,0){5}}
\put(0,25){\line(1,0){5}}
\put(13,5){$\scriptstyle{#1}$}
\put(-18,2){$\cdot$}
\put(-18,7){$\cdot$}
\end{picture}
}

\def\brcapn{ \brcapv{n} }
\def\cbrcapnsh{ \bsymcatsh{\brcapn} }

\def\brcupn{ \brcupv{n} }

\def\cbrcupnsh{ \bsymcatsh{\brcupn} }

\def\brccpn{ \brccpv{n} }
\def\cbrccpnsh{ \bsymcatsh{\brccpn} }

\def\tmrn{ \tmrv{n} }

\def\tmrpnk{ \tmrpv{n}{k} }
\def\tmrpmk{ \tmrpv{m}{k} }
\def\brtwptn{ \brtwpv{2}{2n} }
\def\brtwpn{ \brtwpv{2}{n} }
\def\brtwpon{ \brtwpv{1}{n} }
\def\brtwpoi{ \brtwpv{1}{i} }
\def\brtwpoto{ \brtwpvv{1}{\ztl+1}{80} }

\def\brtwpfnmo{ \brtwpfvv{n-1}{70} }

\def\cbrtwptnsh{ \bsymcatsh{\brtwptn} }

\def\brtwpfn{ \brtwpfv{n} }
\def\cbrtwpfnsh{ \bsymcatsh{\brtwpfn} }

\def\brrtpon{ \brrtpv{1}{n} }

\def\brlfpon{ \brlfpv{1}{n} }

\def\Dsymbim#1{ \bsymcat{#1}_{\xdrv} }

\def\Dbrtwptn{ \Dsymbim{\brtwptn} }
\def\Zbrtwptn{ \bsymcat{\brtwptn\,} }

\def\Dobrotn{ \Dsymbim{\gobrotn} }
\def\Dobrmtn{ \Dsymbim{\gobrmtn} }
\def\Kobrmtn{ \Kbsymbim{\gobrmtn} }
\def\Zobrotn{ \bsymcat{\gobrotn} }

\def\Kidbrtn{ \Kbsymbim{\,\gidbrtn} }
\def\Didbrtn{ \Dsymbim{\,\gidbrtn} }
\def\Zidbrtn{ \bsymcat{\,\gidbrtn} }
\def\Kidbrtn{ \Kbsymbim{\,\gidbrtn} }

\numberwithin{equation}{section}

\title[A categorification of the stable $\SUt$ WRT invariant
 of links in $\Sot$]
{A categorification of the stable $\SUt$ Witten-Reshetikhin-Turaev invariant
 of links in $\Sot$}

\author[L.~Rozansky]{Lev Rozansky}
\address{
L.~Rozansky\\
Department of Mathematics\\
University of North Carolina at Chapel Hill\\
CB \# 3250, Phillips Hall\\
Chapel Hill, NC 27599
}
\email{rozansky@math.unc.edu}
\thanks{The work of L.R. was supported in part by the NSF grant DMS-0808974}

\begin{document}
%\draft
%\begin{titlepage}
\maketitle
\begin{abstract}
The \WRT\ invariant of a link $\xL$ in $\Sot$ at sufficiently high
values of the level $\xr$ can be expressed as an evaluation of a
special polynomial invariant of $\xL$ at $q=\exp(\pi i/\xr)$. We
categorify this polynomial invariant by associating to $\xL$
a bi-graded homology whose graded Euler characteristic is equal to
this polynomial.

If $\xL$ is presented as a circular closure of
a tangle $\xtau$ in $\Sot$, then the homology of
$\xL$ is defined as the \Hhom\ of the $\xaHn$-bimodule associated
to $\xtau$ in\cx{Kh1}. This homology can also be expressed as a
stable limit of the Khovanov homology of the circular closure of $\xtau$ in
$\Sh$ through a \cbr\ with high twist.

\end{abstract}
%\end{titlepage}
\tableofcontents

\section{Introduction}

\subsection{The stable \WRT\ invariant of links in a 3-sphere with handles}

Let $\xIn(\xyO)\in\Zsqqi$ be a Laurent series of $q$ which is a
topological invariant of an object $\xyO$ (\eg  a link in a
3-manifold). A general idea of a (weak) categorification, as presented in\cx{Kh1}, is to
associate to $\xyO$ a $\ZZ\oplus\ZZ$-graded vector space $\xyH(\xyO) =
\bigoplus_{i,j\in\ZZ}\xyHij(\xyO)$, the first degree being of homological nature,
such that its graded Euler characteristic is
equal to $\xIn(\xyO)$:
\xlee{eqa:01}
\xIn(\xyO) = \sum_{i,j\in\ZZ}
(-1)^i\,q^j\,\dim\xyHij(\xyO).
\xeee
A full categorification extends this assignment to
a functor from the category of link cobordisms
to the category of homogeneous maps between bigraded vector
spaces.
% $\hat{\xyC}\colon\xyH(\xyO)\rightarrow\xyH(\xyO\p)$
%to a cobordism $\xyC$ between the objects $\xyO$ and $\xyO\p$.

Let $\xL$ be a framed link in a 3-manifold $\xM$. The $\SUt$ \fWRT\ (\WRT) invariant
$\ZrLM$ is a $\IC$-valued topological invariant of the pair $\pLM$
which depends on an integer number $\xr\geq 2$. We assume that the components
of $\xL$ are `colored' by the fundamental representation of $\SUt$
and if $\xL$ is an empty link, then we omit it from notations.

Generally, the dependence of
$\ZrLM$ on $\xr$ seems random
%, although the surgery formula
%suggests that $\ZrLM$ might be `close' to a polynomial of $q =\exp(\pi i/r)$. Yet,
and the non-polynomial nature of $\ZrLM$ presents
a challenge for anyone trying to categorify it.
However, there is a special class of 3-manifolds (\tsswh) for
which $\ZrLM$ is almost polynomial.
Let $\Shck$ denote a 3-sphere $\Sh$ from which
$\xyk$ 3-balls are cut. Let $\Shgk$ denote oriented manifold constructed
by gluing $2\xyg$ spherical boundary components  of $\Shcv{2\xyg + \xyk}$
pairwise. A \emph{\tswgh} $\Shg$ is defined as $\Shvv{\xyg}{0}$:
$\Shg = \Shvv{\xyg}{0}$. Alternatively, $\Shg$
can be constructed by gluing together two \hbds\ of genus
$\xyg$ through the identity isomorphism of their boundaries.

A link $\xL$ in $\Shg$ can be constructed by taking a
$(n_1,\ldots, n_{2\xyg})$-tangle $\xtau$ in $\Shctg$ with
pairwise matching valences $n_i=n_{i+\xyg}$ and applying the pairwise gluing
to the boundary components of $\Shctg$ and to the tangle endpoints
which reside there. If at least one of the valences  is odd,
then the \WRT\ invariant  associated with the resulting link is zero
for any $\xr$,
hence we will always assume that all valences are even.
For a $(2n_1,\ldots, 2n_{2\xyg})$-tangle $\xtau$ in
$\Shctg$ we define the critical level:
\ylee{eqa:2}
\rcrvv{\Shctg}{\xtau} = \max(n_1,\ldots,n_{\xyg}) + 2.
\yeee
We define the critical level $\rcrShgL$ of a link $\xL\subset\Shg$ as the
minimum of $\rcrvv{\Shctg}{\xtau}$ over all the presentation of
$\xL$ as the closure of a tangle.

The following theorem is easy to prove:
\begin{theorem}
For a link $\xL$ in a \tswgh\ $\Shg$ there exists a unique rational
function $\JnSgLq\in \Qq$ such that if $\xr\geq \rcrShgL$, then the
modified \WRT\ invariant
\ylee{eqa:1a}
\ZtrLShg = \frac{\ZrLShg}{\ZrShg}
\yeee
is equal to the evaluation of $\JnSgLq$ at  $q=\exp(i\pi/\xr)$:
\xlee{eqa:1}
\ZtrLShg = \JnSgLq|_{q=\exp(i\pi/\xr)}.
\xeee
\end{theorem}

We call $\JnSgLq$ the \stWRTi\ of $\xL\subset\Shg$.
We expect that Laurent series expansion of $\JnSgLq$ at $q=0$ can be categorified
in the sense of \ex{eqa:01}.

If $\xyg=0$, then $\Shg$ is a 3-sphere $\Sh$. In this case
$\rcrvv{\Sh}{\xL}=2$ and $\JnLq$ is the Jones polynomial whose
categorification was constructed in\cx{Kh1}.

In this paper we consider the case of $\xyg=1$, that is, we study
$\Sho=\Sot$. It turns out that the corresponding \stWRTi\
$\JnLotq$ is again a Laurent polynomial, and we construct its
categorification.

%\subsection{The stable toy \WRT\ theory}
\subsection{A 3-dimensional \orTQFT}
\label{ss:ortqft}

The \stWRTi\ $\JnSgLq$ comes from a \strtWRTt. Let us recall the
definition of a 3-dimensional \orTQFT\ (a generalization to $n$
dimensions is obvious).

Let $\rF$ be a field with an involution
\xlee{eq:finv}
\finvm\colon\rF\longrightarrow\rF.
\xeee
Usually, $\rF=\IC$, $\finvm$ being the complex conjugation.
%, but in our case $\rF$ is
%the field $\Qq$ of rational functions of $q$ with the involution $\dsym$ defined
%by the relation $\qsym = \qi$.

Let $\xrS$ be the
set of \adcbs: its elements are some closed oriented connected
2-manifolds with marked points. $\xrS$ always includes the empty
surface $\emptyset$.
%In our case, other elements of
%$\xrS$ are $\Stv{n}$:  2-spheres with $n$ marked points.

Let $\bbM$ be a 3-dimensional manifold with a boundary
consisting of $k$ connected components from $\xrS$, each having $n_i$ ($i=1,\ldots,k$) marked points.
An \emph{\idfn} of the boundary of $\bbM$ is a choice of a
diffeomorphism
between every connected component of $\del\bbM$
%We say that the boundary $\del\bbM$ is \emph{\idfd} if we specify
%a homeomorphism between every connected component of $\del\bbM$
and its `standard copy' in $\xrS$.

A $\bfn$-tangle in $\bbM$, where $\bfn=(n_1,\ldots,n_k)$, is an embedding of a disjoint union of segments
and circles in $\bbM$ such that the endpoints of segments are
mapped one-to-one to the marked points on the boundary $\del\bbM$.
Our tangles are assumed to be framed.
Let $\tyTngbnM$ be the set of all $\bfn$-tangles in $\bbM$
distinguished up to boundary fixing ambient isotopy.
% and let $\tyTngM=\bigcup_{\bfn}\tyTngbnM$ be the set of all tangles in $\bbM$.

A diffeomorphism $\bbM\rightarrow\bbM$ is called \emph{\yrlt} if it acts
trivially on the boundary of $\bbM$. Let $\yMpM$ denote the
mapping class group of \yrlhm s of $\bbM$. This
group acts on $\tyTngbnM$, and we refer to the elements of the quotient
$\yTngbnM=\tyTngbnM/\yMpM$ as \emph{\rtng s}, that is, \rtng s are
tangles distinguished up to \yrlhm s.

Let $\xrM$ be the set of \adcm s: its elements are
some oriented 3-manifolds with independently oriented boundary components, each boundary
component being admissible. The set $\xrM$ must include the
manifolds $\adS\times\bbI$ for all $\adS\in\xrS$ and $\bbI=[0,1]$.
Also $\xrM$ must be closed with respect to two operations. The
first operation is the disjoint union of admissible manifolds. The
second operation is the gluing of two diffeomorphic connected
components of $\del\bbM$, provided that they have opposite
relative
orientations with respect to the orientation of $\bbM$.

%\nidm
%
%pairs $(\xtau,\bbM)$, $\xtau\in\tyTngbnM$, where $\bbM$ is a
%special oriented 3-manifold $\bbM$ with (independently) oriented boundary $\del \bbM$ satisfying the property that
%it is a disjoint union of the boundaries of
%$\xrS$, marked points being the endpoints of the tangle $\xtau$.
%The set $\xtrM$ must contain all manifolds of the form
%$\adS\times\bbI$, where
%$\adS\in\xrS$ and we use a notation $\bbI=[0,1]$. Also $\xtrM$ must be closed with respect to two operations. The
%first operation is the disjoint union of admissible pairs. The
%second operation is the gluing of two diffeomorphic connected
%components of $\del\bbM$, provided that they have opposite
%relative
%orientations with respect to the orientation of $\bbM$.
%
%\nidm

To an admissible boundary $\adS\in\xrS$ the \TQFT\ associates
a Hilbert space $\cH(\adS)$ over $\rF$ with three isomorphisms
\xlee{eq:invmd}
\dflp,\dsym \colon\cH(\adS)\longrightarrow\cHv(\adS),\qquad
\finvm\colon\cH(\adS)\longrightarrow\cH(\adS).
\xeee
Here $\cHv$ is the
dual of $\cH$ and the maps $\dsym$ and $\finvm$ are `anti-linear':
they involve the involution\rx{eq:finv} of the field $\rF$. The
maps\rx{eq:invmd} should satisfy the relations
\xlee{eq:relinv}
\dflpv{(\finv{\xxv})}=\dsymv{\xxv},\qquad
\dsymv{(\finv{\xxv})}=\dflpv{\xxv},\qquad
\finv{\finv{\xxv}}=\xxv.
%\qquad \forall\xxv\in\cH(\adS).
\xeee
for all $\xxv\in\cH(\adS)$.
If any two of three isomorphisms\rx{eq:invmd} are defined,
then the third one is determined by either of the first two relations\rx{eq:relinv}.
Slightly abusing notations, we will use the same notations $\dflp$
and $\dsym$ for the inverses of the isomorphisms\rx{eq:invmd}; in
each case it will be clear whether a direct or an inverse isomorphism is used.

%and the involution $\dsym$, in contrast to $\dflp$, involves the involution
%of the field $\rF$.

To a disjoint union of admissible boundaries
\TQFT\ associates a tensor product of Hilbert spaces over $\rF$:
\ylee{eq:tnprh}
\cH(\adS_1\sqcup\cdots\sqcup\adS_k) =
\cH(\adS_1)\otimes\cdots\otimes\cH(\adS_k).
\yeee
For the empty boundary $\cH(\emptyset)=\rF$.

Let $\bbM$ be an admissible manifold. The relative orientation of
$\bbM$ and connected components of its boundary
separates $\del\bbM$ into two disjoint components: the \inb\
boundary and the \outb\ boundary: $\del\bbM =
\xinv{\del\bbM}\sqcup\xoutv{\del\bbM}$.
%We say that the boundary of $\bbM$ is \emph{\idfd} if we fix a diffeomorphism between every
%connected component of $\del\bbM$ and a standard copy
Define
\ylee{eq:dfhsp}
\cHdbM=\cH\xoutv{\del\bbM}\otimes\cHv\xinv{\del\bbM}.
\yeee
For every \idfn\ of $\del\bbM$  the
%Fix the diffeomorphism
%between each connected component $\adS$ of $\del\bbM$ and its
%`standard copy'. Then
\TQFT\ provides a \stmp, which maps \rtng
s $\xtau$ in $\bbM$ to elements of the Hilbert space of its boundary:
\xlee{eq:stmptqft}
\xstmp{\xdummy}:\yTngbnM\longrightarrow \cHdbM
%= \cH\xoutv{\del\bbM}\otimes\cHv\xinv{\del\bbM}
,\qquad
(\xtau,\bbM)\mapsto\xstmp{\xtau,\bbM}.
\xeee
%

%Then
%to an admissible pair $(\xtau,\bbM)$
%the \stmp\ of the \TQFT\ associates  an element
%$\xstmp{\xtau,\bbM}\in\cH\xoutv{\del\bbM}\otimes\cHv\xinv{\del\bbM}$.

The \stmp\ should satisfy the following axioms:
\begin{description}

\item[Change of orientation of the boundary component]
A change in the orientation of a boundary component of $\bbM$
results in the application of the $\dflp$ map to the corresponding factor in the
Hilbert space
%$\cH\xoutv{\del\bbM}\otimes\cHv\xinv{\del\bbM}$.
$\cHdbM$.

\item[Change of orientation of $\bbM$]
If $\bbM\p$ is the manifold $\bbM$ with reversed orientation, then
$\xstmp{\xtau,\bbM\p} =
\dsymv{\xstmp{\xtau,\bbM}}\in\cHvdbM$
%\cHv\xoutv{\del\bbM}\otimes\cH\xinv{\del\bbM}$
(since we did not change the orientation of the boundary
components, the \inb\ boundary component of $\bbM$ becomes the
\outb\ boundary component of $\bbM\p$ and vice versa).

\item[Disjoint union] For two admissible pairs $(\xtauo,\bbM_1)$
and $(\xtaut, \bbM_2)$, the \stmp\ of their disjoint union is
$\xstmp{\xtauo\sqcup\xtaut,\bbM_1\sqcup\bbM_2} =
\xstmp{\xtauo,\bbM_1}\otimes\xstmp{\xtaut,\bbM_2}$.

\item[Gluing]

Suppose that the same admissible boundary component $\adS$ appears
in the \inb\ and in the \outb\ parts of $\bbM$. Let $\bbM\p$ be
the manifold constructed by gluing these components together (according to their identifications with
the `standard copy'). Then
$\xstmp{\xtau,\bbM\p}$ is the result of canonical pairing $\xprng$ applied
to the tensor product $\cH(\adS)\otimes\cHv(\adS)$ within the
Hilbert space
%$\cH\xoutv{\del\bbM}\otimes\cHv\xinv{\del\bbM}$.
$\cHdbM$.
\end{description}

If follows from the change of orientation axioms and
relations\rx{eq:relinv} that if $\bbM\p$ is the manifold $\bbM$ in
which the orientation of $\bbM$ and of its boundary are reversed
simultaneously, then $\xstmp{\xtau,\bbM\p} =
\wfinv{\xstmp{\xtau,\bbM}}$.

In the \strtWRTt\ the base field $\rF$ is $\Qq$ -- the field of rational
functions of $q$ with rational coefficients. Admissible boundaries
are  2-spheres with $2n$ marked points, $n\geq 0$ and admissible
3-manifolds are $\Shgk$ and their disjoint unions.
In this paper we will consider a more restricted (`\toy') version of this
theory. Namely, the admissible 3-manifolds are only $\Shv{0}=\Sh$
(a 3-sphere), $\Shvv{0}{1}=\bbB^3$ (a 3-ball), $\Shvv{0}{2}=\StI$,
$\Shv{1}=\Sot$ and their disjoint unions.
%
%We will show in subsections\rw{sss:relstwrt} and\rw{ss:relstwrt}
%that the invariants $\JnLq = \xstmp{\xL,\Sh}$ and $\JnLotq = \xstmp{\xL,\Sh}$
%are Laurent polynomials of $q$ and satisfy the property\rx{eqa:1}.

We will show in subsections\rw{sss:relstwrt} that $\JnLq =
\xstmp{\xL,\Sh}$ is equal to the Jones polynomial of $\xL$ and we
will prove in subsection\rw{ss:relstwrt}
that the invariant $\JnLotq = \xstmp{\xL,\Sot}$ is a Laurent
polynomial of $q$  and satisfies the
property\rx{eqa:1}, which is this case takes the form
\xlee{eq:relfr}
\ZrLSot = \xstmp{\xL,\Sot}|_{q=\exp(i\pi/\xr)}
\xeee
for $\xr\geq \rcrSotL$,
because $\ZrSot=1$.
\subsection{A \Zgrdd\ \walg\ categorification}

In this paper we will construct a \walg\ categorification of the
\toy\ \strtWRTt.

Let us recall basic facts about a \Zgrdd\ version of a \walg\ categorical
3-dimensional \TQFT. We use the same set of admissible boundaries
$\xrS$ and the set of admissible 3-manifolds $\xrM$ as in previous
subsection.
To an admissible boundary $\adS\in\xrS$ an \alcat\ \TQFT\
associates a \Zgrdd\ \xalg\ $\cA(\adS)$ and an additive category
$\cC(\adS) = \Dbgmd{\cA(\adS)}$, that is, the bounded derived category
of \Zgrdd\ $\cA(\adS)$-modules. Each \xalg\ has a canonical
involution
\xlee{eq:alginv}
\dflp\colon \cA(\adS)\longrightarrow\cAop(\adS)
\xeee
which
preserves \Zgrd.
There are two equivalence functors
\xlee{eq:eqfns}
\dflp,\dsym\colon \Dbgmd{\cA(\adS)}\longrightarrow\Dbgmd{\cAop(\adS)}.
\xeee
The functor $\dflp$ is covariant: it turns a complex of
$\cA(\adS)$-modules into a complex of $\cAop(\adS)$-modules with
the help of the involution\rx{eq:alginv} and shifts its homological and $\ZZ$
gradings by an amount depending on $\adS$. The functor $\dsym$ is
contravariant: it maps objects to their duals and then shifts
their degree by an amount depending on $\adS$.
%
%The functor $\dflp$ is covariant and originates from the \xalg\
%involution $\dflp$, while the functor $\dsym$ is contravariant: it maps
%objects to their duals thus reversing grading shifts.
%Another
%contravariant equivalence functor
The contravariant functor
\ylee{eq:fcnjg}
\finvm\colon \Dbgmd{\cA(\adS)}\longrightarrow\Dbgmd{\cA(\adS)}
\yeee
is the composition of $\dflp$ with the inverse of $\dsym$ or the
other way around. In other words, $\finvm$ turns a complex of
$\cA(\adS)$-modules into the dual complex of
$\cAop(\adS)$-modules, then replaces $\cAop(\adS)$ by
$\cA(\adS)$ with the help of the involution\rx{eq:alginv} and finally performs a degree shift
which depends on $\adS$. Three
functors satisfy the
relations\rx{eq:relinv}, where this time $\xxv$ is an object of
$\Dbgmd{\cA(\adS)}$.

% which extends obviously to the covariant category equivalence
%$\dflp:\Dbmd{\cA(\adS)}\rightarrow\Dbmd{\cAop(\adS)}$.
%Duality generates a contravariant functor
%$\dsym:\Dbmd{\cA(\adS)}\rightarrow\Dbmd{\cAop(\adS)}$.
To a disjoint union of admissible boundaries an \alcat\ \TQFT\
associates a tensor product of algebras
\ylee{eq:tnpal}
\cA(\adS_1\sqcup\cdots\sqcup\adS_k) =
\cA(\adS_1)\otimes\cdots\otimes\cA(\adS_k).
\yeee
and, consequently, the category
\xlee{eq:tnpcat}
\cC(\adS_1\sqcup\cdots\sqcup\adS_k) =
\bDbgmd{ \cA(\adS_1)\otimes\cdots\otimes\cA(\adS_k) }.
%\cC(\adS_1)\otimes\cdots\otimes\cC(\adS_k).
\xeee
%
%(the latter relation is not a general property of
%categorical \TQFT s, but we postulate it here because the theory,
%that we construct in this paper, has it).
For an empty boundary $\cA(\emptyset)=\IQ$ and $\cC(\emptyset)$ is
the category of bounded complexes of $\ZZ\oplus\ZZ$-graded vector spaces over $\IQ$
(the first grading is homological and the second is related to
\Zgrd\ of algebras). Equivalently, this is a category of $\ZZ\oplus\ZZ$-graded vector spaces over $\IQ$
(both categories are related by taking homology).

Define
\ylee{eq:dfbalg}
\cAdbM = \cA\xoutv{\del\bbM}\otimes\cAop\xinv{\del\bbM},\qquad
\cCdbM = \bDbgmd{\cAdbM}.
\yeee
For every identification of the boundary of an admissible manifold $\bbM$
the \alcat\ \TQFT\ provides an \obmp
\begin{gather}
%\xlee{eq:obmptqft}
\label{eq:obcat}
\xobmp{\xdummy}:\yTngbnM\longrightarrow
\cCdbM,\qquad
%\bDbgmd{\cAdbM},\qquad
%\cA\xoutv{\del\bbM}\otimes\cAop\xinv{\del\bbM}},
%\times\Dbmdb{\cAop\xinv{\del\bbM}},
%\\
%\nonumber
(\xtau,\bbM)\mapsto\xobmp{\xtau,\bbM}.
%\xeee
\end{gather}
If $\del \bbM=\emptyset$, so that $\xtau$ is a link $\xL$ in
$\bbM$, we take the homology of the complex $\xobmp{\xL,\bbM}$ and
refer to it as the \sthm\ of the link $\xL\subset\bbM$:
\ylee{eq:khhomdf}
\Hstb(\xL,\bbM) = \Hmb(\xobmp{\xL,\bbM}).
\yeee
%
%\Hmb\xlrb{\symcat{\xL}} = \HKhb(\xL)

%will assume that the homology is taken and we will
%refer to $\xobmp{\xL,\bbM}$ as the homology of the link $\xL\subset\bbM$.

The \obmp\ should satisfy the following axioms:
\begin{description}
\item[Change of orientation of $\bbM$]
If $\bbM\p$ is the manifold $\bbM$ with reversed orientation, then
$\xobmp{\xtau,\bbM\p} =
\dsymv{\xobmp{\xtau,\bbM}}\in\bDbgmd{
%\cAop\xoutv{\del\bbM}\otimes\cA\xinv{\del\bbM}}$.
\cAopdbM}$.

\item[Disjoint union] For two admissible pairs $(\xtauo,\bbM_1)$
and $(\xtaut, \bbM_2)$, the \obmp\ of their disjoint union is
$\xobmp{\xtauo\sqcup\xtaut,\bbM_1\sqcup\bbM_2} =
\xobmp{\xtauo,\bbM_1}\otimes\xobmp{\xtaut,\bbM_2}$.

\item[Change of orientation of a boundary component]
A change in the orientation of a boundary component of $\bbM$
results in the application of the $\dflp$ functor to the corresponding factor in the
category of \ex{eq:obcat}.
%$\Dbmdb{\cA\xoutv{\del\bbM}}
%\times\Dbmdb{\cAop\xinv{\del\bbM}}$.

\item[Gluing]

Suppose that the same admissible boundary component $\adS$ appears
in the \inb\ and in the \outb\ parts of $\bbM$. Let $\bbM\p$ be
the manifold constructed by gluing these components together. Then
$\xobmp{\xtau,\bbM\p}$ is the result of derived tensor product $\Loti$ applied
to the category $\bDbgmd{\cA(\adS)\otimes\cAop(\adS)}$ within the
category $\bDbgmd{\cAdbM}$.
\end{description}

An application of \Grtn 's $\Kz$-functor to the elements of a
\walg\ categorical \TQFT\ construction produces an ordinary \TQFT\ described in
the previous subsection.
This time $\rF$ is the field $\Qq$ of rational functions of $q$
and $\finv{q} = \qi$. The \Grtg\ $\Kz\xlrb{\Dbmd{\cA(\adS)}}$ is a
module over $\Zqqi$, the multiplication by $q$ corresponding to
the translation $\qsho$ of \Zgrd.
A Hilbert space $\cH(\adS)$ is the \Grtg\
$\QKz\xlrb{\Dbmd{\cA(\adS)}}$
defined as an extension to $\Qq$:
\xlee{eq:kzcat}
\cH(\adS)=\QKz\xlrb{\Dbmd{\cA(\adS)}}=\Kz\xlrb{\Dbmd{\cA(\adS)}}\otimes_{\Zqqi}\Qq.
\xeee

There is a canonical isomorphism
$\cHv(\adS)=\Kz\xlrb{\Dbmd{\cAop(\adS)}}$, because for two objects
$\zzA\in\Dbmd{\cA(\adS)}$ and $\zzB\in\Dbmd{\cAop(\adS)}$ there is
a pairing $\Kz(\zzB) \xprng \Kz(\zzA) = \chi\xlrb{\Tor(\zzB,\zzA)}$,
where $\chi$ is the Euler characteristic. Hence $\Kz$ reduces the
functors\rx{eq:eqfns} to the corresponding
involutions\rx{eq:invmd}.

%and the multiplication by $q$ corresponds to \Zdgr\ translation $\qsho$.
Overall, for every admissible manifold $\bbM$ there should be a commutative
triangle
\xlee{eq:cmtr}
\xymatrix{
&\yTngbnM
\ar[dr]^-{\xstmp{\xdummy}}
\ar[dl]_-{\xobmp{\xdummy}}
\\
\cCdbM
\ar[rr]^-{\Kz}
&&
\cHdbM
%\Dbmd{\cAdbM}
}
\xeee
and if a manifold $\bbM\p$ is constructed by gluing together two
matching boundary components of $\bbM$, then there should be a
commutative prism
\xlee{eq:cmprism}
\xymatrix{
\yTngbnM
\ar[rr]^-{\xglu}
%\ar[d]_-(0.25){\spPsp}|(0.525)\hole
\ar[dd]_-{\xstmp{\xdummy}}
\ar[dr]^-{\xobmp{\xdummy}}
&&
\yTngbnMp
\ar[dd]^-(0.25){\xstmp{\xdummy}}|(0.5)\hole
\ar[dr]^-{\xobmp{\xdummy}}
\\
&
\cCdbM
\ar[rr]^-(0.3){\Loti}
\ar[dl]^-{\Kz}
&&
\cCdbMp
\ar[dl]^-{\Kz}
\\
\cHdbM
\ar[rr]^-{\xprng}
&&
\cHdbMp
}
\xeee

\subsection{A brief account of the results and the plan of the paper}
We construct a weak categorification of the \tstyWRTt. Our
construction is based on \Zgrdd\ \xring s $\xaHn$ introduced in\cx{Kh2}.
In the same paper, a complex $\Ktau$ of $\xaHn\otimes\xaHopm$-modules is
associated to a diagram of a \ttngtmtn\ $\xtau$.

In this paper we associate the \xring\
$\xaHn$ and a derived category of \Zgrdd\ $\xaHn$-modules $\DbgHn$
to a 2-sphere with $2n$ marked points. A tangle $\xtau$ in a
3-ball can be considered as a \ttngztn, hence we associate to it
the same complex $\Ktau$ this time considered as an
object of $\DbgHn$. A link $\xL$ in $\Sh$ can be constructed by gluing two tangles in
3-balls along the common boundary. As explained in\cx{Kh2}, the $\xaHn$-modules in
$\Ktau$ are projective, hence the gluing rule implies that \sthm\
of $\xL\subset\Sh$ coincides with the homology introduced in
\cx{Kh1}.

The boundary category of the manifold $\StI$ with $2m$ marked
points on the `in' boundary and $2n$ marked points on the `out'
boundary is $\Dbgmd{\xaHn\otimes\xaHopm}$. If a tangle $\xsg$ in $\StI$
is represented by a \ttngtmtn\ $\xtau$ then we set again
$\xobmp{\xsg}=\Ktau$. We prove that if two different \ttngtmtn s
$\xtauo$ and $\xtaut$ represent the same tangle in $\StI$, then
the complexes $\Ktauo$ and $\Ktaut$ are quasi-isomorphic, that is,
although they may be homotopically inequivalent, they represent
isomorphic objects in derived category
$\Dbgmd{\xaHn\otimes\xaHopm}$.

If a link $\xL$ in $\Sot$ is presented as a closure of a
\ttngtntn\ $\xtau$ within $\Sot$, then in accordance with the
gluing rule we define the \sthm\ of $\xL$ as the \Hhom\ of
$\Ktau$: $\Hst(\lSot{\xL}) = \xHHlv{\Ktau}$. We prove that this
homology does not depend on the choice of a tangle $\xtau$ which
represents $\xL$.

Finally, we suggest a practical method of computing the \sthm\
$\Hst(\lSot{\xL})$. If $\xL\subset\Sot$ can be constructed by a
$\Sot$ closure of a \ttngtntn\ $\xtau$, then it is known that the
stable invariant $\xstmp{\xL,\Sot}$ can be approximated (up to the
coefficients at high powers of $q$) by the Kauffman bracket of the
$\Sh$ closure of $\xtau$ through the torus braid with high twist
number. We show that the same is true in categorified theory:
Khovanov homology of the $\Sh$ closure of $\xtau$ through a high
twist torus braid approximates the \sthm\ of $\xL$ in $\Sot$.
In fact, a close relation between the Hochschild homology of
$\xaHn$ and the Khovanov homology of a high twist torus link was
first observed by Jozef Przytycki\cx{PrzH} in case of $n=1$.

In Section\rw{s:tlcat} we review the definition and properties of
Kauffman bracket and \TLa. We define the \tstyWRTt\ and prove that
it satisfies the axiom requirement. Then we review Khovanov
homology, Bar-Natan's universal categorification of the \TLa\cx{BN1} and
bi-module categorification of \TLa\cx{Kh2}.

In Section\rw{s:res} we present main results of the paper:
we define the categorified theory, explain how \sthm\ is related to Khovanov homology through
torus braid closure and conjecture the structure of \Hhom\ and
cohomology of \xring s $\xaHn$.

In Section\rw{s:algpro} we prove that
the categorified theory is
well-defined and satisfies the axioms.
In Section\rw{s:braid} we prove that torus braid closures of
tangles within $\Sh$
approximate the \sthm\ of their closures within $\Sot$.
In Section\rw{s:jwpst} we review the properties of \JWp s and prove the relation\rx{eq:relfr} between the complete
and stable \WRT\ invariants of links in $\Sot$.

\subsection*{Acknowledgements}

The work on this paper was started jointly with Mikhail Khovanov.
In particular, we worked out the proofs of Theorems\rw{th:exunsti}
and\rw{th:exsot} together. I am indebted to Mikhail
for numerous discussions, explanations, critical comments and encouragement.
%
%This paper is a spinoff of a joint project with Mikhail
%Khovanov which is
%dedicated to the study of categorification complexes of \cbr s and their
%relation to the categorification of the Witten-Reshetikhin-Turaev
%invariant of links in $\Sot$. I am deeply indebted to Mikhail for
%numerous discussions and suggestions.
%
%I would like to thank Slava Krushkal for sharing the results of
%his ongoing research. I am also indebted to organizers of the M.S.R.I.
%workshop `Homology Theories of Knots and Links' which stimulated
%me to write this paper.

This work is supported by the NSF grant DMS-0808974.

\section{Tangles and links}% in $\Sh$ and in $\Sot$}
\label{s:tlcat}
\subsection{Basic definition}
\subsubsection{Tangles}

All tangles and links in this paper are framed, and in pictures we
assume blackboard framing. We use a notation
\ylee{eq:1.2}
\xygraph{
!{0;/r1.5pc/:}
[u(0.5)]
!{\hover}
!{\hcap}
[u(0.5)l(0.25)]
}
\;\; = \;\;
\xygraph{
!{0;/r1.5pc/:}
[u(0.5)]
!{\xcapv@(0)}
[u(0.45)r(0.23)]
*{\symfr\;\scriptstyle{1}}
[u(1.5)]
}
\yeee
for a framing twist.

For a 3-manifold $\bbM$ with a boundary we have defined the set of
tangles $\tyTngM$, the relative mapping class group $\yMpM$ and
the set of \rtng s $\yTngM=\tyTngM/\yMpM$ in
subsection\rw{ss:ortqft}.
%
%-manifolds $\SgI$, where $\bbI=[0,1]$ and $\xSg$ is either $\IR^2$ or $\St$,
%play an especially important role in our constructions.
%We use a notation $\bbI=[0,1]$.
For a 2-dimensional manifold $\xSg$ being either $\IR^2$ or $\St$, let
$\tyTngSgmn\subset\tyTngSg$ denote the set of \ttngmn s, that is, tangles with $m$
marked points on the `in' boundary $\xSg\times\{0\}$ and with $n$
marked points on the `out' boundary $\xSg\times\{1\}$.
We use a shortcut notation $\tyTngSgn=\tyTngSgnn$ and we use
similar notations for \rtng s with $\tyTng$ replaced by $\yTng$.
Note that $\yMpv{\xIRt}$ is trivial, so
$\yTng(\xIRt)=\tyTng(\xIRt)$.
%
%Denote $\bbI=[0,1]$. The manifold $\StI$ will be of particular
%interest to us.
%
%
%Let $\xSg$ be either $\IR$ or $\St$. We are particularly interested
%in 3-manifolds $\bbM=\SgI$, where $\bbI=[0,1]$.
%
%a 2-dimensional manifold (in this paper
%$\xSg=\xIRt,\St$). A tangle in $\xSgint$ is an embedding
%of a disjoint union of segments and circles into $\xSgint$
%such that the endpoints of segments are mapped into marked points on
%$\xSg\times\{0\}$ and on $\xSg\times\{1\}$.
%
%Let $\yTngSg$ denote the set of all tangles in $\xSgint$
%distinguished up to ambient isotopy
%and let $\yTngSgmn\subset\yTngSg$ denote the subset of \ttngmn s.
%We use a shortcut notation $\yTngSgn=\yTngSgnn$.
%
%A diffeomorphism of $\xSgint$ is called \emph{\yrlt} if it acts
%trivially on its boundary $\xSg\times\{0,1\}$. Let $\yMpSg$ denote
%the \rmcg\ of $\xSgint$. This group acts on $\yTngSg$,
%and the quotient $\tyTngSg=\yTngSg/\yMpSg$ is called the set of
%\rtng s, that is, \rtng s are tangles distinguished up to \yrlhm
%s.

The set $\tyTngSg$ of all tangles in $\SgI$ and its quotient $\yTngSg$
have the structure of a \sgrp: the
multiplication corresponds to the composition of tangles
$\xtauo\tcmp\xtaut$. We use the notation $\tyTngopSg$ to denote
the same \sgrp\ with the opposite order of composing elements.
The \sgrp\ $\tyTngSg$ is `quiver-like',
because the composition of tangles is meaningful only when the
numbers of end-points match, otherwise we set it to
zero; in other words, the \sgrp s $\tyTngSgn$ stand at the
vertices of the quiver and the sets $\tyTngSgmn$ stand at its
oriented edges.

The braid group $\yBrSgn$ is a subgroup of $\tyTngSgn$ and, in
particular, the Artin braid group $\yBrn$ is a subgroup of
$\yTng(\xIRt)$.
We use the
following notations for some important framed braids in $\yBrn$:
\begin{gather}
\vspace*{-2cm}
\brrtpon =
\setlength{\unitlength}{1mm}
\begin{picture}(35,10)(-9,3)
\thl
\put(0,0){\oval(10,10)[tl]}
\put(0,5){\line(1,0){15}}
\put(15,10){\oval(10,10)[br]}
\put(0,0){\line(-1,2){1.8}}
\put(-5,10){\line(1,-2){2.2}}
\put(20,0){\line(-1,2){2.2}}
\put(15,10){\line(1,-2){1.8}}
\put(2.5,7.5){$\cdots$}
\put(7.5,0){$\cdots$}
\put(-6,-4){$\scriptstyle{1}$}
\put(1,-4){$\scriptstyle{2}$}
\put(19,-4){$\scriptstyle{n}$}
\put(19,8){$\circ$}
\put(21,8){$\scriptstyle{1}$}
\end{picture},
\qquad
\brlfpon =
\setlength{\unitlength}{1mm}
\begin{picture}(35,10)(-9,3)
\thl
\put(0,10){\oval(10,10)[bl]}
\put(15,0){\oval(10,10)[tr]}
\multiput(3,5)(2.5,0){5}{\line(-1,0){0.5}}
\put(0,0){\line(1,3){3.5}}
\put(12.5,0){\line(1,3){3.5}}
\put(-1,-4){$\scriptstyle{1}$}
\put(10,-4){$\scriptstyle{n-1}$}
\put(20,-4){$\scriptstyle{n}$}
\put(-6,8){$\circ$}
\put(-8.5,8){$\scriptstyle{1}$}
\put(7,8){$\cdots$}
\put(4,0){$\cdots$}
\end{picture}
\\
\brtwpn = \brrtpon\tcmp\brlfpon =
\setlength{\unitlength}{1mm}
\begin{picture}(35,30)(-9,10)
\thl
\put(0,0){\oval(10,10)[tl]}
\put(0,5){\line(1,0){20}}
\put(20,10){\oval(10,10)[r]}
\multiput(0,15)(2.5,0){8}{\line(-1,0){0.5}}
\put(0,20){\oval(10,10)[bl]}
\put(1,0){\line(0,1){4}}
\put(19,0){\line(0,1){4}}
\put(1,6){\line(0,1){19}}
\put(19,6){\line(0,1){19}}
\put(-5,20){\line(0,1){5}}
\put(7.5,20){$\cdots$}
\put(7.5,0){$\cdots$}
\put(-6,-4){$\scriptstyle{1}$}
\put(1,-4){$\scriptstyle{2}$}
\put(19,-4){$\scriptstyle{n}$}
\put(-6,20){$\circ$}
\put(-9,20){$\scriptstyle{2}$}
\end{picture}
,\qquad
\gobron = \xlrB{ \brrtpon }^n,
\\
\vspace*{4cm}
\nonumber
\\
\nonumber
\end{gather}
the latter braid representing the full rotation of $n$ strands
accompanied by a framing shift.

\subsubsection{Tangles in $\IRtint$}

The most commonly used tangles are tangles in $\IRtint$ and we
refer to them simply as tangles. Tangles may be considered as
morphisms of the tangle category: an object of this category is a
set of $m$ linearly ordered points, an \ttngmn\ $\xtau$ is a
morphism between $m$ points and $n$ points and a composition of
morphisms is a composition of tangles. Thus the set of all tangles
$\yTng=\yTng(\xIRt)$ is the \sgrp\ in this category.
%
%The most commonly used tangles are tangles in $\IRtint$,
%%The tangles in $\IRtint$ are the most commonly used tangles,
%we refer to them simply as tangles
%and denote their \sgrp\ simply as
%%and we use a shortcut notation for their set:
%$\yTng=\yTng(\xIRt)$.

The \sgrp\ $\yTng$ has three special involutions %of $\yTng$:
%%
%\ylee{eq:anthom}
%\xymatrix@C=1cm{\yTng\ar[r]^-{\dsym,\,\dflp} & \yTng }.
%\yeee
%%
%$\yTng \xrightarrow{\;\dsym,\,\dflp\;} \yTng$.
%
\xlee{eq:tinvs}
\dflp,\dsym\colon\yTng\longrightarrow\yTngop,\quad
\yTngmn\longrightarrow\yTngnm,\qquad
\finvm\colon\yTng\longrightarrow\yTng,\quad\yTngmn\longrightarrow\yTngmn
\xeee
satisfying the relations\rx{eq:relinv}.
The first two turn a \ttngmn\ $\xtau$ into a \ttngnm. The flip  $\dflp$ rotates
$\xtau$ by $180^{\circ}$ about an axis which lies in the blackboard plane
and is
perpendicular to the time-line. The duality  $\dsym$
performs the flip and then switches all crossings of the tangle into the
opposite ones. In other words, $\xtaud$ is the mirror image of
$\xtau$ with respect to the plane which is perpendicular to the
time line. $\finvm$ just switches all crossings of a tangle.

%
%
%The duality map $\yTng\xrightarrow{\vee}\yTng$ is a \sgrp\
%anti-homomorphism turning a \ttngmn\ into the \ttngnm\ which
%is its mirror image (that is, $\xtaud$ is constructed by rotating
%$\xtau$ by $180^{\circ}$ about an axis lying in the blackboard plane and
%the perpendicular to the
%time-line, and then switching all crossings into the opposite
%ones).
%
%The flip map $\yTng\xrightarrow{\xflp}\yTng$ is a \sgrp\
%anti-homomorphism turning a \ttngmn\

\subsection{\TLb\ tangles and \xcrmt s}
\subsubsection{Definitions}

A tangle is \emph{\flt} if it can be presented by a diagram
without crossings. A \flt\ tangle is called \TLb\ (\TL) if it is
boundary connected, that is,
if it does not contain disjoint circles.
We denote \TL\ tangles by the letter $\xlam$ and
$\yTL\subset\yTng$ denotes the set of all \TL\ tangles.
% Let $\yTL\subset\yTng$ denote the subset of \TL\ tangles.
%
Since \TL\ tangles have no crossings, then $\finv{\xlam}=\xlam$ and the action of $\dflp$ and $\dsym$ coincide:
$\dflpv{\xlam} = \dsymv{\xlam}$.

The \TL\ \ttngztn s are called \emph{\xcrmtn s}, and their set is
denoted $\yCrn= \yTLv{0,2n}$. All \xcrmt s form the set
$\yCr = \bigcup_{n}\yCrn\subset\yTL$. We denote \crmt s by letters $\xal$
and $\xbet$.
%
%A special subset $\yCr\subset\yTL$ is formed by \ttngztn s known
%as \emph{\crmt s}: $\yCrn = { \yTLv{0,2n} }$. We denote \crmt s by letters $\xal$
%and $\xbet$.

\subsubsection{The structure of \TL\ tangles}

%Throughout the paper we will use the letters $\xal$ and $\xbet$
%%$\xlam$ to denote \TL\ tangles and we will use the letter $\xal$
%to denote  \TL\ \ttngztn s.
%These tangles are also known as \emph{\crmt s.}

Let $\gcupni$ and $\gcapni$,
$1\leq i\leq n-1$, denote the elementary cup and cap tangles:
\ylee{eq:1.1aa}
%\underbrace{
\gcupni=\xygraph{
!{0;/r1.5pc/:}
[r(0.25)u(0.5)]
!{\xcapv@(0)}
[u(0.5)r(1)]
*{\cdots}
[r(01)u(0.5)]
!{\xcapv@(0)}
[r(0.5)u(1)]
!{\vcap-}
[r(1.5)]
!{\xcapv@(0)}
[u(0.5)r(1)]
*{\cdots}
[r(01)u(0.5)]
!{\xcapv@(0)}
[u(1.5)l(3.5)]
*{\scriptstyle{i}}
[r(1)]
*{\scriptstyle{i+1}}
[l(3.5)]
*{\scriptstyle{1}}
[r(6)]
*{\scriptstyle{n}}
}
%}
,
\quad\quad
\gcapni=
\xygraph{
!{0;/r1.5pc/:}
[r(0.25)u(0.5)]
!{\xcapv@(0)}
[u(0.5)r(1)]
*{\cdots}
[r(01)u(0.5)]
!{\xcapv@(0)}
[r(0.5)]
!{\vcap}
[r(1.5)u(1)]
!{\xcapv@(0)}
[u(0.5)r(1)]
*{\cdots}
[r(01)u(0.5)]
!{\xcapv@(0)}
[d(0.5)l(3.5)]
*{\scriptstyle{i}}
[r(1)]
*{\scriptstyle{i+1}}
[l(3.5)]
*{\scriptstyle{1}}
[r(6)]
*{\scriptstyle{n}}
}.
\yeee

For a positive integer $d\leq \frac{n}{2}$
let $\stI=(i_1,\ldots,i_d)$ be a sequence of positive integers
such that $i_k<n-2k+2$ for all $k\in\{1,\ldots,d\}$.
A \emph{\aptg} $\gcapnI$
is a $(n,n-2d)$-tangle which
can be presented as a composition of $d$ tangles of the form
$\gcapv{m}{i}{0.75}$:
\ylee{aes2.1a}
\gcapnI =
\gcapv{n-2d+2}{i_d}{2}\tcmp\cdots\tcmp
\gcapv{n-2}{i_2}{1.5}
\tcmp
\gcapv{n}{i_1}{0.75}.
\yeee
A \emph{\uptg} $\gcupnI$ is defined similarly:
\ylee{aes2.2a}
\gcupnI =
\gcupv{n}{i_1}{-0.75}
\tcmp
\gcupv{n-2}{i_2}{-1.25}
\tcmp
\cdots
\tcmp
\gcupv{n-2d+2}{i_d}{-2.25}
.
\yeee

Let $\cstInt$, where $\zt=n-2\zd$, be the set of all sequences
$\stI$ mentioned above
with an additional condition that $i_{k+1}\geq i_k-1$ for all
$k\in\{2,\ldots,\zd\}$.

The following proposition is obvious:
\begin{proposition}
\label{pr:3}
For every \TL\ \ttngmn\ $\xlam$ there
exists a number $\ztl$ and a unique presentation
\xlee{aes2.3a}
\xlam = \gcupnI\tcmp
\gcapmJ,\qquad
\stI\in\cstIntl,\quad\stJ\in\cstImtl,
\qquad
\ztl=n-2\zdl=m-2\zdlp.
\xeee
\end{proposition}
The number $\ztl$
is called a \emph{\thrd},
it equals the number of strands that go through from the bottom to
the top of the tangle. Obviously, $\ztl\leq m,n$ and the numbers $n-\ztl$,
$m-\ztl$ are even.

If $\ztl=0$ then the \TL\ tangle $\xlam$ is called \emph{\tct}.
%Let $\yTLc\subset\yTL$ denote the set of all \tct\ tangles.
A \tct\ \ttngtmtn\ $\xlam$ has a unique presentation $\xlam =
\xal\tcmp\xbetd$, where $\xal$ is a \TL\ \ttngztn\ and $\xbet$ is
a \TL\ \ttngztm.

\subsection{Tangles in admissible manifolds}

\subsubsection{Tangles in $\Bt$}
The \rmcg\ of $\Bt$ is trivial, and there is an obvious canonical bijective
map
\xlee{eq:bijbt}
\xymatrix{
\yTngzn \ar[r]^-{=} &
\yTngBtn=\tyTngBtn.
}
\xeee

\subsubsection{Tangles in $\StI$}
We call the tangles of $\tyTngSt$ \emph{spherical} tangles and we
call the tangles of the quotient $\yTngSt=\tyTngSt/\yMpSt$ \emph{\rgh} \sphtng s.
%
%For a fixed product structure in $\Sot$ there is an obvious surjective \sgrp\
%homomorphism $\yTng\xrightarrow{\hoam}\yTngSt$.

For a fixed product structure in $\StI$ there are obvious
surjective  homomorphisms (with respect to tangle composition)
$\hoho$ and $\hoam$:
%%
%\ylee{eq:hom}
%\xymatrix@C=2cm{
%\yTng \ar[r]^-{\hoam}
%\ar@/^2pc/[rr]^-{\hoho}
%&
%\yTngSt \ar[r]^-{/\yMpSt}
%&
%\tyTngSt
%}
%\yeee
%%
%
\ylee{eq:hom}
\xymatrix@R=0.3cm{
& \yTng
\ar[ld]_-{\hoho}
\ar[rd]^-{\hoam}
\\
\tyTngSt
\ar[rr]_-{\xdummy/\yMpSt}
&&
\yTngSt
}
\yeee

\begin{theorem}
\label{th:rtng}
The kernel of the homomorphism $\yTng\xrightarrow{\hoho}\tyTngSt$
is a congruence on $\yTng$ generated by the
braid $\brtwpn$, that is, an equivalence
$\xlrb{\xtauo\sim\xtaut\;\Leftrightarrow\;\hoam(\xtauo) \isotp\hoam(\xtaut)}$
%coincides with
is
the
minimal equivalence which includes the relation
$\brtwpn\sim\gidbrn$ and respects the tangle
composition.
\end{theorem}

The \rmcg\ of $\Stint$ is $\ZZ_2$. Let $\xtw$ (twist) denote its
generator. It acts on spherical tangles by composing them with the
braid $\gobron$: if $\xsigm$ is a spherical \ttngmn, then
\ylee{eq:compbr}
\xtw(\xsigm) = \hoam\xlrB{\gobron}\tcmp\xsigm.
%= \hoam\xlrB{\xtau\tcmp\gobrom}.
\yeee
Hence we get the following extension of Theorem\rw{th:rtng}:
\begin{theorem}
\label{th:rgheqv}
The kernel of the homomorphism
$\yTng\xrightarrow{\hoam}\yTngSt=\tyTngSt/\yMpSt$
%, which is a composition of $\hoam$ and the quotient homomorphism,
is the congruence generated
by the braids $\brtwpn$ and $\gobron$.
\end{theorem}

The homomorphisms $\hoho$ and $\hoam$ transfer the involutions $\dsym$, $\dflp$
and $\finvm$ from $\yTng$ to $\tyTngSt$
and to $\yTngSt$.

%\subsubsection{Circular closure of tangles}
\subsubsection{Links in $\Sh$ and in $\Sot$}
For a 3-manifold $M$ let $\yLnkM$ denote the set of links in
$M$ up to ambient isotopy and let $\tyLnkM$ be the set of links in
$M$ up to diffeomorphism:
$\tyLnkM=\yLnkM/\yMpStM$, where $\yMpStM$ is the mapping class
group of $M$.
%
% be the
%quotient by the action of the mapping class group $\yMpStM$ of
%$M$. In other words, $\tyLnkM$ is the set of links up to
%diffeomorphisms of $M$.
We use a shortcut notation $\yLnk=\yLnkSh=\tyLnkSh$ for links in
$\Sh$.

Let $\clShv{\xtau}$ denote a circular closure of a \ttngnn\ within
$\Sh$ and let $\clSotv{\xsigm}$ denote a circular closure of a
\sphtng\ within $\Sot$. Sometimes we also use an abbreviated
notation $\clSotv{\xtau} = \xlrb{\lSot{\hoam(\xtau)}}$. Thus we
have two closure maps
\ylee{eq:tmps}
\xymatrix@R=0.5cm{
& \yTngnn \ar[dl]_-{\clShv{\xdummy}} \ar[dr]^-{\clSotv{\xdummy}}
\\
\yLnkSh
&&
\yLnkSot
}
\yeee
Both are trace-like and invariant under the involution
$\dflp$:
%
%\xlee{eq:clinv1}
\begin{gather}
\label{eq:clinv1}
\clstv{\xtauo\tcmp\xtaut}=\clstv{\xtaut\tcmp\xtauo},
\\
\label{eq:clinv2}
\clstv{\dflpv{\xtau}} = \clstv{\xtau},
\end{gather}
%\xeee
%
where $\ast$ stands for $\Sh$ or $\Sot$.

%Let $\clShv{\xtau}$ denote a circular closure of a \ttngnn\
%within $\Sh$. Thus, there is a trace map
%$\yTng\xrightarrow{\clShv{-}}\yLnk$.
%%, where $\linsh$ is the set of
%%framed unoriented links in $\Sh$.
%
%Let $\clSotv{\xsigm}$ denote a circular closure of a spherical
%tangle $\xsigm$ within $\Sot$. We will also use an abbreviated
%notation $\clSotv{\xtau} = \xlrb{\lSot{\hoam(\xtau)}}$. This
%closure generates two maps forming a commutative diagram:

The quotients over the action of the mapping class groups
combine into the following commutative diagram:
\ylee{eq:cmsqq}
\xymatrix@C=2cm{
\yTng
\ar[r]^-{\hoho}
\ar[rd]_-{\hoam}
&
\tyTngSt
\ar[r]^-{\clSotv{\xdummy}}
\ar[d]
&
\tyLnkSot
\ar[d]
\\
&
\yTngSt
\ar[r]^-{\clSotv{\xdummy}}
&
\yLnkSot
}
\yeee

The mapping class group of $\Sot$ is $\ZZ_2\times\ZZ_2$ with
generators $\xtw$ (twist) and $\dflp$ (flip), that is,
\ylee{eq:mpgnac}
\xtw\clSotv{\xtau} = \xlrB{\lSot{\gobron\tcmp\xtau}},\qquad
\dflpv{\clSotv{\xtau}} = \clSotv{\xtauf}.
\yeee
\begin{theorem}
\label{th:soeq}
The equivalence relation
$\xlrb{\xsigmo\sim\xsigmt\;\Leftrightarrow\clSotv{\xsigmo}\homeom\clSotv{\xsigmt}}$
within $\tyTngSt$
is generated by the relation
$\xsigmo\tcmp\xsigmt\sim\xsigmt\tcmp\xsigmo$. The same equivalence
within $\yTngSt$ is generated by two relations: $\xsigmo\tcmp\xsigmt\sim\xsigmt\tcmp\xsigmo$
and $\dflpv{\xsigmo} \sim\xsigmo$.
\end{theorem}
%

%\nidm
%%
%%
%\begin{theorem}
%\label{th:soeq}
%The \Sotec\ relation is generated by three `moves':
%\begin{enumerate}
%\item $\xtauo\tcmp\xtaut \seqr \xtaut\tcmp\xtauo$ for any \ttngnm\ $\xtauo$
%and any \ttngmn\ $\xtaut$;
%\item $\brtwptn\tcmp\xtau \seqr \xtau$, where $\xtau$ is any \ttngnn;
%
%\item $\gobron\tcmp\xtau\seqr\xtau$, where $\xtau$ is any \ttngnn.
%
%\end{enumerate}
%\end{theorem}
%\nidm

\section{Quantum invariants of links and tangles}

%\subsection{Kauffman bracket, \TLa\ and \JWp s}

\subsection{The \TLa}%\ and \JWp s}

A detailed account of \TLa s and \JWp s can be found in the book
by Lou Kauffman and Sostenes Lins\cx{KaLins}. Here we summarize
relevant facts and set our notations.

\subsubsection{Definitions}

%A tangle is \emph{\flt} if it can be presented by a diagram
%without crossings. A \flt\ tangle is connected if it does not
%contain disjoint circles.  $\rTL$ denotes the set of connected
%\flt\ tangles.
A \TLa\ $\cTL$ is a ring generated, as a module, freely by \TL\
tangles, that is, by the
elements $\clam$, $\xlam\in\yTL$ over the ring $\Zqqi$.
The product in $\cTL$ corresponds to the
composition of tangles and is denoted by $\tcmp$. If the numbers of endpoints do not match,
then the product is defined to be zero. Any disjoint circle appearing
in the composition is replaced by the factor $-(q + \qi)$:
\xlee{eq:ae1.1}
%\widehat{\lcir}
\Bsymalg{\lcir}
 = -(\qpqi).
\xeee

The Kauffman bracket relation
\xlee{eq:1.4}
\Bsymalg{\xcrsp}
\;\;=\;\;
\qvh\;\;
\Bsymalg{\xpver}
\;\;+\;\;
\qvmh\;\;
\Bsymalg{\xphor}
\xeee
associates a $\cTL$ element
\xlee{eq:1.3a}
\ctau = \sltl \xcalt\, \clam,\qquad
%\xcalt\in\Zqqi.\quad
\xcalt = \sum_{i\in\ZZ}\xcalit\,q^i
\xeee
to any
tangle $\xtau$.
Note that
%in our conventions
the element $\ctau$ depends on the
framing of $\xtau$:
\vspace*{-0.5cm}
\xlee{ae1.2a}
\Bsymalg{\xvfro\hspace*{-0.2cm}}
\;\; = \;\;
-q^{\frac{3}{2}}\;\;
\Bsymalg{\;\xvert\hspace*{-0.5cm}}.
\xeee
%
%%
%\ylee{eq:1.4a}
%\ctau = \sltln \xcalt\, \clam,\qquad
%%\xcalt\in\Zqqi.\quad
%\xcalt = \sum_{i\in\ZZ}\xcalit\,q^i
%\yeee
%%

The homomorphisms\rx{eq:tinvs} extend to the \TLa
\xlee{eq:tinvtl}
\dflp,\dsym\colon\cTL\longrightarrow\cTLop,\quad
\cTLmn\longrightarrow\cTLnm,\qquad
\finvm\colon\cTL\longrightarrow\cTL,\quad
\cTLmn\longrightarrow\cTLmn
\xeee
by their action
on generating \TL\ tangles, the involutions $\dsym$ and $\finvm$ being
accompanied by the involution of the base ring
$\finvm\colon\Zqqi\rightarrow\Zqqi$,
%defined by its action on $q$:
$\finv{q}=\qi$.
This extension is well-defined, because the involutions preserve
the Kauffman bracket relation\rx{eq:1.4} as well as the unknot
invariant formula\rx{eq:ae1.1}.

%
%
%The flip involution $\dflp$ extends in an obvious
%way, because it preserves the Kauffman bracket
%relation\rx{eq:1.4}. The duality involution $\dsym$ also extends
%to \TLa, if it is accompanied by the involution of the base ring
%$\Zqqi\xrightarrow{\dsym}\Zqqi$ defined by its action on $q$:
%$\qsym=q^{-1}$.

\subsubsection{Subrings of the \TLa\ and their \bmdl s}

As a $\Zqqi$-module, $\cTL$ contains submodules $\cTLmn$ ($n-m$ is even) generated
by \ttngmn s. The `diagonal' submodules $\cTLnn$, which we denote
for brevity as $\cTLn$, are, in fact, subrings of $\cTL$.
Obviously, $\cTLmn$ is a module over the ring $\cTLn\otimes\cTLmop$.

%The ring $\cTL$ contains subrings $\cTLn$ generated by \ttngnn s
%and submodules $\cTLmn$ generated by
The empty tangle $\xnot$ is the only \TL\ \ttngzz,
%There is only one \TL\ \ttngzz: the empty tangle $\xnot$,
hence
the ring $\cTLz$ is canonically isomorphic to $\Zqqi$.
A \ttngzz\ $\xL$ is just a link, and the corresponding Kauffman
bracket equals its Jones polynomial: $\aL = \JnLq$. The cyclic
closure of an \ttngnn\ in $\Sh$ produces a trace on the ring
$\cTLn$: $\cTLn\xrightarrow{(\lSh{-})}\cTLz=\Zqqi$.

\subsection{The \JWp s}
\label{sss:jwp}
We use two versions of the \TLa\ defined over the rings other
than $\Zqqi$. The \TLa\ $\QcTL$ is generated by \TL\ tangles over
the field $\Qq$ of rational functions of $q$ and
$\cTLpinf$ is the \TLa\ over the field $\Zsqqi$ of
Laurent power series. An injective homomorphism $
\Qq\hookrightarrow\Zsqqi$ generated by the Laurent series expansion at $q=0$
%in powers of $q$
produces an injective homomorphism of \TLa s
$\QcTL\hookrightarrow\cTLpinf$.

The algebra $\QcTLn$ contains
central
mutually orthogonal idempotent elements $\jwpxnm$ ($0\leq m\leq n$, $n-m\in 2\ZZ$) known as \JWp s:
\xlee{eq:1.5}
\jwpxnm\tcmp \jwpxnmp =
\begin{cases}
\jwpxnm, & \text{if $m=m\p$,}
\\
0, & \text{if $m\neq m\p$,}
\end{cases}
\qquad
\smmnev \jwpxnm = \gidbrn.
\xeee
These projectors are defined by the relations\rx{eq:1.5} and by
the conditions
\ylee{eq:1.5a}
\acapnJ\tcmp\jwpxnm = 0,\qquad\text{if $d>n-m$},
\yeee
where $d$ is the number of caps in $\acapnJ$.
We denote the images of \JWp s
$\jwpxnm$ under the homomorphism $\Qq\hookrightarrow\Zsqqi$ by the same symbols
$\jwpxnm$.

All projectors are invariant under the involutions\rx{eq:tinvtl}:
% $\dsym$ and $\dflp$:
%
\xlee{eq:invpr}
\dsymv{\jwpxnm} = \dflpv{\jwpxnm} = \bjwpxnm = \jwpxnm.
\xeee

Projectors $\jwpxnm$ are central in the following sense:
\xlee{eq:cntpr}
\jwpxnk\tcmp \xx = \xx\tcmp\jwpxmk\qquad\text{for any $\xx\in\QcTLmn$.}
\xeee
Consequently, the sums
$\jwpxstm = \sum_{k=0}^\infty \jwpxvv{m+2k}{m}$ form a complete set of mutually orthogonal central
projectors of the \TLa\ \TL:
$\sum_{m=0}^\infty \jwpxstm = \sum_{n=0}^\infty \gidbrn$.

The most famous and widely used projector is $\jwpxnn$ denoted
simply as $\jwpn$, but here we will need other projectors too,
especially $\jwpxtnz$.

Split \TL\ tangles have the form $\xlam = \xbet\tcmp \xalf$, where
$\xal$ and $\xbet$ are \xcrmt s. If $\xlam$ is \tct, then
$\xlam\tcmp\xlamp$ and $\xlamp\tcmp\xlam$ are \tct\ for any \TL\
tangle $\xlamp$. Hence \tct\ \TL\ tangles form a two-sided ideal
$\cfQTLct\subset\cfQTL$ with submodules
$\cfQTLcttmtn\subset\cfQTLct$  generated by
\ttngtmtn s. Obviously,
\xlee{eq:tlsplit}
\cfQTLcttmtn =
\cfQTLztn\otQq\cfQTLtmz
\xeee

An alternative definition of the \JWp\ $\jwpxtnz$ comes from
the following theorem:
\begin{theorem}
\label{th:unqprj}
There exists a unique element $\jwpxtnz\in\cfQTLcttn$ such that
\def\xgam{\xal}
\xlee{eq:cndpr}
\jwpxtnz\tcmp\symalg{\xgam}=\symalg{\xgam}
\xeee
for
any $\xgam\in\yCrn$. This element is idempotent.
\end{theorem}
\begin{proof}
The elements $\psbaf$ form a basis of $\cfQTLcttn$, hence
an element $\jwpxtnz\in\cfQTLcttn$ has a unique presentation
$\jwpxtnz = \sum_{\xal,\xbet\in\yCrn}\cxzab\psbaf$. The
condition\rx{eq:cndpr} determines the coefficients $\cxzab$:
\xlee{eq:1.5b1}
\jwpxtnz =
\sabTLztn\mtBiab
\symalg{\xbet\tcmp\xalf},
\xeee
where
a symmetric matrix $(\mtBab)_{\xal,\xbet\in\yTLztn}$ is defined
by the formula
\xlee{eq:1.5b}
\mtBab = \symalg{\xalf\tcmp\xbet} = (-q-\qi)^{\nccab}.
\xeee
It is easy to verify that the element\rx{eq:1.5b1} is idempotent.
\end{proof}
%
%
%The latter projector has a simple
%expression in terms of \crmt s.
%%\TL\ \ttngztn s also known as crossingless matchings.
%Define
%where
%%$\xalf$ is a \TL\ \ttngtnz\ which is a mirror image of $\xal$ and
%$\nccot$ is the number of disjoint circles in
%the \ttngzz\ $\xalf\tcmp\xbet$. Then
%
%Split \TL\ tangles have the form $\xlam = \xbet\tcmp \xalf$, where
%$\xal$ and $\xbet$ are \xcrmt s. If $\xlam$ is \tct, then
%$\xlam\tcmp\xlamp$ and $\xlamp\tcmp\xlam$ are \tct\ for any \TL\
%tangle $\xlamp$. Hence \tct\ \TL\ tangles form a two-sided ideal
%$\cfQTLct\subset\cfQTL$ which is the image of the central projector
%$\jwpxstz = \sum_{n=0}^{\infty} \jwpxtnz$. The total projector
%$\jwpxstz$ is the unit of $\cfQTLct$ considered as an algebra.

%As an algebra, $\cfQTLct$ has a unit element $\xIdv{asdf}$

%Define $\cTLcttmtn = \cTLtmtn\cap\cTLct$. Obviously,
%
Let $\cfQTLcttmtn\subset\cfQTLct$ be the submodule generated by
\ttngtmtn s. Obviously,
\xlee{eq:tlsplitx}
\cfQTLcttmtn =
\cfQTLztn\otQq\cfQTLtmz
\xeee
and the homomorphism
\ylee{eq:prhm}
\cfQTLtmtn \xrightarrow{\;\jwphxstz\;}\cfQTLcttmtn,\qquad
\jwphxstz(\xx) = \jwpxtnz\tcmp\xx = \xx\tcmp\jwpxtmz
\yeee
is surjective.

\subsection{A \xstbl\ \fWRT\ theory}

Let us describe in more detail the \sttWRTt\ that we want to
categorify.

\subsubsection{Basic structure}

As we mentioned at the end of subsection{ss:ortqft},
its base field $\rF$ is the field $\Qq$ of rational functions of
$q$. The involution $\finvm$ is defined by its action on $q$:
$\finv{q} = q^{-1}$. Admissible boundaries are 2-spheres $\Stn$ with
$2n$ marked points. Admissible manifolds are disjoint
unions of four 3-manifolds: $\Bt$, $\StI$, $\Sh$ and $\Sot$. In
fact, all admissible manifolds are generated by $\Bt$ and $\StI$
through disjoint union and gluing. Therefore the \TQFT\ is
defined by the choice of Hilbert spaces $\cH(\Stn)$ and \stmp s
for \rtng s in $\Bt$ and in $\StI$, because the other \stmp s are
determined by the axioms.

\subsubsection{Hilbert spaces}

To a 2-sphere with $2n$ marked points the \sttoyt\ associates a $\Qq$-module
$\cHStn=\cfQTLztn$,
%$\cTL\cfQTL$
which, by definition, is a free $\Qq$-module
generated by (\Kbr s of) \xcrmtn s. The module $\cHvStn = \cfQTLtnz$ is
canonically dual to $\cHStn$, the pairing being defined by the
\Kbr\ of the composition of tangles:
$\symalg{\xalf}\xprng\symalg{\xbet} = \symalg{\xalf\tcmp\xbet}$.
The involutions\rx{eq:invmd} are the corresponding
involutions\rx{eq:tinvtl} restricted to $\cfQTLtnz$.

%
%\nidm
%
%Then the involutions
%$\cTLztn \xrightarrow{\;\dsym,\,\dflp\;} \cTLtnz$ identify the
%module $\cHStn$ with its dual, thus defining two non-degenerate
%bilinear forms on $\cHStn$:
%%
%\ylee{eq:blf}
%\blndv{\lal}{\lbet} = \symalg{\xald\tcmp\xbet},\qquad
%\blnfv{\lal}{\lbet} = \symalg{\xalf\tcmp\xbet}.
%\yeee
%%
%
%\nidm
%
%Suppose that the boundary $\del\bbM$ of an oriented 3-manifold $\bbM$
%consists of connected components, each being diffeomorphic to
%$\St$, $\nmkin$ of those being oriented inwards and $\nmkout$
%being oriented outwards.  Denote the numbers of marked points on
%these components as $\bfm=m_1,\ldots,m_{\nmkin}$ and
%$\bfn=n_1,\ldots,n_{\nmkout}$. Then the \sttWRTt\ associates the
%Hilbert module
%%
%\xlee{eq:hmdl}
%\cHdM = \cHStv{n_1}\otQq\cdots\otQq\cHStv{n_{\nmkout}}\otQq
%\cHvStv{m_1}\otQq\cdots\otQq\cHvStv{m_{\nmkin}}
%\xeee
%%
%to the boundary of $\bbM$ and also
%the provides the \stmp
%%
%\ylee{eq:tqftm}
%%\xymatrix{
%%\tyTngbmnM
%%\ar[r]^-{\qftbrv{\xdummy}{\bbM}} &
%%\cHdM
%%}
%\qftbrv{\xdummy}{\bbM}\colon\tyTngbmnM\longrightarrow\cHdM.
%\yeee
%%
%We may drop the index $\bbM$ in the notation
%$\qftbrv{\xdummy}{\bbM}$ of the \stmp\ when it is clear what is $\bbM$.
%
%\nidm

\subsubsection{A \stmp\ for tangles in $\Bt$ and in $\StI$}
If we orient the boundary of an oriented 3-ball $\Bt$ in the `out' direction and put $2n$ marked points
on it, then using the bijection\rx{eq:bijbt} we set
$\psymalg{\xtau,\Bt}
%\qftbrBt{\xtau}
= \ctau$.
%
%If the boundary of $\Bt$ is oriented in the `in' direction, then
%we set $\qftbrBt{\xtau} = \psymalg{\xtauf}$.

%\subsubsection{A \stmp\ for tangles in $\StI$}
We set the orientation of the boundary component $\St\times\{0\}$
as `in', the orientation of the boundary component $\St\times\{1\}$ as
`out' and put $2m$ and $2n$ marked points on them.
%Then, according to the general rule\rx{eq:hmdl},
According to \ex{eq:tlsplit},
%$\cHStv{n}\otQq\cHvStv{m} = \cfQTLcttmtn$,
%
\xlee{eq:hmdsi}
\cHStv{n}\otQq\cHvStv{m} = \cfQTLcttmtn,
%\cH\big(\del\StItmtn\big) = \cHStv{n}\otQq\cHvStv{m} = \cfQTLcttmtn.
\xeee
hence the \stmp\ has the form
%We define the \stmp
%
\xlee{eq:stih}
%\sphsymalg{\xdummy}\colon \yTngSttmtn\longrightarrow\cfQTLcttmtn
\sphsymalg{\xdummy}\colon \yTngSt\longrightarrow\cfQTLct.
\xeee
%
%%$\sphsymalg{\xdummy}\colon \tyTngSttmtn\longrightarrow\cfQTLcttmtn$
%must be a homomorphism with respect to the tangle composition
%$\tcmp$. We define the map
The gluing axiom implies that it must be a homomorphism
and we define it by the following theorem:
\begin{theorem}
\label{th:trend1}
There exists a unique homomorphism\rx{eq:stih} % spherical Kauffman bracket homomorphism\\
%$\tyTngSt\xrightarrow{\sphsymalg{-}}\cTLct$ and a unique map
%$\tyLnkSot\xrightarrow{\;\JnqSotv{-}\;}\Zqqi$
such that the
diagram
%
%\marginpar{this diagram has to be removed and preceding formula
%corrected}
%\xlee{eq:1.5b2xx}
%\vcenter{
%\xymatrix@C=1.5cm{
%\yTngtmtn \ar@{->>}[r]^-{\hoho}
%\ar[d]^-{\symalg{-}}
%&
%\tyTngSttmtn
%\ar[d]^-{\sphsymalg{-}}
%\\
%\cTLtmtn
%\ar[r]^-{\jwphxstz}
%&
%\cfQTLcttmtn
%}}
%\xeee
%%
%
\xlee{eq:1.5b2}
\vcenter{
\xymatrix@C=1.5cm{
\yTng \ar@{->>}[r]^-{\hoam}
\ar[d]^-{\symalg{\xdummy}}
&
\yTngSt
\ar[d]^-{\sphsymalg{\xdummy}}
\\
\cTL
\ar[r]^-{\jwphxstz}
&
\cfQTLct
}}
\xeee
is commutative, that is, for any $\xtau\in\yTngtmtn$
\xlee{eq:cmsq1}
\sphsymalg{\hoam(\xtau)} =  \jwpxtnz\tcmp\symalg{\xtau}.
\xeee
\end{theorem}
\begin{proof}
Since the homomorphism $\hoam$ is surjective, it is sufficient to
show that $\ker\hoam$ is `untangled' by the composition of the
\Kbr\ $\symalg{\xdummy}$ and the projector $\jwphxstz$. According
to Theorem\rw{th:rgheqv}, $\ker\hoam$ is generated by the braids
$\yvspoh\brtwptn\;$ and $\gobrotn$, so we have to show that
$\jwpxtnz\tcmp\ctau = \jwpxtnz$ for $\xtau$ being either of
these $\ker\hoho$-generating braids.
%
%Let $\xtau$ be the braid $\yvspoh\brtwptn\;$ or the braid $\gobrotn$.
%Theorem\rw{th:rgheqv} says that these braids generate the kernel
%of the homomorphism $\hoho$.
%Therefore the existence of the homomorphism $\sphsymalg{-}$
%follows from the relations $\jwphxstz\xlrb{\ctau} =
%\jwphxstz\xlrB{\bsymalg{\gidbrtn}}$.
%
The latter relations follow from the
the formula\rx{eq:1.5b1} and from the
isotopies $\xalf\tcmp\xtau\isotp\xalf$ %and  $\xtau\tcmp \xal \isotp\xal$
which hold true for any \TL\ \ttngztn\ $\xal$.
\end{proof}

%It is easy to see that the map\rx{eq:stih} is a homomorphism with
%respect to the tangle composition $\tcmp$, as required by the
%gluing axiom.

\subsubsection{The invariant of links in $\Sh$ and in $\Sot$}
\label{sss:relstwrt}
Since $\Sh$ can be constructed by gluing together two 3-balls
$\bbB^3$, the gluing property dictates that the \toy\ invariant of
a link $\xL$ in $\bbB^3$ is its \Kbr: $\Btsymalg{\xL} =
\symalg{\xL}$.

%\subsubsection{An invariant of links in $\Sot$}

If the manifold $\StI$ has $2n$ marked points on both boundary
components, then the corresponding Hilbert module\rx{eq:hmdsi}
becomes the module of endomorphisms:
\xlee{eq:ahmdsi}
\cH\big(\del(\StI)\big) =  \cfQTLcttn =\End_{\Qq}(\cfQTLztn).
\xeee
\begin{theorem}
\label{th:trTL}
A trace of an element $\xx\in\cfQTLcttntn$, considered as an
endomorphism of $\cfQTLztn$, is equal to its circular closure
within $\St$:
\xlee{eq:trTL}
\Tr_{\cfQTLztn} \xx = \symalg{\lSh{\xx}}.
\xeee
\end{theorem}
\begin{proof}
It is sufficient to verify the formula for $\xx =
\xbet\tcmp\xalf$, where $\xal$ and $\xbet$ are \TL\ \ttngztn s.
The only diagonal element in the matrix of $\xx$ in the basis of
\TL\ tangles comes from the tangle $\xbet$, and the corresponding
matrix element is $\symalg{\xalf\tcmp\xbet}$, which is equal to the
\Kbr\ of the closure of $\xbet\tcmp\xald$ within $\Sh$.
\end{proof}

If a \rgh\ link $\xL\subset\Sot$ is presented as a circular closure of a
\rgh\ \sphtng\ $\xsigm\in\yTngSttntn$, then the gluing axiom of a \TQFT\ says that
the invariant of $\xL$ must be equal to the trace of the \stmp\ of
$\xsigm$:
\xlee{eq:trend}
\Sotsymalg{\xsigm} = \Tr_{\cfQTLztn} \sphsymalg{\xsigm}.
\xeee
The presentation $\xsigm=\hoam(\xtau)$ together with
\eex{eq:trTL} and\rx{eq:cmsq1}, allows us to recast \ex{eq:trend} in the following
equivalent form:
\xlee{eq:trend1}
\Sotsymalg{\hoam(\xtau)} =
\xbrb{\lSh{\jwpxtnz\tcmp\symalg{\xtau}}}.
\xeee
In other words, the map
\xlee{eq:Sotmp}
\Sotsymalg{\xdummy}\colon\yLnkSot \rightarrow\Qq
\xeee
should provide the commutativity of the right square of the
diagram
\xlee{eq:trend2}
\label{eq:1.5b3}
\vcenter{\xymatrix@C=2cm@R=1.5cm{
\yTngtntn \ar@{->>}[r]^-{\hoam}
\ar[d]^-{\symalg{-}}
&
\yTngSttntn
\ar[d]^-{\sphsymalg{-}}
\ar@{->>}[r]^-{ (\lSot{-})}
&
%\linsot
\yLnkSot
\ar[d]^-{\Sotsymalg{-}}
\\
\cfQTLtntn
\ar[r]^-{\jwpxtnz\tcmp\xdummy}%{\jwphxtnz}
%\ar@/_1.5pc/[rr]_-{\xbrb{\lSh{\jwpxtnz\tcmp\symalg{\xdummy}}}}
&
\cfQTLcttntn
\ar[r]^-{(\lSh{-})}
&\Qq
}}
\xeee
(the commutativity of the left square is a particular case of
Theorem\rw{th:trend1}).

\begin{theorem}
There exists a unique map\rx{eq:Sotmp} such that the right square
of the diagram\rx{eq:trend2} is commutative
\end{theorem}
\begin{proof}
Since the map $(\lSot{-})$ is surjective, then according to
Theorem\rw{th:soeq}, we have to check relations
\ylee{eq:chrl}
\lrbc{\lSh{\sphsymalg{\xsigmo\tcmp\xsigmt}}} =
\lrbc{\lSh{\sphsymalg{\xsigmt\tcmp\xsigmo}}},\quad
\lrbc{\lSh{\sphsymalg{\dflpv{\xsigm}}}} =
\lrbc{\lSh{\sphsymalg{\xsigm}}},
\yeee
where $\xsigm$, $\xsigmo$ and $\xsigmt$ are arbitrary
\sphtng s. Present them as $\hoho$ images of ordinary tangles,
then \ex{eq:cmsq1} reduces these relations to
\ylee{eq:chrl1}
\xbrb{\lSh{\jwpxtnz\tcmp\psymalg{\xtauo}\tcmp\psymalg{\xtaut}}} =
\xbrb{\lSh{\jwpxtnz\tcmp\psymalg{\xtaut}\tcmp\psymalg{\xtauo}}},
\quad
\xbrb{\lSh{\jwpxtnz\tcmp\psymalg{\dflpv{\xtau}}}}=
\xbrb{\lSh{\jwpxtnz\tcmp\psymalg{\xtau}}}.
\yeee
The first of this relations follows easily from the trace-like
property\rx{eq:clinv1} of the circular closure of a tangle within $\Sh$
and from the commutativity\rx{eq:cntpr} of projectors, while the
second one follows from the invariance of the projector\rx{eq:invpr}
and of the circular closure\rx{eq:clinv2} under the involution $\dflp$.
\end{proof}

\begin{theorem}
\label{pr:Lpol}
The link invariant $\Sotsymalg{\xL}$ defined by \ex{eq:trend1} is a Laurent polynomial of
$q$.
\end{theorem}

Formula\rx{eq:trend1} indicates that relation\rx{eq:relfr} between
the full and \xstbl\ \WRT\ invariants of links in $\Sot$ follows
from the following theorem:
\begin{theorem}
\label{th:relfr}
For any \ttngtntn\ $\xtau$ there is a relation
\xlee{eq:relstwrt1}
\ZrtSot =
\xbrb{\lSh{\jwpxtnz\tcmp\symalg{\xtau}}}|_{q=\exp(i\pi/\xr)}
\xeee
if $\xr\geq n+2$.
\end{theorem}
This theorem is well-known, but we will provide its proof in
subsection\rw{ss:relstwrt} for completeness.

\section{Categorification}

\subsection{A triply graded categorification of the Jones
polynomial}
\label{sss:trgr}
In\cx{Kh1} M.~Khovanov introduced a categorification of
the Jones polynomial of links. To a diagram $\xL$ of a
link he associates a complex of graded vector spaces over $\IQ$
%modules
%
\xlee{ae1.5}
\dL = \lrbc{ \cdots \rightarrow \dLi \rightarrow \dLio\rightarrow\cdots}
\xeee
so that
if two diagrams represent the same link then the corresponding
complexes are homotopy equivalent, and the graded Euler
characteristic of $\dL$ is equal to the Jones polynomial of $\xL$.
As a complex of vector spaces, $\dL$ is homotopically equivalent
to its homology known as Khovanov homology of the link $\xL$:
$\Hmb\xlrb{\symcat{\xL}} = \HKhb(\xL)$.

Thus, overall, the complex\rx{ae1.5} has two gradings:
the first one was the homological grading of the complex, the corresponding degree being
equal to $i$, and the
second grading was the grading related to powers of $q$.
%the first one was
%the grading related to powers of $q$ and the second one was the
%homological grading of the complex itself, the corresponding
%degree being equal to $i$.
%
In this paper we adopt a slightly different convention which is
convenient for working with framed links and tangles. It is
inspired by matrix factorization categorification\cx{KR1} and its
advantage is that it is no longer necessary to assign orientation to
link strands in order to obtain the grading of the categorification
complex\rx{ae1.5} which would make it invariant under the second
Reidemeister move.

To a framed link
diagram $\xL$ we associate a $\ZZ \oplus\ZZ\oplus\ZZ_2$-graded complex\rx{ae1.5} with
degrees $\dgo$,  $\dgt$ and $\dgh$.
The first two gradings are of the same nature as in\cx{Kh1} and, in
particular, $\dgo\dLi=i$. The third grading is an inner grading of
chain modules defined modulo 2 and of homological
nature, that is, the homological parity of an element of $\dL$,
which affects various sign factors, is the sum of $\dgo$ and
$\dgh$. Both homological degrees are either integer or
half-integer simultaneously, so the homological parity is integer
and takes values in $\ZZ_2$. The $q$-degree $\dgt$ may also take
half-integer values.

%The
%first grading is the homological grading $i$: $\dgo\dLi=i$.
%The second and third gradings are inner gradings of individual chain modules
%$\dLi$. The second grading
%is the $\ZZ$-grading associated with powers of $q$. The third grading is
%of homological nature and it is defined only modulo 2. The first
%and third degrees  take integer or half-integer values simultaneously, so
% their sum, which is the total homological grading, is integer
%and takes values in $\ZZ_2$.

Let $\tgrshv{l}{m}{n}$ denote the shift of three degrees by $l$,
$m$ and $n$ units respectively\footnote{
Our degree shift is defined in such a way that if an object $M$
has a homogeneous degree $n$, then the shifted object $M[1]$ has a
homogeneous degree $n+1$.
}. We use abbreviated notations
$$
\tgrsshv{m}{l} = \tgrshv{m}{l}{0},\qquad
\hqshv{m} = \tgrsshv{m}{m},\qquad
\qshv{m} = \tgrshv{m}{0}{0},\qquad
\qtshv{m} = \tgrshv{m}{0}{m}.
$$
%as well as the following `power' notation:
%$$
%\tgrshv{m}{l}{n}^k = \tgrshv{km}{kl}{kn}.
%$$

After the grading modification, the categorification
formulas of\cx{Kh1} take the following form:
the module associated with an unknot is still $\ZZ[x]/(x^2)$ but with
a different degree assignment:
\begin{eqnarray}
\label{ae1.6}
&
\Bsymcat{\lcir}=\ZZ[x]/(x^2)\,
%[0,-1,1]
\tgrshv{-1}{0}{1},
%\qquad
\\
&\dgt 1 = 0, \quad \dgt x = 2,
\quad\dgo 1 = \dgo x = \dgh 1=\dgh x =0,
\end{eqnarray}
and the categorification complex of a crossing is the same as
in\cx{Kh1} but with a different degree shift:
\xlee{ae1.7}
%\xygraph{
%!{0;/r1.5pc/:}
%[u(0.5)]
%!{\xoverv}
%[u(1.5)r(0.53)]
%*{\smcat}
%}
\Bsymcat{\xcrsp}
\;\;=\;\;
\Bigg(\;\;
%\vcenter{{
%%%%%%\begin{CD}
%\xygraph{
%!{0;/r1.5pc/:}
%[u(0.5)]
%!{\xunoverv}
%[u(1.5)r(0.5)]
%*{\smcat}
%}
\Bsymcat{\xpver}
\;\tgrshv{\vthf}{-\vthf}{\vthf}
%\;\;+\;\;
%%%%%@>\xmrf>>
\xrightarrow{\;\;\;\;\xmrf\;\;\;\;}
%\xygraph{
%!{0;/r1.5pc/:}
%[u(0.5)]
%!{\xunoverh}
%[u(0.5)l(0.5)]
%*{\smcat}
%}
\Bsymcat{\xphor}
\;\tgrshv{-\vthf}{\vthf}{-\vthf}
\vspace*{18pt}
%%%%%%%\end{CD}
%}}
\;\;
\Bigg),
\xeee
where $f$ is either a multiplication or a comultiplication of the
ring $\ZZ[x]/(x^2)$ depending on how the arcs in the \rhs are
closed into circles.
The resulting categorification complex\rx{ae1.5} is invariant
up to homotopy under the second and third Reidemeister moves, but
it acquires a degree shift under the first Reidemeister move:
\vspace*{-0.5cm}
\xlee{ae1.8}
\Bsymcat{\xvfro\hspace*{-0.2cm}}
\;\; = \;\;
\Bsymcat{\;\xvert\hspace*{-0.5cm}}
\tgrshv{\vthh}{-\vthf}{-\vthf}.
\xeee
It is easy to see that the whole categorification complex\rx{ae1.5} has a
homogeneous degree $\dgh$.

\subsection{The universal categorification of the \TLa}
D.~Bar-Natan \cxw{BN1} described the universal category $\dTL$, whose
Grothendieck \Kzg\ is the \TLa\ $\cTL$ considered as a $\Zqqi$-module.
We will use this category with obvious adjustments required by the new
grading conventions.
%
%First, observe that for two \TL\ tangles $\xlamo,\xlamt\in$ the following modules are canonically
%isomorphic:
%%
%\ylee{eq:caniso}
%asdf
%\yeee

\subsubsection{A homotopy category of complexes of a \Zgrdd\
additive category}
\label{sss:hotcatcom}

\hyphenation{Gro-then-dieck}

A split \Grtg\ $\Gz(\xctC)$ of an additive category $\xctC$ is
generated by the images $\Gz(\obA)$ of its objects modulo the
additive relation $\Gz(\obA\oplus\obB) = \Gz(\obA) + \Gz(\obB)$.

A \Grtg\ $\Kz(\xctC)$ of an abelian category $\xctC$ is generated
by the images $\Kz(\obA)$ of its objects modulo the exact sequence
relation: $\Kz(\obA) - \Kz(\obB) + \Kz(\obC) = 0$, if there is an
exact sequence $\obA\rightarrow\obB\rightarrow\obC$.

A \Grtg\ $\Kz(\xctC)$ of a triangulated category $\xctC$ is
generated by the images $\Kz(\obA)$ of its objects modulo the
translation relation $\Kz(\obA[1])=-\Kz(\obA)$ and the
exact triangle relation: $\Kz(\obA) - \Kz(\obB) + \Kz(\obC)=0$, if
$\obA$, $\obB$ and $\obC$ form an exact triangle
$\obA\rightarrow\obB\rightarrow\obC\rightarrow A[1]$.

If the category $\xctC$ is \Zgrdd, that is, there is a shift
functor $\qsho$ and modules $\Hom(\obA,\obB)$ are \Zgrdd, then the
groups $\Gz(\xctC)$ and  $\Kz(\xctC)$ are modules over $\Zqqi$,
the multiplication by $q$ corresponding to the shift $\qsho$.

For an additive category $\cctC$, let $\xhKb(\cctC)$ denote the
homotopy category of bounded complexes and $\xhKm(\cctC)$ -- the
similar category of bounded from above complexes over $\cctC$. These categories are triangulated. Suppose that
$\cctC$ is \Zgrdd\ and generated, as an additive category, by objects $\obEo,\ldots,\obEN$ and their \Zgrd\
shifts, that is, an object $\xA$ of $\cctC$ has a form
\xlee{eq:obgr}
\xA = \bpaoN\bigoplus_{j\in\ZZ}\aja\obEa\qshj,
\xeee
where $\aja\in\ZZ_{\geq 0}$ are multiplicities of the shifted objects
$\obEa\qshj$.
Suppose further that $\Kz(\obEo),\ldots,\Kz(\obEN)$ generate freely
$\Gz(\cctC)$ as a module over $\Zqqi$, so the presentation\rx{eq:obgr} is unique.

An object of $\xhKb(\cctC)$ has the form
\xlee{eq:obkom}
\xbA =
\lrbc{\cdots\rightarrow\xAi\rightarrow\xAio\rightarrow\cdots},\qquad
\xAi =
\bpaoN
\bigoplus_{j\in\ZZ}
\aija\obEa \qshj.
\xeee
%
%where $\aija\in\ZZ_{\geq 0}$ are multiplicities of the shifted objects
%$\obEa\qshj$.
We call the objects $\xAi$  \emph{\qcmds} and we
refer to objects $\obEa$ with non-zero multiplicities as
\emph{constituent objects} of the complex $\xbA$.
The functor $\cctC\hookrightarrow\xhKv(\cctC)$, $\xA\mapsto (0\rightarrow\xA\rightarrow 0)$ generates the
isomorphism of modules $\Gz(\cctC) = \Kz(\xhKv(\cctC))$ and
\xlee{eq:grmpbd}
\Kz(\xbA) = \sum_{i\in\ZZ}\smaoN\sum_{j\in\ZZ} (-1)^i
\aija\,
q^{j}\,\Kz(\obEa).
\xeee

Define the \qpord\ of an object\rx{eq:obgr}
as $\yordq{\xA} = \xminv{j\colon \aja\neq 0}$.
A complex $\xbA$ in $\dTLp$ is \emph{\qpb} if $\lmii\yordq{\xAmi} =
+\infty$.
Let $\xhKpq(\cctC)\subset\xhKm(\cctC)$ be the full subcategory of
\qpb\ complexes. The category $\xhKpq(\cctC)$ is triangulated.
Define $\Kzp(\xhKpq(\cctC))$ as a module over formal Laurent
series $\Zsqqi$ freely generated by the elements $\Kzp(\obEa)$,
$1\leq\inoa\leq\xlN$ and define the map $\Kzp\colon\Ob\xhKpq(\cctC)\rightarrow
\Kzp(\xhKpq(\cctC))$ by the formula similar to \ex{eq:grmpbd}:
\xlee{eq:grmpbdi}
\Kzp(\xbA) = \sum_{i\in\ZZ}\smaoN\sum_{j\in\ZZ} (-1)^i
\aija\,
q^{j}\,\Kzp(\obEa).
\xeee
Since the complex $\xbA$ is \qpb, the sum over $i$ in this
equation is well-defined.

\subsubsection{The additive category $\dTLtl$}

For two \TL\ \ttngmn s $\xlamo$ and $\xlamtw$,
% be two \TL\ \ttngmn s.
let
\xlee{eq:2a.1}
\xlamo\#\xlamtw=\xSoo \sqcup\cdots\sqcup\xSok
\xeee
denote a union of
disjoint circles produced by gluing together the matching end-points of $\xlamo$
and $\xlamtw$. A \emph{\fltc} $\cSg$
from $\xlamo$ to $\xlamtw$ is a compact orientable surface with a
specified diffeomorphism between its boundary and
$\xlamo\#\xlamtw$.

Let $\coblot$ be a $\ZZ\oplus\ZZt$-graded module with free
generators $\hSg$ associated with \fltc s $\cSg$ and having
degrees
\ylee{eq:dgrcob}
\dgt\hSg  = \dgh \hSg = \shlf(m+n) -\chi(\cSg),
\yeee
where $\chi(\cSg) = 2 - \#\text{holes} - \#\text{handles}$ is
the Euler characteristic of $\cSg$.

Let $\zpo,\zpt,\zph,\zpf$ be 4 distinct points inside $\cSg$. An
associated \ftur\ is a formal relation
$$\hSgv{12} + \hSgv{34} = \hSgv{13} + \hSgv{24},$$
where $\cSgij$ is a \fltc\ constructed by cutting small
neighborhoods of $\zpi$ and $\zpj$ out of $\cSg$ and then gluing
together the boundaries of the cuts.

By definition, $\dTLtlmn$ is an additive $\ZZ\oplus\ZZt$-graded category generated by objects $\dlam$
which are indexed by \TL\ \ttngmn s  $\xlam$. A module of
morphisms is
\ylee{eq:modmor}
\HmTlmnlot = \coblot/(\text{\ftur s}).
\yeee
%
%$\HmTlmnlot$ is generated by elements $\hSg$ associated with \fltc
%s $\cSg$, modulo \ftur s

%
%
%Let $\coblot$ be a $\ZZ\oplus\ZZt$-graded module generated freely
%
% A cobordism $\cSg$ is assigned $q$- and
%$\ZZt$-degrees:
%%
%\ylee{eq:dgrcob}
%\dgt\cSg  = \dgh \cSg = \shlf(m+n) -\chi(\cSg),
%\yeee
%%
%where $\chi(\cSg) = 2 - \#\text{holes} - \#\text{handles}$ is
%the Euler characteristic of $\cSg$.
%
%Let $\zpo,\zpt,\zph,\zpf$ be 4 distinct points inside $\cSg$. An
%associated \ftur\ is a formal relation
%$$\cSgv{12} + \cSgv{34} = \cSgv{13} + \cSgv{24},$$
%where $\cSgij$ is a \fltc\ constructed by cutting small
%neighborhoods of $\zpi$ and $\zpj$ out of $\cSg$ and then gluing
%together the boundaries of the cuts.
%
%By definition, $\dTLtlmn$ is an additive $\ZZ\oplus\ZZt$-graded category generated by objects $\dlam$
%which are indexed by \TL\ \ttngmn s  $\xlam$. A module of morphisms
%$\HmTlmnlot$ is generated by \fltc s
%modulo \ftur s. Morphisms are composed as cobordisms.
%
%%It is easy to see that the module $\Hom_{\dTLtl}(\xlamo,\xlamtw)$
%%is freely generated by
%%\ftur s allow us to reduce the number of generators.
%
%

A \fltc\ $\cSg$ is called
\emph{\rdcd} if it is a disjoint union of connected surfaces
$\cSg=\cSg_1\sqcup\cdots\sqcup\cSg_k$, such that each $\cSg_i$ is either a 2-disk $\bbB^2$ or
a 2-torus  with a hole; since each $\cSg_i$ has a single boundary component, there are specified
diffeomorphisms between the boundaries $\del\cSg_i$ and circles $\xSov{i}$ of
\ex{eq:2a.1}.
\begin{theorem}[\cxw{Kh2},\cx{BN1}]
\label{pr:rpc}
The module $\HmTlmnlot$ is generated freely by
\rdcd\ \fltc s.
\end{theorem}

The following corollary of this theorem is easily established by
relating two types of \rdcd\ \fltc s (a 2-disk and a 2-torus with a hole)
to the generators $1$ and $x$ of the module\rx{ae1.6}.
\begin{corollary}
There is a canonical isomorphism
\xlee{eq:homtliso}
\HmTlmnlot =
\zsymcat{\lSh{\xlamtw\tcmp\dsymv{\xlamo}}}
\tgrshv{\tfrac{m+n}{2}}{0}{\tfrac{m+n}{2}},\qquad
\xlamo,\xlamtw\in\yTLmn.
\xeee
\end{corollary}
Moreover, since
$(\lSh{\xlamtw\tcmp\dsymv{\xlamo}})=(\lSh{\xlamo\tcmp\dsymv{\xlamtw}})=
(\lSh{\dsymv{\xlamtw}\tcmp\xlamo})$,
there is a canonical isomorphism
$\zsymcat{\lSh{\xlamtw\tcmp\dsymv{\xlamo}}} =
\zsymcat{\lSh{\xlamo\tcmp\dsymv{\xlamtw}}}
=\zsymcat{\lSh{\dsymv{\xlamo}\tcmp\xlamtw}}$ and in view of
\ex{eq:homtliso} there are canonical isomorphisms between the
modules of morphisms
\xlee{eq:homsiso}
\HmTlmnlot =
\Hom_{\dTLtlmn}\xlrb{\dlamtw,\dlamo} =
\Hom_{\dTLtlmn}\xlrb{\zsymcat{\dsymv{\xlamo}},\zsymcat{\dsymv{\xlamtw}}}
.
\xeee

An \emph{\efltc} is a \rdcd\ cobordism which is either a saddle cobordism
%involving just two arcs
or an \emph{\xmult}, which is a connected sum of an identity
cobordism and a 2-dimensional torus.
%(the latter cobordism is also known as `$x$-multiplication').

\begin{proposition}
\label{pr:rpc1}
Every \rfltc\ can be presented as a composition of
%\elmt\ ones.
\efltc s.
\end{proposition}

A composition of \TL\ tangles generates a bifunctor
$\tcmp\colon\dTLtlmn\otimes\dTLtllm\rightarrow\dTLtlln$,
if we apply
the categorified version of the rule\rx{eq:ae1.1} in order to remove
disjoint circles:
\xlee{ae1.01}
\Bsymcat{\lcir}= \cnot \tgrshv{1}{0}{1} + \cnot\tgrshv{-1}{0}{1},
\xeee
%
%where $\xnot$ is the empty \TL\ \ttngzz.
%and two rules for
%cobordism-related morphisms: a disjoint compact oriented surface
%of genus $g$ in a cobordism is converted into a factor $-2$, if
%$g=1$ and $0$, if $g\neq 1$.

%
%if a cobordism contains a disjoint
%surfa then the corresponding morphism is zero, and if a cobordism.
Thus the category
$\dTLtl = \bigoplus_{m,n}\dTLtlmn$ acquires the monoidal
structure.

\subsubsection{The universal category $\dTL$}

 The universal category
$\dTLmn$ is the homotopy category of bounded complexes  over
$\dTLtlmn$: $\dTLmn = \xhKb(\dTLtlmn)$. In other words, in accordance
with the general formula\rx{eq:obkom}, an object of $\dTL$ is a complex
\xlee{ae1.8a}
\xbA =
\lrbc{\cdots\rightarrow\xAi\rightarrow\xAio\rightarrow\cdots},\qquad
\xAi =
\bigoplus_{\substack{j\in\ZZ \\ \xmu\in\ZZ_2}}
\oltlmn \ajilam\,\dlam \tgrshv{j}{0}{\mu}.
\xeee
%
%where non-negative numbers $\ajilam$ are multiplicities; since the
%complex is bounded, they are non-zero for only finitely many
%values of $i$. We call the objects $\xAi$ of $\dTLtlmn$ \emph{\qcmds} and we
%refer to objects $\dlam$ with non-zero multiplicities as
%\emph{constituent objects} of the complex $\xbA$.

The total category $\dTL$ is a formal sum of categories:
$\dTL=\bigoplus_{m,n}\dTLmn$. It inherits the monoidal structure
of $\dTLtl$ which comes from the composition of tangles.
%, and we
%use the notation $\dTLop$ for the category $\dTL$ with reversed
%order of the composition of objects.

A categorification map $\mpcat\colon\yTng\rightarrow\dTL$ turns a framed
tangle diagram $\xtau$ into a complex $\dtau$ according to the
rules\rx{ae1.6} and\rx{ae1.7}, the morphism $\xmrf$ in the
complex\rx{ae1.7} being the saddle cobordism.
%A composition of
%tangles becomes a composition bi-functor
%$\dTL\times\dTL\rightarrow\dTL$ if we apply
%the categorified version of the rule\rx{eq:ae1.1} in order to remove
%disjoint circles:
%%
%\xlee{ae1.01x}
%\Bsymcat{\lcir}= \cnot \tgrshv{1}{0}{1} + \cnot\tgrshv{-1}{0}{1},
%\xeee
%%
%where $\xnot$ is the empty \TL\ \ttngzz.

%The categorification map $\mpcat$ transfers the
%involutions\rx{eq:tinvs} into involutive covariant and,
%respectively, contravariant functors
%%
%\xlee{eq:finvsx}
%\dflp,\dsym\colon \dTL\rightarrow\dTL.
%\xeee
%%
%The functor $\dflp$ preserves the $\ZZZt$ grading, while the
%functor $\dsym$ reverses it. It is easy to see that this
%convention is compatible with \eex{ae1.6} and\rx{ae1.7}.

%The category $\dTL$ has full subcategories $\dTLmn$ and
We use a
shortcut $\dTLn= \dTLv{n,n}$.
The category
%$\dTLz$
$\dTLz$
is generated by a single object $\cnot$. Hence
it is equivalent to
%the abelian category of free $\ZZ\oplus\ZZt$-graded modules. Hence the category $\dTLz$ is
the homotopy category of free
$\ZZ\oplus\ZZt$-graded modules and it has a homology functor. The
homology functor
%. Hence the category $\dTLz$ is
%abelian and
%the homology functor
applied to a complex $\symcat{\xL}$ of a link
$\xL\subset\Sh$ yields, by definition, the link homology:
$\Hmb\xlrb{\symcat{\xL}} = \HKhb(\xL)$.
%
%we define the homology functor
%$\HKh\colon\dTLz\rightarrow
%\tZmod$, where $\mgmod$ denotes the category of
%$\ZZ\oplus\ZZ\oplus\ZZt$-graded modules .
%
A circular closure $(\lSh{\xtau})$ of \ttngnn s $\xtau$ within
$\Sh$ extends to a \Shcl\ functor
\xlee{eq:fnct}
%(\lSh{-})\colon\dTLn\longrightarrow\dTLz=
%\xhKv(\cvgmod{\ZZ}).
%\\
\xymatrix@C=1.5cm{
\dTLn
\ar[r]^-{(\lSh{-})}
&
\dTLz=\xhKb(\cvgmod{\ZZ}).
}
\xeee

\subsubsection{\Grtg\ map $\Kz$}
In accordance with the general rules of subsection\rw{sss:hotcatcom},
$\Kz(\dTL)=\Gz(\dTLtl) = \cTL$ and there is a commutative diagram
%Overall,
%we have the following commutative diagram:
%
\begin{equation}
\label{eq:cmdgrm}
\xymatrix@C=1.5cm@R=0.3cm{
& {}\dTL \ar[dd]^{\Kz}
\\
\yTng \ar[ur]^{\mpcat} \ar[dr]^{\mpalg}
\\
& {}\cTL
}
\end{equation}
%
%%
%\begin{equation}
%\label{eq:cmdgrm}
%\xymatrix@C=1.5cm@R=0.3cm{
%& {}(\dTL,\tcmp) \ar[dd]^{\Kz}
%\\
%(\yTng,\tcmp) \ar[ur]^{\mpcat} \ar[dr]^{\mpalg}
%\\
%& {}(\cTL,\tcmp)
%}
%\end{equation}
%
where the map $\Kz$ turns the complex\rx{ae1.8a} into the
sum\rx{eq:1.3a}:
\xlee{ae1.9a}
\Kz(\xbA)
=
\sltln\sum_{j\in\ZZ} \xcalj \,q^j\,\clam,
\qquad
\xcalj = \sum_{\substack{i\in\ZZ \\ \xmu\in\ZZ_2}}
(-1)^{i+\xmu} \ajilam.
\xeee
%
%Since the complex is bounded, the sum in the expression for
%$\xcalj$ is finite.

In addition to $\dTL$ we consider
the categories $\dTLp$ and $\dTLpq\subset\dTLp$ defined in
accordance with subsection\rw{sss:hotcatcom}. Obviously,
$\Kzp(\dTLpq) = \cTLpinf$.
%
%a category $\dTLp$ of complexes
%bounded from above, that is, the multiplicity coefficients in the
%sum\rx{ae1.8a} are zero for $i$ above certain value.
%
%Define the $q^+$ order of an object of $\dTLtlmn$
%%
%\ylee{ae1.8z1x}
%\xA =
%%\bigoplus_{j\in\ZZ} \oltl
%\bigoplus_{\substack{j\in\ZZ \\ \xmu\in\ZZ_2}} \oltlmn
%\ajlam\,\dlam
%%\qshv{j},
%\tgrshv{j}{0}{\mu}
%\yeee
%%
%as
%%of $\dTLtl$ with a presentation
%%Define the $q^+$ order of an object $\xA$ of $\dTLtl$ with
%%presentation\rx{ae1.8z1}:
%$\yordq{\xA} = \xinfv{j\colon \ajlam\neq
%0}$.
%A complex $\xbA$ in $\dTLp$ is \emph{\qpb} if $\lmii\yordq{\xAmi} =
%+\infty$.
%Let $\dTLpq\subset\dTLp$ denote the full subcategory of \qpb\
%complexes.
%For a \qpb\ complex,
%the sum in the expression\rx{ae1.9a} for $\xcalj$ is
%finite, hence  the map $\dTLpq\xrightarrow{\Kz}\cTLpinf$ is
%well-defined.
%

%
%We also
%define an associated bi-functor
%$\TrSh\colon\dTLtn\times\dTLtn\rightarrow\tZmod$ by its action on
%a pair of tangle categorification complexes:
%%
%\ylee{ae1.9a1}
%\TrSh\bigl(\dtauo,\dtaut\bigr) = \HKh \Bigl(\symcat{\lSh{\xtauo\tcmp\xtaut}}
%\Bigr).
%\yeee
%%

\subsubsection{Involutive functors $\dflp$, $\dsym$ and $\finvm$}
\label{sss.dfun}
Let $\dTLop$ denote the category $\dTL$ in which the composition
of tangles is performed in reverse order.
We define a covariant functor $\dflp$ and contravariant functors $\dsym$, $\finvm$
% and, respectively, contravariant involutive functors
%
\xlee{eq:finvs}
\dflp,\dsym\colon \dTL\longrightarrow\dTLop,\quad
\dTLmn\longrightarrow\dTLnm,\qquad
\finvm\colon\dTL\longrightarrow\dTL,\quad
\dTLmn\longrightarrow\dTLmn
\xeee
by the requirement that their action on generating objects $\dlam$
matches their action on underlying \TL\ tangles $\xlam$ and that
$\dflp$ should preserve degree shifts while $\dsym$ and $\finvm$
should invert them.  The action of the functors\rx{eq:finvs} on morphisms is
established with the help of isomorphisms\rx{eq:homsiso}.

It is
easy to see by checking the action of $\dflp$ and  $\dsym$ on
\eex{ae1.7} and\rx{ae1.01} that the maps of the diagram\rx{eq:cmdgrm}
intertwine the actions of $\dflp$ and $\dsym$ on $\yTng$, $\cTL$ and
$\dTL$.

%by their action on the generating objects:
%$\dflpv{\dlam} = \dsymv{\dlam} = \zsymcat{\dflpv{\xlam}} =
%\zsymcat{\dsymv{\xlam}}$ and by the requirement that $\dflp$
%preserves the $\ZZZt$ grading, while $\dsym$ reverses it. It is
%easy to see by checking the action of $\dflp$ and  $\dsym$ on
%\eex{ae1.7} and\rx{ae1.01} that the maps of the diagram\rx{eq:cmdgrm}
%intertwine the actions of $\dflp$ and $\dsym$ on $\yTng$, $\cTL$ and
%$\dTL$. The action of the functors\rx{eq:finvs} on morphisms is
%established with the help of isomorphisms\rx{eq:homsiso}.
%
%
%
%The action of the involutions $\dflp$ and $\dsym$ on tangles
%defines their action on objects of the category $\dTLtl$. We
%promote $\dflp$ and $\dsym$ to involutive covariant and,
%respectively, contravariant endofunctors
%%promote $\dflp$ to a covariant endofunctor and promote $\dsym$ to a contravariant endofunctor
%acting on $\dTLtl$  (the same cobordism can produce morphisms acting both
%ways) and consequently on $\dTL$. The functor $\dflp$ preserves
%the gradings of $\dTL$, while $\dsym$ reverses their degrees.

For two complexes $\xbA$ and $\xbB$ in $\dTLmn$ there is a
canonical isomorphism extending that of \ex{eq:homtliso}:
\ylee{eq:1.9}
\Hmmn(\xbA,\xbB) =
\zsymcat{\lSh{\xbB\tcmp\dsymv{\xbA\!}}}\tgrshv{\tfrac{m+n}{2}}{0}{\tfrac{m+n}{2}}.
\yeee
%

%\nidm
%
%\subsubsection{Duality functor}
%\label{sss:dfun}
%A dual of an \ttngmn\ $\xtau$ is the \ttngnm\ tangle
%$\xtaud$ which is its mirror image. Duality extends to an
%isomorphism $\cTL \xrightarrow{\dsym} \cTLop$ combined with the
%isomorphism of the ground ring $\Zqqi\xrightarrow{\dsym}\Zqqi$, such that
%$\dsymv{q} = q^{-1}$. Furthermore, duality establishes an
%isomorphism $\cTLpinf\xrightarrow{\dsym}\cTLminfop$, where
%$\cTLminf$ is the analog of $\cTLpinf$ constructed over the ring
%$\Zsqiq$ of Laurent series in $q^{-1}$.
%
%\nidm
%
%
%
%Duality functor $\dsym$ extends to a
%contravariant equivalence functor $\dTL\xrightarrow{\dsym}\dTLop$, where
%$\dTLop$ is the same category as $\dTL$, except that the
%composition of tangles is performed in reversed order. The functor
%$\dsym$ also switches the signs of all three gradings of $\dTL$.
%
%%For two \ttngtnz s $\xtauo$ and $\xtaut$ there is a canonical
%%isomorphism
%%%
%%\ylee{eq:1.9}
%%\Hmtnz\bigl(\dtauo,\dtaut\bigr) = \symcat{\xtaut\tcmp \xtauod}\tgrshnzn
%%\yeee
%%%
%
%For two complexes $\xbA$ and $\xbB$ in $\dTLmn$ there is a
%canonical isomorphism
%%
%\ylee{eq:1.9}
%\Hmmn(\xbA,\xbB) =
%\xlrb{\lSh{\xbA\tcmp\dsymv{\xbB}}}\tgrshv{\tfrac{m+n}{2}}{0}{\tfrac{m+n}{2}}
%\yeee
%%
%
%\nidm

\subsubsection{A \tct\ subcategory}

An additive \tct\ category $\dTLtlcttmtn\subset\dTLtltmtn$ is a
full subcategory whose objects are \tct\ \TL\ tangles $\zsymcat{\xbet\tcmp\xalf}$, where
$\xal\in\yCrm$ and $\xbet\in\yCrn$. There is an obvious functor
\xlee{eq:fnproTL}
\dTLtltmz\otimes\dTLtlztn\longrightarrow\dTLtlcttmtn,\qquad\xlrb{\dbet,\zsymcat{\xalf}}\mapsto
\zsymcat{\xbet\tcmp\xalf}.
\xeee

A \tct\ category $\dTLcttmtn\subset\dTLtmtn$ is a full subcategory
whose objects are $\zsymcat{\xbet\tcmp\xalf}$, where
$\xal\in\yCrm$ and $\xbet\in\yCrn$. There is an obvious functor
\xlee{eq:fnproTLx}
\dTLtmz\otimes\dTLztn\longrightarrow\dTLcttmtn,\qquad\xlrb{\dbet,\zsymcat{\xalf}}\mapsto
\zsymcat{\xbet\tcmp\xalf}.
\xeee
\begin{theorem}
\label{th:tlspl}
The functor\rx{eq:fnproTL} is an equivalence of categories.
\end{theorem}

\begin{proof}
Obviously, this functor acts bijectively on objects and
injectively on morphisms. The surjectivity of its action on
morphisms is established with the help of Theorem\rw{pr:rpc}: a
\rdcd\ \fltc\ between the tangles $\xalo\tcmp\dflpv{\xbet_1}$ and
$\xalt\tcmp\dflpv{\xbet_2}$ is a disjoint union of a cobordism
between $\xalo$ and $\xalt$ and a cobordism between $\xbet_1$ and
$\xbet_2$.
\end{proof}

The additive category equivalence\rx{eq:fnproTL} implies the
equivalence of homotopy categories
\xlee{eq:kpqtl}
\xhKpq (\dTLtltmz\otimes\dTLtlztn)\longrightarrow\dTLcpqtmtn,
\xeee
where $\dTLcpqtmtn=\xhKpq\xlrb{\dTLtlcttmtn}$ is the full subcategory
of $\dTLpqtmtn$.

\subsection{Bimodule categorification}
\label{sss:bimcat}
\subsubsection{The rings $\xaHn$}
M.~Khovanov\cx{Kh2} defined the algebras %$\ZZ$-graded rings
$\xaHn$ as the sums of rings of morphisms between the objects $\dal$ or, equivalently, $\dald$, $\xal\in\yCrn$:
\xlee{eq:dfhn}
\label{eq:1.10}
\xaHn =
\bigoplus_{\xal,\xbet\in\yCrn} \Hmtnz \xlrb{ \dal,\dbet }=
\bigoplus_{\xal,\xbet\in\yCrn} \Hmtnz \xlrb{ \dald,\dbetd },
\xeee
the isomorphisms being particular cases of those of
\ex{eq:homsiso}.
The rings $\xaHn$
%$\xaHn$
have
%$\ZZ$-grading
$\ZZ\oplus\ZZt$-grading
associated with $q$-related \Zgrd\ and \Ztgrd\ of the
category $\dTL$ (in fact,
$\dgh\xaHn=0$).
%all elements of $\xaHn$ have zero \Ztdgr).
The first of isomorphisms\rx{eq:homsiso} provides
a canonical isomorphism
\xlee{eq:algopis}
\dflp\colon\xaHn\rightarrow\xaHopn.
\xeee
%

%Note that there are canonical isomorphisms
%%
%\begin{equation}
%\nonumber
%%\label{eq:obisos}
%\begin{split}
%\Hmtnz \xlrb{ \dald,\dbetd } & = \Hmtnz\xlrb{\dbetd,\dald} =  \Hmztn \xlrb{ \dal,\dbet}
%\\
%& = \symcat{\xald\tcmp\xbet}\tgrshnzn
%= \symcat{\xbetd\tcmp\xal}\tgrshnzn.
%\end{split}
%\end{equation}
%%
%The first of them generates the isomorphism
%$\dflp\colon\xaHn\rightarrow\xaHopn$.

Denote $\xaHdnm=\xaHn\otimes\xaHopm$. We  use  abbreviated
notations  $\xaHne=\xaHn\otimes\xaHopn$ and $\KbgHdnm=\xKbgHdnm$.
The isomorphism\rx{eq:algopis} generates the isomorphism
$\dflp\colon\xaHdnm\rightarrow\xaHdmn$ and consequently an equivalence functor
\xlee{eq:dfeqfun}
\dflpp\colon\KbgHdnm\rightarrow\KbgHdmn.
\xeee

\subsubsection{Bimodules from tangles}
The categorification maps
\xlee{eq:1.10a}
%\rTNGmn\xrightarrow{\mpbim} \DbHdnm,
%\rTNGmn
\mpKbim\colon\yTngtmtn
\rightarrow \KbgHdnm
%,\qquad
%\Ktau=\bigoplus_{\substack{\xal\in \rTLtmz\\ \xalp\in\rTLtnz }}
%\xbrapal,
\xeee
are defined in\cx{Kh2} by the formula
%
%For a \ttngmn\ $\xtau$ the first author defines a complex of
%modules
%
%\begin{equation}
%%\label{eq:dfmdl}
%\begin{split}
%\Ktau & = \bigoplus_{\substack{\xal\in \yCrm\\ \xbet\in\yCrn }}
%\Hmmn\xlrb{\dald,\zsymcat{\xbetd\tcmp\xtau} }
%%\\ &
%= \bigoplus_{\substack{\xal\in \yCrm\\ \xbet\in\yCrn }}
%\zsymcat{ \xbetd\tcmp\xtau\tcmp\xal }\tgrshmzm.
%\end{split}
%\end{equation}
%%
%
\begin{equation}
\label{eq:dfmdl}
\begin{split}
\Ktau & = \bigoplus_{\substack{\xal\in \yCrm\\ \xbet\in\yCrn }}
\Hmtmtn\xlrb{\dbet,\zsymcat{\xtau\tcmp\xal} }
%\\ &
= \bigoplus_{\substack{\xal\in \yCrm\\ \xbet\in\yCrn }}
\zsymcat{ \xbetd\tcmp\xtau\tcmp\xal }\tgrshnzn.
\end{split}
\end{equation}

Comparing \eex{eq:dfhn} and\rx{eq:dfmdl}, it is easy to see that
%In particular, it is obvious that
%
\xlee{eq:1.11a}
\Kidbrtn = \xaHn,
\xeee
where $\xaHn$ is considered as a module over $\xaHne$.

\begin{theorem}[\cxw{Kh2}]
\label{th:cmptp}
The map $\mpKbim$ translates the composition of tangles into the
tensor product over the intermediate ring:
if $\xtauo$ is a \ttngtltm\ and $\xtaut$ is a \ttngtmtn, then
there is a canonical isomorphism
\xlee{eq1.11a1a}
\Ksymbim{\xtaut\tcmp\xtauo} = \Ktaut\otimes_{\xaHm}\Ktauo.
\xeee
\end{theorem}

%, while
%the category $\DbHdnm$ is endowed with an extra
%$\ZZ_2$-grading to match the $\ZZ_2$-grading of the universal category
%$\dTLtmtn$.

The map\rx{eq:1.10a} restricted to \TL\ tangles extends to a
functor $\spfK\colon\dTLtmtn\rightarrow\KbgHdnm$, which maps an
object $\dlam$ to $\Klam$ and translates \fltc s between \TL\ tangles $\xlam$ and
$\xlamp$ into homomorphisms between the
modules $\Klam$ and $\Klamp$. Moreover, the categorification maps\rx{eq:1.10a} can be threaded though the
universal category:
\xlee{eq:1.11b}
\xymatrix@C=0.2cm@R=1cm{
&\yTngtmtn \ar@<-1ex>[dl]_-{\symcat{-}} \ar[dr]^-{\Ksymbim{-}}
\\
\dTLtmtn \ar@{->>}[rr]^-{\spfK}
&&
\KbgHdnm
%\KtlbgHdnm
%\ar@{}[r]|{\hspace*{0.35cm}\subset}
%&
%\KbgHdnm
}
\xeee
%
%where the full subcategory $\KtlbgHdnm\subset\KbgHdnm$ is the
%bounded homotopy category of $\xaHdnm$-modules of the form $\Klam$,
%$\xlam\in\yTLtmtn$.

%where $\spfK$ is a special functor which maps the objects $\dlam$
%into $\Klam$ and translates \fltc s between \TL\ tangles $\xlam$ and
%$\xlamp$ into homomorphisms between the
%modules $\Klam$ and $\Klamp$.

The covariant involutive functor $\dflp$ and contravariant
involutive functors $\dsym$ and $\finvm$
\xlee{eq:invbim}
%\begin{gather}
%\nonumber
\dflp,\dsym\colon \KbgHdnm \rightarrow \KbgHdmn,\quad
%\\
%\nonumber
\finvm\colon\KbgHdnm\rightarrow\KbgHdnm
%\end{gather}
\xeee
are defined by the formulas
\xlee{eq:fundefbim}
\dflpv{\xbM} = \dflppv{\xbM}\tgrshmmn,\qquad
\dsymv{\xbM} = \ddulv{\xbM}\tgrshminmn,\qquad
\finv{\xbM} = \ddulv{ (\dflppv{\xbM}) } \tgrshteq{-2m},
\xeee
where $\xbM$ is a complex in $\KbgHdnm$, $\ddulv{\xbM}$ is the dual complex and
$\dflpp$ is the equivalence functor\rx{eq:dfeqfun}.
%$\dflppv{\xbM}$ is the complex of $\xaHdmn$-modules obtained from $\xaHdnm$-modules of $\xbM$ by
%applying the isomorphisms\rx{eq:algopis}.
It is easy to see that the maps of the diagram\rx{eq:1.11b}
intertwine  the actions\rx{eq:tinvs}, \rx{eq:finvs}
and\rx{eq:invbim} of $\dflp$, $\dsym$ and $\finvm$. In particular,
\xlee{eq:dfleq}
\dflpv{\dtau} = \zsymcat{\dflpv{\xtau}}
\xeee
for $\xtau\in\yTngtntn$.

%
%Overall, there is a commutative diagram of maps respecting tangle
%composition:
%%
%\ylee{eq:1.13}
%\xymatrix@C=2cm@R=2cm{
%\rTNGmn
%\ar[r]^-{\mpalg}
%\ar[d]^-{\mpcat}
%\ar[dr]^-(0.25){\mpbim}
%&
%\cTLmn
%\\
%\dTLmn
%\ar[r]^-{\spfn}
%\ar[ur] |{\hole}_-(0.8){\Kz}
%&
%\DbHdnm
%\ar[u]_-{\Kz}
%}
%\yeee
%%

%\nidm
%
%The modules $\wcuptnI$, $\stI\in\cstIn$, form a complete
%set of indecomposable projective modules of $\xaHn$. Moreover, the
%the functor $\spfn$ restricted to the full subcategory
%$\dTLztn$ establishes its equivalence to the category $\KbHpn$ of bounded
%complexes of projective $\xaHn$-modules.
%
%
%
%The modules
%$\bsymbim{ \gcuptnIp\tcmp\gcaptnI }$, $\stI,\stIp\in\cstIn$ form a
%complete set of indecomposable projective modules of $\xaHne$.
%Let $\dTLtltnct$ be the full additive subcategory of $\dTLtl$
%generated by \tct\ \TL\ \ttngtntn s, that is, $\dTLtltnct$ is
%generated by the objects $\bsymcat{ \gcuptnIp\tcmp\gcaptnI }$. The
%homotopy category of bounded complexes
%$\dTLtnct=\xhKb(\dTLtltnct)$ is the full subcategory of $\dTL$,
%and the restriction of $\spfn$ establishes its equivalence to the
%category $\KbHepn$ of bounded projective $\xaHne$-modules.
%Similarly, $\spfn$ establishes an equivalence between
%the homotopy category $\dTLtnctm = \xhKm(\dTLtltnct)$ of complexes bounded from above with
%the category $\KmHepn$.
%
%\nidm

\section{Derived category of $\xaHdnm$-modules and Hochschild
homology}

\subsection{A quick review of derived categories of modules}
\subsubsection{Projective resolution formula}
For a (graded) ring $\aR$,  $\DbmdR$ denotes the bounded derived
category of (graded) $\aR$-modules, $\cvxpmod{\aR}$ denotes the category
of (graded) projective modules of $\aR$, while $\xhKbpR$ and $\xhKmpR$ denote the
bounded and bounded from above homotopy categories of (graded) projective
$\aR$-modules. The following commutative diagram is a practical guide
for working with $\DbmdR$:
\xlee{eq:prdiag}
\xymatrix@C=1.5cm@R=1.5cm{
\xhKbpR
\ar@{^{(}->}[r]
\ar@{^{(}->}[d]
&
\xhKmpR
%\ar[r]^-{\Kz}
%&
%\Kz(\xhKmpR)
\\
\xhKbR
\ar@{->>}[r]^-{\spfKD}
\ar[ur]^-{\spPK}
&
\DbmdR
\ar@{^{(}->}[u]_-{\spP}
%\ar[r]^-{\Kz}
%&
%\Kz(\DbmdR)
%\ar@{^{(}->}[u]
}
\xeee
Here the arrows $\hookrightarrow$ denote full subcategory
injective functors, the arrow $\twoheadrightarrow$ denotes a
functor with surjective action on objects, while $\spP$ and
$\spPK$ are functors of projective resolution.

If $\xbM$ is a complex in $\xhKbR$, then usually, with a slight
abuse of notations, $\spfKD(\xbM)$ is denoted simply as $\xbM$,
and we will follow this convention.

Since the functor $\spP$ is fully injective, while $\spfKD$ is
surjective, the structure of the derived category $\DbmdR$ is
completely determined by the functor $\spPK$, which has a
convenient presentation. Denote $\aRe=\aRd$ and define a
projective resolution of $\aR$ as $\aRe$-module to be a
complex of $\aRe$-modules
\ylee{eq:prresr}
\spPr(\aR) =
(\cdots\rightarrow\spPr(\aR)_{-2}\rightarrow\spPr(\aR)_{-1}\rightarrow\spPr(\aR)_0)
\yeee
with a homomorphism $\spPr(\aR)_0\rightarrow\aR$ such that the
total complex
$\cdots\rightarrow\spPr(\aR)_{-1}\rightarrow\spPr(\aR)_0\rightarrow\aR$
is acyclic. Then $\spPK$ acts on the complexes of $\xhKbR$ by
tensoring them with $\spPr(\aR)$ over $\aR$:
\xlee{eq:fprtpr}
\spPK = \spPr(\aR)\otimes_{\aR}\xdummy.
\xeee

The isomorphism of objects $\xbM$ and $\xbN$ of $\DbmdR$ is called
\emph{\qim} and is denoted as $\xbM\qie\xbN$.

\subsubsection{Semi-projective bimodules}

Suppose that $\aR$ is a tensor product:
$\aR=\aRo\otimes\aRtop$. A $\aRotd$-module is called \emph{\sprj} if
it is projective separately as a $\aR$-module and as a
$\aRtop$-module. Denote by $\KbspRot$  the bounded homotopy
category of \sprj\ $\aRotd$-modules and consider the following version of the commutative
diagram\rx{eq:prdiag}:
\xlee{eq:prddiag}
\xymatrix@C=1cm@R=1cm{
\KbspRot
\ar[r]^-{\spPsp}
\ar@{^{(}->}[d]
&
\xhKmpRot
\\
\xhKbRot
\ar@{->>}[r]^-{\spfKD}
\ar[ur]^-{\spPK}
&
\DbmdRot
\ar@{^{(}->}[u]_-{\spP}
}
\xeee
The projective resolution functor $\spPK$ can be expressed with
the help of \ex{eq:fprtpr}: $\spPK =
\spP(\aRo)\otimes_{\aRo}(\xdummy)\otimes_{\aRt}\spP(\aRt)$.
However, its restriction $\spPsp$ to complexes of \sprj\ modules
has a simpler expression, because it is sufficient to tensor with
a projective resolution of one of two rings:
\xlee{eq:ostprj}
\spPsp = \spP(\aRo)\otimes_{\aRo}\xdummy = \xdummy\otimes_{\aRt}
\spP(\aRt).
\xeee

\subsubsection{Hochschild homology}

For any ring $\aR$ consider a ring $\aRe=\aRd$. The
\Hhom\ and cohomology functors are defined by the
diagrams
\xlee{eq:diagsHH}
\xymatrix@C=1.5cm@R=1.5cm{
\xhKmpRe
\ar[dr]^-{\Hmb(\xdummy\otimes_{\aRe}\aR)}
\\
\DbmdRe
\ar@{^{(}->}[u]_-{\spP}
\ar[r]^-{\xHHl(\xdummy)}
&
\Qgmodp
}
\qquad\qquad
\xymatrix@C=1.5cm@R=1.5cm{
\xhKmpRe
\ar[dr]^-{\;\Hmub(\Hom_{\aRe}(\xdummy,\aR))}
\\
\DbmdRe
\ar@{^{(}->}[u]_-{\spP}
\ar[r]^-{\xHHu(\xdummy)}
&
\Qgmodq
}
\xeee
where $\Qgmodp$ and $\Qgmodq$ denote the categories of \Zgrdd\
vector spaces over $\IQ$ with homological \Zgrd\ bound from above
and from below. In other words, \Hhom\ and cohomology are
are functors  $\Tor_{\aRe}(\xdummy,\aR)$ and
$\Ext_{\aRe}(-,\aR)$.

If $\xbM$ is a bounded complex
of $\aRe$-modules, then
\xlee{eq:flHHR}
\xHHlv{\fKDbM} = \Hmb(\xbM\otimes_{\aRe}\spPr(\aR)).
\xeee

\subsection{Split \TL\ tangles and projective $\xaHdnm$-modules}

\subsubsection{Derived category of $\xaHdnm$-modules}% and \tct\ universal category}

%A $\xaHdnm$-module is called \emph{\sprj} if it is projective
%separately as a $\xaHn$-module and as a $\xaHopm$-module.
Let
$\KbspHdnm\subset\KbgHdnm$ denote the bounded homotopy category of \sprj\
$\xaHdnm$-modules. The following is easy to prove:
\begin{theorem}[\cxw{Kh2}]
For any \TL\ \ttngtntm\ $\xlam$, the $\xaHdnm$-module $\Klam$ is
\sprj\ with possible degree shift and, consequently,
the image of the functor $\spfK$ lies within $\KbspHdnm$.
%$\KtlbgHdnm\subset\KbspHdnm$.
% the image of the categorification map \rxw{eq:1.10a} lies within subcategory $\KbspHdnm\subset\KbgHdnm$.
\end{theorem}

The next theorem strengthens this result for \tct\ \TL\ tangles:
\begin{theorem}[\cxw{Kh2}]
\label{th:splprj}
The $\xaHdnm$-modules $\prjPba=\zKsymbim{\xbet\tcmp\xald}\tgrshmzm$,
$\xal\in\yCrm$, $\xbet\in\yCrn$ form a complete
list of indecomposable projective $\xaHdnm$-modules.
\end{theorem}
%

%For a graded ring $\aR$ we use a shortcut notation
%$\Dbmd{\aR}=\xDbgmd{\aR}$ for the bounded derived category of
%graded $\aR$-modules.

Since $\xaHv{0}=\ZZ$ and, consequently, $\xaHdv{n}{0}=\xaHn$,
Theorem\rw{th:splprj} implies that $\Kal$, $\xal\in\yCrn$ form the full list of
indecomposable projective modules of $\xaHn$. This fact has three corollaries.
%The first one establishes the \Grt\ of $\DbhHdnm$:
%%
%\ylee{eq:grthdnm}
%\Kz\xlrb{\DbhHdnm} =
%\yeee
%%
The first one is obvious:
\begin{corollary}
The map $\zKsymbim{\xdummy}\colon\yTngzn\rightarrow\KbgHn$ threads
through the homotopy category of projective modules:
\ylee{eq:thrball}
\xymatrix@C=1.5cm{
\yTngzn
\ar@/^1.5pc/[rr]^-{\zKsymbim{\xdummy}}
\ar[r]_-{\zKsymbim{\xdummy}}
&
\xhKbpHn
\ar@{^{(}->}[r]
&
\KbgHn
}
\yeee
\end{corollary}

The second corollary is a consequence of an obvious relation
\xlee{eq:prisom}
\zKsymbim{\xbet\tcmp\xald} = \Kbet\otimes\zKsymbim{\xald}.
\xeee
%
%we
%come to the following corollary:
%
\begin{corollary}
The isomorphism\rx{eq:prisom} generates  a canonical equivalence
of categories
\ylee{eq:eqprecat}
\cvxpmod{\xaHdnm} =  ( \cvxpmod{\xaHn})\otimes(\cvxpmod{\xaHopm})
\yeee
and consequently
\xlee{eq:eqdcat}
\xhKmpqHdnm = \xhKpq\xlrB{ (\cvxpmod{\xaHn})\otimes(\cvxpmod{\xaHopm})}.
%\qquad\xhKbpHn\otimes\xhKbpv{\xaHopm},\qquad
%\DbgHdnm = \DbgHn\otimes\Dbgmd{\xaHopm}.
\xeee
\end{corollary}

The third corollary comes from the combination of \eex{eq:homtliso}, \rx{eq1.11a1a} and\rx{eq:fundefbim}: for
$\xal,\xbet\in\yCrn$
\begin{multline}
%\label{eq:seqiso}
\nonumber
\Hom_{\dTLtnz}(\dal,\dbet) = \zsymcat{\xald\tcmp\xbet}\qtshv{n} =
\zKsymbim{\xald}\otimes_{\xaHn}\Kbet\qtshv{n}
\\
= \ddulv{\Kal}\otimes_{\xaHn}\Kbet=
\Hom_{\xhKbpHn}\xlrb{\Kal,\Kbet}.
\end{multline}
\begin{corollary}
The functor
\xlee{eq:funeqztn}
\spfK\colon\dTLtlztn\xrightarrow{\;\;=\;\;}\cvxpmod{\xaHn},\quad\dTLztn \xrightarrow{\;\;=\;\;} \xhKbpHn
%,\quad \dTLpqztn\xrightarrow{\;\;=\;\;} \xhKmpqHdnm
\xeee
establishes an equivalence of categories.
\end{corollary}
%
%We also need a slightly modified version of this equivalence:
%%
%\xlee{eq:funeqmztn}
%\spfK\colon\dTLtnz \xrightarrow{\;\;=\;\;} \xhKbpHn
%\xeee
%
A combination of equivalences\rx{eq:eqdcat},\rx{eq:fnproTL}
and\rx{eq:funeqztn} leads to the equivalence of categories
\xlee{eq:bigeqm}
\spfK\colon \dTLcpqtmtn\xrightarrow{\;\;=\;\;}\xhKmpqHdnm.
\xeee

Let $\zDsymcat{\xdummy}$ be a composition of the categorification map
$\zKsymbim{\xdummy}$ and the functor $\spfKD$:
\ylee{eq:cmfuncp}
\xymatrix@C=1.5cm{
\yTngtmtn
\ar@/^2pc/[rr]^-{\zDsymcat{\xdummy}}
\ar[r]_-{\zKsymbim{\xdummy}}
&
\KbspHdnm
\ar[r]_-{\spfKD}
&
\DbgHdnm
}
\yeee

\begin{theorem}
\label{th:hpbss}
The elements $\QKz\xlrb{ \zDsymbim{\xbet\tcmp\xald} }$, $\xal\in\yCrm$, $\xbet\in\yCrn$ form a
basis of the $\Qq$ vector space $\QKz(\DbgHdnm)$.
\end{theorem}

\begin{corollary}
There are canonical isomorphisms
\ylee{eq:cisokgr}
\QKz\xlrb{\DbgHdnm} = \cfQTLcttmtn,\qquad
\Kzp\xlrb{\DbgHdnm} = \cTLctptmtn
\yeee
which identify the basis elements
$\QKz\xlrb{\zDsymbim{\xbet\tcmp\xald}}$ and
$\psymalg{\xbet\tcmp\xald}$.
\end{corollary}
From now on throughout the paper we will use $\cfQTLcttmtn$ and $\cTLctptmtn$ in place
of \Grtg s of $\DbgHdnm$.

\subsubsection{A universal projective resolution}

\begin{theorem}
\label{th:hprthrpq}
The projective resolution functor
$\spP\colon\DbgHdnm\rightarrow\xhKmpHdnm$ threads through \qpb\
complexes, so there is a commutative diagram
%$\DbgHdnm\xrightarrow{\spPpl}\xhKmpqHdnm\hookrightarrow\xhKmpHdnm$.
%\end{theorem}
%
%In other words, there is a commutative diagram
%
\ylee{eq:funthres}
\xymatrix{
\DbgHdnm
\ar@/^1.75pc/[rr]^-{\spP}
\ar[r]_-{\spPpl}
\ar[dr]_-{\spPTL}
&
\xhKmpqHdnm
\ar@{^{(}->}[r]
%\ar[r]
&
\xhKmpHdnm
\\
&
\dTLcpqtmtn
\ar[u]^-{\spfK}_-{=}
}
\yeee
which defines the functor $\spPTL$.
\end{theorem}

Let $\rsPn = \spPTL(\xaHn)$
%\in\Ob\dTLcpqtmtn$
be a `universal'
projective resolution of the $\xaHne$-module $\xaHn$. Since the
functor $\spfK$ transforms the tangle composition into the tensor
product, projective resolution\rx{eq:ostprj} of \sprj\
$\xaHdnm$-modules can be performed universally with the help of
the following
commutative diagram
\xlee{eq:cmdigprs}
\xymatrix{
\dTLtmtn
\ar[rr]^-{\rhPs}
\ar[d]_-{\spfK}
&&
\dTLcpqtmtn
\ar[d]^-{\spfK}_-{=}
\\
\KbspHdnm
\ar[rr]^-{\spPsp}
\ar[dr]_-{\spfKD}
&&
\xhKmpqHdnm
\\
&
\DbgHdnm
\ar[ur]_-{\spP}
}
\xeee
where the universal projective resolution functor $\rhPs$ is
defined similarly to \ex{eq:ostprj}:
\xlee{eq:fnactts}
\rhPs(\xdummy) = \rsPn\tcmp\xdummy = \xdummy\tcmp\rsPm.
\xeee
In fact, the commutativity of the square in the
diagram\rx{eq:cmdigprs} together with the second equality of
\ex{eq:ostprj} proves that $\rsPn\tcmp\xdummy =
\xdummy\tcmp\rsPm$.

\subsubsection{Hochschild homology}

Let $\Qgmodpq$ denote the category of $\ZZ\oplus\ZZ$-graded vector
spaces $V = \bigoplus_{i,j\in\ZZ} V_{i,j}$ whose homological
grading is bounded from above and which are also \qpb\ (see the
definition in subsection\rw{sss:hotcatcom}). Obviously,
$\Kzp\Qgmodpq=\ZZsqqi$ and $\Kzp$ acts on the objects of
$\Qgmodpq$ as the graded Euler characteristic $\chq$. Let
$\QgmodQ\subset\Qgmodpq$ denote the full subcategory whose objects
have the property that their Euler characteristic is a rational
function of $q$, that is, $\chq$ threads through $\Qq$:
\ylee{eq:grecth}
\xymatrix{
\QgmodQ
\ar@/^1.75pc/[rr]^-{\chq}
\ar[r]^-{\chq}
&
\Qq
\ar@{^{(}->}[r]
&
\Zsqqi
}
\yeee

\begin{theorem}
\label{th:hhomtl}
The \Hhom\ of  $\DbgHen$ lies within $\QgmodQ$:
\ylee{eq:hhomqq}
\xymatrix@C=1.5cm{
\DbgHen
\ar[r]_-{\xHHl(\xdummy)}
\ar@/^1.75pc/[rr]^-{\xHHl(\xdummy)}
&
\QgmodQ
\ar@{^{(}->}[r]
&
\Qgmodpq
}
\yeee
\end{theorem}

%By definition of the $\Tor$ functor, if the module $\aM$ is
%projective, then its \Hhom\ is just the homology of the tensor
%product: $\xHHlM = \Hmb(\aM\otimes_{\aRe}\aR)$. This formula has
%an important corollary:
%%
%\begin{theorem}
%\label{th:hhtr}
%Let $\xbA$ and $\xbB$ be two complexes in $\dTLtn$. If $\xbB$ is
%\tct, then there is a canonical isomorphism
%%
%\xlee{eq:1.13a}
%%\xHHlvv{\spfn(\xbA\tcmp\xbB)}{\xaHne}
%%= \TrSh(\xbA,\xbB).
%\Hmb\xlrb{
%\spfn(\xbA)\otimes_{\xaHne}\spfn(\xbB)
%} =
%%\HKhb
%\Hmb(\lSh{\xbA\tcmp\xbB}),
%\xeee
%%
%where $(\lSh{\xbA\tcmp\xbB})$ is the \Shcl\
%functor\rx{eq:fnct} applied to the composition $\xbA\tcmp\xbB$.
%\end{theorem}

\section{Results}
\label{s:res}

\subsection{Categorification}
\label{ss:catfic}
Admissible boundaries in \strtWRTt\ are 2-spheres $\Stn$ with $2n$ marked
points. Associated \xring s are $\cA(\Stn)=\xaHn$, canonical
involution\rx{eq:alginv} being the homomorphism\rx{eq:algopis}.
Then by the axiom\rx{eq:tnpcat}
\ylee{eq:catsphs}
\cC(\Stn) = \DbgHn,\qquad
\cC(\Stm\sqcup\Stn) = \DbgHdnm,
\yeee
where in the last equation we assume that the boundary component
$\Stm$ is `in', while the boundary component $\Stn$ is `out'.
Gluing formulas force us to define the category $\empc$ associated
with empty boundary as $\QgmodQ$, however we can not define the
action of functors $\dsym$ and $\finvm$ on it.

The
categorification map for 3-ball tangles is defined as
the unique map $\xobmpd\colon\yTngBtn\rightarrow\DbgHn$ which makes
the following diagram commutative
\ylee{eq:cattbt}
\xymatrix@C=1.5cm@R=1.5cm{
\yTngzn
\ar[r]^-{=}
\ar[d]_-{\zKsymbim{\xdummy}}
\ar[dr]^-{\mpDcat}
&
\yTngBtn
\ar@{-->}[d]^-{\xobmpd}
\\
\xhKbpHn
\ar@{^{(}->}[r]
&
\DbgHn
}
\yeee

Since gluing two 3-balls together produces a 3-sphere,
the gluing axiom determines the categorification map for links in
$\Sh$. Derived tensor product in $\DbgHn$ coincides with the
ordinary tensor product in $\xhKbpHn$, hence in view of
Theorem\rw{th:cmptp}, a link $\xL\subset\Sh$ is mapped into its Khovanov homology:
$\xobmpv{\xL}{\Bh}= \zsymcat{\xL}$.

The categorification map  $\xobmpd\colon
\yTngSt\rightarrow\DbgHdnm$
for \rtng s in $\StI$ is defined with the help of the following
theorem:
\begin{theorem}
\label{th:exunsti}
There exists a unique categorification map
$\xobmpd\colon\yTngSt\rightarrow\DbgHdnm$ such that
the following diagram is commutative:
\xlee{eq:cdiagsti}
\xymatrix@C=1.5cm@R=1.5cm{
\yTngtmtn
\ar[d]_-{\zKsymbim{\xdummy}}
\ar[dr]^-{\mpDcat}
\ar@{->>}[r]^-{\hoam}
&
\yTngSttmtn
\ar@{-->}[d]^-{\xobmpd}
\\
\KbspHdnm
\ar[r]_-{\spfKD}
&
\DbgHdnm
}
\xeee
\end{theorem}

\begin{theorem}
\label{th:tristi}
The diagram\rx{eq:cmtr}
\ylee{eq:cmsti}
\xymatrix@C=0cm{
& \yTngSttmtn
\ar[rd]^-{\xobmpd}
\ar[dl]_-{\xstmpd}
\\
\DbgHdnm
\ar[rr]^-{\Kz}
&&
\cfQTLcttmtn
}
\yeee
is commutative.
\end{theorem}

Since $\Sot$ can be constructed by gluing together the boundary
components of $\StI$, the gluing axiom requires that the categorification map for links in $\Sot$
should be determined by the following theorem:
\begin{theorem}
\label{th:exsot}
There exists a unique homology map
$\HKhbd\colon\yLnkSot\rightarrow \QgmodQ$ such that
the following diagram is commutative:
\ylee{eq:cdiagsts}
\xymatrix @C=1.5cm @R=1.5cm{
\yTngSttntn
\ar@{->>}[r]^-{\clSotv{\xdummy}}
\ar[d]^{\xobmpd}
&
\yLnkSot
\ar@{-->}[d]^-{\Hstbd}
\\
\DbgHdnm
\ar[r]^-{\xHHl(\xdummy)}
&
\QgmodQ
}
\yeee
\end{theorem}
\begin{theorem}
\label{th:cmtrisot}
The diagram\rx{eq:cmtr}
\xlee{eq:cmsts}
\xymatrix@C=0cm{
& \yLnkSot
\ar[rd]^-{\xstmpd}
\ar[dl]_-{\Hstbd}
\\
\QgmodQ
\ar[rr]^-{\chq}
&&
\Qq
}
\xeee
is commutative.
\end{theorem}

Next theorem shows that the universal resolution $\rsPn$ allows us to compute
the homology $\xobmpv{\xL}{\Sot}$ as a Khovanov-type homology within $\Sh$ and without reference to
the \xring s $\xaHdnm$.
\begin{theorem}
\label{th:hhclsh}
The following diagram is commutative:
\xlee{eq:comdgcomp}
\xymatrix@C=1.5cm@R=1cm{
\yTngtntn
\ar[r]^-{\hoam}
\ar[d]_-{\mpcat}
&
\yTngSttntn
\ar[r]^-{\clSotv{\xdummy}}
&
\yLnkSot
\ar[rd]^-{\Hstbd}
\\
\dTLtntn
\ar[r]^-{\rhPs}
&
\dTLcpqtntn
\ar[r]^-{\Hmb(\lSh{-})}
&
\Qgmodpq
&
\QgmodQ
\ar@{_{(}->}[l]
}
\xeee
\end{theorem}
In other words, if a link  in $\Sot$ is presented as a $\Sot$
closure of a \ttngtntn\ $\xtau$, then its \sthm\
%the homology $\xobmpv{\xL}{\Sot}$
is
isomorphic to the homology of the composition
$\dtau\tcmp\rsPn$ closed within $\Sh$:
\xlee{eq:cltlst}
%\Hmb\xlrb{\xobmpv{\hoam(\xtau)}{\Sot}}
\Hstb(\xtau,\Sot)
= \Hmb\xlrb{\lSh{\dtau\tcmp\rsPn}}.
\xeee
The universal projective resolution complex $\rsPn$ is
approximated by categorification complexes of torus braids with
high twist. As a result, \ex{eq:cltlst} provides an effective
method of computing \sthm\ of links in $\Sot$ by approximating it
with Khovanov homology of their `torus braid closures' within
$\Sh$.

\subsection{Infinite \cbr\ as a projective resolution
of $\xaHn$}

\subsubsection{Torus braids yield a projective resolution of
$\xaHn$}

%The complex $\prbC$ mentioned in Theorem\rw{th:cC} originates from \cbr s.

Let $\stA$ be a set of pairs of integer numbers confined within
a certain angle on a square lattice:
\xlee{eq:angset}
\stA = \{ (i,j)\in\ZZ^2\,|\,
%i,j\in\ZZ,\;\;
i\geq 0,\;\;i\leq j\leq
2i\}
\xeee
and let $\stAs{l}{k}$ denote
the shifted set $\stA$:
\xlee{eq:angsh}
\stAs{l}{k} = \{(i,j)\,|\,(i-k,j-l)\in\stA\}.
\xeee
We also introduce a special notation
\xlee{eq:sn1}
\stAxtnm = \stAs{\shlf m\zt^2 + m\zt - \shlf\zt + \shlf n}{\shlf
m\zt^2}.
\xeee

A complex $\xbA\in\Ob\dTLptn$ of \ex{ae1.8a}
is called \emph{\otbl} if the multiplicities $\ajmilam$
%$\xmtmijmlA$
%$\amimijml$
are non-zero only when $(i,j)\in\stA$. Obviously, in this case the
complex is an object of $\dTLpqtn$.
A complex
$\xbA$ is called \emph{\tct} if the multiplicities $\ajmilam$ are non-zero only
when the tangles $\xlam$ are \tct.

A \emph{truncation} of a complex $\xbA$ is defined as follows:
$\xtrnvv{m}{\xbA} = (\xAv{-m}\rightarrow
\xAv{-m+1}\rightarrow\cdots)$.

Let $\xbA$ be an object of $\dTLptn$. We use a notation
$\xbAc$ for a particular complex with special properties, which represents a homotopy equivalence class of
$\xbA$.

In Section\rw{s:braid} we construct special categorification
complexes of \cbr s.
\begin{theorem}
\label{th:spm}
There exists a sequence of special complexes $\cobrmtnsh$, $m=1,2,\ldots$,
such that
\begin{enumerate}
\item
a \TL\ tangle $\xlam$ with a
\thrd\ $\ztl$ may appear in $\cobrmtnsh$
only within a shifted angle region:
\xlee{eq:1.22}
\ajmilam > 0\qquad \text{ only if $(i,j)\in\stAxtlnm$};
\xeee
\item
there is an isomorphism of truncated complexes
\xlee{eq:a.10a}
\xtrntmov{\cobrmotnsh}\cong\xtrntmov{\cobrmtnsh}.
\xeee
\end{enumerate}
\end{theorem}
According to \ex{eq:a.10a}, braid complexes
$\xtrntmov{\cobrmtnsh}$ `stabilize' and it turns out that their
`stable limit' is the universal resolution $\rsPn$.
\begin{theorem}
\label{th:complex}
There exists a particular \otbl\ complex $\rsPnc\in\Ob\dTLpqtn$ representing the
homotopy equivalence class of $\rsPn$, such that the following
truncated complexes are isomorphic:
\xlee{eq:1.23}
\xtrntmov{\cobrmtnsh} \cong \xtrntmov{\rsPnc}.
\xeee
\end{theorem}

A combination of this theorem with formula\rx{eq:cltlst} leads to
a practical method of computing homology of a link in $\Sot$
presented as a closure of a \ttngtntn\ $\xtau$:
\begin{theorem}
\label{th:brappr}
For a \ttngtntn\ $\xtau$
there is a canonical isomorphism of
homologies
\xlee{eq:1.24}
\Hsti(\xtau,\Sot)= \HKhv{i}\xlrB{
\lSh{\xtau\tcmp\gbrmtn\,}}
%}
\qquad
%\text{for $ \uhdv{\dtau}-2m+2\leq i\leq \uhdv{\dtau}$.}
\xeee
for $i\geq \uhdv{\dtau}-2m+2$.%\leq i\leq \uhdv{\dtau}$.
\end{theorem}
Note that
since the complexes $\rsPnc$ and $\cobrmtnsh$ are trivial in
positive homological degrees, both sides of \ex{eq:1.24} are trivial for $i>\uhdv{\dtau}$.

\subsection{A conjecture about the structure of the \Hhom\ and
cohomology of the algebra $\xaHn$}

\Hhom\ and cohomology of $\xaHn$ is related to the homology of the
circular closure of the trivial braid within $\Sot$. Indeed,
combining \ex{eq:1.11a} and the definition of the categorification
map $\Hstbd\colon\yLnkSot\rightarrow\QgmodQ$ via Theorem\rw{th:exsot} we get
the isomorphism
\xlee{eq:1.49b}
\xHHl(\xaHn) =\Hstb\xlrB{\,\lSot{\gidbrtn}}.
\xeee

The algebra $\xaHn$ has a Frobenius trace with $q$-degree equal to $-2n$, hence its \Hchom\ is
dual to its \Hhom\ up to a degree shift:
\xlee{eq:1.52a}
\xHHu(\xaHn) = \dsmdv{\xHHl(\xaHn)}\qshv{2n}.
\xeee
This relation can be proved also with the help of universal categorification:
% explicitly through the following chain of isomorphisms:
%
%%%
%\begin{equation}
%\nonumber
%\begin{split}
%\xHHu(\xaHn) &=\Ext^\bullet_{\xaHe}(\xaHn,\xaHn) =
%\Hmb\xlrb{\Hom_{\xaHne}(\spfKCn,\xaHn)}
%=\Hmb\xlrB{ \dsymv{\spfKCn}\otimes_{\xaHne} \xaHn }
%\\
%&=\dsymv{\Hmb\xlrB{ \spfKCn\otimes_{\xaHne} \xaHn }}\qshv{2n}
%= \Hstb\xlrB{\lSot{\gidbrtn}}\qshv{2n} = \xHHl(\xaHn)\qshv{2n}.
%\end{split}
%\end{equation}
%%
%%
%
\begin{equation}
\nonumber
\begin{split}
\xHHu(\xaHn) &=
%\Ext^\bullet_{\xaHe}(\xaHn,\xaHn) =
\Hmb\xlrb{\Hom_{\xaHne}(\spP(\xaHn),\xaHn)}
=\Hmb\xlrb{ \dsmdv{\spP(\xaHn)}\otimes_{\xaHne} \xaHn }
=\Hmb\xlrb{ \dsymv{\spP(\xaHn)}\otimes_{\xaHne} \xaHn }\qshv{2n}
\\
&
=\Hmb\xlrb{\spfK(\dsymv{\rsPn})\otimes_{\xaHne}\xaHn }\qshv{2n}
= \Hmb\xlrb{\lSh{\dsymv{\rsPn}}}\qshv{2n}
=\dsmdv{\Hmb\xlrb{\lSh{\rsPn}}}\qshv{2n}
\\
&
=\dsmdv{\Hmb\xlrb{ \spP(\xaHn)\otimes_{\xaHne} \xaHn }}\qshv{2n}
=\dsmdv{\xHHl(\xaHn)}\qshv{2n}.
\end{split}
\end{equation}
Here the first and last equalities are the definitions\rx{eq:diagsHH} of the
Hochschild homology and cohomology, the third equality comes from
the second of equations\rx{eq:fundefbim}, while the seventh and fifth
equalities come from commutative triangles
\ylee{eq:cmtrigs}
\xymatrix@C=2.5cm{
\dTLcptn
\ar[d]_-{\spfK}
\ar[dr]^-{(\lSh{-})}
\\
\xhKmp(\xaHne)
\ar[r]_-{\Hmb(\xdummy\otimes_{\xaHne}\xaHn)}
&\Qgmodp
}
\qquad\qquad
\xymatrix@C=2.5cm{
\dTLcqtn
\ar[d]_-{\spfK}
\ar[dr]^-{(\lSh{-})}
\\
\xhKmq(\xaHne)
\ar[r]_-{\Hmb(\xdummy\otimes_{\xaHne}\xaHn)}
&\Qgmodq
}
\yeee
The first of these triangles is a part of the commutative
diagram\rx{eq:vbgcdg} (the $q^+$-boundedness condition plays no
role there), while the second diagram is its analog for complexes
bounded from below.

Let $\Tbrnnm$ denote the torus link which can be presented as $n$-cabling of the
unknot with framing number $m$: $\Tbrnnm = \xlrB{\lSh{\gbrmtn}}$,
and let $\Tbrnmnm=\dsymv{\Tbrnnm}$ denote its mirror image. Note
that the links are framed and each of $2n$ components of the link
$\Tbrnnm$ has a self-linking number $m$.
Combining Theorem\rw{th:brappr} with \eex{eq:1.49b} and\rx{eq:1.52a}
we arrive at the following theorem:
%
%The following is the
%result of Theorem\rw{th:brappr} when $\xtau=\gidbrtn$:
\begin{theorem}
There are isomorphisms
\ylee{eq:1.53}
\xHHmi(\xaHn) = \HKh_{-i}(\Tbrnnm),\qquad
\xHHui(\xaHn) = \HKh_{i}(\Tbrnmnm)\qshtn
%\qquad\text{for $i\leq 2m-2$}.
\yeee
for $i\leq 2m-2$
\end{theorem}
These isomorphisms were first observed by Jozef Przytycki\cx{PrzH} in
case of $n=1$.

The homologies of torus links can be evaluated for sufficiently
small $n$ and $m$ with the help of computer programs. This
experimental data led us to a conjecture regarding the
structure of the \Hchom\ of $\xaHn$ as a commutative algebra.

Consider the following lists of variables:
\def\bfx{\mathbf{x}}
\def\bfa{ \mathbf{a} }
\def\bfth{ \boldsymbol{\theta} }
\def\Irel{ I_{\mathrm{rel} } }
\begin{flalign}
\nonumber
\bfx &= x_1,\ldots,x_{2n} & \dgt x_i &=2 & \dgo x_i &= 0
\\
\nonumber
\bfa &= a_1,\ldots, a_n & \dgt a_i & = -2i-2 & \dgo a_i &= 2i
\\
\nonumber
\bfth &= \theta_1,\ldots,\theta_n & \dgt \theta_i &= -2i+2 & \dgo\theta_i &=
2i-1
\end{flalign}
%
%Here $\dgq$ is the $q$-degree and $\dgh$ is the homological
%degree.
The variables $\bfth$ have odd homological degree,  hence they anti-commute:
\begin{equation}
\nonumber
\theta_i\theta_j = - \theta_j\theta_i,\qquad \theta_i^2=0.
\end{equation}
\begin{conjecture}
The \Hchom\ of the algebra $\xaHn$ has the following graded commutative algebra structure:
\ylee{eq:1.54}
\xHHu(\xaHn) = \IQ[\bfx,\bfa,\bfth]/\Irel,
\yeee
where $\Irel$ is the ideal generated by relations
\begin{gather}
\nonumber
x_1^2 =\cdots = x_{2n}^2 = 0,\qquad x_1+\cdots + x_{2n} = 0,
\\
\nonumber
a_i p_i(\bfx) = \theta_i p_i(\bfx) = 0\qquad \text{for
$\forall p_i\in\IQ[\bfx]$ such that $\deg_{\bfx}p_i(\bfx) = i$},
\end{gather}
where in the second relation $p_i\in\IQ[\bfx]$ is any polynomial of
homogeneous $\bfx$-degree $i$.
\end{conjecture}

\section{Proofs of \TQFT\ properties}
\label{s:algpro}

\subsection{A derived category of $\xaHdnm$-modules}
\subsubsection{A special \Zgrdd\ algebra}

Consider an \xring\ $\aR$ which has some special properties
shared by all \xring s $\xaHdnm$.

A finite-dimensional non-negatively \Zgrdd\ \xring\ $\aR =
\bigoplus_{j\geq 0}\aRrj$ with an involution $\dflp\colon\aR\rightarrow\aRop$ is called \emph{\xcnv} if its zero-degree subalgebra
$\aRrz$ is generated by a finite number of mutually orthogonal
idempotents $\ideo,\ldots,\ideN$:
\ylee{eq:idemdef}
\idea\ideb =
\begin{cases}
\idea, & \text{if $\idia=\idib$,}
\\
0, & \text{if $\idia\neq\idib$,}
\end{cases}
\qquad
\idea+\cdots+\ideN = \xIdv{\aR},
\yeee
where $\xIdv{\aR}$ is the unit of $\aR$.
%
%We assume that $\aR$ is finite-dimensional and there is an
%isomorphism $\dflp\colon \aR\longrightarrow\aRop$. Furthermore, we
%assume that $\aR$ is non-negatively graded: $\aR =
%\bigoplus_{j\geq 0}\aRrj$ and that the zero-degree subalgebra
%$\aRrz$ is generated by a finite number of mutually orthogonal
%idempotents $\ideo,\ldots,\ideN$:
%%
%\ylee{eq:idemdef}
%\idea\ideb =
%\begin{cases}
%\idea, & \text{if $\idia=\idib$,}
%\\
%0, & \text{if $\idia\neq\idib$.}
%\end{cases}
%\yeee
%%
A \xcnv\ \xring\ $\aR$ has $\xlN$ indecomposable projective
modules $\idPa = \aR\idea$ and $\xlN$ one-dimensional \irred\
modules $\idSa = (\aR/\aRrp) \idea$, where $\aRrp = \bigoplus_{j\geq
1}\aRrj$.

\begin{theorem}[\cxw{Kh2}]
\label{th:khmod}
The modules $\idPa$ and $\idSa$, $\idia=1,\ldots,\xlN$ form a
complete list of indecomposable projective and, respectively,
\irred\ $\aR$-modules. The elements $\Kz(\idSa)$ generate freely
$\Kz(\DbmdR)$.
\end{theorem}

\begin{theorem}
Each $\aR$-module $M$ has a resolution
$\spP(M) = (\cdots\rightarrow\xAo\rightarrow\xAz)$ such that
$\yordq{\xAv{-i}}\geq i$.
%of the form
%%
%\ylee{eq:resmodm}
%\spP(M) = (\cdots\rightarrow\xAo\rightarrow\xAz),\qquad
%\xAi = \bpaoN\bigoplus_{j\in\ZZ} \aja\idPa\qshj,
%\yeee
%%
%such that $\aja=0$ if $\idia<i$.
\end{theorem}

\begin{proof}
Let $\Hom_j(\idPa,\idPb)$ denote the \qdgr\ $j$ part of
$\Hom(\idPa,\idPb)$. Then it is easy to see that
\xlee{eq:homzdg}
\Hom_{<0}(\idPa,\idPb)=0,\qquad
\Hom_0(\idPa,\idPb) =
\begin{cases}
\IQ,& \text{if $\idia=\idib$,}
\\
0, & \text{if $\idia\neq\idib$.}
\end{cases}
\xeee
and the generator of $\Hom_0(\idPa,\idPa)$ is the identity
homomorphism.

Consider a resolution of the $\aR$ module $M$. If a constituent projective module of the
resolution complex has a non-trivial zero-degree homomorphism
coming from it to another module $\idPa$, then this pair can be
contracted. After these contractions, if a constituent projective module does not
have a homomorphism originating from it, then it will contribute
to homology, because its zero-degree part can not be annihilated
by incoming homomorphisms. Since the homology of a resolution must
be concentrated in the zero homological degree, it follows that
$\yordq{\xAv{-i-1}}\geq\yordq{\xAv{-i}}$.
\end{proof}

\begin{corollary}
\label{cr:rprthrpq}
The projective resolution functor
$\spP\colon\DbmdR\rightarrow\xhKmpR$ threads through the
subcategory of \qpb\ complexes:
\ylee{eq:prresthrd}
\xymatrix{
\DbmdR
\ar@/^2pc/[rr]^-{\spP}
\ar[r]_-{\spPpl}
&
\xhKmpqR
\ar@{^{(}->}[r]
&
\xhKmpR
}
\yeee
\end{corollary}

This corollary allows us to use projective resolutions for the calculation of
$\Kz$ of objects of $\DbmdR$. Indeed, there is a commutative
diagram
\xlee{eq:comdgkz}
\xymatrix{
\xhKmpqR
\ar[r]^-{\Kzp}
&
\Kzp(\xhKmpqR)
%\ar[d]^-{=}
\ar@{->>}[d]
\\
\DbmdR
\ar@{^{(}->}[u]^-{\spPpl}
\ar[r]^-{\Kz}
&
\Kzp(\DbmdR)
}
\xeee
where by definition $\Kzp(\DbmdR) =
\Kz(\DbmdR)\otimes_{\Zqqi}\Zsqqi$. The right vertical map is
surjective, because $\Kzp(\DbmdR)$ is generated by $\Kz$-images of
objects in $\DbmdR$, but all those objects can be mapped into $\Kzp(\DbmdR)$
through the resolution $\spPpl$.

\begin{theorem}
\label{th:rabss}
The elements $\Kz(\idPa)$ form a basis in $\QKz(\DbmdR)$.
\end{theorem}

\begin{proof}
According to theorem\rw{th:khmod}, the elements $\Kz(\idSa)$
generate freely $\Kz(\DbmdR)$, hence they form a basis of
$\Kzp(\DbmdR)$ and $\dim\Kzp(\DbmdR)=\dim\QKz(\DbmdR)=\xlN$.
The right vertical map in the diagram\rx{eq:comdgkz} is
surjective, hence $\xlN$ elements $\Kz(\idPa)$ generate
$\Kzp(\DbmdR)$. Therefore they are linearly independent there and in $\QKz(\DbmdR)$,
so they form a basis of $\QKz(\DbmdR)$.
%
%Consider a sequence of homomorphisms
%%
%\ylee{eq:sqhom}
%\xymatrix{
%\DbmdR\ar[r]^-{\Kz}
%&
%\Kz(\DbmdR)
%\ar@{^{(}->}[r]
%&
%\QKz(\DbmdR)
%\ar@{^{(}->}[r]
%&
%\Kzp(\DbmdR)
%}
%\yeee
%%
%
%
%Commutative diagram\rx{eq:comdgkz} indicates that the elements
%$\Kz(\idPa)$ generate $\Kzp(\DbmdR)$, hence they are linearly independent there. At the same
%time, they belong to $\QKz(\DbmdR)$, and they keep their linear independence, hence they form a basis.
%\qed
\end{proof}

Corollary\rw{cr:rprthrpq} implies that the space $\Kzp(\DbmdR)$
has a $\Zsqqi$-valued symmetric bilinear pairing
\xlee{eq:prtor}
\blpbv{\Kzp(\xbA)}{\Kzp(\xbB)}=\chq\xlrb{\TorR(\xbAf,\xbB)},
\xeee
where, by definition, $\TorR$ is the homology of the derived tensor
product:
\ylee{eq:dftor}
\TorR(\xbAf,\xbB) = \Hmb\xlrb{\dflpv{\spPpl(\xbA)}\otimes_{\aR}\xbB} =
\Hmb\xlrb{\xbAf\otimes_{\aR}\spPpl(\xbB)}.
\yeee

\begin{proposition}
\label{pr:ratval}
The pairing\rx{eq:prtor} restricted to $\QKz(\DbmdR)$ takes values
in $\Qq$.
\end{proposition}
\begin{proof}
According to\cx{Kh2}, the irreducible modules $\idSa$ and projective modules $\idPa$
have the property
\xlee{eq:dimhoms}
\blpbv{\Kzp(\idSa)}{\Kzp(\idPb)} =
\dmq(\dflpv{\idSa}\otimes_{\aR}\idPb) =
\begin{cases}
1,&\text{if $\xal=\xbet$,}
\\
0, &\text{if $\xal\neq \xbet$.}
\end{cases}
\xeee
Since the elements $\Kzp(\idSa)$ and $\Kzp(\idPa)$ belong to
$\QKz(\DbmdR)$ and form (dual) bases there,
equation\rx{eq:dimhoms} implies that the pairing\rx{eq:prtor} on
$\QKz(\DbmdR)$ takes values in its base field $\Qq$.
\end{proof}

The following is obvious:
\begin{proposition}
\label{pr:tnseasy}
If \xring s $\aRo$ and $\aRt$ are \xcnv\ then $\aRo\otimes\aRt$ is
also \xcnv.
\end{proposition}
In particular, if $\aR$ is \xcnv\ then $\aRe=\aR\otimes\aRop$ is
also \xcnv.

\begin{theorem}
\label{th:hhomratg}
The \Hhom\ of $\DbgRe$ lies within $\QgmodQ$:
\ylee{eq:hrhomqq}
\xymatrix@C=1.75cm{
\DbgRe
\ar[r]_-{\xHHl(\xdummy)}
\ar@/^1.75pc/[rr]^-{\xHHl(\xdummy)}
&
\QgmodQ
\ar@{^{(}->}[r]
&
\Qgmodpq
}
\yeee
\end{theorem}
\begin{proof}
If $\xbM$ is a bounded complex
of $\aRe$-modules, then
\hyphenation{Hoch-schild}
according to the definition of the
%Hochschild homology
\Hhom,
%\ex{eq:flHHR},
$$\chq\xlrb{\xHHlv{\fKDbM}} =
\chq\xlrb{\Tor_{\aRe}(\xbM,\aR)}
%\chq\xlrb{\Hmb(\xbM\otimes_{\aRe}\spPr(\aR))}
=\blpbv{\xbM}{\aRop},$$
hence by Proposition\rw{pr:tnseasy}
$\chq\xlrb{\xHHlv{\fKDbM}}\in\Qq$ and by definition
$\xHHlv{\spfKD(\xbM)}\in\QgmodQ$.
\end{proof}

\subsubsection{Algebras $\xaHdmn$ and the universal resolution}

The \xring s $\xaHdnm$ are \xcnv. In view of
Proposition\rw{pr:tnseasy}, it is sufficient to check that $\xaHn$
is \xcnv. Indeed, it is easy to see from the definition\rx{eq:1.10}
that the zero-degree part of $\xaHn$ consists of identity
endomorphisms of objects $\dal$, and those morphisms are mutually
orthogonal idempotents.

Since the \xring s $\xaHdnm$ are \xcnv,
Theorems\rw{th:splprj},\rw{th:hpbss},\rw{th:hhomtl} and\rw{th:hprthrpq}
are particular cases of
Theorems\rw{th:khmod},\rw{th:rabss},\rw{th:hhomratg} and Corollary\rw{cr:rprthrpq}.

\begin{theorem}
\label{th:kzprj}
The map
$\Kzp\colon\dTLcpqtn\rightarrow\cTLctptn$ maps the \upr\ $\rsPn$ to the \JWp\ $\jwpxtnz$:
\xlee{eq:grmprpr}
\Kzp(\rsPn) = \jwpxtnz.
\xeee
\end{theorem}
\begin{proof}
For $\xal\in\yCrn$, the $\xaHn$-module $\dal$ is projective.
Projective resolution is unique up to homotopy, hence there is a
homotopy equivalence $\spfK(\rsPn)\otimes_{\xaHn}\Kal\hteqv\Kal$.
The functor $\spfK$ translates tangle composition into tensor
product and establishes a category equivalence\rx{eq:bigeqm},
hence $\rsPn\tcmp\dal\hteqv\dal$. Applying $\Kzp$ to both sides of
this relation, we find that $\Kzp(\rsPn) \tcmp\ \psymalg{\xal} =
\psymalg{\xal}$. Since $\Kzp(\rsPn)\in\cTLctptmtn$, Theorem\rw{th:unqprj} implies
\ex{eq:grmprpr}
\end{proof}

\begin{corollary}
The solid arrows of following diagram are commutative:
\xlee{eq:fbigdg}
\xymatrix{
\KbspHdnm
\ar[rrr]^(0.4){\spfKD}
&&&
\DbgHdnm
\ar[dddr]^(0.6){\Kz}|(0.34)\hole
\ar@{_{(}->}[dd]_(0.7){\spPTL}|(0.5)\hole
\\
&
\yTngtmtn
\ar[dd]_(0.3){\mpalg}
\ar[ul]_(0.4){\mpKbim}
\ar[dl]_-{\mpcat}
\ar@{->>}[rrr]^(0.35){\hoam}
\ar[urr]^(0.35){\mpDcat}
&&&
\yTngSttmtn
\ar[dd]^-{\mpalg}
\ar@{-->}[ul]_-{\mpcat}
\\
\dTLtmtn
\ar[uu]^-{\spfK}
\ar[dr]_-{\Kz}
\ar[rrr]^(0.7){\rhPs}|(0.417)\hole
&&&
\dTLcpqtmtn
\ar[ddr]^(0.35){\Kzp}|(0.515)\hole
\\
&
\cTLtmtn
\ar[rrr]^-{\jwphxstz}
&&&
\cfQTLcttmtn
\\
&&&&
\cTLctptmtn
\ar@{^{(}->}[u]
}
\xeee
%
%
%
%%
%\ylee{eq:crcdg}
%\xymatrix{
%\KbspHdnm
%\ar[rrr]^{\spfKD}
%&&&
%\DbgHdnm
%\ar@{_{(}->}[dd]_-{\spPTL}
%\ar[dddr]^-{\Kz}
%\\
%\\
%\dTLtmtn
%\ar[uu]_-{\spfK}
%\ar[dr]^-{\Kz}
%\ar[rrr]^-{\rhPs}
%&&&
%\dTLcpqtmtn
%\ar[dr]^-{\Kzp}
%\\
%&\cTLtmtn
%\ar[rrr]^-{\jwphxstz}
%&&&
%\cTLctptmtn
%\\
%}
%\yeee
\end{corollary}
\begin{proof}
The commutativity of the `skewed' bottom horizontal square follows from
Theorem\rw{th:kzprj}. The vertical square is a part of the
diagram\rx{eq:cmdigprs}. The commutativity of the right vertical
triangle follows from the commutativity of the
diagram\rx{eq:comdgkz} in the case when $\aR=\xaHdnm$. The
vertical upper-left triangle is the diagram\rx{eq:1.11b} and the
vertical lower-left triangle is the diagram\rx{eq:cmdgrm}.
Finally, the upper horizontal triangle is the definition of the
functor $\mpDcat$.
\end{proof}

\begin{remark}
Since $\xaHn\otimes_{\xaHn} \xaHn=\xaHn$, it follows that
$\spP(\xaHn)\otimes_{\xaHn}\spP(\xaHn) \hteqv \spP(\xaHn)$ and,
consequently,
\ylee{eq:cmpures}
\rsPn\tcmp\rsPn\hteqv\rsPn.
\yeee
This relation together with \ex{eq:grmprpr} suggests that $\rsPn$
is a categorification of the projector $\jwpxtnz$, but we think
that a proper setting for this statement would be a simultaneous
categorification of all projectors $\jwpxnm$ which would allow the verification
of the categorified orthogonality and completeness conditions\rx{eq:1.5}, so we leave it as a
conjecture.
\end{remark}

\subsubsection{Basic properties of Hochschild homology}

For two \xring s $\aRo$ and $\aRt$, let $\xbM$ and $\xbN$ be complexes of
$\aRt\otimes\aRoop$-modules and, respectively,
$\aRo\otimes\aRtop$-modules. Then the Hochschild homologies of
their derived tensor products are canonically isomorphic:
\wlee{eq:cndisocmh}
\xHHl(\xbM\Loti_{\aRo}\xbN) = \xHHl(\xbN\Loti_{\aRt}\xbM).
\weee
If $\xbM$ and $\xbN$ are \sprj, then their derived tensor products
coincide with the ordinary ones, so there is
%isomorphism\rx{eq:cndisocmh} implies
a simpler canonical
isomorphism
\xlee{eq:cnpisocmh}
\xHHl(\xbM\otimes_{\aRo}\xbN) = \xHHl(\xbN\otimes_{\aRt}\xbM).
\xeee

Suppose that an algebra $\aR$ has an involution
$\dflp\colon\aR\rightarrow\aRop$. It determines an involutive
functor $\zdflp\colon\xhKbR\rightarrow\xhKbv{\aRop}$, which turns
a $\aR$-module into a $\aRop$-module with the help of the
isomorphism $\dflp$. The \xring\ $\aRe=\aR\otimes\aRop$ is
canonically isomorphic to its opposite, hence $\dflp$ generates an
automorphism $\dflpe\colon\aRe\rightarrow\aRe$ and a corresponding
autoequivalence functor $\dflpe\colon\xhKbRe\rightarrow\xhKbRe$.

\begin{theorem}
For a complex of $\aRe$-modules $\xbM$ there is a canonical
isomorphism
\xlee{eq:isohhfl}
\xHHl(\xbM) = \xHHl(\dflpev{\xbM}).
\xeee
\end{theorem}
\begin{proof}
It is easy to see that the isomorphism $\dflp\colon\aR\rightarrow\aRop$ also establishes
an isomorphism of $\aRe$-modules
$\dflp\colon\aR\rightarrow\dflpev{\aR}$. Now the
isomorphism\rx{eq:isohhfl} is established by a chain of canonical
isomorphisms
\ylee{eq:chisohhfl}
\xHHl(\xbM) = \Tor_{\aRe}(\xbM,\aR) =
\Tor_{\aRe}(\dflpev{\xbM},\dflpev{\aR}) =
\Tor_{\aRe}(\dflpev{\xbM},\aR) = \xHHl(\dflpev{\xbM}).
\yeee
\end{proof}

\subsubsection{Hochschild homology and a closure within $\Sh$}

\begin{theorem}
\label{th:cmdgdtl}
The following diagram is commutative:
\xlee{eq:cmdgdtl}
\xymatrix@C=1cm@R=1cm{
\DbgHen
\ar[r]^-{\xHHl(\xdummy)}
\ar@{^{(}->}[d]_-{\spPTL}
&
\QgmodQ
\ar@{^{(}->}[d]
\\
\dTLcpqtn
\ar[r]^-{(\lSh{-})}
&
\Qgmodpq
}
\xeee
\end{theorem}
\begin{proof}
Consider the following diagram:
\xlee{eq:vbgcdg}
\xymatrix@R=1.5cm{
\xhKmpq(\xaHne)
\ar[rrd]^{\Hmb(\xdummy\otimes_{\xaHne}\xaHn)}
\ar[rr]^-{=}
&&
%\xhKmpq(\xaHopn)\otimes\xhKmpq(\xaHn)
\xhKpq\xlrB{ (\cvxpmod{\xaHopn})\otimes(\cvxpmod{\xaHn})}
\ar[d]^-{\Hmb(\xdummy\otimes_{\xaHn}\xdummy)}
\\
\DbgHen
\ar[u]^-{\spPpl}
\ar[d]_-{\spPTL}
\ar[r]^-{\xHHl(\xdummy)}
&
\QgmodQ
\ar@{^{(}->}[r]
&
\Qgmodpq
\\
\dTLcpqtn
\ar[rr]^-{=}
\ar@/^4pc/[uu]^-{\spfK}_-{=}
\ar[rru]^-{(\lSh{-})}
&&
%\dTLpqtnz\otimes\dTLpqztn
\xhKpq (\dTLtltnz\otimes\dTLtlztn)
\ar@/_8pc/[uu]_-{\spfK\otimes\spfK}^-{=}
\ar[u]_-{\xdummy\tcmp\xdummy}
}
\xeee
The lower left elementary triangle in it coincides with the
diagram\rx{eq:cmdgdtl}. All other elementary triangles, as well as
the outer frame, are commutative, hence the lower left elementary
triangle is also commutative.
\end{proof}

\subsection{Proof of categorification properties}
%So far, we have proved the commutativity of solid arrows in the
%diagram of \fg{fg:cdsti}.
%
%\begin{figure}
%\label{fg:cdsti}
%\ylee{eq:fbigdg}
%\xymatrix{
%\yTngtmtn
%\ar@{->>}[rrr]^-{\hoam}
%\ar[dr]^-{\mpKbim}
%\ar[dddr]_(0.42){\mpcat}
%\ar[dd]_-{\mpalg}
%&&&
%\yTngSttmtn
%\ar[dd]_(0.25){\mpalg}|(0.5)\hole
%\ar@{-->}[dr]^-{\mpcat}
%%\ar@{-->}@/^4cm/[dddr]^-{\mptl}
%\\
%&
%\KbspHdnm
%\ar[rrr]^(0.25){\spfKD}
%&&&
%\DbgHdnm
%\ar@{^{(}->}[dd]^-{\spPTL}
%\ar[dl]^-{\QKz}
%\\
%\cTLtmtn
%\ar[rrr]^(0.65){\jwphxstz}|(0.25)\hole|(0.375)\hole
%&&&
%\cfQTLcttmtn
%%\ar@{^{(}->}[dd]|(0.5)\hole
%\\
%&
%\dTLtmtn
%\ar[uu]_(0.3){\spfK}
%\ar[ul]^-{\Kz}
%\ar[rrr]^(0.25){\rhPs}
%&&&
%\dTLcpqtmtn
%%\ar[dl]^-{\Kzp}
%%\\
%%&&&
%%\cTLctptmtn
%}
%\yeee
%%
%\caption{A commutative categorification diagram for $\StI$}
%\end{figure}
%
%
%
%\begin{figure}[ht]
%%\label{fg:cdsti}
%\ylee{eq:fbigdg}
%\xymatrix{
%\KbspHdnm
%\ar[rrr]^(0.4){\spfKD}
%&&&
%\DbgHdnm
%\ar[dddr]_(0.55){\QKz}|(0.34)\hole
%\ar@{_{(}->}[dd]_(0.7){\spPTL}|(0.5)\hole
%\\
%&
%\yTngtmtn
%\ar[dd]_(0.3){\mpalg}
%\ar[ul]_(0.4){\mpKbim}
%\ar[dl]_-{\mpcat}
%\ar@{->>}[rrr]^(0.35){\hoam}
%\ar[urr]^(0.35){\mpDcat}
%&&&
%\yTngSttmtn
%\ar[dd]^-{\mpalg}
%\ar@{-->}[ul]_-{\mpcat}
%\\
%\dTLtmtn
%\ar[uu]^-{\spfK}
%\ar[dr]_-{\Kz}
%\ar[rrr]^(0.7){\rhPs}|(0.43)\hole
%&&&
%\dTLcpqtmtn
%\\
%&
%\cTLtmtn
%\ar[rrr]^-{\jwphxstz}
%&&&
%\cfQTLcttmtn
%}
%\yeee
%%
%\caption{A commutative categorification diagram for $\StI$}
%\label{fg:cdsti}
%\end{figure}
%
The proof of Theorem\rw{th:exunsti} is based on a theorem which we
will prove later:
\begin{theorem}
\label{th:bghmeq}
The following complexes are homotopy equivalent within
$\dTLcpqtmtn$:
\xlee{eq:thteqv}
\rhPs\sprbs{\Zbrtwptn} \hteqv \rhPs\sprbs{\Zobrotn} \hteqv
\rhPs\sprbs{\Zidbrtn}.
\xeee
\end{theorem}

\begin{proof}[Proof of Theorem\rw{th:exunsti}]
The commutativity of solid arrows in the diagram\rx{eq:fbigdg} together with the
injectivity of the functor $\spPTL$ mean that homotopy
equivalences\rx{eq:thteqv} imply \qim s
\xlee{eq:thriso}
\Dbrtwptn \qie \Dobrotn \qie \Didbrtn.
\xeee
According to Theorem\rw{th:rgheqv}, $\ker\hoam$ is generated by the
braids $\brtwpn$ and $\gobron$, hence \ex{eq:thriso} implies the
existence of the map $\xobmpd$. Its
uniqueness follows from the
surjectivity of the map $\hoam$.
\end{proof}
\begin{proof}[Proof of Theorem\rw{th:tristi}]
It is easy to see that the commutativity of solid arrows in the
diagram\rx{eq:fbigdg} together with the commutativity of the
upper horizontal triangle with a dashed side
(Theorem\rw{th:exunsti}) and surjectivity of the map $\hoam$ implies the full commutativity of the
diagram and, in particular, the commutativity of the right
vertical triangle claimed by Theorem\rw{th:tristi}.
%
%. If we replace there $\cTLctptmtn$ by
%$\cfQTLcttmtn$ then we get the commutative triangle of Theorem\rw{th:tristi}.
\end{proof}

\begin{proof}[Proof of Theorem\rw{th:exsot}]
Consider a commutative diagram
\ylee{eq:longcdiag}
\xymatrix @C=1.5cm @R=1.5cm{
\yTngtntn
\ar@{->>}[r]^-{\hoam}
%\ar[d]_-{\mpKbim}
\ar[dr]_-{\mpDcat}
&
\yTngSttntn
\ar@{->>}[r]^-{\clSotv{\xdummy}}
\ar[d]^{\xobmpd}
&
\yLnkSot
\ar@{-->}[d]^-{\Hstbd}
\\
%\KbspHne
%\ar[r]^-{\spfKD}
&
\DbgHen
\ar[r]^-{\xHHl(\xdummy)}
&
\QgmodQ
}
\yeee
The surjectivity of the maps $\hoam$ and $\clSotv{\xdummy}$
together with Theorem\rw{th:soeq} implies that the existence of
the homology map $\HKhbd$ would follow from the
isomorphisms
\ylee{eq:twoisoms}
\xHHl\xlrb{\Dsymcat{\xtaut\tcmp\xtauo}} =
\xHHl\xlrb{\Dsymcat{\xtauo\tcmp\xtaut}},\qquad
\xHHl\xlrb{\zDsymcat{\dflpv{\xtau}}} = \xHHl\xlrb{\Dsymcat{\xtau}},
\yeee
which should hold for any $\xtauo\in\yTngtntm$,
$\xtaut\in\yTngtmtn$ and $\xtau\in\yTngtntn$. The first
isomorphism follows from the isomorphisms\rx{eq1.11a1a}
and\rx{eq:cnpisocmh}, while the second isomorphism follows from
the isomorphisms\rx{eq:dfleq} and\rx{eq:isohhfl}.
\end{proof}
\begin{proof}[Proof of Theorem\rw{th:cmtrisot}]
The diagram\rx{eq:cmsts} coincides with the right face of the
solid cube in the following diagram:
\xlee{eq:bgcbdg}
\xymatrix@C=0.2cm@R=1cm{
&&\DbmdHne
\ar[rr]^-{\xHHl(\xdummy)}
\ar[dd]_(0.3){\spPpl}|(0.5)\hole
&
&
\QgmodQ
\ar@{^{(}->}[dd]|(0.5)\hole
\\
&
\yTngtntn
\ar@{-->>}[rr]_(0.3){\hoam}
\ar@{-->}[dl]_-{\mpcat}
&&
\yTngSttntn
\ar[ul]_-{\xobmpd}
\ar[dd]^(0.72){\mpalg}
\ar@{->>}[rr]^(0.35){\clSotv{\xdummy}}
&&
\yLnkSot
\ar[d]^-{\mpalg}
\ar[ul]_-{\Hstbd}
%{\xobmpd}
\\
\dTLtn
\ar@{-->}[rr]_-{\jwphxstz}
&&
\dTLcpqtn
\ar[dr]^-{\Kzp}
\ar[rr]^(0.3){\clShv{\xdummy}}|(0.47)\hole
&&
\Qgmodp
\ar[dr]^-{\chq}
&
\Qq
\ar@{^{(}->}[d]
\\
&&
&
\cTLctptn
\ar[rr]^{\clShv{\xdummy}}
&&
\Zsqqi
}
\xeee
The commutativity of all other solid cube faces has been established (in
particular, the vertical back face is the diagram\rx{eq:cmdgdtl}).
Hence the commutativity of the right face follows from the
surjectivity of the map
${\clSotv{\xdummy}}\colon\yTngSttntn\rightarrow\yLnkSot$.
\end{proof}

\begin{proof}[Proof of Theorem\rw{th:hhclsh}]
The addition of dashed arrows to the solid cube in the
diagram\rx{eq:bgcbdg} preserves the commutativity, because the
dashed arrows together with the left face of the solid cube form a
part of the diagram\rx{eq:fbigdg}. The diagram\rx{eq:comdgcomp} is
a part of the whole diagram\rx{eq:bgcbdg}.
\end{proof}

%\section{Topological proofs}
%\label{s:qiso}

%\subsubsection{Cobordism movies as morphisms between categorification complexes}

\subsection{Categorification complexes of \qtriv\ tangles}

An \elcb\ $\eleps$ between two tangle diagrams $\xtau$ and $\xtaup$
is a cobordism of one of the
following types: creation or annihilation of a disjoint circle, a
saddle cobordism, a Reidemeister move. To an \elcb\ of each type one
associates a special morphism between the corresponding
categorification complexes $\dtau\xrightarrow{\deleps}\dtaup$. In
particular, to a saddle cobordism one associates the corresponding \fltc\ acting on
constituent objects $\dlam$ in the complex $\dtau$.

A \emph{\cmov} is a sequence of \elcb s:
$\beleps=\elepso,\ldots,\elepsk$. A morphism $\dbeleps$
associated to a \xmov\ is a composition of \elmt\ morphisms:
$\dbeleps = \delepsk\cdots\delepso$.

\begin{theorem}
\label{th:cbsgn}
If two \cmov s $\beleps$ and $\beleps\p$ yield isotopic
cobordisms, then the corresponding morphisms are homotopy
equivalent up to a sign: $\dbeleps\hteqv\pm\dbeleps\p$.
\end{theorem}

A \Rmov\ is a \cmov\ which is a sequence of Reidemeister moves: $\zbr =
\rho_1,\ldots,\rho_k$.
% of Reidemeister moves.

\begin{definition}
\label{df:qtriv}
A \ttngtntn\ $\xtau$ is \emph{\qtriv} if it satisfies two
conditions:
\begin{enumerate}
\item
For any $\xal\in\yTngtnz$ (that is, for $\xal$ being a flipped
\crmt)
%\TL\ \ttngtnz\ $\xal$
there is a \Rmov\
$\zbral$ transforming $\xal\tcmp\xtau$ into $\xal$ such that
for any \fltc\  $\xphaot$ between two tangles $\xal$ and
$\xalp$ the following diagram is commutative:
\def\eleps{ \xphaot }
\xlee{eq:2.8cob}
\vcenter{
\xymatrix@C=1.5cm{
\xalo\tcmp\xtau \ar[r]^-{\eleps\tcmp\xIdtau}
\ar[d]^-{\yfhmao}
&
\xalt\tcmp\xtau
\ar[d]^-{\yfhmat}
\\
\xalo \ar[r]^-{\eleps}
&
\xalt
}
}
\xeee
that is, the \cmov s $\yfhmao\,(\eleps\tcmp\xIdtau)$  and
$\eleps\,\yfhmao$ are isotopic ($\xIdtau$ denotes the identity cobordism
between $\xtau$ and $\xtau$).

\def\xalp{ \beta }
\def\xalpz{ \beta }
\item
There exists a \crmt\ $\xalp\in\yCrn$
%\TL\ \ttngztn\ $\xalp$
and a \Rmov\ $\yfhmap$
transforming $\xtau\tcmp\xalp$ into $\xalp$, such that for any
flipped \crmt\ $\xal$
the following \Rmov s are isotopic:
\xlee{eq:2.9cob}
\xymatrix@C=2cm{
\xal\tcmp\xtau\tcmp\xalpz\;\;\;
\ar@/^1.5pc/[r]^-{\yfhma\tcmp\xIdapz}
\ar@/_1.5pc/[r]^-{\xIda\tcmp\yfhmap}
&
\xal\tcmp\xalpz
}
\xeee

\end{enumerate}
\end{definition}
\def\xalp{ \beta }
\def\xalpz{ \beta }

\begin{remark}
\label{rm:qtriv}
If $\eleps$ is an \xmult, then the first condition of this
definition is satisfied automatically, so if $\xtau$ is \qtriv,
then the second condition is satisfied for any \efltc.
\end{remark}

To a \ttngtntn\ $\xtau$ we associate a functor
$\htau\colon\dTLtltnz\rightarrow\dTLtnz$ which acts by composing
with $\dtau$: $\htau = \xdummy\tcmp\dtau$.
\begin{theorem}
\label{th:qrtfiso}
If a \ttngtntn\ $\xtau$ is \qtriv, then the tangle composition functor $\htau$ is
isomorphic to the injection functor
$\dTLtltnz\hookrightarrow\dTLtnz$.
\end{theorem}

\begin{proof}
As an additive category, $\dTLtltnz$ is generated freely by
objects $\dal$, where $\xal$ are flipped \crmt s:
$\xal\in\yTngtnz$.

According to Definition\rw{df:qtriv}, for any $\xal$ the tangles
$\xal$ and $\xal\tcmp\xtau$ are isotopic,
%for any $\xal$ there exists a \Rmov\ $\zbral$ transforming $\xal\tcmp\xtau$ into $\xal$,
hence there is a homotopy equivalence
$\dal\tcmp\dtau\hteqv\dal$.
Hence, by the definition of functor isomorphism, it remains to
prove that for every $\xal$ there exits a particular homotopy
equivalence  $\hefa\colon\zsymcat{\xal\tcmp\xtau}\rightarrow\dal$
such that for any pair of \fcrmt s $\xalo$, $\xalt$ and for any
\fltc\ $\xphaot$ between $\xalo$ and $\xalt$
the following diagram is
commutative:
\xlee{eq:2.8b}
\vcenter{
\xymatrix@C=1.5cm{
\symcat{\xalo\tcmp\xtau} \ar[r]^-{\hphaot\tcmp\xIdtau}
\ar[d]^-{\hefao}
&
\symcat{\xalt\tcmp\xtau}
\ar[d]^-{\hefat}
\\
\dalo \ar[r]^-{\hphaot}
&
\dalt
}
}
\xeee
In view of Proposition\rw{pr:rpc1}, it is sufficient to prove this commutativity for
$\xphaot$ being an \efltc, that is, either an \xmult\ of a saddle
cobordism.

Theorem\rw{th:cbsgn} says that since the cobordism
diagram\rx{eq:2.8cob} is commutative, the following diagram is
commutative up to a sign:
\xlee{eq:2.8f}
\label{eq:2.8a}
\vcenter{
\xymatrix@C=1.5cm{
\symcat{\xalo\tcmp\xtau} \ar[r]^-{\hphaot\tcmp\xIdtau}
\ar[d]^-{\xfhmao}
\ar@{}[dr]|{\displaystyle \pm}
&
\symcat{\xalt\tcmp\xtau}
\ar[d]^-{\xfhmat}
\\
\dalo \ar[r]_-{\hphaot}
&
\dalt
}
}
\xeee

Since the \Rmov s\rx{eq:2.9cob} are isotopic, according to
Theorem\rw{th:cbsgn} there exists a sign factor $\sgmua$ such that
\xlee{eq:2.9d}
\sgmua\xfhma\tcmp\xIdapz \hteqv \xIda\tcmp\xfhmapz.
\xeee
We choose homotopy equivalences as $\hefa = \sgmua\xfhma$. The `up
to a sign commutativity' of the diagram\rx{eq:2.8f} implies that
there exists a sign factor $\smmuf=\pm 1$ such that the diagram
\xlee{eq:2.8m}
\vcenter{
\xymatrix@C=1.5cm{
\symcat{\xalo\tcmp\xtau} \ar[r]^-{\hphaot\tcmp\xIdtau}
\ar[d]^-{\sgmuao\xfhmao}
&
\symcat{\xalt\tcmp\xtau}
\ar[d]^-{\sgmuat\xfhmat}
\\
\dalo \ar[r]^-{\smmuf\hphaot}
&
\dalt
}
}
\xeee
is commutative. It remains to prove that $\smmuf=1$.

Consider the following diagram:
\ylee{eq:barcmd}
\xymatrix@C=4cm@R=2.5cm{
\symcat{\xalo\tcmp\xtau\tcmp\xalpz}
\ar[r]^-{\xphaot\tcmp\xIdtau\tcmp\xIdapz}
\ar@/_2pc/[d]_-{\sgmuao\xfhmao\tcmp\xIdapz}
\ar@/^2pc/[d]^-{\xIdv{\xalo}\tcmp\xfhmapz}
&
\symcat{\xalt\tcmp\xtau\tcmp\xalpz}
\ar@/^2pc/[d]^-{\sgmuat\xfhmat\tcmp\xIdapz}
\ar@/_2pc/[d]_-{\xIdv{\xalt}\tcmp\xfhmapz}
\\
\symcat{\xalo\tcmp\xalpz}
\ar[r]^-{\xphaot\tcmp\xIdapz}
&
\symcat{\xalt\tcmp\xalpz}
}
\yeee
The inner squeezed square is commutative, because cobordisms
$\xphaot$ and $\xfhmapz$ act on different parts of the composite
tangles $\xalo\tcmp\xtau\tcmp\xalpz$ and
$\xalt\tcmp\xtau\tcmp\xalpz$.
The commutativity of the left and
right faces is equivalent to \ex{eq:2.9d} for $\xalo$ and $\xalt$.
Hence the whole diagram is commutative. Compare the
commutativity of its outer bloated square with the tangle composition of
the diagram\rx{eq:2.8m} with $\xbet$:
\ylee{eq:barcmd}
\xymatrix@C=2cm@R=1.5cm{
\symcat{\xalo\tcmp\xtau\tcmp\xalpz}
\ar[r]^-{\hphaot\tcmp\xIdtau\tcmp\xIdapz}
\ar[d]_-{\sgmuao\xfhmao\tcmp\xIdapz}
&
\symcat{\xalt\tcmp\xtau\tcmp\xalpz}
\ar[d]^-{\sgmuat\xfhmat\tcmp\xIdapz}
\\
\symcat{\xalo\tcmp\xalpz}
\ar[r]^-{\smmuf\hphaot\tcmp\xIdapz}
&
\symcat{\xalt\tcmp\xalpz}
}
\yeee
Since $\xphaot$ is an \efltc, the linear map
$\hphaot\tcmp\xIdapz\colon\symcat{\xalo\tcmp\xalpz}\rightarrow\symcat{\xalt\tcmp\xalpz}$
is non-zero, hence $\smmuf=1$.
\end{proof}

%The category equivalence\rx{eq:kpqtl} leads us to the following
\begin{corollary}
\label{cr:qrtfiso}
If a \ttngtntn\ $\xtau$ is \qtriv, then the tangle composition functor
$$\htau\colon\dTLcpqtn\longrightarrow\dTLcpqtn,\qquad\htau = \xdummy\tcmp\dtau$$ is
isomorphic to the identity functor and, in particular,
for any complex $\xbA$ in $\dTLcptntm$ there
is a homotopy equivalence
\xlee{eq:2.4a}
\xbA\tcmp\xtau \hteqv \xbA.
\xeee
\end{corollary}
\begin{proof}
Use the category equivalence\rx{eq:kpqtl} to replace $\dTLcpqtn$
with $\xhKpq (\dTLtltnz\otimes\dTLtlztn)$. The functor $\htau$
acts as identity on the $\dTLtltnz$ factor and its action on the
$\dTLtlztn$ factor is equivalent to identity by
Theorem\rw{th:qrtfiso}
\end{proof}

%\begin{lemma}
%\label{lm:hteqtau}
%If a \ttngtntn\ $\xtau$ is \qtriv,
%then for any complex $\xbA$ in $\dTLcptntm$ there
%is a homotopy equivalence
%%
%\xlee{eq:2.4ax}
%\xbA\tcmp\xtau \hteqv \xbA.
%\xeee
%%
%\end{lemma}
%%
\begin{proof}[Proof of Theorem\rw{th:bghmeq}]
According to the definition\rx{eq:fnactts} of the functor $\rhPs$,
the homotopy equivalences\rx{eq:thteqv} are explicitly
\xlee{eq:hmeqvs}
\rsPn\tcmp\Zbrtwptn \hteqv
\rsPn\tcmp\Zobrotn \hteqv
\rsPn.
\xeee

It is easy to see that the tangles $\yvspoh\brtwptn\;$ and
$\gobrotn$ are \qtriv\ (in fact, for these tangles any
$\xbet\in\yCrn$ satisfies the second condition of
Definition\rw{df:qtriv}). Since $\rsPn$ is an object of $\dTLcpqtn$, the homotopy
equivalences\rx{eq:hmeqvs} follows from
that of \ex{eq:2.4a}.
\end{proof}

\section{Properties of the universal categorification complex of a
\cbr}
\label{s:braid}

\subsection{A \mtcn\ structure of a chain complex}

In many instances, in order to simplify a complex
through homotopy equivalence, we will use
its presentation as a \mtcn, that is, as a multiple application of
cone construction. A general theory of this procedure is related
to Postnikov structures and it is described, \eg, in\cx{GM}. We
need only a tiny bit of this theory, as it applies to homotopy
categories.

Let $\xChmA$ be a category of bounded from
above \chcpls\ associated with an additive
category $\xctA$. An object of $\xChmA$ is a \chcpl\
\ylee{ae1.ch1}
(\xbA,\xbd)  = (\cdots \rightarrow
\xAi\xrightarrow{\xdi}\xAio\rightarrow\cdots\rightarrow\xAv{k}),
\yeee
and a morphism
between two chain complexes is
%%a \chmp\ defined as
%a `\mmp', that is,
a sequence of morphisms $\xbf =( \ldots,\yfi,\ldots )$:
\xlee{ae1.10d}
\vcenter{\xymatrix{
\xbA \ar[d]^-{\xbf} &&
\cdots\ar[r]^-{\xdimo} & \xAi \ar[r]^-{\xdi} \ar[d]^-{\yfi} &
\xAio
\ar[r]^-{\xdio} \ar[d]^{\yfio} & \cdots
\\
\xbB &&
\cdots\ar[r]^-{\xdpimo} & \xBi \ar[r]^{\xdpi} & \xBio
\ar[r]^-{\xdpio} & \cdots
}
}
%,\qquad\qquad
%\xdpio\,\yfio = \yfi\,\xdio
\xeee
%
%which commutes with the chain differential: $\xdpi\,\yfi = \yfio\,\xdi$ for all $i$.
An associated homotopy category $\xKhmA$ has the same objects as
$\xChmA$, while the morphisms are \chmp s (that is, \mmp s which
commute with differentials) up to homotopy.

A \mtcn\ in the category $\xChmA$ is a family of complexes $(\xbAa)_{\ina\in\inA}$,
where $\inA$ is an index set with a grading function
$\inh\colon\inA\rightarrow\ZZ$
such that $\inh(\inA)$ is bounded from above and $\inh^{-1}(n)$ is
finite for any $n\in\ZZ$. The complexes of the \mtcn\
%with adjacent \xhdgr\
are connected by \mmp s
$\xbAa[-1]\xrightarrow{\xbfaap}
\xbAap$ if $\inh(\ina)< \inh(\inap)$, and the \mmp s satisfy the
condition
\ylee{ae1.10d2}
\xbfvv{\ina}{\inapp}\xbdv{\ina} +
\xbdv{\inapp}\xbfvv{\ina}{\inapp} +
\sum_{\substack{\inap\in\inA\\\inh(\ina)<\inh(\inap)<\inh(\inapp)}}
\xbfvv{\inap}{\inapp}\xbfvv{\ina}{\inap} = 0.
\yeee
It guarantees that the \mtcn\ determines a `\tcomp' $(\xbA,\xbd)$ in $\xChA$, whose
\qcmds\ are direct sums of \qcmds\ of $\xbAa$ and differentials
are the sums of \mmp s $\xbfvv{\ina}{\inap}$:
$\xbA = \bigoplus_{\ina\in\inA}\xbAa$, $\xbd =
\sum_{\ina,\inap\in\inA}\xbfvv{\ina}{\inap}$.
%$\xlrb{\bigoplus_{\ina\in\inA}\xbAa,\sum_{\ina,\inap\in\inA}\xbfvv{\ina}{\inap}}$.
If a complex $\xbAd$ of $\xChA$ is presented as a \tcomp\ of
a \mtcn, then we say that $\xbAd$ has a \mtcn\ structure and we refer
to $\xbAa$ as \xcnstc es.

The following easy proposition explains why the \mtcn\ structure helps
to simplify a complex within its homotopy equivalence class.

%$i\in\ZZ$ connected by \mmp s $\xbAi[-1]\xrightarrow{\xbfij}\xbAj$ for
%$i\leq j$, such that $\xbfii = \xbdi$ and
%%
%\ylee{ae1.10d1}
%\sum_{i\leq k\leq j} \xbfvv{i}{k}\xbfvv{k}{j} = 0\qquad\text{for
%$i\leq j$.}
%\yeee
%%
%Note that this condition implies that $\xbAi\xrightarrow{\xbfvv{i}{i+1}}\xbAio$ are \chmp
%s.
%
%A \mtcn\ determines an object
%of $\xChA$, whose \qcmds\ are direct sums of \qcmds\ of $\xbAi$
%and differentials are sums of \mmp s $\xbfij$:
%$\xlrb{ \bigoplus_{i} \xbAi,\sum_{i,j}\xbfij}$.

%A \mtcn\ structure of a complex helps to describe its homotopy
%transformations due to the following easy proposition.
\begin{proposition}
\label{pr:hommcn}
For a homotopy equivalent family of complexes $\xbAa\p\hteqv\xbAa$
with the same index set $\inA$
there exist \mmp s $\xbfvv{\ina}{\inap}\p$ such that the resulting \mtcn\ is
homotopy equivalent to the original one:
$\xlrb{ \bigoplus_{\ina\in\inA} \xbAi,\sum_{\ina,\inap\in\inA}\xbfvv{\ina}{\inap}}\hteqv
\xlrb{ \bigoplus_{\ina\in\inA} \xbAa\p,\sum_{\ina,\inap\in\inA}\xbfvv{\ina}{\inap}\p}$.
\end{proposition}
%In other words, one can simplify \xcnstc es knowing that \mmp s
%can be adjusted accordingly.

In this paper \mtcn\ structures emerge when categorification
complexes of tangle compositions are considered. For example,
consider a composition $\dtau\tcmp\xbA$ of two complexes in $\dTLn$,
the first object being a categorification
complex of a \ttngnn\ $\xtau$. Constituent complexes of
$\dtau\tcmp\xbA$
are formed by the compositions $\dtau\tcmp\dlam =
\symcat{\xtau\tcmp\xlam}$, where $\xlam$ are \xcnstt s of $\xbA$.
%, while $\inh(\xlam)$ is equal to the homological degree of the
%\qcmd\ of $\xbA$ which contains $\dlam$.
Now we can use
Reidemeister moves in order to simplify the tangles
$\xtau\tcmp\xlam$, knowing that the \mmp s of the \mtcn\ can be
adjusted accordingly, so that the modified \tcomp\ will be homotopy
equivalent to the original complex $\dtau\tcmp\xbA$.

\subsection{A cup sliding trick}
\label{ss:trick}
In this section we will prove Theorem\rw{th:spm} by induction over
$m$. In proving it for $m=1$ and in deducing the case of $m+1$ from
the case of $m$ we will use a special trick.
%(once with a slight
%modification involving the addition of extra straight strands to
%some tangles).

Let $\xgami$, $i=1,2,\ldots$ be a sequence of \ttngvv{i}{\ip}s (in our applications
$\ip=i$ or $\ip=i+2$) such
that any \uptg\ can slide through them:
%Suppose that there is a sequence of \ttngvv{i}{i} $\xgami$
%($i=1,2,\ldots$), such that  \uptg s can slide through them:
%
\xlee{eq:2.5b1}
\xgami\tcmp
\gcupvI{i}{-0.5}
%\gcupnI
\isotp
%\gcupnI
\gccupv{\ip}{\tilde{\stI}}{-0.5}
\tcmp\xgamv{i-2d},
\xeee
where $d$ is the number of cup arcs in $\gcupnI$. Suppose that all tangles
$\xgami$ have special categorification complexes
$\spsymbcat{\xgami}$ and we have to construct a special
categorification complex for the composition $\xtaup=\xgamn\tcmp\xtau$,
where $\xtau$ is a \ttngmn\ with a special categorification
complex $\spsymbcat{\xtau}$.

First, we represent $\symcat{\xtaup}$ by the tangle-composition of
complexes
$\symcat{\xgamn}\tcmp\spsymbcat{\xtau}$ (the choice of
$\symcat{\xgamn}$ does not matter). We view this composition as a
\mtcn\ constructed by replacing every \xcnsto\ $\dlam$ of
$\spsymbcat{\xtau}$ with the composition
$\symcat{\xgamn}\tcmp\dlam$. In order to transform
the \mtcn\ $\symcat{\xgamn}\tcmp\spsymbcat{\xtau}$ homotopically into the special complex
$\spsymbcat{\xtaup}$, we use the presentation\rx{aes2.3a}
for $\xlam$ and isotopy\rx{eq:2.5b1} in order to replace this
composition with a homotopy equivalent one
\xlee{eq:2.5b2}
%\bsymcat{\gcupvI{n\p}{-0.5}}
\bsymcat{ \gccupv{\np}{\tstI}{-0.5} }
\tcmp\spsymbcat{\xgamv{\ztl}}\tcmp
\bsymcat{\gcapvJ{m}{0.75}},
\xeee
where $\ztl$ is a \thrd\ of $\xlam$.
The homotopy equivalence transformation of complexes
$\symcat{\xgamn}\tcmp\dlam$ into the complexes\rx{eq:2.5b2} may
change the morphisms of the \mtcn\ in a rather non-trivial way,
but this does not matter, since in this section we are
interested mostly in the position of a \xcnsto\ within the complex
$\spsymbcat{\xtaup}$. This position is determined by the position of the
original object $\dlam$ within $\spsymbcat{\xtau}$ and the
structure of the special complex $\spsymbcat{\xgamv{\ztl}}$:
\begin{proposition}
\label{pr:trick}
If an object $\dlamp$ appears in a position $\dlamp\tgrsshji$ in the
complex $\spsymbcat{\xgamv{\ztl}}$, then the corresponding object
in the composition\rx{eq:2.5b2} appears in the same position
$\bsymcat{\gccupv{\np}{\tstI}{-0.5}\tcmp
\xlamp\tcmp\gcapvJ{m}{0.75}}\tgrsshji$.
\end{proposition}
\begin{proof}
The homological degree $\dgo$ is additive with respect to the
tangle composition. The $q$-degree $\dgt$ is not additive because
of \ex{ae1.01} required to remove disjoint circles, but the
tangle composition
$\gccupv{\np}{\tstI}{-0.5}\tcmp \xlamp\tcmp\gcapvJ{m}{0.75}$
does not contain disjoint circles, hence $\dgt$ is also additive
in\rx{eq:2.5b2}. Now the claim of the proposition follows from the
fact that the degrees $\dgo$ and $\dgt$ of cap and cup tangles
in\rx{eq:2.5b2} are zero.
\end{proof}

\subsection{A special categorification complex of a single twist
\cbr}

Recall the definitions\rx{eq:angset}--\rxw{eq:sn1} of the set
$\stA$ and its shifts. We also use an abbreviated notation
$\stAxtn = \stAxvv{\zt}{n}{1} =
\stAs{\shlf(\zt^2+\zt+n)}{\shlf\,\zt^2}$.

%that $\stA$ be a set of pairs of integer numbers confined within
%a certain angle: $\stA = \{ (i,j)\,|\,i,j\in\ZZ,\;\;i\geq 0,\;\;i\leq j\leq
%2i\}$ and $\stAs{l}{k}$ denotes the shifted set $\stA$:
%$\stAs{l}{k} = \{(i,j)\,|\,(i-k,j-l)\in\stA\}$.
%
%We also denote $\stAxtn = \stAs{\shlf(\zt^2+\zt+n)}{\shlf\,\zt^2}$.

\begin{theorem}
\label{th:stro}
For a \cbr\ $\gobron$ there exists a special categorification
complex
\\
$\cobronsh$ with the following property:
\ylee{eq:a.22x}
\cjmilam > 0,\qquad\text{only if
$(i,j)\in\stAxtln$}.
%\stAs{\shlf  (\ztl^2 + \ztl+ n)}{\shlf \ztl^2}$}.
\yeee
\end{theorem}

The proof of this theorem is based on two lemmas.

Define the following tangle notations:
\begin{gather}
\brcapn=
\setlength{\unitlength}{1mm}
\begin{picture}(30,10)(-9,3)
\thlb
\put(0,0){\oval(10,10)[tl]}
\put(0,5){\line(1,0){10}}
\put(10,0){\oval(10,10)[tr]}
\put(0,0){\line(0,1){4}}
\put(0,6){\line(0,1){4}}
\put(10,0){\line(0,1){4}}
\put(10,6){\line(0,1){4}}
\put(2.5,7.5){$\cdots$}
\put(2.5,0){$\cdots$}
\put(-6,-4){$\scriptstyle{1}$}
\put(-0,-4){$\scriptstyle{2}$}
\put(15,-4){$\scriptstyle{n}$}
\put(7,-4){$\scriptstyle{n-1}$}
\end{picture},
\qquad
\brcupn=
\setlength{\unitlength}{1mm}
\begin{picture}(30,10)(-9,3)
\thlb
\put(-1,10){\oval(10,10)[bl]}
%\put(0,5){\line(1,0){10}}
\multiput(2.5,5)(1.5,0){5}{\line(-1,0){0.5}}
\put(11,10){\oval(10,10)[br]}
\put(0,0){\line(0,1){10}}
%\put(0,6){\line(0,1){4}}
\put(10,0){\line(0,1){10}}
%\put(10,6){\line(0,1){4}}
\put(3,7.5){$\cdots$}
\put(3,0){$\cdots$}
\put(-7,12){$\scriptstyle{1}$}
\put(-0,12){$\scriptstyle{2}$}
\put(16,12){$\scriptstyle{n}$}
\put(7,12){$\scriptstyle{n-1}$}
\end{picture},
\\
\brccpn = \brcupn\tcmp\brcapn =
\setlength{\unitlength}{1mm}
\begin{picture}(30,25)(-9,3)
\thlb
\put(-1,15){\oval(10,10)[bl]}
\put(0,5){\line(1,0){10}}
\multiput(2.5,10)(1.5,0){5}{\line(-1,0){0.5}}
\put(11,15){\oval(10,10)[br]}
\put(0,6){\line(0,1){9}}
%\put(0,6){\line(0,1){4}}
\put(10,6){\line(0,1){9}}
%\put(10,6){\line(0,1){4}}
\put(0,0){\oval(10,10)[tl]}
\put(10,0){\oval(10,10)[tr]}
\put(0,0){\line(0,1){4}}
\put(10,0){\line(0,1){4}}
\put(3,12.5){$\cdots$}
\put(3,0){$\cdots$}
\put(-6,-4){$\scriptstyle{1}$}
\put(-0,-4){$\scriptstyle{2}$}
\put(15,-4){$\scriptstyle{n}$}
\put(7,-4){$\scriptstyle{n-1}$}
\end{picture}.
\end{gather}
\vspace*{0.5cm}
and a \rshp: $\ryA=\{(i,i)\,|\,i\in\ZZ,\;i\geq 0\}$.

\begin{lemma}
\label{lm:cpcp}
The tangle $\brccpn$ has a special categorification complex
$\cbrccpnsh$ with two properties:
\begin{enumerate}
\item
$\cbrccpnsh$ has a \rshp\ $\ryA\hqshv{2-n}$;
%\tgrsshv{-1}{-1}^{n-2}$;
\item
the object
$\cidbrn$ does not appear in $\cbrccpnsh$.
\end{enumerate}
\end{lemma}
\begin{proof}
Let $\cbrcapnsh$ denote the standard categorification complex of
$\brcapn$. Since Kauffman splicing of crossings in the diagram of
this tangle does not produce disjoint circles, it follows from
\ex{ae1.7}
that $\cbrcapnsh$ has the shape
$\ryA\hqshv{\frac{2-n}{2}}$
%\tgrsshv{-\shlf}{-\shlf}^{n-2}$.
By the same argument, the
standard categorification complex $\cbrcupnsh$ has the same shape.

Now we use the trick of subsection\rw{ss:trick}, where
%$\xgami = \brcupvv{i+2}{-57}{65}{-60}$,
$\xgami = \brcupv{i}$,
$\xtau = \brcapn$ and $\xtaup = \brccpn$.
After the sliding of cups, a
\xcnsto\ of $\cbrcapnsh$
corresponding to a \TL\ tangle
$\xlam = \gcupvI{n-2}{-1.1}
%\gcupnIp
\tcmp\gcapnJ$
turns into the complex
%
%In order to construct $\cbrccpnsh$, we compose constituent
%objects $\dlam$ of $\cbrcapnsh$ individually with $\cbrcupn$ and
%use the second Reidemeister move to simplify the compositions. We
%use the arc presentation\rx{aes2.3a} for the \TL\ \ttngnnmt\ $\xlam$ and slide its cup
%part $\gcupnmtIp$ above the arc of $\brcupn$:
%%
%\xlee{eq:a.2}
%\brcupn\tcmp\xlam = \gcupnIp\tcmp \brcupvv{n-2\zdlp}{-75}{90}{-80} \tcmp
%\gcapnI.
%\xeee
%%
%Thus for $\brcupn\tcmp\xlam$ we choose the categorification complex
%%
%
\xlee{eq:a.3}
%\bsymcat{\gcupnIp}
\bsymcat{ \gccupv{n}{\tstI}{-0.5} }
\tcmp
%\bsymcatsh{\brcupvv{n-2\zdlp}{-75}{90}{-80}}
\bsymcatsh{\brcupvv{\ztl+2}{-65}{80}{-70}}
\tcmp \bsymcat{\gcapnJ},
\xeee
where $\ztl$ is a \thrd\ of $\xlam$.
The first claim of the lemma follows from the fact
the middle complex of the composition\rx{eq:a.3} has the shape
%$\ryA\tgrsshv{-\shlf}{-\shlf}^{n-2\zdlp-2}$.
$\ryA\hqshv{-\frac{\ztl}{2}}$.
%\tgrsshv{-\shlf}{-\shlf}^{\ztl}$.
The second claim of
the lemma follows from the fact that the middle tangle in the
composition\rx{eq:a.3} is of type $(\ztl,\ztl+2)$: since $\ztl\leq
n-2$, only the tangle with \thrd\ up to $n-2$ can
form in the composition.
% the tangle $\gidbrn$ can not appear in the composition.

%$\zdl\geq 1$, so there is at least one cap in the
%left tangle in the \lhs of
%\ex{eq:a.2}.\myqed

\end{proof}

\begin{lemma}
\label{lm:wrap}
The tangle $\brtwpfn$ has a special categorification complex
$\cbrtwpfnsh$ with the following properties:
\begin{itemize}
\item
$\cbrtwpfnsh$ has an \otblc\ $\stA\tgrsshv{-n}{-n+1}$ if $n\geq 2$;
\item
the object $\cidbrn$ appears in $\cbrtwptnsh$ only in the position
$\cidbrn\tgrsshv{1}{-1}^{n-1}$
and its multiplicity there is $1$.
\end{itemize}
\end{lemma}
\begin{proof}
We prove the lemma by induction over $n$. When $n=1$ the claim is
obvious.

Consider now a general value of $n$. The tangle $\brtwpfn$ can be
presented as a composition of three tangles
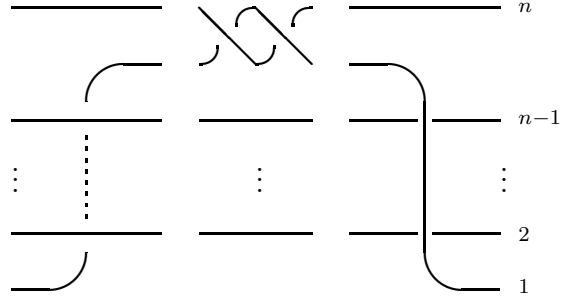
\xlee{eq:a.5}
\setlength{\unitlength}{1mm}
%\newlength{\xspc}
%\setlength{\xspc}{5mm}
\begin{picture}(50,50)(-50,-10)
\put(0,0){\oval(20,10)[bl]}
\put(-10,0){\line(0,1){20}}
\put(-20,20){\oval(20,10)[tr]}
\put(0,2.5){\line(-1,0){9}}
\put(-20,2.5){\line(1,0){9}}
\put(0,17.5){\line(-1,0){9}}
\put(-20,17.5){\line(1,0){9}}
\put(0,32.5){\line(-1,0){20}}
\put(-25,25){\line(-1,1){7.5}}
\put(-25,30){\oval(5,5)[tl]}
\put(-32.5,27.5){\oval(5,5)[br]}
\put(-32.5,25){\line(-1,1){7.5}}
\put(-32.5,30){\oval(5,5)[tl]}
\put(-40,27.5){\oval(5,5)[br]}
\put(-25,2.5){\line(-1,0){15}}
\put(-25,17.5){\line(-1,0){15}}
\put(-45,17.5){\line(-1,0){20}}
\put(-45,2.5){\line(-1,0){20}}
\put(-45,32.5){\line(-1,0){20}}
\put(-65,0){\oval(20,10)[br]}
\put(-45,20){\oval(20,10)[tl]}
\multiput(-55,4.5)(0,1.5){8}{\line(0,1){0.5}}
\put(-65,8){$\vdots$}
\put(0,8){$\vdots$}
\put(-32.5,8){$\vdots$}
\put(2.5,-5.5){$\scriptstyle{1}$}
\put(2.5,1.5){$\scriptstyle{2}$}
\put(2.5,17){$\scriptstyle{n-1}$}
\put(2.5,32){$\scriptstyle{n}$}
\end{picture}
\yeee
%
%\vspace{1.5cm}
We use the categorification complex for the middle tangle
which is based on the following categorification complex of the
double-crossing \ttngtt:
\ylee{eq:a.6}
\bsymcat{\dtwst} =
\xlrB{
\bsymcat{\dpar}\tgrsshv{1}{-1} \longrightarrow \bsymcat{\dccp}
\longrightarrow\bsymcat{\dccp}\tgrsshv{-2}{1}
}.
\yeee
The special complex $\cbrtwpfnsh$ is constructed by
tangle-composing the objects of this complex with the first and third tangles of
the decomposition\rx{eq:a.5} and using their special categorification complexes.
As a result, $\cbrtwpfnsh$ is a
`double-cone' composed out of the following three complexes
\ylee{eq:a.7}
\bsymcatsh{\brtwpnmoe}\tgrsshv{1}{-1},\quad \cbrccpnsh,\quad
\cbrccpnsh\tgrsshv{-2}{1},
\yeee
where the first tangle $\brtwpnmoe$ is the tangle $\brtwpfnmo$ together with an
extra straight strand on top. Now the claims of the lemma follow
from the induction assumption applied to the first complex and
Lemma\rw{lm:cpcp} applied to the second and third complexes.
\end{proof}

\begin{proof}[Proof of Theorem\rw{th:stro}]
We prove the theorem by induction over $n$. When $n=1$, its claim
is obvious.

In order to construct a special complex for the \cbr\ $\gobron$,
we present it as a composition of two braids:
\xlee{eq:a.8}
\gobron \isotp \brtwpon \tcmp \gobrolnmo\;\;,
\xeee
where the braid $\gobrolnmo$ is constructed by adding an extra
straight strand to the braid $\gobronmo$.
%Then we follow the same steps as in the proof of Lemma\rw{lm:cpcp}.
%Then we use the trick of subsection\rw{ss:trick} with a slight modification.
For a \ttngvv{k}{l}\ $\xtau$, let $\uxtau$ denote a
\ttngvv{k+1}{l+1}\ constructed by adding an extra straight line at
the bottom of $\xtau$.
We
construct the special complex $\bsymcatsh{\gobrolnmo}$ by replacing
every constituent object $\dlam$ of the
special complex $\bsymcatsh{\gobronmo}$, which exists by the
assumption of induction, with the object $\udlam$.
Then we use the trick of subsection\rw{ss:trick}
with $\xtau = \gobrolnmo\;$, $\xgami = \brtwpoi$ and $\xtaup =
\gobron$ with a slight modification: the \uptg s that slide through
$\xgami$ must be of the form $\gcuplnmoI$, but these are exactly
the \uptg s that appear in the decomposition of \xcnstt s $\uxlam$
of the complex $\bsymcat{\gobrolnmo}$. In other words, if
%$\xlam=\gcuplnmoI\tcmp\gcaplnmoJ$
%is a \xcnstt\ of the special complex $\bsymcatsh{\gobrolnmo}$,
$\xlam=\gcupnmoI\tcmp\gcapnmoJ$
is a \xcnstt\ of the special complex $\bsymcatsh{\gobronmo}$,
then the sliding isotopy
is
\ylee{eq:a.9}
\brtwpon\tcmp\gcuplnmoI\tcmp\gcaplnmoJ \isotp
\gcuplnmoI\tcmp\brtwpoto \tcmp\gcaplnmoJ
\yeee
%
%
%we choose the
%following categorification complex
%
%
%
%%
% Then
%we construct the special categorification complex of the
%composition $\brtwpon\tcmp\uxlam$ by simplifying this tangle
%through Reidemeister moves. We use the presentation\rx{aes2.3a}
%for $\xlam$ and slide all cups through $\brtwpon$:
%%
%\ylee{eq:a.9}
%\brtwpon\tcmp\gcuplnmoIp\tcmp\gcaplnmoI =
%\gcuplnmoIp\tcmp\brtwpoto \tcmp\gcaplnmoI
%\yeee
%%
and we choose the special categorification complex
\xlee{eq:a.10}
\bsymcatsh{ \brtwpon\tcmp\uxlam} =
\bsymcat{\gcuplnmoI}\tcmp\bsymcatsh{ \brtwpoto}\tcmp
\bsymcat{\gcaplnmoJ}.
\xeee
%
%According to Lemma\rw{lm:wrap}, this complex has the shape
%$\stA\tgrsshv{\shlf-\ztl}{\shlf-\ztl}$.
By the assumption of
induction, the object $\dlam$ of the complex
$\bsymcatsh{\gobronmo}$ lies within the shape
$\stAxtlnmo$.
%$\stA\tgrsshv{\shlf\ztl^2}{\shlf(\ztl^2+\ztl+n-1)}$.
According to Lemma\rw{lm:wrap} and Proposition\rw{pr:trick}, the complex\rx{eq:a.10} that
$\xlam$ yields, has the shape
$\stA\hqshv{\shlf-\ztl}$.
%\tgrsshv{\shlf-\ztl}{\shlf-\ztl}$.
Combining these two shapes
together, we find that the ultimate contribution of $\xlam$ to
$\cobronsh$ lies within the shape $\stAxv{\ztl-1}{n}$,
%$$\stA\tgrsshv{\shlf(\ztl-1)^2}{\shlf((\ztl-1)^2 + (\ztl-1) + n)},$$
which corresponds to \ttngnn\ of \thrd\ $\ztl-1$. The
complex\rx{eq:a.10} contains only the tangles of \thrd\ $t\leq
\ztl+1$ and parity opposite to that of $\ztl$.
Since $\stAxv{\zt}{n}\subset\stAxv{\zt\p}{n}$ for $\zt\geq\zt\p$,
the tangles with
\thrd\ $t\leq \ztl-1$ satisfy the claim of Theorem\rw{th:stro}.
The tangles with \thrd\ $t=\ztl+1$ originate from the object
$\cidbrto$ within the complex $\bsymcatsh{ \brtwpoto}$. According
to the second claim of Lemma\rw{lm:wrap}, it has a special degree
shift relative to the vertex of
$\stA\hqshv{\shlf-\ztl}$.
%\tgrsshv{\shlf-\ztl}{\shlf-\ztl}$.
Because of this shift, the
tangles with \thrd\ $t=\ztl+1$ fall within the shape
$\stAxv{\ztl+1}{n}$,
%$$\stA\tgrsshv{\shlf(\ztl+1)^2}{\shlf((\ztl+1)^2 + (\ztl+1) + n)},$$
hence they also satisfy the claim of Theorem\rw{th:stro}.
\end{proof}

%\subsection{A special categorification complex of a general \cbr}
\subsection{A categorification complex of a \cbr\ as an
approximation to the universal resolution of $\xaHn$}
%\subsection{Proof of Theorem\rw{th:spm}}
%\ and a projective resolution of $\xaHn$}

%Denote $\stAxtlnm = \stAs{\shlf m\ztl^2 + m\ztl - \shlf\ztl + \shlf n}{\shlf
%m\ztl^2}$. Note that $\stAxvv{\ztl}{n}{1}= \stAxtln$.
%
%\begin{theorem}
%\label{th:spm.x}
%There exists a sequence of special complexes $\cobrmtnsh$, $m=1,2,\ldots$, with the
%following two properties:
%\begin{enumerate}
%\item
%$\cjmilam > 0$ only if $(i,j)\in\stAxtlnm$;
%\item
%there is an isomorphism of truncated complexes
%%
%\ylee{eq:a.10a}
%\xtrntmov{\cobrmotnsh}\cong\xtrntmov{\cobrmtnsh}.
%\yeee
%%
%
%\end{enumerate}
%\end{theorem}
%\proof
%\linebreak

%Now we can prove Theorem\rw{th:spm} by induction.

\begin{proof}[Proof of Theorem\rw{th:spm}]
%\begin{proof}[]
%
We define the complexes $\cobrmtnsh$ recursively and prove the
theorem by induction over $m$. Theorem\rw{th:stro} establishes the first
property for $m=1$. Suppose that the theorem holds for $m$. We
construct the special complex $\cobrmotnsh$ by presenting the braid
$\gobrmotn$ as a composition
\ylee{eq:a.11}
\gobrmotn = \gobrotn\tcmp\gobrmtn.
\yeee
Next, we split the special complex $\cobrmtnsh$ into two pieces by
homological degree:
\xlee{eq:a.11a}
\cobrmtnsh = \Cone\xlrB{\xtrngtmv{\cobrmtnsh}[1]\longrightarrow\xtrntmov{\cobrmtnsh}
},
\xeee
and tangle-compose these pieces individually with $\cobrotn$. By the
assumption of induction, the complex $\xtrntmov{\cobrmtnsh}$ is \tct, so
according to \ex{eq:2.4a}, there is a homotopy equivalence
\ylee{eq:a.11a1}
\cobrotn\tcmp\xtrntmov{\cobrmtnsh}\;\; \hteqv\;\; \xtrntmov{\cobrmtnsh}.
\yeee
Thus it remains to construct a special categorification complex
\xlee{eq:a.11a2}
\Bspcc{\cobrotn\;\tcmp\;\xtrngtmv{\cobrmtnsh}}
\xeee
and substitute it for the first term in the cone\rx{eq:a.11a}.
We build this complex thought the trick of
subsection\rw{ss:trick}: we slide the cups of  a \xcnstt\ $\xlam=
\gcuptnI\tcmp \gcaptnJ$ of
%
%by using Reidemeister moves in order to simplify the
%complexes representing the composition of $\gobrotn$ with
%constituent \TL\ tangles$\xlam$ of
the truncated special complex $\xtrngtmv{\cobrmtnsh}$.
%We use the presentation\rx{aes2.3a} for $\xlam$ and slide the cups
through $\gobrotn$:
\ylee{eq:a.12}
\gobrotn\tcmp
\gcuptnI\tcmp \gcaptnJ \isotp
\gcuptnI\tcmp\gobrotl \tcmp\gcaptnJ,
\yeee
(note the cancelation of framing twists),
thus constructing a special complex for the composition
$\gobrotn\tcmp\xlam$:
\xlee{eq:a.13}
\bsymcatsh{\gobrotn\tcmp\xlam} = \bsymcat{\gcuptnI}\tcmp
\bsymcatsh{\gobrotl}\tcmp\bsymcat{\gcaptnJ}.
\xeee

The special complex\rx{eq:a.11a2}
consisting of complexes\rx{eq:a.13} satisfies the first property
of Theorem\rw{th:spm}. Indeed, a constituent object $\dlamp$ of
the complex\rx{eq:a.13} satisfies the property $\ztlp\leq\ztl$ and
by Theorem\rw{th:stro} it lies within the shape $\stAxtlptn$. At
the same time, by the assumption of induction, the object
$\dlam$ of $\xtrngtmv{\cobrmtnsh}$ lies within the shape
$\stAxtltnm$. Now it is easy to check that the place of the object
$\dlamp$ in the complex\rx{eq:a.11a2} will be within the shape
$\stAxvv{\ztlp}{2n}{m}$.

Finally, it is obvious that the cone
\ylee{eq:a.14}
\cobrmotnsh =
\Cone\xlrB{\Bspcc{\cobrotn\;\tcmp\;\xtrngtmv{\cobrmtnsh}}[1]
\longrightarrow\xtrntmov{\cobrmtnsh}}
\yeee
satisfies the second property of Theorem\rw{th:spm}.
%\myqed
%\qed
\end{proof}

%\subsection{Infinite \cbr\ as a projective resolution of $\xaHn$}

%Theorem\rw{th:spm} indicates that there exists a semi-infinite
%complex

% as well
%as the tangle $\gidbrto$.$\cidbrto$

\begin{proof}[Proof of Theorem\rw{th:complex}]
Isomorphisms\rx{eq:a.10a} imply the existence of a special complex
$\prbCsn\in\dTLcptn$ such that
\xlee{eq:spcC}
\xtrntmov{\cobrmtnsh} \cong \xtrntmov{\prbCsn}.
\xeee
The property\rx{eq:1.22}
implies that truncated complexes $\xtrntmov{\cobrmtnsh}$ are \otbl\ and
\tct, hence $\prbCsn$ is \otbl\ and \tct.

It remains to show that $\prbCsn$ satisfies the defining property
of the complex $\rsPn$: $\spfK(\prbCsn)\qie\xaHn$. The homotopy
equivalence\rx{eq:thteqv} implies the \qim\
$\Dobrotn \qie \xunt.$ Since the map
$\mpDcat\colon\yTngtntn\rightarrow\DbgHen$ converts the tangle
composition into the tensor product, this relation implies
$\Dobrmtn \qie \xunt$, which means that the homology of the
complex $\Kobrmtn$ is zero in all degrees except at degree zero
where it is isomorphic to $\xaHn$:
\ylee{eq:a3.1}
\Hmlv{i}\xlrB{\Kobrmtn} =
\begin{cases}
0 & \quad\text{if $i < 0$},
\\
\xunt & \quad\text{if $i = 0$ }.
\end{cases}
\yeee
In view of the isomorphisms\rx{eq:spcC}, $\spfn(\prbCsn)$ has the
same property:
\xlee{eq:a3.2}
\Hmlv{i}\xlrB{\spfK(\prbCsn) }
=
\begin{cases}
0 & \quad\text{if $i < 0$},
\\
\xunt & \quad\text{if $i = 0$ }.
\end{cases}
\xeee
Since $\prbCsn$ is \tct, its image $\spfK(\prbCsn)$ is projective, so
\ex{eq:a3.2} means that $\spfK(\prbCsn)$ is a projective resolution
of $\xunt$ and we set $\rsPnc=\prbCsn$.
\end{proof}

\begin{proof}[Proof of Theorem\rw{th:brappr}]
The isomorphism\rw{eq:cltlst} and the use of special complexes for
$\rsPn$ and $\cobrmtn$ allows us to rewrite \ex{eq:1.24}
as the isomorphism between homologies
\ylee{eq:isomhml}
\Hmli\xlrB{\lSh{\dtau\tcmp\rsPnc}} = \Hmli\xlrB{\lSh{\dtau\tcmp\cobrmtnsh}}.
\yeee
According to
\ex{eq:1.23}, the \qcmds\ in both complexes are canonically
isomorphic in homological degrees up to $\uhdv{\dtau}-2m+1$, hence their
homologies are isomorphic in degrees up to $\uhdv{\dtau}-2m+2$.
\end{proof}

\appendix

\section{\JWp s and \WRTp}
\label{s:jwpst}

\subsection{\JWp s}
The most famous \JWp\ $\jwpn$ is the idempotent element
of the \TLa\ $\QcTLn$ generated by \TL\ \ttngnn s over the field $\Qq$ of
rational functions of $q$ and
%The idempotent $\jwpn$ is
defined by
the property that for any \TL\ \ttngnn\ $\xlam$,
\xlee{eq:4.1}
\jwpn\tcmp\clam = \clam\tcmp\jwpn =
\begin{cases}
\jwpn, & \text{if $\xlam = \gidbrn $,}
\\
0, & \text{if $\xlam\neq \gidbrn $.}
\end{cases}
\xeee
%
%for any \TL\ \ttngnn\ $\xlam$.

To define the other \JWp s of $\QcTLn$, recall that as a
$\ZZ$-module, $\QcTLn$ is freely generated by \TL\ tangles, and
each \TL\ tangle has a presentation\rx{aes2.3a}. Since this
presentation plays a central role in our calculations, we will use
a special notation
\xlee{eq:4.1a}
\xlamIJm = \gcupnI\tcmp\gcapnJ,\qquad\stI,\stJ\in\cstInm,
\quad 0\leq m\leq n,\quad n-m\in 2\ZZ.
\xeee
Now we define a
`\xprot'
$\prclam$ by inserting the \JWp\ $\jwpv{\ztl}$ between the cup and
cap parts of $\xlamIJm$:
\xlee{eq:4.2}
\prclamIJm = \bsymalg{\gcupnI}\tcmp
\jwpv{m}
\tcmp \bsymalg{\gcapnJ}.
\xeee
%
%(in fact,  there is no tangle $\xlamt$, and $\dlamt$ denotes just
%an element of the \TLa\ $\QcTLn$).
Since $\jwpn - \bsymalg{\gidbrn}$ can be presented as a linear combination
of \TL\ tangles with \thrd\ less than $n$, the relation between
the $\QcTLn$ generators $\clam$ and the \xprot s $\prclam$ is
upper-triangular: $\prclam-\clam$ is a linear combination of \TL\
tangles with \thrd\ less than $\ztl$. Hence \xprot s $\prclam$ also
form a set of free generators of $\QcTLn$. For $m$ such that
$m<n$ and $n-m$ is even, let
$\QcTLprnm\subset\QcTLn$ be a submodule generated by all $\prclam$
such that $\ztl = m$. Since $\prclam$ generate $\QcTLn$,
the latter is a direct sum
\ylee{eq:4.3}
\QcTLn = \bigoplus_{m} \QcTLprnm.
\yeee
We will see shortly that $\QcTLprnm\subset\QcTLn$ are, in fact, two-sided
ideals.

For
$\stI,\stJ\in\cstInm$ the composition $\gcapnJ\tcmp\gcupnI$ is a
\ttngmm. Let $k$ denote the number of disjoint circles in it, then
\xlee{eq:4.3a}
\bsymalg{\gcapnJ\tcmp\gcupnI} = (-q-\qi)^k\,\clam,
\xeee
where $\xlam$ is the \TL\ \ttngmm, which corresponds to the
`connected' part of the composition.
Define the $\Zqqi$-valued coefficients
$(\mtBIJ)_{\stI,\stJ\in\cstInm}$
of a symmetric matrix $\mtB$
 by the formula
\ylee{eq:4.4}
\mtBIJ = \begin{cases}
(-q-\qi)^k, & \text{if $\xlam = \gidbrv{m}{1}$},
\\
0 & \text{otherwise.}
\end{cases}
%\\
%\bsymalg{\gcapnJ\tcmp\gcupnI} = \mtBIJ\,\bsymcat{\gidbrm} + \cdots,
\yeee
where $\cdots$ stands for a linear combination of \TL\ \ttngmm s
whose \thrd\ is less than $m$.
\begin{proposition}
Projected tangles satisfy a simple composition formula:
%%
%\ylee{eq:4.4a}
%\bsymcat{ \gcupnI\tcmp\jwpm\tcmp\gcapnJ} \tcmp
%\bsymcat{\gcupnIp \tcmp\jwpmp \tcmp\gcapnJp}
%=
%\mtBvv{\stIp}{\stJ}\,\bsymcat{ \gcupnI\tcmp\jwpm\tcmp\gcapnJp}
%\yeee
%%
%
\xlee{eq:4.4a}
%\bprsymalg{ \underbrace{ \gcupnI\tcmp\gcapnJ}_{\xlam}}\tcmp
%\bprsymalg{ \underbrace{\gcupnIp\tcmp\gcapnJp}_{\xlamp}} = \delta_{\ztl\ztlp}
%\mtBvv{\stIp}{\stJ}\,\bprsymalg{\gcupnI\tcmp\gcapnJp},
\prclamIJm\tcmp\prclamIJmp = \delta_{m m\p} \prclamIJpm,
\xeee
where $\delta_{m m \p}$ is the Kronecker symbol.
\end{proposition}
\proof
We use the formula \ex{eq:4.3a} for the composition of cup and cap
tangles:
\ylee{eq:4.4a1}
\bsymalg{\gcapnJ\tcmp\gcupnIp} = (-q-\qi)^{k}\,
\symalg{\xlampp}
%\bsymalg{ \underbrace{\gcupvIpp{\ztl}{-0.5}\tcmp
%\gcapvJpp{\ztlp}{0.75}}_{\xlampp}},
\yeee
%
%where $k$ is the number of disjoint circles in the composition
%$\gcapnJ\tcmp\gcupnIp$ and the \ttngvv{\ztlp}{\ztl}\ $\xlampp$ is
%the `connected' part of that composition.
Then the
composition in the \rhs of \rx{eq:4.4a} takes the form
\ylee{eq:4.4a2}
\prclamIJm\tcmp\prclamIJmp = (-q-\qi)^{k}\,
\bsymalg{\gcupnI}\tcmp\jwpv{m} \tcmp \symalg{\xlampp}
\tcmp \jwpv{m\p} \tcmp \bsymalg{\gcapnJp}.
\yeee
The property\rx{eq:4.1} of \JWp s indicates that this expression
is equal to zero, unless $m=m\p$
and $\xlampp = \gidbrv{m}{1}$, in which case
$(-q-\qi)^k=\mtBvv{\stIp}{\stJ}$ and we
obtain the \rhs of \ex{eq:4.4a}.\myqed

\begin{corollary}
\label{cr:tsid}
Submodules $\QcTLprnm\subset\QcTLn$ are two-sided
ideals.
\end{corollary}

Define the special elements of $\QcTLprnm$
\xlee{eq:4.5}
\jwpxnm =
\sum_{\stI,\stJ\in\cstInm}\mtBiIJ\, \prclamIJm
=
\sum_{\stI,\stJ\in\cstInm}\mtBiIJ\;\;
\bsymalg{\gcupnI}\tcmp\jwpm\tcmp\bsymalg{\gcapnJ},
\xeee
where $\mtBi$ is the inverse matrix of $\mtB$.

\begin{proposition}
An element $\jwpxnm$ satisfies the following property:
\ylee{eq:4.6}
\jwpxnm\tcmp\prclam = \prclam\tcmp\jwpxnm =
\begin{cases}
\prclam, &\text{if $\ztl = m$,}
%$\xlamt \in\QcTLprnm$,}
\\
0, & \text{if $\ztl\neq m$ }
\end{cases}
\yeee
\end{proposition}
\proof Present $\prclam$ in the form\rx{eq:4.2} and use the
formula\rx{eq:4.4a}.\myqed

\begin{corollary}
The element $\jwpxnm$ is an idempotent
projecting $\QcTLn$ onto $\QcTLprnm$
%, they project $\QcTLn$ onto
%$\QcTLprnm$:
%%
%\ylee{eq:4.7}
%\jwpxnm\tcmp\QcTLn = \QcTLprnm\tcmp\jwpxnm =
%\QcTLprnm
%\yeee
%%
and
\xlee{eq:4.8}
\smmnev \jwpxnm = \cidbrn.
\xeee
\end{corollary}
Hence $\jwpxnm$ are \JWp s mentioned in subsection\rw{sss:jwp}.

Consider a left $\cTLn$-module $\cTLprmn=\cTLmn\tcmp\jwpm$.
An element $\xx\in\QcTLn$ determines a multiplication endomorphism
$\cTLprmn\xrightarrow{\xmltx}\cTLprmn$, $\xmltx(\ay) =
\xx\tcmp\ay$.

\begin{proposition}
The $\Sh$ closure of the composition $\xx\tcmp\jwpxnm$ is
proportional to the trace of $\xmltx$ over $\cTLprmn$:
%The trace of $\xmltx$ is equal to the $\Sh$ closure of the
%composition $\xx\tcmp\jwpxnm$:
%
\xlee{eq:4.8a}
%(-1)^m\,\frac{q^{m+1}-q^{-m-1}}{q-\qi}\,
\xlrb{\lSh{\jwpxnm\tcmp\xx}} = \cpmsh\,\Tr_{\cTLprmn}\xmltx.
\xeee
\end{proposition}
\proof
If $\xx\in\QcTLprnmp$ and $m\p\neq m$, then both sides of this
equation are equal to zero, hence we may assume that
$\xx\in\QcTLprnm$. Moreover, since $\QcTLprnm$ is generated by the
elements\rx{eq:4.2}, we may assume that $\xx=\prclamIJm$. The the
\lhs of \ex{eq:4.8a} can be calculated:
\begin{multline}
\label{eq:4.8b}
%\ylee{eq:4.8b}
\xlrb{\lSh{\jwpxnm\tcmp\xlamIJm}} =
\xlrb{\lSh{\xlamIJm}}
\\
= \xlrB{
\bsymalg{\gcupnI}\tcmp \jwpv{m} \tcmp \bsymalg{\gcapnJ} }
=
\xlrB{ \jwpv{m} \tcmp \bsymalg{\gcapnJ} \tcmp \bsymalg{\gcupnI} }
= \mtBIJ \xlrb{\lSh{\jwpm}}.
\end{multline}
The module $\cTLprmn$ is generated freely by the elements
\xlee{eq:4.8c}
\bsymalg{\gcupnIp}\tcmp\jwpm,\qquad \stIp\in\cstInm.
\xeee
The matrix elements
$(\xmltx)_{\stIp\stJp}$ of
the endomorphism $\xmltx$ with respect to this basis are easy to
evaluate:
\begin{multline}
\nonumber
%\ylee{eq:4.8b1}
\xmltx \xlrb{\bsymalg{\gcupnIp}\tcmp\jwpm}
= \bsymalg{\xlamIJm}\tcmp\bsymalg{\gcupnIp}\tcmp\jwpm
\\
= \bsymalg{\gcupnI}\tcmp
\jwpm
\tcmp \bsymalg{\gcapnJ}
\tcmp\bsymalg{\gcupnIp}\tcmp\jwpm
= \mtBvv{\stIp}{\stJ}\,\bsymalg{\gcupnI}\tcmp\jwpm,
\end{multline}
hence $(\xmltx)_{\stIp\stJp} =
\delta_{\stI\stJp}\,\mtBvv{\stIp}{\stJ}$ and
\xlee{eq:4.8b2}
\Tr_{\cTLprmn}\xmltx = \sum_{\stIp\in\cstInm}
(\xmltx)_{\stIp\stIp} = \mtBIJ.
\yeee
This equation together with \ex{eq:4.8b} prove the
proposition.\myqed
%
%
%In view of Corollary\rw{cr:tsid} we may assume that
%$\xx\in\QcTLprnm$. Moreover, since $\QcTLprnm$ is generated freely
%by the elements\rx{eq:4.1a}, we may assume that $\xx=\xlamIJm$.
%Now $\jwpxnm\tcmp\prclamIJm=\prclamIJm$ and
%%
%\ylee{eq:4.8a1}
%\xlrB{\lSh{\prclamIJm}} = \xlrB{
%\bsymalg{\gcupnI}\tcmp \jwpv{m} \tcmp \bsymalg{\gcapnJ} }
%=
%\xlrB{ \jwpv{m} \tcmp \bsymalg{\gcapnJ} \tcmp \bsymalg{\gcupnI} }
%= \mtBIJ \xlrb{\lSh{\jwpm}}.
%\yeee
%%

%
%The \JWp\ $\jwpxnm$ is defined as a
%projector onto $\QcTLprnm$, that is, $\jwpxnm\tcmp\QcTLprnm =
%\QcTLprnm$.

\begin{corollary}
\label{cr:int}
If $\xx\in\cTLn$, then $\xlrb{\lSh{\jwpxnm\tcmp\xx}} \in\Zqqi$.
\end{corollary}
In other words, although the expression for the \JWp\ $\jwpxnm$
involves rational functions of $q$, the closure
$\xlrb{\lSh{\jwpxnm\tcmp\xx}}$ is purely polynomial.

\begin{proof}[Proof of Corollary\rw{cr:int}]
%\pr{Corollary}{cr:int}
We are going to use \ex{eq:4.8a}. First of all, note that
\ylee{eq:4.8b3}
\cpmsh = (-1)^m\,\frac{q^{m+1}-q^{-m-1}}{q-\qi}\in\Zqqi,
\yeee
hence we just have to prove $\Tr_{\cTLprmn}\xmltx\in\Zqqi$. We
will show that the matrix elements $(\xmltx)_{\stIp,\stJp}$ with
respect to the basis\rx{eq:4.8c} belong to $\Zqqi$.
It is sufficient to verify this claim for $\xx=\xlamIJm$, and we
leave the details to the reader.
%\myqed
\end{proof}

\begin{proof}[Proof of Theorem\rw{pr:Lpol}]
%\proof{Proposition}{pr:Lpol}
This theorem is a particular case ($m=0$) of
Corollary\rw{cr:int}.
\end{proof}

\subsection{The \WRT\ invariant of links in $\Sot$}
\label{ss:relstwrt}
Let us introduce another \ttngnn\ notation:
\ylee{eq:4.9}
\tmrn =
%\brtwptn =
\setlength{\unitlength}{1mm}
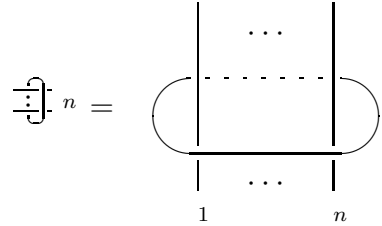
\begin{picture}(35,15)(-9,10)
\thl
\put(0,10){\oval(10,10)[bl]}
\put(0,5){\line(1,0){20}}
\put(20,10){\oval(10,10)[r]}
\multiput(0,15)(2.5,0){8}{\line(-1,0){0.5}}
\put(0,10){\oval(10,10)[tl]}
\put(1,0){\line(0,1){4}}
\put(19,0){\line(0,1){4}}
\put(1,6){\line(0,1){19}}
\put(19,6){\line(0,1){19}}
%\put(-5,20){\line(0,1){5}}
\put(7.5,20){$\cdots$}
\put(7.5,0){$\cdots$}
%\put(-6,-4){$\scriptstyle{1}$}
\put(1,-4){$\scriptstyle{1}$}
\put(19,-4){$\scriptstyle{n}$}
%\put(-6,20){$\circ$}
%\put(-9,20){$\scriptstyle{2}$}
\end{picture}
\\
\nonumber
\\
\nonumber
\yeee
We refer to the circle in this picture as \emph{\xmrd}. Let
$\tmrpnk$ denote the element of $\QcTLn$ constructed by replacing
the \xmrd\ of $\tmrn$ with a $k$-cable on which the \JWp\ $\jwpk$
is placed.

%Let us use the following notation: $\xDln = (-1)^\frac{a}{a}$

The dependence of the \TLa\ element $\yvspoh\tmrpnk$ on the value of $k$
can be made more explicit, if we multiply it
%If we multiply the element $\tmrpnk$
by the \lhs of \ex{eq:4.8},
slide the cups of \ex{eq:4.5} through the \xmrd\ and use the well-known formula
\ylee{eq:4.9a}
\jwpm\tcmp\tmrpmk =
(-1)^{k(m+1)}\;\frac{\qpmv{(k+1)(m+1)}}{q^{m+1}-q^{-(m+1)}}\;\jwpm.
\yeee
Then we obtain the following formula for $\tmrpnk$:
\begin{equation}
\label{eq:4.9a1}
\begin{split}
\tmrpnk & =
\xlrB{ \smmnev \jwpxnm }\tcmp\tmrpnk =
\sum_m \sum_{\stI,\stJ\in\cstInm}\mtBiIJ
\bsymalg{\gcupnI}\tcmp\jwpm\tcmp\tmrpmk\tcmp\bsymalg{\gcapnJ}
\\
&=\smmnev (-1)^{k(m+1)}
\;\frac{\qpmv{(k+1)(m+1)}}{q^{m+1}-q^{-(m+1)}}\;
\jwpxnm.
\end{split}
\end{equation}
%\yeee
%

\begin{proof}[Proof of Theorem\rw{th:relfr}]
The closure $\cltSot$ of a \ttngnn\ $\xtau$ in $\Sot$ can be
constructed by a surgery on a \xmrd\ of the link
$\xlrB{\lSh{\xtau\tcmp\tmrn}}$. Hence, the \WRT\ invariant of
$\cltSot$ is expressed by the Reshetikhin-Turaev surgery
formula:
\ylee{eq:4.10}
\ZrtSot = -\frac{ q - \qi}{2\xr}
%\sum_{0\leq k \leq \xr-2}
\sum_{k=0}^{\xr-2}
(-1)^k\, (q^{k+1}-q^{-k-1})\,
\xlrB{
\lSh{
\ctau\tcmp \tmrpnk
}
}\Bigr|_{\qepr}.
\yeee
If we replace $\tmrpnk$ by the formula\rx{eq:4.9a1} and use the formula
\begin{multline}
%\ylee{eq:4.11}
\nonumber
%\sum_{0\leq k \leq\xr-2}\xlrb{q^{l(k+1)} - q^{-l(k+1)}} =
\sum_{k=0}^{\xr-2} (-1)^{km}\, \xlrb{ q^{(k+1)(m+1)} -
q^{-(k+1)(m+1)} }
\xlrb{ q^{k+1} - q^{-k-1}}\bigr|_{\qepr}
\\
=
\begin{cases}
-2\xr, & \text{if $m=2\xr\yk$, $\yk\in\ZZ$},
\\
2\xr, & \text{if $m=2\xr\yk-2$, $\yk\in\ZZ$},
\\
0 & \text{otherwise},
\end{cases}
\end{multline}
%\yeee
%
then we find that $\ZrtSot=0$ when $n$ is odd and obtain
the formula
\ylee{eq:1.6}
\ZrtSot = \Biggl(\sum_{0\leq \yk \leq \frac{n}{\xr}}
\xlrb{\xlShv{\jwpxtnv{2\yk r}\tcmp\ctau} }
\;
+ \sum_{1\leq \yk\leq\frac{n+1}{\xr}}
\xlrb{\xlShv{\jwpxtnv{2\yk r-1}\tcmp\ctau}}\Biggr)\Biggr|_{\qepr}
\yeee
when $n$ is even.
If $\xr\geq n+2$
%is big enough,
then only one term survives in the sums of
the \rhs of this equation,
and we get \ex{eq:relstwrt1}.
%hence
%%
%\xlee{eq:1.7}
%\ZrtSot = \xlrb{\xlShv{\jwpxtnz\tcmp\ctau}}|_{\qepr},\qquad
%\text{if $\xr \geq n+2$.}
%\xeee
%%
%
%
% when $n$ is even.
\end{proof}

\subsection{Torus braids and the \JWp\ $\jwpxtnz$}
\label{ss:tbrjw}

Define the $q^+$ order of an element
\xlee{eq:5.1}
\xx = \sltl\sum_{j\in\ZZ}
\xcaljt\,q^j\,\clam\in\cTLpinf
\xeee
as
$\yordq{\xx}=\xinfv{j\colon\exists \xlam\colon \xcaljt\neq
0}$.
\begin{definition}
A sequence of elements $\xx_1,\xx_2,\ldots\in\cTLpinf$ has a limit
$\lmii\xx_i = \xx$ if $\lmii\yordq{\xx-\xx_i} = +\infty$.
\end{definition}
In other words, the limit $\lmii\xx_i=\xx$ means that the
coefficients at lower powers of $q$ in the
presentations\rx{eq:5.1} of the elements $\xx_i$ stabilize progressively
as $i$ grows, and $\xx$ is the Laurent series in $q$ formed by the
stable coefficients.

The following is obvious:
\begin{proposition}
If the limit $\lmii\xx_i$ exists, then it is unique.
\end{proposition}

\begin{theorem}
The \TLa\ elements corresponding to \cbr s with high twist converge to the \JWp\
$\jwpxtnz$:
\xlee{eq:5.1a}
\lmmi\abrmtn = \jwpxtnz.
\xeee
\end{theorem}
\proof
We multiply $\abrmtn$ by the identity element\rx{eq:4.8} and use
the well-known formula
\ylee{eq:5.2}
\abron\tcmp\jwpn = (-1)^n q^{\frac{1}{2}n(n+2)} \jwpn
\yeee
in order to express $\abrmn$ as a sum of projectors with growing powers of $q$:
\ylee{eq:5.3}
\abrmtn = \abrmtn \tcmp \sum_{m=0}^{2n} \jwpxtntm =
\sum_{m=0}^{2n} q^{2mn(n+1)}\jwpxtntm.
\yeee
Obviously, all terms in the sum in the \rhs tend to zero except
the term with $m=0$, which carries the zeroth power of $q$.\myqed

%%%%%%%%%%%%%%%%%%%%%%%%%%%%%%%%%%%%%%%%

\end{document}

%********************************
\nidm

\setlength{\unitlength}{1mm}
\begin{picture}(60,40)
\put(20,30)
{
$
\xygraph{
!{0;/r0.5pc/:}
!{\xbendu-@(0)}
[l(0.5)]
!{\xcaph@(0)}
%[u(2.5)l(0.5)]
%!{\xbendl-}
}
$
}
\end{picture}
$$ \Bigl\langle
\setlength{\unitlength}{0.2mm}
\begin{picture}(50,50)(-20,-20)
\put(0,0){\oval(10,10)[tl]}
\put(0,5){\line(1,0){20}}
\put(20,10){\oval(10,10)[r]}
\multiput(0,15)(2.5,0){8}{\line(-1,0){0.5}}
\put(0,20){\oval(10,10)[bl]}
\put(1,0){\line(0,1){4}}
\put(19,0){\line(0,1){4}}
\put(1,6){\line(0,1){19}}
\put(19,6){\line(0,1){19}}
\put(-5,20){\line(0,1){5}}
\put(7.5,20){$\scriptstyle{\cdots}$}
\end{picture}_{n,i}
\Bigl\rangle
$$

$$x + y =
\setlength{\unitlength}{1mm}
\begin{picture}(50,50)(-20,-20)
\put(0,0){\oval(10,10)[tl]}
\put(0,5){\line(1,0){20}}
\put(20,10){\oval(10,10)[r]}
\multiput(0,15)(2.5,0){8}{\line(-1,0){0.5}}
\put(0,20){\oval(10,10)[bl]}
\put(1,0){\line(0,1){4}}
\put(19,0){\line(0,1){4}}
\put(1,6){\line(0,1){19}}
\put(19,6){\line(0,1){19}}
\put(-5,20){\line(0,1){5}}
\put(7.5,20){$\cdots$}
\end{picture}
$$

$$
\setlength{\unitlength}{1mm}
\begin{picture}(50,50)(-20,-20)
\put(0,0){\oval(10,10)[bl]}
\put(-5,0){\line(0,1){20}}
\put(-10,20){\oval(10,10)[t]}
\put(-20,0){\oval(10,10)[br]}
\put(-3,2){\line(1,0){3}}
\put(-3,18){\line(1,0){3}}
\put(-7,2){\line(-1,0){18}}
\put(-7,18){\line(-1,0){18}}
\put(-20,-5){\line(-1,0){5}}
\end{picture}
$$

$$
\left(\brtwptn\right)\qquad\gbrmn
$$

%\end{document}

A \JWp\ $\jwpn$ is a special idempotent element of the $n$-strand
\TLa\ $\cTLn$ which annihilates cap and cup tangles.
The coefficients in its expression in terms of \TLb\ tangles are
rational (rather than polynomial) functions of $q$. This suggests
that the categorification $\ctjwn$ of $\jwpn$ in the
universal tangle category $\dTLn$ constructed by D.~Bar-Natan\cx{BN1}
should be presented by a semi-infinite \chcpl. In fact, there are
two mutually dual categorifications: the complex $\ctjwpn$ which is bound from
below and the complex $\ctjwmn$ which is bound from above. We will
talk mostly about $\ctjwpn$, since the story of $\ctjwmn$ is totally
similar.

%
%Its
%expression in terms of \TLb\ tangles involves rational functions of
%$q$, which suggests that the categorification of $\jwpn$ in the
%universal tangle category $\dTLn$, constructed by D.~Bar-Natan\cx{BN1},
%should be presented by a semi-infinite \chcpl.

The construction of $\ctjwpn$ by
B.~Cooper and S.~Krushkal\cx{CK} is based upon the
Frenkel-Khovanov formula for $\jwpn$ and requires inventing morphisms
between constituent \TL\ tangles as well as non-trivial `thickening'
of the complex.

Our approach is rather straightforward: the
categorified projector $\ctjwpn$ is obtained as a limit of
appropriately shifted
categorification complexes of \cbr s
(\ie braid analogs of torus links) with high \clckw\ twist (the
other projector $\ctjwmn$ comes from high \cclckw\ twists).
The `limit' means that  $\ctjwn$
can be presented as a cone:
\xlee{eq:int1}
\ctjwpn\hteqv
%\CnBv{\wbCmnp\rightarrow\cbrmns},
\CnBv{\Ohp\big(2m(n-1)\big)\longrightarrow\cbrmns},
\xeee
where
$\gbrmn$ is a \cbr\ with $m$ full \clckw\ rotations of $n$ strands,
%$\symcat{-}$ is the categorification complex,
$\symcats{-}$ is the
categorification complex with a special grading shift, and
$\Ohp(k)$ denotes a \chcpl\ which starts only at the homological degree
$k$. Theorem\rw{th:cnpr} imposes even stronger restrictions on
%$\wbCmnp$.
the complex $\Ohp\big(2m(n-1)\big)$ in \ex{eq:int1}.

%$\cbrmns$
%is an appropriately shifted categorification
%complex of a \cbr\ which induces $m$ full rotations of its
%$n$ strands,
%%has $m$ full twists,
%while
%%$\wbCmnp$
%$\Ohp\big(2m(n-1)\big)$
%is a
%\chcpl\ which starts only at the homological degree $2m(n-1)$.
%Theorem\rw{th:cnpr} imposes even stronger restrictions on
%%$\wbCmnp$.
%this complex.

The advantage of our approach is that one can use \cbr s with high
twist as approximations to $\ctjwpn$ in a computation of \Kh\ of a
spin network which involves \JWp s:
% (\ie a spin network):
if a spin network $\xnu$ is constructed by connecting $\jwpn$ to an \ttngnn\
$\xtau$ such that $\symcat{\xtau}\hteqv\Ohp(k)$, while a spin network $\xnum$ is constructed
by replacing $\jwpn$ in $\xnu$ with $\gbrmn$, then homology of
$\symcat{\xnu}$ coincides with the shifted homology of
$\symcat{\xnum}$ up to degree $k + 2m(n-1)-1$. Thus one may say that
there is a stable limit
\xlee{eq:stlimsn}
\symcat{\xnu} =
\lim_{m\rightarrow+\infty}\symcats{\xnum}.
\xeee
We will make define homological limits more precisely in
subsection\rw{sss.homcal}.

The practical
importance of the relation between $\symcat{\xnu}$
and $\symcat{\xnum}$ stems from the fact that $\xnum$ is an
ordinary link and its homology
can be computed by
existing efficient computer programs even for high values
of $m$.
%
%
%the formula\rx{eq:int1}
%guarantees that replacing $\ctjwn$ with $\cbrmns$ does not change
%the homology of such a `link' up to a certain homological degree which
%goes to infinity as $m\rightarrow +\infty$.
%Such an approximation is useful, since an insertion of a \cbr\ $\gbrmn$
%creates an ordinary link whose homology can be computed by
%existing efficient computer programs which can handle high values
%of $m$.

The simplest example of the spin network $\xnu$ is the unknot
`colored' by the $n+1$-dimensional representation of
$\mathrm{SU}(2)$ with the help of the projector $\jwpn$. Its
categorification complex is homologically approximated by homology
of torus links $\mathrm{T}_{n,mn}$ which appear as cyclic closures of
$\gbrmn$. This homology has been studied by Marko Stosic\cx{St}, who
observed that it stabilizes at lower degrees as $m$ grows. This is
a particular case of the `stable limit'\rx{eq:stlimsn}.
% formula for a spin network:
%%
%\ylee{eq:stlimsn}
%\symcat{\xnu} =
%\lim_{m\rightarrow+\infty}\symcats{\xnum}.
%\yeee
%%

In Section\rw{s:notres}
we explain all notations and conventions
which are used in the paper. In particular, in
subsection\rw{sss:trgr} we define a non-traditional grading of
\Kh, which is convenient for our computations.
Then we formulate our results.

%In subsection\rw{ss:not} we explain all notations and conventions
%which are used in the paper. In particular, in
%subsection\rw{sss:trgr} we define a non-traditional grading of
%\Kh, which is convenient for our computations.
%In subsection\rw{ss:res} we formulate our results.

In Section\rw{s:elhomcal} we review basic facts about homological
`calculus' required to work with limits of categorification
complexes of high twist \cbr s. In Section\rw{s:cbr} we construct
a sequence of categorification complexes of
\cbr s related by special \chmp s. This sequence yields $\ctjwpn$
as its homological limit. In Section\rw{s:prfs} we use homological
calculus of Section\rw{s:elhomcal} in order to prove that
$\ctjwpn$ is a categorification of the \JWp.

\section{Notations and results}
\label{s:notres}
\subsection{Notations}
\label{ss:not}
\subsubsection{Tangles and \TLa}
All tangles in this paper are framed and we assume the blackboard
framing in pictures. We use the symbol $\symfr\; k$ to indicate an
addition of $k$ framing twists to a tangle strand:
\xlee{ae1.1b}
\xygraph{
!{0;/r1.5pc/:}
[u(0.5)]
!{\hover}
!{\hcap}
[u(0.5)l(0.25)]
%*{\bullet}
}
\;\; = \;\;
%-q^{\frac{3}{2}}\;\;
\xygraph{
!{0;/r1.5pc/:}
[u(0.5)]
!{\xcapv@(0)}
[u(0.45)r(0.23)]
*{\symfr\;\scriptstyle{1}}
[u(1.5)]
%*{\bullet}
}
\xeee

A tangle is called \emph{flat} if it can be presented by a diagram
without crossings. A flat tangle is called \emph{connected} or
\emph{\TLb} (\TLba) if
it does not contain disjoint circles. Let $\rTNG$ denote the set of all
framed tangles,
$\rTNGmn$ -- the set of \ttngmn s and $\rTNGn$ -- the set of
\ttngnn s.
We adopt similar notations for the set $\rFT$ of flat tangles and for the
set $\rTL$ of \TLba-tangles.

A \TLa\ $\cTLn$ over the ring of Laurent polynomials $\Zqqi$ is
generated by elements $\clam$ corresponding to \TL\ \ttngnn s $\xlam$, multiplication
coming from the composition of tangles, modulo the relation
\xlee{ae1.1}
%\widehat{\lcir}
\Bsymalg{\lcir}
 = -(\qpqi).
\xeee
which is needed to remove the disjoint circles that may appear in the composition
of \TLb\ tangles. Equivalently, $\cTLn$ is generated by elements
$\ctau$ corresponding to
\ttngnn s $\xtau$ modulo the relation\rx{ae1.1} and the Kauffman bracket
relation
\xlee{ae1.2}
\Bsymalg{\xcrsp}
\;\;=\;\;
\qvh\;\;
\Bsymalg{\xpver}
\;\;+\;\;
\qvmh\;\;
\Bsymalg{\xphor}.
\xeee
In other words, there is a map
$\rTNGn\xrightarrow
%{\salg}
{\amap}
\cTLn$
which maps a \ttngnn\ $\xtau$ into a \TLa\ element
\xlee{ae1.2a0}
\ctau = \sltln \xcalt\, \clam,\qquad
%\xcalt\in\Zqqi.\quad
\xcalt = \sum_{i\in\ZZ}\xcalit\,q^i
\xeee
with only finitely many coefficients $\xcalit$ being non-zero.

If two tangles differ only by the framing of their strands, then
the corresponding algebra elements differ by the $q$
power factor coming from the following relation associated with
the first Reidemeister move:
\xlee{ae1.2a}
\Bsymalg{\xvfro\hspace*{-0.2cm}}
\;\; = \;\;
-q^{\frac{3}{2}}\;\;
\Bsymalg{\;\xvert\hspace*{-0.5cm}}
\xeee

A \ttngzz\ $\xL$ is a framed link, so $\symalg{\xL}$
%the application of the map $\amap$ to it
is
the framing dependent Jones polynomial defined by the
Kauffman bracket.

Flat \ttngmn s modulo\rx{ae1.1} or, equivalently, \ttngmn s
modulo\rx{ae1.1} and\rx{ae1.2} generate a bimodule $\cTLmn$ over
$\cTLm^{\op}\otimes\cTLn$ and a composition of tangles produces a
homomorphism
$\cTLkm\otimes_{\cTLm}\cTLmn\rightarrow\cTLkn$. The total \TLa\ is the
union of all algebras and bimodules $\cTL = \bigcup_{m,n}\cTLmn$,
the product being the composition of tangles or zero if the
valences do not match.

%and $\cTLn$.

We use the notations $\QcTL$ and $\cTLpinf$ for \TLa s defined over
the ring $\Qq$ of rational functions of $q$ and over the ring
$\Zsqqi$ of Laurent power series.
A sequence of injective homomorphisms
$\Zqqi\hookrightarrow\Qq\hookrightarrow\Zsqqi$, the latter one
generated by the expansion in powers of $q$,
produce a sequence of injective homomorphisms of the corresponding
\TLa s and tangle bimodules.
%We will also use the ring $\Zsqiq$
%of Laurent power series in $q^{-1}$ and the corresponding \TLa\
%$\cTLminf$.

%Expansion in powers of $q$
%produces an injective homomorphism $\Qq\rightarrow\Zsqqi$ and the
%latter generates an injective homomorphism of the corresponding
%\TLa s.
%
%Let $\Zqqi$ be the ring of Laurent polynomials in $q$, $\Zsqqi$ --
%the ring of formal Laurent series and $\Qq$ -- the ring of
%rational functions of $q$. We will define three `flavors' of \TLa\ over
%these rings

%Let $\Opqm$ denote any element of $\cTLpinf$ of the form
%$\sltln\sum_{i\geq m} \xcali\,q^i\,\clam$.
%We define a \emph{\qord} of an element $\yal\in\cTLpinf$ as
%$\yordq{\yal} = \xsupv{m\colon \yal = \Opqm}$.
%%
%
%\begin{definition}
%A sequence of elements
%%
%\wlee{ae1.2a1}
%\yal_1,\yal_2,\ldots\in\cTLpinf%
%\weee
%%
%has a limit
%$\lim_{k\rightarrow \infty} \yal_k = \ybet$, $\ybet\in\cTLpinf$
%if $\lmii\yordq{\ybet-\yal_k} = +\infty$.
%\end{definition}

\subsubsection{The \JWp}

Let $\gcupni\in\rTLvv{n-2}{n}$ and $\gcapni\in\rTLvv{n}{n-2}$,
$1\leq i\leq n-1$, denote the following \TL\ tangles:
\ylee{ae1.3}
%\underbrace{
\gcupni=\xygraph{
!{0;/r1.5pc/:}
[r(0.25)u(0.5)]
!{\xcapv@(0)}
[u(0.5)r(1)]
*{\cdots}
[r(01)u(0.5)]
!{\xcapv@(0)}
[r(0.5)u(1)]
!{\vcap-}
[r(1.5)]
!{\xcapv@(0)}
[u(0.5)r(1)]
*{\cdots}
[r(01)u(0.5)]
!{\xcapv@(0)}
[u(1.5)l(3.5)]
*{\scriptstyle{i}}
[r(1)]
*{\scriptstyle{i+1}}
[l(3.5)]
*{\scriptstyle{1}}
[r(6)]
*{\scriptstyle{n}}
}
%}
,
\quad\quad
\gcapni=
\xygraph{
!{0;/r1.5pc/:}
[r(0.25)u(0.5)]
!{\xcapv@(0)}
[u(0.5)r(1)]
*{\cdots}
[r(01)u(0.5)]
!{\xcapv@(0)}
[r(0.5)]
!{\vcap}
[r(1.5)u(1)]
!{\xcapv@(0)}
[u(0.5)r(1)]
*{\cdots}
[r(01)u(0.5)]
!{\xcapv@(0)}
[d(0.5)l(3.5)]
*{\scriptstyle{i}}
[r(1)]
*{\scriptstyle{i+1}}
[l(3.5)]
*{\scriptstyle{1}}
[r(6)]
*{\scriptstyle{n}}
}
\yeee
Their compositions $\xUni = \gcupni\, \gcapni$ are the standard
generators of the \TLa\ $\cTLn$.

The \JWp\ $\jwpn\in\QcTLn$ is the unique non-trivial idempotent element satisfying the
condition
\xlee{ae1.4}
\acapni\;\jwpn =0,\qquad 1\leq i\leq n-1.
\xeee
The \JWp\ also satisfies the relation
\xlee{ae1.4a}
\jwpn\;\acupni =0,\qquad 1\leq i\leq n-1.
\xeee

We denote the corresponding idempotent element of
$\cTLpinf_n$  corresponding to $\jwpn$ as $\jwpnp$.
% denotes the corresponding idempotent element of $\cTLpinf_n$.

\subsubsection{Basic notions of homological algebra}
Let $\xChA$ be a category of \chcpls\ associated with an additive
category $\xctA$. An object of $\xChA$ is a \chcpl\
\ylee{ae1.ch1}
\xbA  = (\cdots \rightarrow
\xAio\xrightarrow{\xdio}\xAi\rightarrow\cdots),
\yeee
and a morphism
between two chain complexes is a \chmp\ defined as a \mmp
\xlee{ae1.10d}
\vcenter{\xymatrix{
\xbA \ar[d]^-{\xbf} &&
\cdots\ar[r]^-{\xdit} & \xAio \ar[r]^-{\xdio} \ar[d]^-{\yfio} & \xAi
\ar[r]^-{\xdi} \ar[d]^{\yfi} & \cdots
\\
\xbB &&
\cdots\ar[r]^-{\xdpit} & \xBio \ar[r]^{\xdpio} & \xBi
\ar[r]^-{\xdi} & \cdots
}
}
%,\qquad\qquad
%\xdpio\,\yfio = \yfi\,\xdio
\xeee
which commutes with the chain differential: $\xdpio\,\yfio = \yfi\,\xdio$ for all $i$.
A cone of a \chmp\ $\xbA\xrightarrow{\xbf}\xbB$ is a complex
\ylee{ae1.10b1}
\Cnbf
=
\lrbc{
\vcenter{
\xymatrix@C=1.5cm@R=0.5cm{
\cdots \ar[dr] \ar[r] & \xAi
\ar@{}[d] |{\oplus} \ar[r]^-{(-1)^i\xdi} \ar[dr]^{\xfi} &
\xAimo
\ar@{}[d] |{\oplus}
\ar[r] \ar[dr]& \cdots
\\
\cdots \ar[r] & \xBio\ar[r]_{\xdio} & \xBi \ar[r] & \cdots
}
}
}
\yeee
There are two special \chmp s
%$\chdlbf\colon
$\xbB\xrightarrow{\idlbf}\Cnbf$ and
$\Cnbf[-1]\xrightarrow{\chdlbf}\xbA$ associated to
the cone:
\ylee{ae1.10b2}
\xymatrix{
\xbB[-1] \ar[d]^-{\idlbf}&&
\cdots \ar[r] &
\xBio \ar[r] \ar[d]^-{0\oplus \xId}
&
\xBi \ar[r] \ar[d]^-{0\oplus \xId}
&
\cdots
\\
\Cnbf[-1] \ar[d]^-{\chdlbf} &&
\cdots \ar[r] &
\xAi \oplus \xBio \ar[r] \ar[d]^-{\xId\oplus 0} &
\xAimo\oplus \xBi \ar[r] \ar[d]^-{\xId\oplus 0} &
\cdots
\\
\xbA &&
\cdots \ar[r]
&
\xAi \ar[r]
&
\xAimo \ar[r]
&
\cdots
}
\yeee
These complexes and \chmp s form a \dstt:
\xlee{ae1.ch2}
\xymatrix{
\xbA\ar[r]^-{\xbf} &
\xbB \ar[r]^-{\idlbf} &
\Cnbf \ar[r]^-{\chdlbf} &
\xbA[1]
}.
\xeee

The homotopy category of complexes $\xKhA$ has the objects and
morphisms of $\xChA$ distinguished only up to homotopy
equivalence. The notion of a cone extends to $\xKhA$ and there
are additional relations: $\Cnv{\idlbf} \hteqv \xbA[1]$ and
$\Cnv{\chdlbf} \hteqv \xbB[1]$.

\subsubsection{A triply graded categorification of the Jones
polynomial}
\label{sss:trgr}
In\cx{Kh1} the first author (M.K.) introduced a categorification of
the Jones polynomial of links. To a diagram $\xL$ of a
link he associated a complex of graded modules
\xlee{ae1.5}
\dL = \lrbc{ \cdots \rightarrow \dLi \rightarrow \dLimo\rightarrow\cdots}
\xeee
so that
if two diagrams represent the same link then the corresponding
complexes are homotopy equivalent, and the graded Euler
characteristic of $\dL$ is equal to the Jones polynomial of $\xL$.

Thus, overall, the complex\rx{ae1.5} had two gradings: the first one was
the grading related to powers of $q$ and the second one was the
homological grading of the complex itself, the corresponding
degree being equal to $i$.
In this paper we adopt a slightly different convention which is
convenient for working with framed links and tangles. It is
inspired by matrix factorization categorification\cx{KR1} and its
advantage is that it is no longer necessary to assign orientation to
link strands in order to obtain the grading of the categorification
complex\rx{ae1.5} which would make it invariant under the second
Reidemeister move.

To a framed link
diagram $\xL$ we associate a triply graded complex\rx{ae1.5} with
degrees $\dgt$, $\dgo$ and $\dgh$.
The first two gradings are of the same nature as in\cx{Kh1} and, in
particular, $\dgo\dLi=i$. The third grading is an inner grading of
chain modules, it is defined modulo 2 and it is of homological
nature, that is, the homological parity of an element of $\dL$,
which affects various sign factors, is the sum of $\dgo$ and
$\dgh$. Both homological degrees are either integer or
half-integer simultaneously, so the homological parity is integer
and takes values in $\ZZ_2$. The $q$-degree $\dgt$ may also take
half-integer values.

%The
%first grading is the homological grading $i$: $\dgo\dLi=i$.
%The second and third gradings are inner gradings of individual chain modules
%$\dLi$. The second grading
%is the $\ZZ$-grading associated with powers of $q$. The third grading is
%of homological nature and it is defined only modulo 2. The first
%and third degrees  take integer or half-integer values simultaneously, so
% their sum, which is the total homological grading, is integer
%and takes values in $\ZZ_2$.

Let $\tgrshv{k}{m}{n}$ denote the shift of three degrees by $k$,
$m$ and $n$ units respectively. We also use abbreviated notations
$$
\tgrsshv{k}{m} = \tgrshv{k}{m}{0},\qquad
\qshv{k} = \tgrshv{k}{0}{0}.
$$
After the grading modification, the categorification
formulas of\cx{Kh1} take the following form:
the module associated with an unknot is still $\ZZ[x]/(x^2)$ but with
a different degree assignment:
\begin{eqnarray}
\label{ae1.6}
&\ZZ[x]/(x^2)\,
%[0,-1,1]
\tgrshv{-1}{0}{1},
%\qquad
\\
&\dgt 1 = 0, \quad \dgt x = 2,
\quad\dgo 1 = \dgo x = \dgh 1=\dgh x =0,
\end{eqnarray}
and the categorification complex of a crossing is the same as
in\cx{Kh1} but with a different degree shift:
\xlee{ae1.7}
%\xygraph{
%!{0;/r1.5pc/:}
%[u(0.5)]
%!{\xoverv}
%[u(1.5)r(0.53)]
%*{\smcat}
%}
\Bsymcat{\xcrsp}
\;\;=\;\;
\Bigg(\;\;
%\vcenter{{
%%%%%%\begin{CD}
%\xygraph{
%!{0;/r1.5pc/:}
%[u(0.5)]
%!{\xunoverv}
%[u(1.5)r(0.5)]
%*{\smcat}
%}
\Bsymcat{\xpver}
\;\tgrshv{\vthf}{\vthf}{-\vthf}
%\;\;+\;\;
%%%%%@>\xmrf>>
\xrightarrow{\;\;\;\;\xmrf\;\;\;\;}
%\xygraph{
%!{0;/r1.5pc/:}
%[u(0.5)]
%!{\xunoverh}
%[u(0.5)l(0.5)]
%*{\smcat}
%}
\Bsymcat{\xphor}
\;\tgrshv{-\vthf}{-\vthf}{\vthf}
\vspace*{18pt}
%%%%%%%\end{CD}
%}}
\;\;
\Bigg),
\xeee
where $f$ is either a multiplication or a comultiplication of the
ring $\ZZ[x]/(x^2)$ depending on how the arcs in the \rhs are
closed into circles.
The resulting categorification complex\rx{ae1.5} is invariant
up to homotopy under the second and third Reidemeister moves, but
it acquires a degree shift under the first Reidemeister move:
\xlee{ae1.8}
%\xygraph{
%!{0;/r1.5pc/:}
%[u(0.5)]
%!{\hover}
%!{\hcap}
%[u(0.5)l(0.25)]
%*{\smcat}
%}
%\;\; = \;\;
%%-q^{\frac{3}{2}}\;\;
%\xygraph{
%!{0;/r1.5pc/:}
%[u(0.5)]
%!{\xcapv@(0)}
%[u(1.5)]
%*{\smcat}
%}
%\!\!\!\!
%\tgrshv{\vthh}{\vthf}{\vthf}.
%%
\Bsymcat{\xvfro\hspace*{-0.2cm}}
\;\; = \;\;
\Bsymcat{\;\xvert\hspace*{-0.5cm}}
\tgrshv{\vthh}{\vthf}{\vthf}.
\xeee
It is easy to see that the whole categorification complex\rx{ae1.5} has a
homogeneous degree $\dgh$ which is equal to the sum of all linking
numbers of $\xL$ (including the self-linking numbers):
$\dgh\dL = $.

\subsubsection{A universal categorification of the \TLa}
D.~Bar-Natan\cx{BN1} described the universal category $\dTL$, whose
Grothendieck \Kzg\ is $\cTL$ considered as a $\Zqqi$-module.
We will use this category with obvious adjustments required by the new
grading conventions.

Let $\dTL\p$ be an additive category whose objects are in
one-to-one correspondence with \TLb\ tangles, morphisms being generated
by tangle cobordisms (see\cx{BN1} for details). The universal category
$\dTL$ is the homotopy category of bounded complexes associated with
$\dTL\p$. In other words, an object of $\dTL$ is a complex
\xlee{ae1.8a}
\xbC =
\lrbc{\cdots\rightarrow\xCipo\rightarrow\xCi\rightarrow\cdots},\qquad
\xCi =
\bigoplus_{\substack{j\in\ZZ \\ \xmu\in\ZZ_2}}
\sltln \cjilam\,\dlam \tgrshv{j}{0}{\mu},
\xeee
where non-negative numbers $\cjilam$ are multiplicities; since the
complex is bounded, they are non-zero for only finitely many
values of $i$.

%
%In other words, a complex as a whole is a sum
%%
%\xlee{ae1.8ab}
%\bigoplus_{\substack{i,j\in\ZZ \\ \xmu\in\ZZ_2}} % \\ \xlam\in\rTLn}}
%\sltln \cjilam\,\dlam \tgrshv{j}{i}{\mu},
%\xeee
%%
%where non-negative numbers $\cjilam$ are multiplicities; since the
%complex is bounded, they are non-zero for only finitely many
%values of $i$.

A categorification map $\rTL\xrightarrow{\mpcat}\dTL$ turns a framed
tangle $\xtau$ into a complex $\dtau$ according to the
rules\rx{ae1.6} and\rx{ae1.7}, the morphism $\xmrf$ in the
complex\rx{ae1.7} being the saddle cobordism. A composition of
tangles turns into a composition bi-functor of $\dTL$ if we apply
the categorified version of the rule\rx{ae1.1} in order to remove
disjoint circles:
\xlee{ae1.01}
%\xlam\sqcup\chlcir
% = \xlam\tgrshv{1}{0}{1} + \xlam\tgrshv{-1}{0}{1}.
%\chlcir
\Bsymcat{\lcir}= \cnot \tgrshv{1}{0}{1} + \cnot\tgrshv{-1}{0}{1},
\xeee
where $\xnot$ is the empty \TL\ \ttngzz.

Overall,
we have the following commutative diagram:
%
%\xlee{ae1.9}
\begin{equation}
\xymatrix@C=1.5cm@R=0.3cm{
& {}\dTL \ar[dd]^{K_0}
\\
{}\rTL \ar[ur]^{\mpcat} \ar[dr]^{\mpalg}
\\
& {}\cTL
}
\end{equation}
where the map $\Kz$ turns the complex\rx{ae1.8a} into the
sum\rx{ae1.2a0}:
\xlee{ae1.9a}
\Kz(\xbC)
%\Kz\bigg(\bigoplus_{\substack{i,j\in\ZZ \\ \xmu\in\ZZ_2}}
%\sltln \cjilam\,\dlam \tgrshv{j}{i}{\mu}\bigg)
=
\sltln\sum_{j\in\ZZ} \xcalj \,q^j\,\clam,
\qquad
\xcalj = \sum_{\substack{i\in\ZZ \\ \xmu\in\ZZ_2}}
(-1)^{i+\xmu} \cjilam.
%\sum_{\substack{i,j\in\ZZ \\ \xmu\in\ZZ_2}} \sltln (-1)^{i+\xmu}
%\cjilam\,q^{j}\,\clam.
\xeee
Since the complex is bounded, the sum in the expression for
$\xcalj$ is finite.

%In addition to $\dTL$ we consider
%% two categories of semi-bounded complexes:
%a category $\dTLp$ of complexes
%bounded from below and a category $\dTLm$ of complexes bounded
%from above (that is, the multiplicity coefficients in the
%sum\rx{ae1.8a} are zero for $i$ below or above certain value).
%We say that a complex of $\dTLp$ is \emph{\qpb} if the multiplicity
%coefficients $\cjilam$ are zero for $j<i$ and a complex of $\dTLm$
%is \emph{\qmb} if $\cjilam$ are zero for $j>i$. For a $q^+$ or a \qmb\
%complex, the sum in the expression\rx{ae1.9a} for $\xcalj$ is
%finite and the map $\Kz$ is well defined.

In addition to $\dTL$ we consider
% two categories of semi-bounded complexes:
a category $\dTLp$ of complexes
bounded from below, that is, the multiplicity coefficients in the
sum\rx{ae1.8a} are zero for $i$ below certain value.
A complex $\xbC$ in $\dTLp$ is \emph{\qpb} if $\lmii\yordq{\xCi} =
+\infty$.
%
%For a complex $\xbC$ in $\dTLp$ we define $\xbnbi(\xbC) = \xinfv{j\colon \exists\xlam\text{ such that }
%\cjilam\neq 0} $. A complex $\xbC$ is \emph{\qpb} if
%$\lmii\xbnbi(\xbC) = +\infty$.
For a \qpb\ complex,
the sum in the expression\rx{ae1.9a} for $\xcalj$ is
finite, hence the element $\Kz(\xbC)$ is well defined.

%We say that a complex of $\dTLp$ is \emph{\qpb} if the multiplicity
%coefficients $\cjilam$ are zero for $j<i$ and a complex of $\dTLm$
%is \emph{\qmb} if $\cjilam$ are zero for $j>i$. For a $q^+$ or a \qmb\
%complex, the sum in the expression\rx{ae1.9a} for $\xcalj$ is
%finite and the map $\Kz$ is well defined.

\subsection{Results}
\label{ss:res}

\subsubsection{Infinite \cbr\ as a \JWp\ in a \TLa}
A braid with $n$ strands is a particular example of a \ttngnn.
A braid is called \emph{\cylb} if it can be drawn on a cylinder $\Sco\times[0,1]$
without intersections. In fact, all \cbr s have the form
$\btcyl^m$, $m\in\ZZ$, where $\btcyl$ is the elementary
counterclock winding \cbr:
\xlee{ae1.10p}
\btcyln \;\;=\;\;
\xygraph{
!{0;/r1.5pc/:}
!{\vover}
[u(1.5)l(0.5)]
!{\xbendr[0.5]}
[u(1.25)l(1.25)]
!{\xbendd[-0.5]}
[u(1.25)l(1)]
!{\xcapv[0.25]@(0)}
[r(1.75)]
!{\xcaph@(0)}
[u(1)]
!{\vover[-1]}
[r(1)]
!{\xbendr[-0.5]}
[u(0.75)l(0.75)]
!{\xbendd[0.5]}
[u(0.5)l(0.5)]
!{\xcapv[0.25]@(0)}
[u(1)l(1.5)]
*{\cdots}
[u(1.5)l(1.5)]
*{\cdots}
[u(0.75)l(1.5)]
*{\scriptstyle{1}}
[r(2.5)]
*{\scriptstyle{n-1}}
[r(1.25)]
*{\scriptstyle{n}}
[d(3)l(3)]
*{\scriptstyle{1}}
[r(1)]
*{\scriptstyle{2}}
[r(2.75)]
*{\scriptstyle{n}}
[u(2.2)l(0.35)]
%%%%%%% negative framing !!!!!!!!
%*{\symfr\;\scriptstyle{-1} }
}
\xeee
%
%
%
%
%\ylee{ae1.10a}
%\xygraph{
%!{0;/r1.5pc/:}
%!{\vunder}
%[u(1.5)l(0.5)]
%!{\xbendr[0.5]}
%[u(1.25)l(1.25)]
%!{\xbendd[-0.5]}
%%
%[u(1.25)l(1)]
%!{\xcapv[0.25]@(0)}
%%
%[r(1.75)]
%!{\xcaph@(0)}
%[u(1)]
%!{\vunder[-1]}
%[r(1)]
%!{\xbendr[-0.5]}
%[u(0.75)l(0.75)]
%!{\xbendd[0.5]}
%%
%[u(0.5)l(0.5)]
%!{\xcapv[0.25]@(0)}
%%
%[u(1)l(1.5)]
%*{\cdots}
%[u(1.5)l(1.5)]
%*{\cdots}
%%
%[u(0.75)l(1.5)]
%*{\scriptstyle{1}}
%[r(2.5)]
%*{\scriptstyle{n-1}}
%[r(1.25)]
%*{\scriptstyle{n}}
%%
%[d(3)l(3)]
%*{\scriptstyle{1}}
%[r(1)]
%*{\scriptstyle{2}}
%[r(2.75)]
%*{\scriptstyle{n}}
%}
%\yeee
%%
%
%The negative framing twist is added to maintain a natural
%cylindrical framing: the ribbons representing framed strands of
%the braid can be drawn on the cylinder.
%
We introduce a special notation for the \cbr\ which corresponds to
$m$ full rotations of $n$ strands:
\ylee{eq:brdf}
\gbrmn = \btcyl^{mn}.
\yeee

Let $\Opqm$ denote any element of $\cTLpinf$ of the form
$\sltln\sum_{i\geq m} \xcali\,q^i\,\clam$.
We define a \emph{\qord} of an element $\yal\in\cTLpinf$ as
$\yordq{\yal} = \xsupv{m\colon \yal = \Opqm}$.

\begin{definition}
\label{df:qlm}
A sequence of elements
\wlee{ae1.2a1}
\yal_1,\yal_2,\ldots\in\cTLpinf%
\weee
has a limit
$\lim_{k\rightarrow \infty} \yal_k = \ybet$, $\ybet\in\cTLpinf$
if $\lmii\yordq{\ybet-\yal_k} = +\infty$.
\end{definition}

The following theorem should be well-known: we do not claim
credit for it and it is an easy corollary of \ex{ae2.m4}.
\begin{theorem}
\label{th:alg}
The \TL\ element corresponding to the infinite \cbr\ equals the
\JWp:
\xlee{ae1.9}
\lim_{m\rightarrow+\infty} q^{\vthf mn(n-1)}\abrmn = \jwpnp,
%\qquad\text{in $\cTLpinf$}.
%%%,\qquad\lim_{m\rightarrow+\infty}(\btcylan)^{-mn} = \jwpn
%%%\qquad\text{in $\cTLminf$}.
\xeee
where $\jwpnp\in \cTLpinf_n$ corresponds to the \JWp\
$\jwpn\in\QcTLn$.
\end{theorem}
In fact, a more general statement is also true:
\xlee{ae1.10}
\lim_{m\rightarrow+\infty} q^{\vthf m(n-1)}\symalg{\btcyln^m} =
\jwpnp,
%\qquad\text{in $\cTLpinf$}.
%,\qquad\lim_{m\rightarrow+\infty}(\btcylan)^{-m} = \jwpn
%\qquad\text{in $\cTLminf$}.
\xeee
but its proof is more technical and we omit it here.

\subsubsection{A bit of homological calculus}
\label{sss.homcal}

Let $\xKhA$ denote the homotopy category of complexes associated
with an additive category $\xctA$ (we have in mind a particular case of $\xKhA = \dTLp$).
%An object of $\xKhA$ is a \chcpl\
%$\xbA  = (\cdots \rightarrow
%\xAio\xrightarrow{\xdio}\xAi\rightarrow\cdots)$ and a morphism
%between two chain complexes is a \chmp
%%
%\xlee{ae1.10dx}
%\vcenter{\xymatrix{
%\cdots\ar[r]^-{\xdit} & \xAio \ar[r]^-{\xdio} \ar[d]^-{\yfio} & \xAi
%\ar[r]^-{\xdi} \ar[d]^{\yfi} & \cdots
%\\
%\cdots\ar[r]^-{\xdpit} & \xBio \ar[r]^{\xdpio} & \xBi
%\ar[r]^-{\xdi} & \cdots
%}
%},
%\qquad\qquad
%\xdpio\,\yfio = \yfi\,\xdio
%\xeee
%%
%%($\xdpio\yfio = \yfi\,\xdio$ for all $i$)
%distinguished up to homotopy
%equivalence.

A \chcpl\ is considered `homologically small' if it ends at a high
homological degree.
Let
$\Ohpm$ denote a complex which ends at $m$-th homological
degree: $\Ohpm = (\cdots \xAv{m+1} \rightarrow\xAv{m})$. We define
a homological order of a complex $\xbA$ as $\yordh{\xbA} =
\xsupv{m\colon\xbA\hteqv\Ohpm}$.

Two complexes connected by a \chmp: $\xbA\xrightarrow{\xbf}\xbB$
are considered `homologically close' if $\Cnbf$ is homologically
small.

A \emph{\chsq} is a sequence of complexes connected by \chmp s:
\ylee{ae1.10d1}
\scA = (\xbAz\xrightarrow{\xbfz} \xbAo
\xrightarrow{\xbfo}\cdots).
\yeee
\begin{definition}
\label{df:cauchy}
A \chsq\ $\scA$ is \emph{\Cch} if $\lmii \yordh{\Cnbfi} = +\infty$.
\end{definition}
\begin{definition}
\label{df:sqlm}
A \chsq\ has a limit: $\dlm\scA = \xbA$, where $\xbA$ is a \chcpl, if
there exist \chmp s $\xbAi\xrightarrow{\xbtfi}\xbA$ such that
they form commutative triangles
\xlee{ae1.10e1}
\cmtr{\xbfi}{\xbtfio}{\xbtfi}{\xbAi}{\xbAio}{\xbA}
\xeee
%
%$\xbtfi \hteqv \xbtfio\,\xbfi$
and $\lmii\yordhr{\Cnv{\xbtfi}} =+\infty$.
\end{definition}

In Section\rw{s:elhomcal} we prove the following homology versions
of standard theorems about limits
(Propositions\rw{pr:chlm},\rw{pr:lmch} and\rw{pr:lmun}):
\begin{theorem}
\label{th:lmt}
A \chsq\ $\scA$ has a limit if and only if it is \Cch.
\end{theorem}
\begin{theorem}
\label{th:lmt2}
The limit of a \chsq\ is unique up to homotopy equivalence.
%If $\scA$ has a limit, then the limit is unique up to homotopy.
\end{theorem}

\subsubsection{Infinite \cbr\ as a \JWp\ in the universal category}

For a tangle diagram $\xtau$ let
$\dtaus$ denote the categorification complex $\dtau$
with a degree shift proportional to the number $\crsv{\xtau}$ of crossings
in the diagram $\xtau$:
\xlee{ae1.10b1}
\dtaus = \dtau\tgrshv{\vthf}{\vthf}{-\vthf}^{\crsv{\xtau}}.
\xeee
%
%where $\crsv{\xtau}$ is the number of crossings in
%the diagram $\xtau$.

In subsection\rw{ss:brchsq} we define a
\chsq\ of \cbr\ categorification complexes connected by special
\chmp s
%
%sequence of special \chmp s
%
%$\btcylcnmns \xrightarrow{\mrfm}  \btcylcnmpons$ which form a \chsq
%
%\xlee{ae1.10cx}
%\xctBn =
%\Big(\xymatrix@C=0.6cm{
%{}\cIdbn \ar[r]^-{\mrfz}
%& {}\btcylcnns \ar[r]^-{\mrfo}
%& \cdots \ar[r]^-{\mrfmmo}&{}
%\btcylcnmns
%%\symcats{\gbrmn}
%\ar[r]^-{\mrfm}
%& {}\btcylcnmpons \ar[r]^-{\mrfmo} &\cdots
%}\Big).
%\xeee
%
\begin{multline}
\label{ae1.10c}
\xctBn =
\Big(
\cidbrn
\xratv{\mrfz}
\cbrons \xratv{\mrfo}
\cdots
\\
\cdots
\xrightarrow{\mrfmmo}
\cbrmns \xraov{\mrfm}
\cbrmons \xrightarrow{\mrfmo}\cdots\Big)
\end{multline}
We prove that
%$\Cnv{\mrfm} \hteqv \Ohp\big(2m(n-1) + 1\big)$,
$\yordhr{\mrfm}\geq2m(n-1) + 1$,
so
$\xctBn$ is a \Csq\ and by  Theorem\rw{th:lmt} it has a unique limit:
$\lim\xctBn =\ctjwpn \in\dTLnp $.

\begin{theorem}
\label{th:enum}
The limiting complex $\ctjwpn$ has the following properties:
\begin{enumerate}
%\item
%There exists a \odct\ \otbl\ complex $\xbCn$ such that
%$\xbCn\tgrsshv{1}{1}\hteqv\Cnv{\xbtfz}$.

\item A composition of $\ctjwpn$ with cap- and \uptg s is contractible:
\ylee{eq:auc}
\ccapni \;\ctjwpn \hteqv \ctjwpn\; \ccupni\hteqv 0.
\yeee
\item The complex $\ctjwpn$ is idempotent with respect to tangle composition:
$\ctjwpn \ctjwpn
\hteqv \ctjwpn$.

\end{enumerate}
\end{theorem}

We provide a glimpse into the structure of $\ctjwpn$.
A complex $\xbC$  in $\cTLn$ is called \emph{\odct} if all
multiplicities $\xcv{j,i,\xIdbn}$ in \ex{ae1.8a} are zero, that is, the tangle
$\gidbrn$ never appears in \qcmds\ $\xCi$.
A complex $\xbC$ in $\cTLn$ is called \emph{\otbl} if the multiplicities
$\cjilam$ of \ex{ae1.8a} satisfy the property
\xlee{ae2.m1}
%\text{$\cjilam=0$ for $j<i$ and for $j>2i$}
\cjilam=0\qquad\text{if $i<0$, or $j<i$, or $j>2i$.}
\xeee
%

%Let
%$$\cidbrn\xraov{\xbtfz}\ctjwpn,\qquad
%\cbrons\xraov{\xbtfo}\ctjwpn,\ldots\quad
%$$
%%

Let $\cbrmns\xraov{\xbtfm}\ctjwpn$ be \chmp s associated
with the limit $\lim\xctBn = \ctjwpn$ in accordance with
Definition\rw{df:sqlm}.
\begin{theorem}
\label{th:cnpr}
There exist  \odct\ \otbl\ complexes $\wbCmn$ such that
$\Cnv{\xbtfm}\hteqv\wbCmn\spshmn\tgrsshv{1}{1}$.
%\tgrsshv{2mn+1}{2m(n-1)+1}$.
\end{theorem}
\noindent
In other words, there exists a distinguished triangle
%%
%\ylee{ae2.m2x}
%\wbCmnp\qsho \xratv{\chdlbtfm} \cbrmns \xrahv{\xbtfm} \ctjwpn
%\xrightarrow{\;\;\;\;\;\;\;} \wbCmnp\tgrsshv{1}{1}
%\yeee
%%
%
\ylee{ae2.m2}
\wbCmn\spshmn\qsho \xratv{\chdlbtfm} \cbrmns \xrahv{\xbtfm} \ctjwpn
\xrightarrow{\;\;\;\;\;\;\;} \wbCmn\spshmn\tgrsshv{1}{1}
\yeee
so there is a presentation
%%
%\xlee{ae2.m4y}
%\ctjwpn \hteqv \CnBv{ \wbCmnp\qsho
%\xrahv{\chdlbtfm} \cbrmns},
%\xeee
%%
%
\xlee{ae2.m4}
\ctjwpn \hteqv \CnBv{ \wbCmn\spshmn\qsho
\xrahv{\chdlbtfm} \cbrmns},
\xeee
%
%where $\wbCmnp = \wbCmn\spshmn$,
where the complex $\wbCmn$ is
\odct\ and \otbl.
%%
%\ylee{ae2.m3x}
%\ctjwpn \hteqv \CnBv{ \wbCmn\tgrsshv{2mn+1}{2m(n-1)}
%\xrahv{\chdlbtfm} \cbrmns}.
%\yeee
%%

At $m=0$ the formula\rx{ae2.m4} becomes
\xlee{ae2.m5}
\ctjwpn \hteqv \CnBv{ \wbCzn\qsho
\xrahv{\chdlbtfz} \cidbrn},
\xeee
where the complex $\wbCzn$ is \odct\ and \otbl.

%
%Let $\cidbrn\xraov{\xbtfz}\ctjwpn$ be the first \chmp\
%associated with the limit $\lim\xctBn =\ctjwpn$ in
%accordance with Definition\rw{df:sqlm}.
%
%\begin{theorem}
%\label{th:cnpr_x}
%There exists a \odct\ \otbl\ complex $\xbCn$ such that
%$\Cnv{\xbtfz}\hteqv\xbCn\tgrsshv{1}{1}$.
%%$\hteqv$.
%\end{theorem}
%%
%In other words, there is a \dstt
%%%
%%\ylee{ae2.m2x}
%%\xymatrix{
%%\xbCn\qsho \ar[r]^-{\chdlbtfz} &
%%{}\cIdbn \ar[r]^-{\xbtfz} &
%%\ctjwpn \ar[r] &
%%\xbCn\tgrsshv{1}{1}
%%}
%%\yeee
%%%
%%
%\ylee{ae2.m2x}
%\xbCn\qsho \xratv{\chdlbtfz} \cidbrn \xrahv{\xbtfz} \ctjwpn
%\xrightarrow{\;\;\;\;\;\;\;} \xbCn\tgrsshv{1}{1}
%\yeee
%%
%and $\ctjwpn \hteqv \Cnchdlbtfz$.

Since $\wbCzn$ is \otbl, the complex $\Cnchdlbtfz$ is
also \otbl, hence it is \qpb\ and $\Kctjwpn$ is well-defined. Also
$\Kctjwpn\neq 0$, because it contains $\cidbrn$ with coefficient 1.
Theorem\rw{th:enum} indicates that
$\Kctjwpn$ satisfies
defining properties of the \JWp, hence by uniqueness it is the \JWp:
\begin{corollary}
The complex $\ctjwpn$ categorifies the \JWp\ in
$\cTLpinf$:
\xlee{eq:catKz}
\Kctjwpn = \jwpn.
\xeee
\end{corollary}

%An obvious corollary of this theorem is that $\ctjwpn$ is
%homotopy equivalent to \qpb\ complex, hence
%$\Kctjwpn$ is well-defined. It also implies that
%$\Kctjwpn$ is non-trivial: it contains $\aIdbn$ with
%coefficient 1. The properties 2 and 3 indicate that $\Kctjwpn$ satisfies
%defining properties of the \JWp, hence by uniqueness it is the \JWp:
%%
%\begin{corollary}
%The complex $\ctjwpn$ categorifies the \JWp\ in
%$\cTLpinf$: $\Kctjwpn = \jwpn$.
%\end{corollary}

%%
%\begin{theorem}
%There exists a limit
%%
%\xlee{ae1.11}
%\lmp{m\rightarrow\infty}
%\lrbc{\btcylcn}^{mn} \sht
%%\,\shmn
%%\tgrshv{\sfnh+1}{\sfnh}{-\sfnh+1}}^{mn}
%=
%\ctjwpn \in\dTLnp
%\xeee
%%
%The semi-bounded complex is a projector:
%$(\ctjwpn)^2\xhte\ctjwpn$ and its composition with $\ccapni$ is
%contractible: $\ccapni\ctjwpn\xhte 0$. The complex $\ctjwpn$ is \qpb\ and it categorifies the
%\JWp:
%%$\Kz$ maps it to the \JWp:
%$\Kz(\ctjwpn) = \jwpn$.
%\end{theorem}

%\section{Cylindrical braids are \xgd}

\section{Elementary homological calculus}
\label{s:elhomcal}

\subsection{Limits in the category of complexes}

Consider a category $\xChA$ of \chcpls\ associated with
an additive category $\xctA$.
An $i$-th \emph{\trnc} $\xtrniv{\xbA}$ of a \chcpl\ $\xbA$ is
the \chcpl\ $\cdots\rightarrow\xAio\xrightarrow{\xdio}\xAi$. An
\trnci\ of a \chmp\ $\xbf$ is defined similarly.

Define an \emph{\isor}
$\ysiobf$
of a chain map $\xbA\xrightarrow{\xbf}\xbB$  as
%$\ysiobf = \xsupv{i\colon\xtrniv{\xbf}\quad\text{is an
%isomorphism}}$.
the largest number $i$ for which a truncated \chmp\ $\xtrniv{\xbf}$ is an
isomorphism of truncated complexes.

\begin{remark}
\label{rm:cnord}
Consider a \dstt\rx{ae1.ch2}.
If $\xbA\hteqv\Ohpm$, then $\ysiov{\idlbf}\geq m$.
\end{remark}

\begin{definition}
A \chsq\
$\scA = (\xbAo\xrightarrow{\xbfo}\xbAt\xrightarrow{\xbft}\cdots)$ is
\emph{\stblz} if
$\lim_{i\rightarrow\infty} \ysiobfi=+\infty$.
\end{definition}

\begin{definition}
A \chsq\ $\scA$ has a \tchlm\ $\chlm\scA=\xbA$ if there exist
\chmp s $\xbAi\xrightarrow{\xbtfi}\xbA$ such that
$\xbtfi = \xbtfio\,\xbfi$ and $\lmii \ysiorv{\xbtfi} = +\infty$.
\end{definition}

The following two theorems are easy to prove:
\begin{theorem}
A \chsq\ has a \tchlm\ if and only if it is \stblz.
%A \tchlm\ of a \stblz\ \chsq\ is unique.
If a \tchlm\ exists then it is unique.
\end{theorem}

%Then there exists a
%unique (up to an isomorphism) \chcpl\ $\xbA$ such that for any
%$N>0$ there exists $N\p$ such that for any
%%$i\geq N$ and any
%$i\geq N\p$ there is an isomorphism of \trnd\ complexes
%$\xtrnNv{\xbAi}\cong\xtrnNv{\xbA}$. In this case we use a notation
%%$\lmii\xbAi
%$\chlm\scA= \xbA$.
%
%The limiting complex $\xbA$ comes with a special sequence of \chmp s
%$\xbAi\xrightarrow{\xbtfi}\xbA$ ($\xbtfi = \cdots\xbtfio\xbtfi$) with the
%property $\xbtfi = \xbtfio\,\xbfi$.

\begin{theorem}
\label{pr:fnctchlm}
Suppose that $\chlm\scA = \xbA$.
Then for a complex $\xbB$ and \chmp s
$\xbAi\xrightarrow{\ybgi}\xbB$ such that $\ybgi = \ybgio\xbfi$,
 there exists a unique \chmp\ $\xbA\xrightarrow{\ybg}\xbB$
such that $\ybgi = \ybg\,\xbtfi$.
\end{theorem}

\begin{definition}
\label{df:chlmmp}
A sequence of \chmp s $\xbA\xrightarrow{\xbfz,\xbfo,\cdots}\xbB$
has a \tchlm\ $\lmii\xbfi = \xbf$ if for any $N$ there exists $N\p$
such that $\xtrnNv{\xbfi} = \xtrnNv{\xbf}$ for any $i\geq N\p$.
\end{definition}

\subsection{Limits in the homotopy category}

Definitions\rw{df:cauchy} and\rw{df:sqlm} extend the notion of a
\stblz\ \chsq\ and its limit to the homotopy category $\xKhA$:
obviously, a \stblz\ \chsq\ is \Cch, while $\chlm\scA=\xbA$ implies
$\lim\scA=\xbA$.

\begin{proposition}
\label{pr:chlm}
A \Csq\ has a limit.
\end{proposition}
\proof
Consider a \Csq\ $\scA$. We construct a special \chcpl\
$\xbAs$ such that $\lim\scA=\xbAs$ in accordance with
Definition\rx{df:sqlm}. Roughly speaking, we take $\xbAz$ and
attach to it the cones $\Cnbfi$ represented by homologically small
complexes, one by one. The result is a sequence $\scAs=\xbApz,\xbApo,\ldots$ of \stblz\
complexes $\xbApi$ such that $\xbApi\hteqv\xbAi$, and
$\xbAs=\chlm\scAs$ is their \tchlm.

Here is a detailed explanation.
By Definition\rw{df:cauchy}, there exist complexes $\ybCi$ such
that
\xlee{ae1.10a1}
\Cnv{\xbfi} \hteqv \ybCi[1] = \Ohpmi,\qquad\lmii m_i=+\infty.
\xeee
The complexes $\xbAi$, $\xbAio$ and $\ybCi$ form exact triangles:
\ylee{ae1.10a2}
\xymatrix{
\ybCi \ar[r]^-{\chdlbfi} & \xbAi \ar[r]^-{\xbfi} & \xbAio \ar[r] &
\ybCi[1]
}
\yeee
and $\xbAio \hteqv \Cnv{\chdlbfi}$. We define recursively a new sequence
of complexes $\scAp = (\xbApz \xrightarrow{\idlbgz} \xbApo\xrightarrow{\idlbgo}\cdots)$ by
the relations $\xbApz = \xbAz$, $\xbApi\hteqv\xbAi$ and
$\xbApio = \Cnv{\ybgi}$, where the \chmp\
$\ybCi\xrightarrow{\ybgi}\xbApi$ is homotopy equivalent to the
\chmp\ $\chdlbfi$. In other words,
\xlee{ae1.10g1}
\xbApio = \Cnv{\ybCi\xrightarrow{\ybgi}
\underbrace{
\Cnv{\ybCimo\xrightarrow{\ybgimo}\cdots\xrightarrow{\ybgt}
\underbrace{
\Cnv{\ybCo\xrightarrow{\ybgo}
\underbrace{
\Cnv{\ybCz\xrightarrow{\chdlbfz}\xbAz
}}_{\xbApo}\;
}
}_{\xbApt}\;
}
}_{\xbApi}\;
}
\xeee

According to Remark\rw{rm:cnord}, $\ysiov{\idlbgi}\geq m_i$,
hence the sequence
$\scAp$ is \stblz, so there exists a chain limit
$\chlm\scAp = \xbAs$ and consequently there is a limit
$\lim\scA=\xbAs$.\myqed

Simply saying, the complex $\xbAs$ is an infinite \mtcn\ extension
of the complex\rx{ae1.10g1}:
\xlee{ae1.10g2}
\xbAs =
\cdots\xrightarrow{\ybgh}\Cnv{\ybCt
\xrightarrow{\ybgt}
\Cnv{\ybCo\xrightarrow{\ybgo}
\Cnv{\ybCz\xrightarrow{\chdlbfz}\xbAz
}
}
}
\xeee

For our applications it is important to express $\Cnbtfz$ in terms
of complexes $\ybCi$. This can be done by rearranging
the infinite \mtcn\rx{ae1.10g2} with the help of associativity of
cone formation, which exists even within the category $\xChA$:
%
%The associativity of cone formation even within the category
%$\xChA$ allows us to reshuffle the infinite \mtcn\rx{ae1.10g2}:
%
\xlee{ae1.10h1}
\xbAs = \Cnv{\ybtC\xrightarrow{\ybtg}\xbAz},\qquad
\ybtC =\cdots\xrightarrow{\ybht}
\Cnv{\ybCt[-1]\xrightarrow{\ybho}\Cnv{\ybCo[-1]\xrightarrow{\ybhz}\ybCz}},
%\qquad
%\xbtfz \hteqv \idlv{\ybtg},
\xeee
so that $\xbtfz \hteqv \idlv{\ybtg}$, and $\Cnbtfz\hteqv\ybtC[1]$
is expressed in terms of complexes $\ybCi$
arranged into an infinite \mtcn\ $\ybtC$. Here is a more formal
statement.
\begin{theorem}
\label{th:rshfl}
For a \Csq\ $\scA$ there exists another \Csq\
$\yctC = (\ybCz
\xrightarrow{\ybhpz} \ybtCo\xrightarrow{\ybhpo}\cdots)$
%, $\ybtCz=\ybCz$,
and \chmp s
$\ybCi[-1]\xrightarrow{\ybhi} \ybtCi$ such that
$\Cnv{\ybhi}=\ybtCio$, $\ybhpi = \idlv{\ybhi}$ and for the
limiting complex $\ybtC=\chlm\yctC$ there exists a \chmp\
$\ybtC\xrightarrow{\ybtg}\xbAz$ such that $\xbAs = \Cnv{\ybtg}$,
$\xbtfz \hteqv \idlv{\ybtg}$ and consequently $\Cnbtfz\hteqv\ybtC[1]$.
\end{theorem}

\proof
%The cones $\Cnbtfi$ of the \chmp s $\xbtfi$ which come with the limit $\lim\scA
%= \xbAs$ can be presented as \mtcn s built of the complexes $\ybCi$ of
%\ex{ae1.10a1}. We will consider the case of $i=0$ (others are similar).
%
%The \mtcn\ presentation of $\Cnbtfz$ comes from the associativity of
%taking a cone even within the category $\xChA$. Consider a general
%case first.
%
Let us recall the associativity of cones in a general setting.
For a \chmp\ $\xbA\xrightarrow{\xbf}\xbB$,
a \chmp\ $\xbC\xrightarrow{\ybg} \Cnbf$
%where $\xbf$ is a \chmp\ $\xbA\xrightarrow{\xbf}\xbB$,
is a sum: $\ybg = \ybgA \oplus\ybgB$
\ylee{ae1.10f1}
\xymatrix{
 & \xbA\ar[d]^-{\xbf}
\\
\xbC \ar[ur]|{[1]}^-{\ybgA} \ar[r]_-{\ybgB}
& \xbB
}
\yeee
where $\xbC\xrightarrow{\ybgA}\xbA[1]$ is a \chmp\ and
$\xbC\xrightarrow{\ybgB}\xbB$ is a \mmp. Now it is obvious
that
\xlee{ae1.10f2}
\Cnv{\xbC\xrightarrow{\ybg}\Cnv{\xbA\xrightarrow{\xbf}\xbB}}
=\Cnv{\Cnv{\xbC[-1]\xrightarrow{\ybgA}\xbA}\xrightarrow{\ybgB\oplus\xbf}\xbB
}.
\xeee

We apply the associativity relation\rx{ae1.10f2} to
\mtcn s\rx{ae1.10g1} consecutively for $i=1,2,\ldots$ in order to rearrange
them, so that $\xbApi = \Cnv{\ybtCi\xrightarrow{\ybtgi}\xbAz}$,
while the complexes $\ybtCi$ and \chmp s $\ybtgi$ are defined
recursively: $\ybtCz=\ybCz$, $\ybtgz = \chdlbfz$, $\ybtCio = \Cnv{\ybhi}$,
while the \chmp s $\ybCi[-1]\xrightarrow{\ybhi} \ybtCi$ and
$\ybtCio\xrightarrow{\ybtgio}\xbAz$
are defined by applying the associativity
relation\rx{ae1.10f2} to the  double cone on the second line of
the following equation:
\begin{equation}
\label{ae1.10f3}
\begin{split}
\xbApio & = \Cnv{\ybCi\xrightarrow{\ybgi}\xbApi}
\\
& = \Cnv{\ybCi\xrightarrow{\ybgi}\Cnv{\ybtCi\xrightarrow{\ybtgi}\xbAz} }
\\
& = \Cnv{\Cnv{\ybCi[-1]\xrightarrow{\ybhi} \ybtCi
} \xrightarrow{\ybtgio}
\xbAz
}
\\
& = \Cnv{\ybtCio\xrightarrow{\ybtgio}\xbAz}.
\end{split}
\end{equation}
Distinguished triangles
\ylee{ae1.10f3}
\xymatrix{
\ybCi[-1]\ar[r]^-{\ybhi}
&
\ybtCi \ar[r]^-{\idlbhi}
&
\ybtCio \ar[r]
&
\ybCi
}
\yeee
determine  \chmp s $\ybhpi=\idlbhi$ of the \chsq\  $\yctC = (\ybtCz
\xrightarrow{\ybhpz} \ybtCo\xrightarrow{\ybhpo}\cdots)$. By
Remark\rw{rm:cnord}, it has a \tchlm: $\chlm\yctC =
\ybtC$, which is an infinite \mtcn:
\ylee{ae1.10f4}
\ybtC =\cdots\xrightarrow{\ybht}
\Cnv{\ybCt[-1]\xrightarrow{\ybho}\Cnv{\ybCo[-1]\xrightarrow{\ybhz}\ybCz}}.
\yeee
The \chmp s $\ybtCi\xrightarrow{\ybhpi}\ybtCio$ satisfy a relation
$\ybtgi = \ybtgio\,\ybhpi$, so by Theorem\rw{pr:fnctchlm} there
exists a unique \chmp\ $\ybtC \xrightarrow{\ybtg}\xbAz$ such that
$\ybtgi=\ybtg\,\ybhtpi$.
It is easy to show that $\xbAs = \Cnv{\ybtC\xrightarrow{\ybtg}\xbAz}$,
and $\xbtfz = \idlv{\ybtg}$, hence $\Cnbtfz\hteqv \ybtC$.\myqed

It is easy to prove the analog of Theorem\rw{pr:fnctchlm}:
\begin{theorem}
\label{pr:spmp}
For a complex $\xbB$ and \chmp s
$\xbAi\xrightarrow{\ybgi}\xbB$ such that $\ybgi \hteqv \ybgio\xbfi$,
 there exists a unique (up to homotopy) \chmp\ $\xbAs\xrightarrow{\ybg}\xbB$
which forms commutative triangles
\xlee{ae1.10f1}
\cmtr{\xbtfi}{\ybg}{\ybgi}{\xbAi}{\xbAio}{\xbB}
\xeee
%
%such that
%$\ybgi \hteqv \ybg\,\xbtfi$.
\end{theorem}

In order to complete the proof
of Theorems\rw{th:lmt} and\rw{th:lmt2},
we need two simple propositions. The first one establishes a
triangle inequality for homological orders of cones.
\begin{proposition}
If three \chmp s form a commutative triangle
\xlee{ae1.10c1}
\xymatrix{
\xbA \ar[r]_-{\xbfAB} \ar@/^1pc/[rr]^-{\xbfAC} &
\xbB \ar[r]_-{\xbfBC} &
\xbC
},\qquad \xbfAC\hteqv \xbfBC\xbfAB.
\xeee
then the homological orders of their cones satisfy the inequalities
\begin{align}
\label{eq:in1}
\yordch{\xbfAB}\geq \min\big(\yordch{\xbfAC},\yordch{\xbfBC}-1\big),
\\
\label{eq:in2}
\yordch{\xbfBC}\geq \min\big(\yordch{\xbfAB}+1,\yordch{\xbfAC}\big).
\end{align}
\end{proposition}
\proof
If \chmp s form a commutative triangle\rx{ae1.10c1}, then their cones form a \dstt
\ylee{ae1.10c2}
\Cnv{\xbfAB}\xrightarrow{\ybgo} \Cnv{\xbfAC} \xrightarrow{\ybgt} \Cnv{\xbfBC}
\xrightarrow{\ybgh}\Cnv{\xbfAB}[1],
\yeee
so the first inequality follows from the relation
$\Cnv{\xbfAB}\hteqv\Cnv{\ybgt}[-1]$ and the second inequality
follows from the relation $\Cnv{\xbfBC} \hteqv \Cnv{\ybgo}$.\myqed

The second proposition says that if a complex is homologically
infinitely small then it is contractible.
\begin{proposition}
\label{pr:ismc}
If $\yordh{\xbA} = +\infty$ then $\xbA$ is contractible.
\end{proposition}
\proof Since $\yordh{\xbA} = +\infty$, there exist complexes
$\xbAi\hteqv\xbA$, such that $\xbAi=\Ohpmi$ and $\lmii m_i =
+\infty$. Consider a sequence of \chmp s establishing homotopy
equivalence between the complexes:
\ylee{ae1.10c3}
\xymatrix@C=0.5cm{
\xbA \ar@<0.6ex>[r]^-{\xbfz}
&
\xbAo
\ar@<0.3ex>[l]^-{\ybgz}
\ar@<0.6ex>[r]^-{\xbfo}
&
\xbAt
\ar@<0.3ex>[l]^-{\ybgo}
\ar@<0.6ex>[r]%^-{\xbft}
&
\cdots
\ar@<0.3ex>[l]
\ar@<0.6ex>[r]
&
\xbAi
\ar@<0.3ex>[l]
\ar@<0.6ex>[r]^-{\xbfi}
&
\xbAio
\ar@<0.3ex>[l]^-{\ybgi}
\ar@<0.6ex>[r]
&
\cdots
\ar@<0.3ex>[l]
}
,\qquad
\yIdAi-\ybgi\xbfi =
%\xbdi\, \xbhi + \xbhi\, \xbdi,
\atcmv{\xbdi}{\xbhi},
\yeee
where $\yIdAi$ is the identity
\chmp\ of $\xbAi$, while $\xbAi[1]\xrightarrow{\xbhi}\xbAi$ is a homotopy \chmp\ (it does
not commute with the chain differential $\xbdi$ in the complex $\xbAi$).

Consider the compositions $\xbhfi = \xbfi\cdots\xbfo\xbfz$,
$\xbhgi = \ybgz\ybgo\cdots\ybgi$ and $\xbhhi =
\xbhgimo\,\xbhi\,\xbhfimo$. It is easy to see
that $\xbhgimo\,\xbhfimo - \xbhgi\,\xbhfi = \atcmv{\xbd}{\xbhhi}$,
hence $\yIdA - \xbhgi\,\xbhfi = \atcmv{\xbd}{\xbchi}$, where
$\xbchi = \xbhhz + \xbhho +\cdots + \xbhhi$.
There is a limit (\cf Definition\rw{df:chlmmp}) $\lmii\xbchi = \xbch$, while
$\lmii\xbhgi\,\xbhfi = 0$, hence $\yIdA = \atcmv{\xbd}{\xbch}$
which means that the complex $\xbA$ is contractible.
\myqed

\begin{proposition}
\label{pr:lmch}
If a \chsq\ $\scA$ has a limit, then it is \Cch.
\end{proposition}
\proof
The inequality\rx{eq:in1} applied to the commutative
triangle\rx{ae1.10e1} says that
$$\yordch{\xbfi}\geq
\min\lrbc{\zordch{\xbtfi},\zordch{\xbtfio}-1},$$
hence the limit $\lmii\zordch{\xbtfi} = +\infty$ implies the \Cch\
property of $\scA$.

\begin{proposition}
\label{pr:lmun}
If a \chsq\ $\scA$ has a limit then it is unique.
\end{proposition}
\proof
If $\scA$ has a limit then by Proposition\rw{pr:lmch} it is \Cch.
Hence it has a special limit $\xbAs$ described in the
proof of Proposition\rw{pr:chlm}. If $\scA$ has another limit
$\xbAp$ with \chmp s $\xbAi\xrightarrow{\xbtfpi}\xbAp$ then by
Theorem\rw{pr:spmp} there is a \chmp\
$\xbAs\xrightarrow{\ybg}\xbAp$ with commutative
triangles\rx{ae1.10f1}. The inequality\rx{eq:in2} says
\ylee{ae1.10c3}
\yordch{\ybg}
\geq \min\big(\zordch{\xbtfi}+1,\zordch{\ybgi}\big).
\yeee
Since both cones in the \rhs become homologically infinitely small
at $i\rightarrow +\infty$, the cone $\Cnv{\ybg}$ is also
homologically infinitely small. Then Proposition\rw{pr:ismc} says
that $\Cnv{\ybg}$ is contractible and as a result
$\xbAp\hteqv\xbAs$.\myqed

We end this section with a theorem which follows easily from
Definition\rw{df:sqlm}.
\begin{theorem}
\label{th:zl}
If a \chsq\ $\scA$ satisfies the property $\lmii\yordh{\xbAi} =
+\infty$ then its limit is contractible: $\lim\scA = 0$.
\end{theorem}

%\section{The properties of the universal complex associated with a
%\cbr}
\section{A \chsq\ of \cbr\ categorification complexes} % universal complexes associated with \cbr s}
\label{s:cbr}
\subsection{A special categorification complex of a \ngbr}

Let $\xsgi$ denote an elementary negative $n$-strand braid:
\ylee{ae2.1}
%\underbrace{
\xsgi=\xygraph{
!{0;/r1.5pc/:}
[r(0.25)u(0.5)]
!{\xcapv@(0)}
[u(0.5)r(1)]
*{\cdots}
[r(01)u(0.5)]
!{\xcapv@(0)}
[r(0.5)u(1)]
%!{\vcap}
!{\vcross}
[r(1.5)u(1)]
!{\xcapv@(0)}
[u(0.5)r(1)]
*{\cdots}
[r(01)u(0.5)]
!{\xcapv@(0)}
[d(0.5)l(3.5)]
*{\scriptstyle{i}}
[r(1)]
*{\scriptstyle{i+1}}
[l(3.5)]
*{\scriptstyle{1}}
[r(6)]
*{\scriptstyle{n}}
}
\yeee

\begin{theorem}
\label{th:prop}
If an $n$-strand braid $\brb$ can be presented as a product of elementary
negative braids: $\brb = \xsgiv{k}\cdots\xsgiv{2}\xsgiv{1}$, then
its categorification complex has a special presentation $\cbrbs$:
\xlee{aea2.1}
\cbrbas =
\Big(\ldots\rightarrow\xCt\rightarrow\xCo\rightarrow\cidbrn\Big)
%\tgrshv{-\vthf}{\vthf}{-\vthf}^k,
\xeee
such that the complex
\xlee{aea2.2}
\xbC = (\ldots\rightarrow\xCt\rightarrow\xCo)\tgrsshv{-1}{-1}
\xeee
is
$\odct$ and $\otbl$.

%$\cbrbs$
%such that $\cbrbs\tgrshv{\vthf}{-\vthf}{\vthf}^k$ is \oldn\ and
%\otbl.

%%
%\xlee{ae2.2}
%\cbrbs =
%%\tgrshv{\vthf}{\vthf}{-\vthf}^k=
%\lrbc{\cdots\rightarrow\xCi\rightarrow\cdots\rightarrow\xCv{1}\rightarrow\cIdbn}
%\tgrshv{-\vthf}{-\vthf}{\vthf}^k,
%\xeee
%%
%such that the multiplicities $\cjilam$ in the formula\rx{ae1.8a} for
%$\xCi$ satisfy the following properties:
%%
%\begin{enumerate}
%\item $\xcv{j,i,\xIdbn}=0$ for $i\geq 1$;
%\item $\cjilam=0$ for $j<i$ and for $j>2i$.
%\end{enumerate}
%%
\end{theorem}

More abstractly, the theorem says that there exists a
\odct\ and \otbl\ complex $\xbC$ and a \chmp\
$\xbC\rightarrow\cidbrn$ such that $\cbrba \hteqv
\CnBv{\xbC\qsho\rightarrow\cidbrn}$.

\proof
Let $\xlam$ be a \TL\ \ttngnn. Fix $i$, $1\leq i\leq n-1$. If the
composition $\gcapni\;\xlam$ does not contain a disjoint circle,
then, in accordance with \ex{ae1.7},
we define the special categorification complex of $\xsgi\xtau$ as
%%
%\xlee{ae2.3}
%\spcc{\symbcat{\xsgi\xlam} } =
%\Big(\symbcat{\xUni\,\xlam}\tgrshv{\vthf}{\vthf}{-\vthf}
%\rightarrow \dlam\tgrshv{-\vthf}{-\vthf}{\vthf} \Big).
%\xeee
%%
%
\xlee{ae2.3}
%\spsymbcat{\xsgi\xlam}\sht
\symcatps{\xsgi\,\xlam}  =
\Big(\symbcat{\xUni\,\xlam}\tgrshv{1}{1}{1}
\rightarrow \dlam \Big)
%\tgrshv{-\vthf}{-\vthf}{\vthf}.
\xeee
%
%Here we factored out the degree shift associated
%with the number of crossings in the braid expression.
If $\gcapni\,\xlam$ contains a disjoint circle, then $\xlam$ must
have the form $\gcupni\xlamp$. Hence
$\xsgi\xlam=\xsgi\,\gcupni\,\xlamp$. The tangle $\xsgi\,\gcupni$
is the same as $\gcupni$ with a positive framing twist, so
according to \ex{ae1.8},
$\bsymcat{\xsgi\,\gcupni} = \ccupni \tgrshv{\vthh}{\vthf}{\vthf}$.
Hence in this case we define the special categorification complex
of $\xsgi\xlam$ simply as shifted $\dlam$:
\xlee{ae2.4}
%\spsymbcat
\symcatps{\xsgi\xlam} = %\sht=
%\Big(
\dlam
\tgrshv{2}{1}{0}.
%\Big).
%\tgrshv{-\vthf}{-\vthf}{\vthf}
\xeee
%
%and we have again factored out the crossing related degree shift.

Now we define a recursive algorithm for constructing the complex
$\cbrbas$. For $\brb = \gidbrn$ we define $\cbrbas = \cidbrn$. Let
$\brb = \xsgiv{k}\cdots\xsgiv{1}$ and
suppose that we have defined its special complex $\cbrbas$. We
define the special categorification complex of a
braid $\brbp=\xsgikpo\brb$ by applying the rules\rx{ae2.3}
and\rx{ae2.4} to all constituent tangles $\xlam$ in the complex
$\cbrbs$ (see the formula\rx{ae1.8a}).

We prove the properties of $\cbrbas$ by induction over $k$.
If $k=0$ then $\brb = \gidbrn$ and the properties of $\cbrbas$ are
obvious.

Suppose that the special categorification complex
$\cbrbas$ of a braid $\brb = \xsgiv{k}\cdots\xsgiv{1}$
has the form\rx{aea2.1} and its tail\rx{aea2.2} is \odct\ and
\otbl.
Consider
%the special categorification complex of
a longer braid
$\brbp=\xsgikpo\brb$. The object $\cidbrn$ may appear in $\cbrbpas$ if
and only if $\xlam=\xIdbn$ and the extra crossing $\xsgikpo$ is
\nsplcd\ in \ex{ae2.3}, hence
%$\cbrbps$ is \oldn.
$\cbrbpas$ has the form\rx{aea2.1} and its tail\rx{aea2.2} is
\odct.

If the negative crossing $\xsgikpo$ is composed with the head
$\cidbrn$ of the complex\rx{aea2.1}, then the formula\rx{ae2.3}
applies and the tangle $\xUv{n}{i_k+1}$ appearing in the tail of
$\cbrbpas$ satisfies the property\rx{ae2.m1}.

%resulting \TL\ tangles satisfy the condition

If the crossing $\xsgikpo$ is composed with a
\TL\ tangle $\xlam$ from the $i$-th \qcmd\ $\xCi$ (see \ex{ae1.8a})
in the tail of the complex $\cbrbps$ with the $q$-degree shift $j$
satisfying the inequality $i-1 \leq j-1 \leq 2(i-1)$, then the
shifted objects in the \rhs of \eex{ae2.3} and\rx{ae2.4}
also satisfy this inequality.\myqed

The picture\rx{ae1.10p} presents a \cbr\ as a product of negative
crossings, hence
\begin{corollary}
%A crossing-shifted special categorification complex of a \cbr\
%has the form
A \cbr\ has a special categorification complex %$\btcylcsnmn$ such that
\xlee{ae2.5}
%\btcylcsnmns
\cbrons = \CnBv{\xbCon\qsho\rightarrow\cidbrn},
\xeee
where the complex
$\xbCon$ is \odct\ and \otbl.
%is \oldn\ and \otbl.
\end{corollary}
%\begin{remark}
%\label{rm:1}
%The formula\rx{ae2.5} indicates that the complex $\cbrmns$
%%$\btcylcsnmns$
%itself is \otbl.
%\end{remark}

\subsection{Special morphisms between \cbr\ complexes}
\label{ss:brchsq}

Relation\rx{ae2.5}  indicates that there is a \dstt\
%
%\ylee{ae2.6x}
%\xymatrix{
%\xbCon\qsho \ar[r] &
%{}\cIdbn \ar[r]^-{\mrfo} &
%%\btcylcsnons
%{}\btcylcnns \ar[r] &
%\xbCon\tgrsshv{1}{1}
%}
%\yeee
%
%
\ylee{ae2.6}
\xbCon\qsho \longrightarrow
\cidbrn \xratv{\mrfo}
\cbrons \longrightarrow
\xbCon\tgrsshv{1}{1}
\yeee
and
\xlee{ae2.6a}
\Cnv{\mrfo} \hteqv \xbCon\tgrsshv{1}{1}.
\xeee
%
%If we `tangle-compose'
%the morphism $\mrfo$ with the identity morphism
%$\btcylcnmns\xrightarrow{=}\btcylcnmns$
%
Composing both sides of the morphism $\mrfo$ with
the \cbr\ complex
%$\btcylcnmns$,
$\cbrmns$,
we get a morphism
%%
%\ylee{ae2.7x}
%\xymatrix{
%\btcylcnmns \ar[r]^-{\mrfm} & \btcylcnmpons
%}
%\yeee
%%
%
\ylee{ae2.7}
\cbrmns \xratv{\mrfm}\cbrmons
\yeee
and
\xlee{ae2.8}
\Cnv{\mrfm} \hteqv \Cnv{\mrfo}\,\cbrmns.
%\hteqv. \xbCon\,\btcylcnmns\tgrsshv{1}{1}.
\xeee
\begin{theorem}
\label{th:2.1}
The cone\rx{ae2.8} can be presented by a shifted complex
\ylee{ae2.9}
\Cnv{\mrfm} \hteqv \xbCmn\tgrsshv{2mn+1}{2m(n-1)+1},
\yeee
%
%$\xbCmn\tgrsshv{}{}$
such that $\xbCmn$ itself is \odct\ and \otbl.
\end{theorem}

The proof is based on a simple geometric lemma:
\begin{lemma}
\label{l:1}
For $n\geq 2$, the following two compositions of framed tangles are isotopic:
%%
%\xlee{ae2.bx}
%\xcapni (\btcyln)^n = (\btcylv{n-2})^{n-2}\;\xcapnit
%\xeee
%%
%
\xlee{ae2.b}
\gcapni \;\gbron = \gbronmt\;\gcapnit
\xeee
where $\gcapnit$ is the tangle $\gcapni$ with double framing twist
on the cap:
\ylee{ae2.10}
\gcapnik=
\xygraph{
!{0;/r1.5pc/:}
[r(0.25)u(0.5)]
!{\xcapv@(0)}
[u(0.5)r(1)]
*{\cdots}
[r(01)u(0.5)]
!{\xcapv@(0)}
[r(0.5)]
!{\vcap}
[r(1.5)u(1)]
!{\xcapv@(0)}
[u(0.5)r(1)]
*{\cdots}
[r(01)u(0.5)]
!{\xcapv@(0)}
[d(0.5)l(3.5)]
*{\scriptstyle{i}}
[r(1)]
*{\scriptstyle{i+1}}
[l(3.5)]
*{\scriptstyle{1}}
[r(6)]
*{\scriptstyle{n}}
[l(3)u(1)]
*{\symfr}
[u(0.5)]
*{\scriptstyle{k}}
}
\yeee
\end{lemma}
\proof
This lemma is geometrically obvious: a cap on a pair of adjacent strands slides down
through the \cbr\ to the
bottom.\myqed

An immediate corollary of \ex{ae2.b} and of the framing change
formula\rx{ae1.8} is the following relation:
%%
%\xlee{ae2.11x}
%\ccapni \btcylcnmns\; \hteqv
%\Big(\btcylcnmnmts\Big)\;\ccapni
%\tgrsshv{n}{n-1}^{2m}.
%\xeee
%%
%
\xlee{ae2.11}
\bsymcats{\gcapni \;\gbrmn} \hteqv \bsymcats{\gbrmnmt\;\gcapni}
\tgrsshv{n}{n-1}^{2m}.
\xeee

%%%%%%%%%%%%%%%%%%%%%%%%%%%%%%

In order to prove Theorem\rw{th:2.1}, we need three simple
propositions.
A \emph{\aptg} $\gcapnI$, where $\stI = \{i_1,\ldots,i_d\}$,
%of \apdg\ $d$
is a $(n,n-2d)$-tangle which
can be presented as a product of $d$ tangles of the form
$\gcapv{m}{i}{0.75}$:
%%
%\ylee{aes2.1ax}
%\maptgnI = \xcapvv{n-2d+2}{i_d}\cdots
%\xcapvv{n-2}{i_2}\,\xcapvv{n}{i_1}.
%\yeee
%%
%
\ylee{aes2.1a}
\gcapnI =
\gcapv{n-2d+2}{i_d}{2}\;\cdots
\gcapv{n-2}{i_2}{1.5}
\;
\gcapv{n}{i_1}{0.75}.
\yeee
A \emph{\uptg} $\gcupnI$ is defined similarly:
%%
%\ylee{aes2.2a}
%\muptgnI = \xcupvv{n-2d+2}{i_d}\cdots
%\xcupvv{n-2}{i_2}\,\xcupvv{n}{i_1}.
%\yeee
%%
%
\ylee{aes2.2a}
\gcupnI =
\gcupv{n}{i_1}{-0.75}
\;
\gcupv{n-2}{i_2}{-1.25}
\cdots
\;
\gcupv{n-2d+2}{i_d}{-2.25}
.
\yeee
The first proposition is obvious:
\begin{proposition}
\label{pr:3}
Every \TL\ \ttngnn\ $\xlam$ has a presentation
\xlee{aes2.3a}
\xlam = \gcupnIp\,\gidbrnmtd\,\gcapnI,\qquad
\nI=\nIp.
%=\cpdlam.
%\nIp=\cpd.
\xeee
\end{proposition}
The number $\cpdlam=\nI=\nIp$ is determined by the tangle $\xlam$
and we call it  the \apdg\ (or \updg) of $\xlam$.
%: $\dgap\xlam = d$.

The second proposition is also obvious:
\begin{proposition}
\label{pr:1}
If at least one of two complexes $\xbCo$ and $\xbCt$ in $\dTLn$ is
\odct\ then their composition $\xbCo\xbCt$ is \odct.
\end{proposition}
Note, however, that even if
both complexes are \otbl,  then their composition is not necessarily
\otbl. The problem is that although the homological degree is
additive with respect to the composition of tangles, the \qdgr\ is
not: if  the composition of two \TL\ tangles creates a disjoint
circle then the rule\rx{ae1.01} creates the shifts of \qdgr\ which
violate additivity. Still, if the upper tangle in the composition has no caps or the
lower tangle has no cups then no circles are created and the
\otblc\ is maintained:
\begin{proposition}
\label{pr:2}
If a complex $\xbC$ in $\dTLn$ is \otbl, then the complexes
%$\dcupnti\xbC$
$\ccupnI\;\xbC$
and
$\xbC\;\ccapnI$
%$\xbC\;\ccapnti$
are also \otbl.
\end{proposition}

%%%%%%%%%%%%%%%%%%%%%%%%%%%%%%%%%

\pr{Theorem}{th:2.1}
In order to construct the \odct\ and \otbl\ complex $\xbCmn$, we
use the presentation
\xlee{ae2.12}
\Cnv{\mrfm} \hteqv
\xbCon\,\cbrmns\tgrsshv{1}{1},
\xeee
which follows from \eex{ae2.8} and\rx{ae2.6a}.
We construct $\xbCmn$ by
simplifying the complexes
%$\dlam\btcylcnmn$
$\bsymcats{\xlam\;\gbrmn}$
for \TL\ \ttngnn s $\xlam$
appearing in the \qcmds\ of $\xbCon$, with the help of the
relation\rx{ae2.11} and then using the special complex
$\cbrmnmtps$
%$\btcylcsnmnmt$
to represent $\cbrmnmts$.

%%%%%%%%%5

%%%%%%%%%%%%%%%%

Let
%$\dlam\qshj$
$\dlam\tgrsshji$
be an object appearing in the $i$-th \qcmd\ of
$\xbCon$ with a non-zero multiplicity (we made its homological degree explicit by
including $i$ in the shift).
%The complex $\xbCon$ is \otbl,
%hence $i\leq j\leq 2i$. The complex $\xbCon$ is also \odct,
%hence $\cpdlam\geq 1$.
%
We apply \ex{ae2.11} consequently to every cap $\gcapni$ appearing
in the \aptg\ $\gcapnI$ in the presentation\rx{aes2.3a} of
$\xlam$:
\begin{multline}
\label{ae2.13}
\dlam\tgrsshji\cbrmns
\\
\hteqv
\Bigg(
\ccupnIp\;
%\btcylcnmnmtdl
%\btcylcsnmnmtdl\sht
\bsymcatps{\gbrv{m}{n-2\cpdlam}{2.5}}\;\ccapnI
\tgrsshv{\alm}{\blm}^{2m} \tgrsshji\Bigg)
\tgrsshv{n}{n-1}^{2m},
\end{multline}
where
\xlee{ae2.14}
\alm = \sum_{k=1}^{\cpdlam-1}(n-2k)
,\qquad
\blm = \sum_{k=1}^{\cpdlam-1}(n-2k-1).
\xeee

The object $\dlam$ comes from a \odct\ complex $\xbCon$,
hence $\cpdlam>0$ and
%the tangles $\muptgnIp$ and $\maptgnI$ are \odct, hence
the complex
in big brackets in the \rhs of \ex{ae2.13} is \odct\ in view of
Proposition\rw{pr:1}. Proposition\rw{pr:2} implies that the complex
$\ccupnIp\;
\bsymcatps{\gbrv{m}{n-2\cpdlam}{2.5}}\;\ccapnI$
%$\cmuptgnIp \btcylcsnmnmtdl\sht \cmaptgnI$
is also \otbl. Since $\dlam$
comes from a \otbl\ complex $\xbCon$, the numbers $i$ and $j$
satisfy inequalities $i\geq 0$ and $i\leq j\leq 2i$. It is easy to check that the numbers
$\alm$ and $\blm$ of \ex{ae2.14} satisfy the same inequalities:
$\blm\geq 0$, $\blm\leq \alm \leq 2\blm$, hence the complex in big
brackets in the \rhs of \ex{ae2.13} is also \otbl. The
complex
%$\xbCon\,\btcylcnmns$
$\xbCon\,\cbrmns$
in the \rhs of \ex{ae2.12} is
composed of complexes\rx{ae2.13}, so this proves
Theorem\rw{th:2.1}.\myqed

\section{Categorified \JWp}
\label{s:prfs}

Consider the \chsq\rx{ae1.10c}. Theorem\rw{th:2.1} implies that
%$\Cnv{\mrfm} \hteqv \Ohp\big(2m(n-1) + 1\big)$,
$\yordhr{\Cnv{\mrfm}} \geq 2m(n-1)+1$,
hence $\xctBn$ is
\Cch\ and it has a unique limit $\lim\xctBn =\ctjwpn \in\dTLnp$.

Now we prove Theorems\rw{th:enum} and Theorem\rw{th:cnpr} which
describe the properties of $\ctjwpn$.

\pr{Theorem}{th:cnpr}
%Apply Theorem\rw{th:rshfl} to
Consider the \chsq\rx{ae1.10c} truncated from below:
\ylee{eq:np1}
\xctBmn =
\Big(
\cbrmns \xraov{\mrfm}
\cbrmons \xrightarrow{\mrfmo}\cdots\Big)\longrightarrow\ctjwpn.
\yeee
According to Theorem\rw{th:rshfl}, the limit $\ctjwpn$ can be
presented as a cone\rx{ae2.m4}, where $\wbCmnp = \wbCmn\spshmn$
and $\wbCmn$ is an infinite
\mtcn:
\begin{multline}
%\ylee{ae3.1}
\nonumber
\wbCmn =\cdots\rightarrow\Cnv{\xbCvn{m+k}\tgrsshv{2kn}{2k(n-1)-1}
\rightarrow
\cdots
\\
\cdots
\rightarrow
\Cnv{
\xbCvn{m+1}\tgrsshv{2n}{2n-3}\rightarrow\xbCmn}
}
\end{multline}
with \odct\ and \otbl\ complexes $\xbCmn$ introduced in
Theorem\rw{th:2.1}. Hence the complex $\wbCmn$ itself is \odct\ and
\otbl.\myqed

%Let us apply Theorem\rw{th:rshfl} to the \chsq\rx{ae1.10c}
%truncated from below:
%
%\pr{Theorem}{th:cnpr}
%According to Theorem\rw{th:rshfl}, the limit $\ctjwpn$ can be
%presented as a cone $\ctjwpn\hteqv\CnBv{\zbCn\qsho\rightarrow
%\cidbrn}$ (we removed tildes from the notations of \ex{ae1.10h1}),
%where $\zbCn$ is an infinite \mtcn:
%%
%\ylee{ae3.1x}
%\zbCn =\cdots\rightarrow\Cnv{\xbCmn\tgrsshv{2mn}{2m(n-1)-1}
%\rightarrow
%\cdots\rightarrow
%\Cnv{
%\xbCvn{1}\tgrsshv{2n}{2n-3}\rightarrow\xbCvn{0}}
%}
%\yeee
%%
%with \odct\ and \otbl\ complexes $\xbCmn$ introduced in
%Theorem\rw{th:2.1}. Hence the complex $\zbCn$ itself is \odct\ and
%\otbl.\myqed

\pr{part 1 of Theorem}{th:enum}
The tangle composition with $\ccapni$ is a `continuous'
functor, that is, it can be applied to both sides of
$\lim\xctBn = \ctjwpn$, hence $\ccapni\;\ctjwpn = \lim\;
\ccapni\;\xctBn$. According to \ex{ae2.11},
\begin{equation}
\nonumber
%\label{ae3.2}
\begin{split}
\ccapni\;\xctBn & = \Big(\ccapni\cidbrn\rightarrow \cdots\rightarrow
\ccapni\,\cbrmns\rightarrow\cdots\Big)
\\
& = \Big( \ccapni\rightarrow\cdots\rightarrow
\cbrmnmts\;\ccapni
\tgrsshv{n}{n-1}^{2m}
\rightarrow\cdots
\Big).
\end{split}
\end{equation}
Since
\ylee{ae3.3}
\yordhb{\cbrmnmts\;\ccapni
\tgrsshv{n}{n-1}^{2m}} = 2m(n-1)\xrightarrow[m\rightarrow +\infty]{}+ \infty,
\yeee
according
to Theorem\rw{th:zl}, $\lim\ccapni\,\xctBn=0$, hence
$\ccapni\,\ctjwpn$ is contractible.\myqed

\begin{remark}
The contractibility of $\ctjwpn\;\ccupni$ is proved similarly.
\end{remark}

\begin{corollary}
\label{cr:odct}
If $\xbC$ is a \odct\ complex in $\dTLnp$, then $\xbC\,\ctjwpn$ is
contractible.
\end{corollary}

\pr{part 2 of Theorem}{th:enum}
According to
%Theorem\rw{th:cnpr},
$\ctjwpn\hteqv\CnBv{\xbCn\qsho\rightarrow\cidbrn}$, where the
%complex $\xbCn$ is \odct. Then
\ex{ae2.m5},
\begin{multline}
\nonumber
%\ylee{ae3.4}
\ctjwpn\,\ctjwpn \hteqv \CnBv{\wbCzn\qsho\longrightarrow\cidbrn}\,\ctjwpn
\\
\hteqv\CnBv{\wbCzn\,\ctjwpn\qsho\longrightarrow\cidbrn\,\ctjwpn}
\hteqv\ctjwpn,
%\yeee
\end{multline}
where we used the fact that $\wbCzn$ is \odct\ and Corollary\rw{cr:odct} in order to establish the last
homotopy equivalence.\myqed

\pr{Theorem}{th:alg}
The complexes $\ctjwpn$, $\wbCmn$ and $\cbrmns$ in \ex{ae2.m4} are
\otbl, hence they are \qpb\ and their $\Kz$ images are
well-defined. Applying $\Kz$ to this equation and taking into account \ex{eq:catKz} and
the definition\rx{ae1.10b1}, we find
\ylee{ae3.5}
\jwpn = q^{\vthf mn(n-1)}\abrmn - q^{2mn+1}\Kz(\wbCmn).
\yeee
The complex $\wbCmn$ is \otbl, so $\yordq{\Kz(\wbCmn)}\geq 0$ and
by Definition\rw{df:qlm} there is a limit\rw{ae1.9}.\myqed

\section{The other projector}
A dual of an \ttngmn\ $\xtau$ is the \ttngnm\ tangle
$\xtaud$ which is its mirror image. Duality extends to an
isomorphism $\cTL \xrightarrow{\dsym} \cTLop$ combined with the
isomorphism of the ground ring $\Zqqi\xrightarrow{\dsym}\Zqqi$, such that
$\dsymv{q} = q^{-1}$. Furthermore, duality established the
isomorphism $\cTLpinf\xrightarrow{\dsym}\cTLminfop$, where
$\cTLminf$ is the analog of $\cTLpinf$ constructed over the ring
$\Zsqiq$ of Laurent series in $q^{-1}$.

Since the relations\rw{ae1.4} and\rx{ae1.4a} are dual to each other,
while the idempotency condition $\jwpn\jwpn=\jwpn$ is duality
invariant, the uniqueness of the \JWp\ implies that it is duality
invariant: $\dsymv{\jwpn} = \jwpn$. Hence the corresponding
idempotents $\jwpnp\in\cTLpinf$ and $\jwpnm\in\cTLminf$ are also
dual to each other: $\jwpnm = \dsymv{(\jwpnp)}$. Taking the dual
of \ex{ae1.9} we find that $\jwpnm$ is the limit of \cbr s with
high positive (that is, \cclckw) twist:
\xlee{ae1.9b}
\lim_{m\rightarrow+\infty} q^{-\vthf mn(n-1)}\aobrmn = \jwpnm,
\xeee
because $\dsymv{\Big(\gbrmn\Big)} = \gobrmn$.

%The dual \JWp\
%$\jwpnm = \dsymv{(\jwpnp)}\in\cTLminfop$ is also an idempotent in
%$\cTLminf$, because the relation $\jwpnm\jwpnm=\jwpnm$ remains
%intact after the reversal of multiplication order.

Duality extends to an
equivalence functor $\dTL\xrightarrow{\dsym}\dTLop$, where
$\dTLop$ is the same category as $\dTL$, except that the
composition of tangles is performed in reversed order. The functor
$\dsym$ also switches the signs of all three gradings of $\dTL$.
Applying the duality functor to the construction of $\ctjwpn$ we
find that there exists a \chsq
\begin{multline}
\label{ae1.10f}
\xctBnd =
\Big(
\cidbrn
\rxratv{\dmrfz}
\cobrons \rxratv{\dmrfo}
\cdots
\\
\cdots
\xleftarrow{\dmrfmmo}
\cobrmns \rxraov{\dmrfm}
\cobrmons \xleftarrow{\dmrfmo}\cdots\Big),
\end{multline}
where $-\xspsh$ denotes the grading shift which is opposite
to\rx{ae1.10b1}.
The system\rx{ae1.10f} is
dual to the system\rx{ae1.10c}
and it has an inductive limit $\ilm \xctBnd =\ctjwmn $, which satisfies
projector properties
\ylee{ae1.10f1}
\ctjwmn \ctjwmn \hteqv \ctjwmn,\qquad
\ccapni \;\ctjwmn \hteqv \ctjwmn\; \ccupni\hteqv 0
\yeee
and has a presentation
\ylee{ae1.10f2}
\ctjwpn \hteqv \CnBv{ \dsymv{\wbCmn}\ospshmn\qshmo
\xrahv{\dsymv{\chdlbtfm}} \cobrmns},
\yeee
%
%where $\wbCmnp = \wbCmn\spshmn$,
where the complex $\wbCmn$ is
\odct\ and \otbl. In particular, at $m=0$ we get the dual of
presentation\rx{ae2.m5}:
\ylee{ae1.10f3}
\ctjwmn \hteqv \CnBv{ \dsymv{\wbCzn}\qshmo
\xrahv{\dsymv{\chdlbtfz}} \cidbrn},
\yeee
where the complex $\wbCzn$ is \odct\ and \otbl.

%$$ \testgr$$
%$$\testgr1 $$
%$$\testgr2 $$
%$$\gbrmnmt\qquad\gbrmn\qquad\tstgro\qquad\tstgrt$$

%$$
%\begindc{\commdiag}[50]
%\obj(0,1){$A$}
%\obj(1,2){\box{$\cidbrn$}}
%\enddc
%$$

%\subsection{Proof of Theorem\rw{th:enum}}

%$\accentset{\bullet}{\rTL}$

\end{document}

%%%%%%%%%%%%%%%%%%%%%%%%%%%%%%%%%%%%%%

\bib{BFN}{article}
{
author={Ben-Zvi, David}
author={Francis, John}
author={Nadler, David}
title={Integral transforms and Drinfeld centers in derived
algebraic geometry}
eprint={arXiv:0805.0157}
}

%\bibitem{BFN} D.~Ben-Zvi, J.~Francis, D.~Nadler, ``Integral
%transforms and Drinfeld centers in derived algebraic geometry,''
%arXiv:0805.0157 [math.AG]

\bib{JB}{article}
{
author={Block, Jonathan}
title={Duality and equivalence of module
categories in noncommutative geometry I}
eprint={arXiv:math/0509284}
}

%\bibitem{Block} J.~Block, ``Duality and equivalence of module
%categories in noncommutative geometry I,'' arXiv:math/0509284.

\bib{BvB}{article}
{
author={Bondal, Alexei}
author={van den Bergh, Michel}
title={Generators and representability of functors in commutative and
noncommutative geometry}
journal={Moscow Mathematical Journal}
volume={3}
year={2003}
pages={1-36}
eprint={arXiv:math.AG/0204218}
}

%\bibitem{BvB} A.~Bondal and M.~ van den Bergh,
%``Generators and representability of functors in commutative and
%noncommutative geometry,'' Mosc.\ Math.\ J {\bf 3}, 1 (2003)
%[arXiv:math.AG/0204218].

\bib{BondalRosly}{misc}
{
author={Bondal, Alexei}
author={Rosly, Alexei}
note={in preparation}
}

\bib{Kapranov}{article}
{
AUTHOR = {Kapranov, Mikhail},
     TITLE = {Rozansky-{W}itten invariants via {A}tiyah classes},
   %JOURNAL = {Compositio Math.},
  %FJOURNAL = {Compositio Mathematica},
    JOURNAL = {Compositio Mathematica},
    VOLUME = {115},
      YEAR = {1999},
    NUMBER = {1},
     PAGES = {71--113},
 }

\bib{KLi}{article}
{
author={Kapustin, Anton}
author={Li, Yi}
title={D-branes in Landau-Ginzburg models and algebraic geometry}
journal={Journal of High Energy Physics}
volume={0312}
year={2003}
pages={005-049}
eprint={arXiv:hep-th/0210296}
}

%\bibitem{KLi} A.~Kapustin and Y.~Li, ``D-branes in Landau-Ginzburg models and algebraic geometry,''
%JHEP {\bf 0312}, 005 (2003) [arXiv:hep-th/0210296].

\bib{KaLi}{article}
{
author={Kapustin, Anton}
author={Li, Yi}
title={Topological correlators in Landau-Ginzburg models with boundaries}
journal={Advances in Theoretical and Mathematical Physics}
volume={7}
date={2004}
pages={727}
eprint={arXiv:hep-th/0305136}
}

%\bibitem{KaLi} A.~Kapustin and Y.~Li, ``Topological correlators in Landau-Ginzburg models with
%boundaries,'' Adv.\ Theor.\ Math.\ Phys.\  {\bf 7}, 727 (2004)
%[arXiv:hep-th/0305136].

\bib{KRS1}{article}
{
author={Kapustin, Anton}
author={Rozansky, Lev}
author={Saulina, Natalia}
title={Three-dimensional topological field theory and symplectic
algebraic geometry I}
journal={Nuclear Physics B}
volume={816}
date={2009}
pages={295-355}
eprint={arXiv:0810.5415}

}

\bib{Ke1}{article}
{
author={Keller, Bernard}
title={Introduction to A-infinity algebras and modules}
journal={Homoogy, Homotopy and Applications}
volume={3}
year={2001}
pages={1-35}
eprint={arXiv:math/9910179}
}

%\bibitem{Ke1} B.~Keller, ``Introduction to A-infinity algebras and modules,''
%Homology, Homotopy and Applications\ {\bf 3}, 1-35 (2001)
%[arXiv:math/9910179].

%\bib{Ke}{article}
%{
%}

%\bib{Rob}{article}
%{
%author={Roberts, Justin}
%}

\bib{catsheaf}{article}
{
author={Moerdijk, Ieke}
title={Introduction to the language of stacks and sheaves}
eprint={arXiv:math.AT/0212266}
}

%\bibitem{catsheaf} I.~Moerdijk, ``Introduction to the language of
%stacks and gerbes,'' arXiv:math.AT/0212266.

\bib{Orlov:MF}{article}
{
author={Orlov, Dmitri}
title={Triangulated categories of singularities and equivalences
between Landau-Ginzburg models}
journal={Matematicheskii Sbornik}
volume={197}
year={2006}
pages={117-132}
translation={
journal={Sbornik Mathematics}
volume={197}
year={2006}
pages={1827-1840}
}
eprint={arXiv:math.AG/0503630}
}

%\bibitem{Orlov:MF} D.~Orlov, ``Triangulated categories of singularities and equivalences
%between Landau-Ginzburg models,'' arXiv:math.AG/0503630.

\bib{Orlov:priv}{misc}
{author={Orlov, Dmitri}
note={private communications}
}

%\bibitem{Orlov:priv} D.~Orlov, private communication.

\bib{RobWil}{article}
{
author = {Justin Roberts and Simon Willerton},
  title = {On the Rozansky-Witten weight systems},
  eprint={arXiv:math/0602653}
}

\bib{RW}{article}
{
author={Rozansky, Lev}
author={Witten, Edward}
title={Hyper-K\"{a}hler geometry and invariants of three-manifolds}
journal={Selecta Mathematica}
volume={3}
year={1997}
pages={401-458}
eprint={arXiv:hep-th/9612216}
}

%\bibitem{RW} L.~Rozansky and E.~Witten, ``Hyper-Kaehler geometry and invariants of three-manifolds,''
%Selecta Math.\  {\bf 3}, 401 (1997) [arXiv:hep-th/9612216].

\bib{Toen}{article}
{
author={To\"{e}n, Bertrand}
title={The homotopy theory of dg-categories and derived Morita
theory}
journal={Inventiones Mathematicae}
volume={167}
year={2007}
pages={615-667}
eprint={arXiv:math.AG/0408337}
}

%\bibitem{Toen} B.~Toen, ``The homotopy theory of dg-categories and derived Morita
%theory,'' Invent.\ Math.\ {\bf 167}, 615 (2007)
%[arXiv:math.AG/0408337].

\bib{ToVe}{article}
{
author={To\"{e}n, Bertrand}
author={Vezzosi, Gabrielle}
title={A note on Chern character,
loop spaces and derived algebraic geometry}
eprint={arXiv:0804.1274}
}

%\bibitem{ToVe} B.~Toen and C.~Vezzosi, ``A note on Chern character,
%loop spaces and derived algebraic geometry,'' arXiv:0804.1274
%[math.AG]

\end{biblist}
\end{bibdiv}
\end{document}

\section{Introduction}
\subsection{Sigma-models and categories}
\label{ss.intr}

Let $M$ be a real $\zN$-dimensional manifold and let
$\xcX=(\xX,\xs)$ be a pair in which $\xX$ is a real
manifold and $\xs$ is a geometric structure on $\xX$ such as a complex
structure or a symplectic structure.
A $\zN$-dimensional topological \sgmd\ (\tsgmd) with a
\emph{\xspt} (also known as the \xwsh\ or
the world-volume) $M$ and a
\xtsp\ $\xcX$
is a quantum field theory based on a
path integral over the infinite-dimensional space of maps $\MpsMX$. The measure on the
space of maps is determined by the structure $\xs$.

We are
interested in 2-dimensional and 3-dimensional \tsgmd s.
%, that is, $\zN=2,3$.

Path-integral based arguments
suggest that if manifolds $\xcX$ with a certain type of structure serve
as \xtsp s for $\zN$-dimensional \tsgmd\ then they form a
\zNc\ $\cC$ with special features. Let us briefly recall these features and
illustrate them by two well-known examples.

The category $\cC$ has a symmetric monoidal structure related to
the
cartesian product of manifolds:
\xlee{bag1.1}
\cC\times\cC\longrightarrow\cC,\qquad
(\xXo,\xso)\times(\xXt,\xst) = (\xXo\times\xXt,\xso\times\xst),
\xeee
where $\xso\times\xst$ is the natural structure on $\xXo\times\xXt$.
The monoidal structure has a unit element $\xcXopt=(\xXopt,\xsopt)$, where
$\xXopt$ is the manifold consisting of a single point and $\xsopt$
is the corresponding trivial structure. This element has the
property
\wlee{bag1.2}
\xcXopt\times\xcX=\xcX.
\weee
For a structured manifold $\xcX$ we define a \zNcmo\ of morphisms
\ylee{bag1.2a}
\cCX \edfn \HomC(\xcXopt,\xcX).
\yeee
This category is known in quantum field theory as the category of
boundary conditions of the \tsgmd\ related to $\xcX$.

The \zNc\ $\cC$ has a contravariant
duality functor
\xlee{bag1.3}
\cC\xmapta{\hve}\cC,\qquad
\zdlv{(\xX,\xs)} = (\xX,\zdlv{\xs} ),
\xeee
such that there is a canonical equivalence between \zNcmos\ of
morphisms:
\xlee{bag1.4}
\HomC(\xcXo,\xcXt) = \HomC(\xcXopt,\zdlv{\xcXo}\times\xcXt)
=\cCov{\zdlv{\xcXo}\times\xcXt}.
\xeee
This equivalence implies that an object
$\xOot\in\HomC(\xcXo,\xcXt)$ determines a functor between \zNcmos
\xlee{bag1.5}
\xPhv{\xOot}\colon\cCXo \longrightarrow\cCXt,
\xeee
which represents a composition of morphisms within $\cC$.
Moreover, a composition of morphisms of $\cC$ corresponds to the composition
of functors\rx{bag1.5}, so the structure of the \zNc\ $\cC$ is
determined by the boundary condition categories $\cCX$ and the
functors\rx{bag1.5}.

Recall two examples of this general construction for
$\zN=2$, that is, when $\cC$ is a 2-category. The first example is related to the A-model.
The structure $\xs$ is a symplectic structure (that is, $\xs$ is a symplectic form on $\xX$), the category of
boundary conditions $\cCX$ is the \FkFc\ $\rFuk(\xcX)$,
its simplest objects being lagrangian submanifolds of $\xX$,
the action of the duality functor is $\zdlv{\xs}=-\xs$, and the
functor $\xPhv{\xOot}$ is the lagrangian correspondence functor
determined by a lagrangian submanifold
$\xcEot\subset\zdlv{\xcXo}\times\xcXt$.

The second example of the 2-category $\cC$ comes from the B-model:
$\xX$ is a \fCY\ manifold,
$\xs$ is its complex structure, the category of boundary conditions
$\cCX$ is the bounded derived category of coherent sheaves
$\xrDv{\xcX}$, its simplest objects being complexes of holomorphic
vector bundles on $\xX$, the duality functor acts trivially:
$\zdlv{\xs}=\xs$, and the functor $\xPhv{\xOot}$ is the
Fourier-Mukai transform corresponding to the object $\xOot$.

\subsection{The 3-category of holomorphic symplectic manifolds}
\label{ss.hcatintr}

For $d=3$ a natural class of \tsgmd\ comes from the Rozansky-Witten model \cx{RW}.
In\cx{KRS1} we studied this \tsgmd\ and its 2-category of boundary conditions  from the path integral viewpoint. In this paper we
attempt to present a mathematical description of the 3-category $\ctLLL$  formed by
these theories and formulate conjectures about it.

%mathematical description and conjectures about the 3-category
%$\ctLLL$ related to the family of these theories.

Objects of $\ctLLL$ are \hlsmm s $\xcX=\Xsom$,
where $\xX$ is a complex manifold and
$\som\in\Omega^{2,0}$ is a \hlsm\ form: it is non-degenerate at every point of $\xX$
and $d \som=0$. If the symplectic form is canonical,
we abbreviate the notation $\Xsom$ down to $\xX$.
The monoidal structure\rx{bag1.1} comes from the product of
manifolds and the sum of their symplectic structures:
$\Xsomo\otimes\Xsomt = \brb{\xXo\times\xXt,\yepiuo(\somo) +
\yepiut(\somt) }$, where $\yepio$ and $\yepit$ are the projections
of $\xXo\times\xXt$ onto $\xXo$ and $\xXt$.
The duality
functor $\hve$ acts on objects by switching the sign of the
symplectic form: $\zdlv{\Xsom} = \Xsomm$.

The main purpose of the paper is to investigate the 2-category
$\cCX$ associated to a \hlsmm\ $\xcX=\Xsom$. We denote it
as
$\ctLLXsom$. The definition of $\ctLLXsom$ for a
general \hlsmm\ is rather complicated and requires a construction
of a \mlcs\ of 2-categories on $\xX$. Therefore we devote much of
this paper to special $\Xsom$ and present an attempt at a general definition
only in Section\rw{s.mcl}.
%
%We will present this
%definition only in Section\rw{s.mcl}, and before that we will use two
%other approaches towards the description of $\ctLLXsom$.

%
%For a general \hlsmm\ $\Xsom$, the
%full structure of $\ctLLXsom$ is rather complicated and we can only
%conjecture its definition. Hence we use two different approaches towards
%the description of $\ctLLXsom$.

\subsection{Algebraic approach}

The first approach towards the description of
the 2-category $\ctLLXsom$
%is algebraic in spirit and it
is based on the fact
that when $\xX$ is a
cotangent bundle of a complex manifold $\xcA$, the category
$\ctLLTsU$ can be described in terms of the properties of $\xcA$.
This description is algebraic in nature and there is no reference
to $\TsU$, so by looking at the definitions one would not see
directly that symplectic structure is involved.

In Section\rw{s.sct3} we study a `toy'  2-category
$\xcMFdbx$,
$\bax=\lvar{\ax}{n}$, which after a minor modification should be equivalent to the
2-category $\ctLLv{(\Ts\IC^n)}$
associated with a \zsaff\ space, that is, with the cotangent bundle
$\Ts\IC^n$.
%Although the category $\xcMFdbx$ may look too simple or
%even artificial, it is at the heart of the \mlcs\ definition of
%$\ctLLXsom$ provided in Section\rw{s.mcl}.

Recall that for a polynomial $\xcW\in\ICbx$, an object of the category of matrix
factorizations $\xcMFbxW$ is a free finite rank \Ztgrdd\ $\ICbx$-module
$\xmM$ with a degree-1 endomorphism (called a curved differential) $\xdD$  satisfying the
condition $\xdD^2 = \xcW\,\xIdv{\xmM}$. The polynomial $\xcW$ is called \emph{a \crvng}.
A \crvng\ of a tensor product of two matrix factorizations over $\ICbx$ is
the sum of their \crvng s.

The simplest \xzobj s of $\xcMFdbx$ are polynomials $\xcW\in\ICbx$.
%Their geometric counterparts in $\ctLLv{(\Ts\IC^n)}$ are \hlgrsm s
%$\yYW$, which are graphs of $\del\xcW$:
%%
%\ylee{aeq1.6}
%\yYW = \{(\bfx,\bfp)\in\Ts\IC^n\,|\, \bfp =
%\boldsymbol{\del}\xcW\}.
%\yeee
%%
The category of morphisms between two polynomials $\xcWo$ and
$\xcWt$ is the
%`\xaug'
category of matrix factorizations of their
difference
\xlee{aeq1.7}
\Hom(\xcWo,\xcWt) =
%\bigcup_{C\in\IC}
\xcMFvv{\bax}{\xcWt-\xcWo},
\yeee
and the composition of morphisms comes from the tensor product of
matrix factorizations over $\ICbx$.

%Let $\CrW$ denote the critical locus of the polynomial $\xcW$:
%$\CrW = \{ \bfx\in\IC^n\,|\,\del\xcW(\bfx)=0\}$.
%A category $\xcMFbxW$ `localizes' to $\CrW$ and if the Hessian of
%$\xcW$ is non-degenerate in the normal directions to $\CrW$ then
%we conjecture that $\xcMFbxW$ is equivalent to $\rDprf(\CrW)$ up
%to a certain categorical `shift'. The intersection
%$\yYWo\cap\yYWt\subset\Ts\IC^n$ projects exactly onto
%$\CrWtmo\subset\IC^n$ and the projection establishes an
%isomorphism between them. The intersection of $\yYWo$ and $\yYWt$ is \xgd\
%exactly when the difference $\xcWt-\xcWo$ is non-degenerate in the
%normal direction to $\CrWtmo$. Hence in this case the categories
%$\Hom(\xcWo,\xcWt)$ and $\Hom(\yYWo,\yYWt)$ are equivalent (up to
%a shift).

In Section\rw{s.sct4} we extend the algebraic construction of
$\ctLLv{(\Ts\IC^n)}$ to the cotangent bundle $\TsU$ of a complex
manifold $\xcA$ by defining algebraically a 2-category $\rDDprfU$
which is supposed to be equivalent to $\ctLLTsU$.
%:
%%
%\xlee{eeq2}
%\rDDprfU = \ctLLTsU.
%\xeee
%%
 Similar to the
\zaffn\ case, the simplest objects of $\rDDprfU$ are labeled
by holomorphic functions $\xcW$ on $\xcA$, and a category of
morphisms between two such functions is
the \xacrv\ version of the derived category of coherent sheaves
%$\rDprf(\xcA)$
$\rDbU$:
\xlee{aeq1.8}
\Hom_{\rDDprfU}(\xcWo,\xcWt) = \rDprfvv{\xcA}{\xcWt-\xcWo}.
\xeee
This category is an analog of the category of matrix
factorizations\rx{aeq1.7} when the algebra $\ICbx$ is replaced by
the differential graded Dolbeault algebra $(\xbOmbU,\dlb)$.
A \xper\ object of $\rDprfUW$ is a \Ztgrdd\ vector bundle
$\xE\rightarrow\xcA$ with a \xacrv\ $(0,1)$ differential $\nbb$ such that
$\nbb^2 = \xcW\,\xIdv{\xE}$, and the composition of morphisms of the
type\rx{aeq1.8} comes from the tensor product of vector bundles.
%(a \crvng\ of the tensor product of bundles is the sum of their
%\crvng s).

\subsection{Geometric approach: the case of a cotangent bundle}

Path-integral analysis in\cx{KRS1} indicates that `in the classical approximation'
the category $\ctLLXsom$ should contain special `geometric'
objects. These objects are
holomorphic fibrations $\ycY\rightarrow\yY$, where
$\yY\subset\xX$ is a lagrangian submanifold.
Generally, fibration objects $\ycY\rightarrow\yY$ have to be
deformed because of quantum corrections, but this is unnecessary in two special
cases.
%an original fibration $\ycY\rightarrow\yY$ is still an
%object of $\ctLLXsom$.
The first case is when $\ycY$ is a \opfib,
that is, the fiber is a point and the object is just the lagrangian submanifold $\yY$
itself. The second case is when $\xX$ is isomorphic to a cotangent bundle:
$\xX\cong\TsU$.

In Section\rw{s.sct6} we study
morphisms between \xfob s of $\ctLLXsom$ for
$\xX=\TsU$.
%objects within a cotangent bundle $\xX=\TsU$.
Here is an overview of our conjectures
for the simplest objects of $\ctLLTsU$ which are lagrangian
submanifolds.

%Section\rw{s.sct6} describes
%the second, geometric, approach to the description of $\ctLLTsU$.
%%is called geometric.
%It was shown in\cx{KRS1} that a holomorphic fibration
%$\ycY\rightarrow\yY$%\subset \xX$
%with a lagrangian base $\yY\subset\xX$ is an object of
%$\ctLLXsom$.
%
%It is based on
%path-integral evidence of\cx{KRS1} suggesting that many
%objects of $\ctLLTsU$ have a simple geometric description: they
%are
%%Let us review the main features of $\ctLLXsom$ using a geometric language.
%%The simplest \zobj s of $\ctLLXsom$ are
%\hlgrsm s $\yY\subset\xX$, which are also \fCY.
%In Section\rw{s.sct6}
%we will consider more general \xzobj s
%represented by holomorphic fibrations $\ycY\rightarrow\yY\subset \xX$ with
%lagrangian bases.

We say that two holomorphic submanifolds $\yZo,\yZt\subset\xX$
have \emph{a \gdint} if  any point $x\in\yZo\cap\yZt$ has has an open
neighborhood $U_x\subset\xX$ which is isomorphic to a neighborhood
of $0$ in $\Tng_x\xX$,
so that $\yZo$ and $\yZt$ are identified with $\Tng_x\yZo$ and
$\Tng_x\yZt$. This condition implies that the intersection
$\yZo\cap\yZt$ is smooth.

%If
Suppose that two lagrangian submanifolds $\yYo,\yYt\subset \xX=\TsU$ are
\fCY,  their intersection is \xgd,\ and
%
%Generally, the category of morphisms $\Hom_{\ctLLXsom}(\yYo,\yYt)$
%is rather complicated, but if $\xX$ is a
%cotangent bundle ($\xX=\TsU$), $\yYo$ and $\yYt$ have a \gdint\
%and
the difference of dimensions $\dim \yYo - \dim(\yYo\cap\yYt)$ is even. Then
the category of morphisms between them
%it
becomes fairly simple:
\xlee{aeq1.2a}
\Hom_{\ctLLXsom}(\yYo,\yYt) = \rDprf(\yYo\cap\yYt).
\xeee
%
%and, in particular, $\End_{\ctLLXsom}(\yY) = \rDprf(\yY)$ for any
%lagrangian submanifold $\yY\subset\xX$.
Moreover, if all intersections between three \fCY\ lagrangian
submanifolds $\yYo$, $\yYt$ and $\yYh$ are \xgd, then
%
%
%Suppose that the \hlgrsm s $\yYo$ and $\yYt$ have a \emph{\gdint},
%that is, $\yYoit$ is locally isomorphic to the intersection of the
%tangent spaces. This condition implies that $\yYoit$ is a smooth
%complex manifold. We conjecture that in this case the category
%$\Hom(\yYo,\yYt)$ (up to a certain `shift' explained in ) is the
%deformation of $\rDprf(\yYoit)$:
%%
%\ylee{aeq1.2}
%\Hom(\yYo,\yYt) = \rDsprfsvv{\yYoit}{\xdfmot},
%\yeee
%%
%the deformation parameter $\xdfmot$ being \xrlb,
%$\dgDlbv\xdfmot\geq 2$.
%%\ and of Dolbeault degree at least 2.
%Furthermore, we
%conjecture that there exists a universal formula expressing
%$\xdfmot$ in terms of Atiyah-type characteristic classes
%determined by the geometry of the neighborhood of $\yYoit$.
%
the composition of morphisms $\cEot\in\Hom(\yYo,\yYt)$ and
$\cEth\in\Hom(\yYt,\yYh)$ is a combination of
pull-backs, tensor product and push-forward
\xlee{aeq1.3}
\cE_{23}\circ\cE_{12} = (\xiosoh)_\ast\Big(
\xiosot^\ast(\cE_{12})\otimes \xiosth^\ast(\cE_{23})
\Big),
\xeee
where $\xiosv{ij}$ are injections
\ylee{aeq1.4}
\xymatrix{ & \yYothi
\ar@{_{(}->}[dl]_{\xiosot}\ar@{^{(}->}[dr]^{\xiosth} \ar@{_{(}->}[d]^{\xiosoh}
\\
\yYoit &\yYoih& \yYtih.
}
\yeee
%
%and the sign $\approx$ indicates that the \lhs is a deformation of
%the \rhs.
In particular, the endomorphism category of a lagrangian
submanifold $\yY\subset\xX$ is its 2-periodic category:
\xlee{bah7.1}
\End_{\ctLLXsom}(\yY) = \rDprf(\yY),
\xeee
and the composition of endomorphisms is just the
tensor product:
\xlee{bah7.2}
\cEo\circ\cEt=\cEo\otimes\cEt
\xeee
for any
$\cEo,\cEt\in\rDprf(\yY)$.

In Section\rw{ss.dlv} we study the geometric description of the
category $\ctLLXsom$ in case of a general \hlsmm\ $\Xsom$. We show
that general features remain the same as for the cotangent bundle
$\xX=\TsU$, but almost everything is
deformed. As we have mentioned, lagrangian submanifolds $\yY\subset\xX$ remain
objects of $\ctLLXsom$, but the categories of
morphisms\rx{aeq1.2a} and the composition rules\rx{aeq1.3} are
deformed. These \tAinf\ deformations are determined by the symplectic geometry
of tubular neighborhoods of the lagrangian
submanifolds involved, as summarized
%in more detail
in subsection\rw{ss.geomgen}.

\subsection{Relation between algebraic and geometric approaches}
\label{ss.relout}
In subsection\rw{ss.reltan} we relate the
algebraic and geometric descriptions of the
2-category of a cotangent bundle.
We describe the restriction of the equivalence functor
\xlee{bah5.1}
\etfe\colon \rDDprfaU
%\longrightarrow
\xmapta{\cong}
\ctLLTsU
\xeee
(where the 2-category $\rDDprfaU$ is a slight modification of
$\rDDprfU$ defined in subsection\rw{ss.aug} )
to the objects of $\rDDprfaU$ whose images in $\ctLLTsU$ admit a
geometric description as holomorphic fibrations.

Let us sketch this relation for
$\xX=\Ts\IC^n$.
To a polynomial $\xcW\in\IC[\bax]$, which is an object of the
2-category $\xcMFdbx$, we associate the graph of its holomorphic
differential
\xlee{bag1.11}
\yYW = \{(\bfx,\bfp)\in\Ts\IC^n\,|\, \bfp =
\boldsymbol{\del}\xcW\}.
\xeee
$\yYW$ is a lagrangian submanifold of $\xX$ and it represents
the object of $\ctLLTsCn$ corresponding to $\xcW$.

Let $\CrW$ denote the critical locus of the polynomial $\xcW$:
$\CrW = \{ \bfx\in\IC^n\,|\,\del\xcW(\bfx)=0\}$.
We say that the critical locus is \xgd, if $\CrW$ is a smooth
manifold and the Hessian of $\xcW$ is non-degenerate in the normal
directions.
A category $\xcMFbxW$ `localizes' to $\CrW$, and if
%the Hessian of $\xcW$ is non-degenerate in the normal directions to $\CrW$ then
$\CrW$ is \xgd then
we conjecture that $\xcMFbxW$ is equivalent to $\rDprf(\CrW)$ up
to a certain categorical `shift' explained in subsection\rw{ss.lcl}. The intersection
$\yYWo\cap\yYWt\subset\Ts\IC^n$ projects onto
$\CrWtmo\subset\IC^n$ and the projection establishes an
isomorphism between them. The intersection of $\yYWo$ and $\yYWt$ is \xgd\
exactly when the difference $\xcWt-\xcWo$
has a \xgd\ critical locus.
%is non-degenerate in the normal direction to $\CrWtmo$.
Hence in this case the categories
$\Hom_{\xcMFdbx}(\xcWo,\xcWt)$ and $\Hom_{\ctLLTsCn}(\yYWo,\yYWt)$
are equivalent (up to a shift).
% (up to a shift).

\subsection{The 2-category of a deformed cotangent bundle}
\label{ss.tcdcb}
%
%In Section\rw{ss.dlv} we study the geometric description of the
%category $\ctLLXsom$ in case of a general \hlsmm\ $\Xsom$. Our
%conclusion is that the general features of the cotangent bundle
%case $\xX=\TsU$ remain the same, but almost everything is
%deformed. Lagrangian submanifolds $\yY\subset\xX$ remain the
%objects of $\ctLLXsom$, but the categories of
%morphisms\rx{aeq1.2a} and the composition rules\rx{aeq1.3} are
%deformed  and the deformations are determined by the the holomorphic
%symplectic structure and geometry
%of the tubular neighborhoods of participating lagrangian
%submanifolds.

Path integral analysis in paper\cx{KRS1} suggests that similarly to
the derived category of coherent sheaves and in contrast to the Fukaya category, the 2-category
$\ctLLXsom$ is local: the category of morphisms between two
lagrangian submanifolds $\Hom_{\ctLLXsom}(\yYo,\yYt)$ is
determined by a small tubular neighborhood of the intersection
$\yYo\cap\yYt$ and the composition of morphisms connecting $\yYo$,
$\yYt$ and $\yYh$ is determined by a small tubular neighborhood of
the triple intersection $\yYo\cap\yYt\cap\yYh$. Therefore, if we
knew the properties of the 2-category $\ctLL$ of a tubular
neighborhood of a lagrangian submanifold $\yY\subset\xX$, we would
know the morphisms involving $\yY$ and their compositions.

In real symplectic geometry, a small tubular neighborhood of a
lagrangian submanifold is symplectomorphic to a small tubular
neighborhood of the zero section of its cotangent bundle. This is
no longer the case in holomorphic symplectic geometry: the holomorphic symplectic
structure of the cotangent
bundle $\TsY$ may have non-trivial deformations and a tubular
neighborhood of $\yY\subset\xX$ may be isomorphic to a tubular
neighborhood of the zero section within this deformed bundle. Thus
in order to apply the locality principle to the study of
$\ctLLXsom$, in Section\rw{ss.dlv} we explore the 2-category $\ctLL$ of a deformed
cotangent bundle of a complex manifold $\xcA$.

The best way for us to describe the 2-category
$\ctLLTsU$ is through its equivalence to the 2-category $\rDDprfU$.
Hence we describe the category $\ctLL$ of the deformed
cotangent bundle of $\xcA$ by constructing a deformation of
$\rDDprfU$. We assume that the deformation parameter $\hdf$ of
$\rDDprfU$ is the same parameter that describes the deformation of
the holomorphic complex structure of $\TsU$ and that the simplest
objects of the deformed category $\rDDprfUSk$ are functions $\xcW$
on $\xcA$, such that the graphs of their holomorphic differentials
$\del\xcW$ are lagrangian submanifolds of the deformed cotangent
bundle $\mtdfTUSk$.

We find that the deformation parameter $\hdf$ is an element of
$\xbOmov{\xcA,\Sb\TU}$, $\deg_{\Sb}\hdf\geq 2$, satisfying the \CMe\ $\dlb\hdf +
\shlf\Pbrv{\hdf,\hdf}=0$, where $\Pbrv{\xblnk,\xblnk}$ is the
Poisson-Schouten bracket on $\xbOmov{\xcA,\Sb\TU}$. The functions
$\xcW$ parameterizing the simplest objects of $\rDDprfUSk$ must
satisfy the equation $\dlb\xcW=\hdf(\del\xcW)$, where $\hdf$ is
regarded as a $(0,1)$ form on $\TsU$ taking values in
polynomial functions on the fibers of $\TsU$. The category of morphisms
between two objects $\xcWo$ and $\xcWt$ of $\rDDprfUSk$ turns out
to be an \tAinf-deformation of\rx{aeq1.8}:
\xlee{bah5.1a}
\Hom_{\rDDprfUSk}(\xcWo,\xcWt) =
 \rDsprfUdot,
\xeee
where the deformation parameter
$\xdfmot=\xcWt-\xcWo+\cdots\in\xbOmv{\xmi}(\xcA,\wedge^{\bullet}\TU)$
satisfies the \CMe\ $\dlb\xdfmot + \shlf\,[\xdfmot,\xdfmot] = 0$
and the bracket
$[\xblnk,\xblnk]$ is the Schouten bracket. The composition
of morphisms turns out to be a non-commutative deformation of the
tensor product that described the composition of morphisms within
$\rDDprfU$.

\subsection{Geometric approach: the general case}
\label{ss.geomgen}

The locality principle applied to the formula\rx{bah5.1a} says that
%a category of morphisms
%between two lagrangian submanifolds $\yYo,\yYt\subset\xX$
%
 for a general \hlsmm\ $\Xsom$
the category of morphisms between lagrangian submanifolds $\yYo$
and $\yYt$ having a \gdint is given by a deformation of
\ex{aeq1.2a}:
\xlee{bah5.3}
\Hom_{\ctLLXsom}(\yYo,\yYt) = \rDprf(\yYo\cap\yYt;\adfmot),
\xeee
where $\adfmot$ is a deformation parameter of a special form:
%$\adfmot = \sum_{\xmi=2}^\infty \adfmoti$,
$\adfmot = \adfmotv{2}+\adfmotv{3}+\cdots$ and
$\adfmoti\in\xbOmv{\xmi}\brb{\xcA,\wedge^{\xmi}\Tng(\yYocct)}$.
Properties of $\rDDprfUSk$ suggest that if one of the exact
sequences $\Tng\yYi\rightarrow\Tng\xX|_{\yYi}\rightarrow\Nrm\yYi$
$i=1,2$
(for
example, the one with $i=1$)
 splits and the other lagrangian submanifold $\yYt$
can be presented as the graph of
$\del\xcW$ with $\xcA=\yYo$ as described in subsection\rw{ss.odc}
\ftnt{We expect that such a presentation exists if
the holomorphic bundle $\Tng\yYt|_{\yYocct}/\Tng(\yYocct)$ admits an $\OnC$
structure.},
then $\adfmot=0$, so that the formula\rx{aeq1.2a}
holds true.

The non-split nature of the exact sequence
$\TY\rightarrow\Tng\xX|_{\yY}\rightarrow\Nrm\yY$
is measured by a class
$\cdfbtY\in\Ext^1_{\yY}(\strsY,\Stat\TY)\subset\Ext^1_{\yY}(\NY,\TY)$, where
$\strsY$ is the structure sheaf of $\yY$ and we used the fact that
$\yY$ is lagrangian, so $\NY=\TsY$. If the exact sequence does not
split or equivalently $\cdfbtY\neq 0$,
%
%However, if the exact sequence
%$\TY\rightarrow\Tng\xX|_{\yY}\rightarrow\Nrm\yY$ does not
%split and hence determines a non-zero element
%$\tdfbtY\in\Ext^1(\NY,\TY)$,
then the category of endomorphisms of $\yY$ is
deformed:
\ylee{bah5.4}
\End_{\ctLLXsom}(\yY) = \rDprf(\yY;\xdfmv{\yY}),
\yeee
where
%$\xdfmv{\yY} = \sum_{i=3}^\infty \xdfmv{\yY,i}$,
$\xdfmv{\yY} = \xdfmv{\yY,3} + \cdots$,
$\xdfmv{\yY,i}\in \xbOmv{\xmi}\brb{\xcA,\wedge^{\xmi}\TY}$
and the leading term $\xdfmv{\yY,3}$
%of the deformation parameter $\xdfmY$
is
quadratic in $\cdfbtY$ and linear in the Atiyah class $\ccrR$ of the tangent bundle
$\TY$ (\cf \ex{beq2.42b}).

To illustrate the deformation of the composition rule\rx{aeq1.3},
consider the case when $\yYo=\yYt=\yYh=\yY$
%a single lagrangian submanifold $\yY\subset \xX$ such that
and $\cdfbtY\neq 0$. If the Atiyah class of the tangent bundle
$\ccrR$ is zero, then, according to subsection\rw{ss.gdgs}, the
endomorphism category of $\yY$ remains undeformed as in
\ex{bah7.1}, but the composition rule\rx{bah7.2} is deformed. For
example, if $\eEo,\eEt\in\rDprf(\yY)$ are two
%complexes of sheaves
holomorphic vector bundles
on $\yY$, then their composition is the deformed tensor product
\xlee{bah7.4}
\eEo\circ\eEt = (\eEo\otimes\eEt)_{\dfalot},
\xeee
where $\dfalot$ is a deformation parameter
$\dfalot=\dfaloto+\dfalott+\cdots$, $\dfaloti \in
\xbOmv{2i+1}\brb{\End(\eEo\otimes\eEt)}$, satisfying the \CMe\
$\dlb\dfalot + \shlf[\dfalot,\dfalot]=0$. The cohomology class
of the leading component of
$\dfalot$ is proportional to $\cdfbtY$ and to the Atiyah classes
of the bundles $\eEo$ and $\eEt$:
\ylee{bah7.5}
\check{\dfal}_{12,3}=
\cdfbtY\spsmb (\hacFEo\hacFEt).
\yeee
Note that $\check{\dfal}_{21,3}=-\check{\dfal}_{12,3}$, so the
composition in the category $\End_{\ctLLXsom}(\yY)$ is
non-commutative due to the deformation\rx{bah7.4}.

If $\cdfbtY=0$, then $\dfalot=0$ and the composition
rule\rx{bah7.2} remains undeformed, however the associator
isomorphism
\ylee{bah7.5a}
\zasoth\colon(\eEo\otimes\eEt)\otimes\eEh\longrightarrow
\eEo\otimes(\eEt\otimes\eEh)
\yeee
may be non-trivial.

\subsection{A \prshf\ definition of the 2-category $\ctLLXsom$}

In Section\rw{s.mcl} we sketch an approach to a rigorous
definition of the 2-category $\ctLLXsom$ as the category of global
sections of a certain \prshf\ $\prsfXsom$ of 2-categories defined
on $\xX$.

%First of all, we define a special symplectic open cover of $\xX$.
Let $\ICnbax$ be the affine space $\IC^\zn$ with standard
coordinates $\bax=\ax_1,\ldots,\ax_{\zn}$ and let
$\xcAbx\subset\ICnbax$ be an open subset inheriting the coordinates.
A \emph{\smrc} is a product of two such subsets $\UxVy$; it has a natural \hlsms\  $\som = \sum_{i=1}^{\zn}d\ay_i\wedge d\ax_i$.
A \emph{\rcch} is a symplectic embedding
\xlee{bai1.1}
\mch\colon\UxVy\rightarrow\xX.
\xeee
The images of these charts form an open cover of
$\xX$.

Since the coordinates $\bax$ on $\xcAbx$ provide a trivialization
of the cotangent bundle $\TsU$, there is a canonical symplectic embedding
$\UxVy\subset\TsU$.
To a \smrc\ $\UxVy$ we associate a 2-category $\ctLLbv{\UxVy}$ which is a
`\mlcz' of the 2-category $\rDDprfU$.
The
2-category $\ctLLbv{\UxVy}$ is a full subcategory of $\rDDprfU$
and
% and it is also similar to the 2-category $\xcMFdbx$.
its simplest objects
%of $\rDDprfv{\UxVy}$
are holomorphic functions $\xcW$ on $\xcA$
satisfying the condition that the associated lagrangian
submanifolds\rx{bag1.11} should lie within $\UxVy\subset\TsU$: for
any point $u\in\xcAbx$, the differential $v=\del_{\bax}\xcW\in\ICnbay$ should
belong to  $\xcVby\subset\ICnbay$.

Thus to a \rcch\rx{bai1.1} we associate the 2-category
$\rDDprfUxVy$. The structure of the \mlc\ \prshf\ $\prsfXsom$
comes from two types of functors defined in subsection\rw{ss.tctf}:
the restriction functor and the \Ldrt. A \smrc\ $\UxVy$ has
a lagrangian `\qfib' formed by subspaces $u\times\xcVby$, where $u\in\xcAbx$.
The restriction functor
$\rsfe\colon\rDDprfUxVy\longrightarrow\ctLLbv{\UxVypp}$ is
associated to a symplectic embedding
$\sfq\colon\UxVypp\hookrightarrow\UxVy$, which preserves the
\qfib. The \Ldrt\ is a special equivalence functor
$\dLtp\colon\rDDprfUxVy \longrightarrow \ctLLbv{\VyUmx}$, which
permutes the lagrangian fibrations $u\times\xcVby$, $u\in\xcAbx$ and
$\xcAbx\times v$, $v\in\xcVby$ of the \smrc\ $\UxVy$.

\subsection{Derived categorical sheaves}

In Section\rw{sec:sheaves} we discuss the relationship between
the RW model and the theory of derived categorical sheaves introduced by B.~Toen
and G.~Vezzosi \cite{ToVe}.  This relationship emerges when the target manifold $X$
is the cotangent bundle of a complex manifold $Y$. In this special case
one can promote the $\ZZ_2$ grading of the RW model to a $\ZZ$-grading
by declaring that natural fiber coordinates of the cotangent bundle
sit in cohomological degree $2$. Objects of the corresponding 2-category
of boundary conditions are naturally associated with sheaves of DG-categories
over $Y$, i.e. with derived categorical sheaves. More precisely,
objects of the kind mentioned above (complex fibrations over $Y$)
correspond to rather special sheaves of DG-categories. However,
we argue that more general sheaves of DG-categories, such as skyscraper
sheaves, can also be related to boundary conditions in the RW model if
we allow fibrations whose fibers are graded manifolds. Conjecturally,
the 2-category of boundary conditions in the $\ZZ$-graded RW model with
target $\TsY $ is a full sub-2-category of the 2-category of derived
categorical sheaves over $Y$. We perform some simple checks of this conjecture.

\subsection*{Acknowledgements}

L.R. is indebted to D.~Arinkin for many patient explanations of the
properties of coherent sheaves. He is also grateful to
V.~Ginzburg for numerous discussions and encouragement.
 A.K. would like to thank D.~Orlov
for the same. A.K. is also grateful to D.~Ben-Zvi, V.~Ostrik, and
L.~Positselski for advice. Both authors would like to thank Natalia Saulina for collaboration on Part I of the paper. The work of A.K. was supported in part by
the DOE grant DE-FG03-92-ER40701. The work of L.R. was supported by
the NSF grant DMS-0808974.

\section{The 3-category of affine spaces}
\label{s.sct3}

%\section{A 3-category of matrix factorizations}
%\label{s.sct3}

\subsection{\Ztpdcat\ of a \cdga}
\label{ztpdcat}

In this section we define the 2-category of boundary conditions corresponding to the RW model whose target is a complex symplectic vector spaces. We also describe the 3-category of all such RW models. Definitions of categories of morphisms between two \xzobjs\
of our 2-categories follow the same general pattern that we are
going to review in this subsection.
%%
%In many descriptions of the 2-category $\ctLLXsom$, the construction
%of morphism categories $\xMor(A,B)$ between \zobj s $A$ and
%$B$ follows the same pattern of a \Ztpdcat\ (\ZtPDC) of a \cdga\
%(\CDGA).
%%
We follow closely the exposition of J.~Block\cx{JB}, replacing
\Zgrdng\ with \Ztgrdng\ when needed.

A commutative curved differential graded algebra (\CDGA) is a triple $\ycdgaA$, where $\xdlA$ is a
\Zgrdd\ associative commutative algebra
%
%\ee
%\label{bae1.1}
$\xdlA = \bigoplus_{i=0}^{\infty} \xdlAuv{i}$
%\eee
%
with an associated \Ztgrdng\
\ee
\label{bae1.1a} %*r
\xdlA = \xdlAe\oplus\xdlAo,\quad
\xdlAe = \bigoplus_{i=0}^\infty \xdlAuv{2i},\quad
\xdlAo = \bigoplus_{i=1}^\infty \xdlAuv{2i+1},
\eee
$\nbb$ is its differential of (possibly inhomogeneous) odd degree not less than 1:
\ee
\xlabel{bae1.2}
\nbb^2=0,\qquad \nbb(\xdlAuv{i})\subset\bigoplus_{j=0}^\infty
\xdlAuv{i+2j+1},
\eee
and a \crvng\ $\xcW$ is a $\nbb$-closed element of $\xdlA$ of even \Ztdgr:
$\xcW\in\xdlAe$, $\;\;\nbb\xcW=0$.
%
%\ee
%\label{bae1.3}
%\xcW\in\xdlAe.
%,\qquad\text{where}\quad
%\xdlAe = \bigoplus_{i=0}^\infty \xdlAuv{2i},\quad
%\xdlAo = \bigoplus_{i=1}^\infty \xdlAuv{2i+1}.
%\eee
%

We adopt the notations $\zdgz{\xblnk}$ and $\zdgt{\xblnk}$ for
\Zdgr\ and \Ztdgr\ respectively and we denote the elements of
$\Zt$ as $\yev$ and $\yod$.

A \Ztdgm\ (\ZtDGM) over a \CDGA\ $\ycdgaA$  is a pair $\cM=\xdgmM$,
%(we may abbreviate the notation down to $\xmM$),
where $\xmM$ is a \Ztgrdd\ module over $\xdlA$, while $\nbbM$ is
its curved differential: $\nbbM$ is a $\IC$-linear map
%$\nbbM\colon\xmM\longrightarrow\xmM$
$\xmM\xrarv{\nbbM}\xmM$, $\zdgt{\nbbM}=\yod$,
%of \Ztdgr\ $\yod$
satisfying the Leibnitz
identity
\ee
\xlabel{bae1.4}
\nbbM(am) = (\nbb a)\,m + (-1)^{\zdgt{a}}\, a\,(\nbbM m),\qquad
\any a\in\xdlA,\quad \any m\in\xmM,
\eee
and having the \crvng\ $\xcW$:
\ee
\xlabel{bae1.5}
%\nbbM^2\,m = \xcW\,m,\qquad \any m\in\xmM.
\nbbM\circ\nbbM =\xcW\, \xId_{\xmM},
\eee
where $\xId_{\xmM}$ is the identity endomorphism of $\xmM$.
The module $\xmM$ can be rolled out into a \tpd\  twisted
complex
\begin{gather}
\xlabel{bae1.5a}
\xymatrix{
\cdots \ar[r]^-{\nbbMo} &
\xmMuv{\yev} \ar[r]^-{\nbbMe} &
\xmMuv{\yod} \ar[r]^-{\nbbMo} &
\xmMuv{\yev} \ar[r]^-{\nbbMe} &
\cdots
}
\\
\nonumber
\nbbMe\circ\nbbMo = \xcW \xId_{\xmMuv{\yod}},\qquad
\nbbMo\circ\nbbMe = \xcW \xId_{\xmMuv{\yev}},
\end{gather}
hence the name of the category.

%Let us introduce the following convenient notation.
Suppose that
two \Ztgrdd\ modules (or vector spaces) $M_1$ and $M_2$ have
endomorphisms $A_1$ and $A_2$ of a similar nature. Then for a
linear map $f\colon M_1\rightarrow M_2$ we use the commutator
notation for the following expression:
\ee
\label{bae1.5a1} %*r
[\cxrA,f] = A_2 f - (-1)^{\zdgt{A}\zdgt{f}} f A_1.
\eee

For two \ZtDGM s
%$\xdgmMo$ and $\xdgmMt$
$\cMo$ and $\cMt$
over a \CDGA\ $\ycdgaA$, the space
of homomorphisms $\Hom_{\xdlA}(\xmMo,\xmMt)$ has a differential $d$:
\ee
\xlabel{bae1.6}
df = [\nbbv{\cxrmM},f],\qquad f\in\Hom_{\xdlA}(\xmMo,\xmMt).
%\nbbv{\xmN}\circ f - (-1)^{\dgZt f}\,f\circ\nbbM.
\eee
Thus \ZtDGM s are objects of a DG-category.

%We define a \emph{homotopy homomorphism} space of
%$\xdgmM$ and $\xdgmN$
%$\cMo$ and $\cMt$
%as the homology of $d$:
%
%\ee
%\label{bae1.7} %*r
%\HmB{
%\xdgmMo,\xdgmMt
%\cMo,\cMt
%} = \rmH_d^\bullet\,\Big( \Hom_{\xdlA}(\xmMo,\xmMt)\Big).
%\eee
%

%We define a tensor product of a \ZtDGM\ $\xdgmM$ over a \CDGA\ $\ycdgaAo$ and a \ZtDGM\
%$\xdgmN$ over a \CDGA\ $\ycdgaAt$ as a \ZtDGM\ over a \CDGA\
%$\ycdgaAot$ by the formula
%%
%\ee
%\label{bae1.8x}
%\xdgmM\otdlA\xdgmN = \big(\xmM\otdlA\xmN,\nbbM\otimes\xId_{\xmN} +
%(-1)^{\dgZt}\otimes\nbbN\big).
%\eee
%%

We define the tensor product of a \ZtDGM\ $\cMo$ over a \CDGA\ $\ycdgaAo$ and a \ZtDGM\
$\cMt$ over a \CDGA\ $\ycdgaAt$ as a \ZtDGM\ over a \CDGA\
$\ycdgaAot$ by the formula
\ee
\label{bae1.8} %*r
\cMo\otdlA\cMt = \big(\xmMo\otdlA\xmMt,\nbbMxo\otimes\xId_{\xmM} +
(-1)^{\zdgt{\xblnk}}\otimes\nbbMxt\big).
\eee
Also we define the dual \ZtDGM\ as
\ee
\label{bae1.9} %*r
\xdulv{\xdgmM} = \big(\xdulv{\xmM},\xdulv{\nbbM}\big).
\eee
Note that $\xdulv{\xdgmM}$ is a \ZtDGM\ over the \CDGA\ $\ycdgaAd$.

A \ZtDGM\ $\cP=\xdgmP$ is called \emph{\xper} if the $\xdlA$-module
$\xmP$ has the form
\ee
\label{bae1.10} %*r
%\cxP = \chxP\otimes_{\xdlAZz} \xdlA,
\xmP = \chmP\otimes_{\xdlAZz} \xdlA,
\eee
where $\chmP$ is a projective \Ztgrdd\ module over $\xdlAZz$. The
\emph{\Ztpdcat} $\rDprfbA$ of a \CDGA\ $\ycdgaA$ is defined as a
graded category whose objects are \xper\ \ZtDGM s, and morphisms
are defined by
\ee
\label{bae1.7} %*r
\HmB{
%\xdgmMo,\xdgmMt
\cPo,\cPt } = \rmH_d^\bullet\,\Big( \Hom_{\xdlA}(\xmPo,\xmPt)\Big).
\eee

One may enhance the category $\rDprfbA$ by
adding new `\xad' objects which are declared isomorphic to the
existing \xper\ objects according to the following rule. A \ZtDGM\
%$\xdgmM$
$\cM$
is called \emph{\xad}, if there
exists a \xper\ \ZtDGM\
%$\xdgmP$
$\cP$ such that for any \xper\ \ZtDGM\
%$\xdgmPp$
$\cPp$
there is an isomorphism
%%
%\ee
%\label{bae1.10ax}
%\HmB{ \xdgmPp,\xdgmM } = \HmB{\xdgmPp,\xdgmP}
%\eee
%%
%and for any other \xad\ \ZtDGM\ $\xdgmMp$ we define
%%
%\ee
%\label{bae1.10a1x}
%\XtB{\xdgmM,\xdgmMp} = \HmB{ \xdgmP,\xdgmMp }.
%\eee
%%
%%
%
\ee
\label{bae1.10a}
\Hom(\cPp,\cM) = \Hom(\cPp,\cP)
\eee
and for any other \xad\ \ZtDGM\ $\cMp$ we define
\ee
\xlabel{bae1.10a1}
\Ext(\cM,\cMp) \edfn \Hom(\cP,\cMp).
\eee

\subsection{Categories of matrix factorizations}

\subsubsection{Definition of the category}

A category of matrix factorizations is a particular case of a
\Ztpdcat\ defined in subsection\rw{ztpdcat}. For a finite set of
commuting variables
\ee
\label{bae1.11p} %*r
\bax=\lvar{\ax}{n}
\eee
consider the algebra of polynomial functions $\xdlA=\xdlAZz = \ICbx$ regarded as  $\ZZ$-graded CDGA placed in zero degree, with zero differential $\nbbA=0$, and the
\crvng\ being a polynomial $\xcW\in\ICbx$. Then the \cxmf\
$\xcMFbxW$ is the corresponding \Ztpdcat:
\ee
\label{bae1.11} %*r
\xcMFbxW \edfn \rDprf\big(\ICbx,0,\xcW\big).
\eee

According to the general definition, an object of $\xcMFbxW$ is a pair
$\cfM=\xmdM$, where $\xmM$ is a free \Ztgrdd\ $\ICbx$-module, while
$\xdDM$ is its curved differential:
\ee
\xlabel{bae1.12}
\xdDM\in\EndICbx(\xmM),\qquad\zdgt{\xdDM}=\yod,\qquad \xdDM^2 =
\xcW\,
\xIdv{\xmM}.
\eee
%
%%where $\xIdv{\xmM}$ is the identity endomorphism of $\xmM$.
%We may abbreviate the notation $\xmdM$ down to $\xmM$.
Morphisms between objects are defined by eq. \rx{bae1.7}. The
tensor product\rx{bae1.8} gives a functor
\ee
\label{bae1.12a} %*r
\xymatrix@C=1.5cm{\xcMFbxWo\times\xcMFbxWt \ar[r]^-{\otCbx} &
\xcMFvv{\bax}{\xcWo+\xcWt}}.
\eee

Let $\bay=\lvar{\ay}{k}$ be another list of variables, generally
of a different length. For $\xcWo\in\ICbx$ and $\xcWt\in\ICby$,
a \xmf\ $\cfMot\in\xcMFvv{\bax,\bay}{\xcWt-\xcWo}$
determines a functor
\ee
\label{bae1.12a1} %*r
\xymatrix@C=1.5cm{\xcMFbxWo \ar[r]^-{\xPhMot} &\xcMFvv{\bay}{\xcWt}}
\eee
which acts by taking a tensor product with $\cfMot$: for a \xmf\
$\cfM\in\xcMFbxWo$,
\ee
\label{bae1.12a2} %*r
\xPhMot(\xmM) = \cfM\otCbx\cfMot\in\xcMFvv{\bay}{\xcWt}.
\eee
Note that since $\xcWo$ cancels from the \crvng\ of the tensor
product\rx{bae1.12a2}, we can forget its $\ICbx$-module structure,
thus turning it into an object of $\xcMFvv{\bay}{\xcWt}$.

All categories $\xcMFbxW$ can be unified into a single 2-category
$\xcMFd$ of Landau-Ginzburg B-models with affine
\tgsp s along
the lines explained in subsection\rw{ss.intr}. This 2-category should be thought of as a 2-category of boundary conditions for the RW model whose target is a point.
An \xzobj\ of $\xcMFd$ is a pair $\cWx$,
%\xvprvv{\bax}{\xcW}$,
$\xcW\in\ICbx$
or, equivalently, a category of matrix factorizations $\xcMFbxW$.
Morphisms between two \xzobj s also form a \cxmf
\def\Hmo#1{ \Hmb{#1} }
%%%%%%%%%%%%%%%%%%
\ee
\label{bae1.12a3} %*r
%\HmB{\xvprvv{\bax}{\xcWo},\xvprvv{\bay}{\xcWt}}=\xcMFvv{\bax,\bay}{\xcWt-\xcWo}.
\Hom_{\xcMFd}\brb{\cWxo,\cWyt}=\xcMFvv{\bax,\bay}{\xcWt-\xcWo},
\eee
and the composition of morphisms corresponds to the composition of
the functors\rx{bae1.12a1}: for two morphisms
%
%\ee
%\label{bae1.12a4}
$\cfMot\in\Hmo{ \cWxo,\cWyt }$ and
$\cfMth\in\Hmo{ \cWyt,\cWzh }$
%,
%\eee
%
we define
\ee
\xlabel{bae1.12a5}
\cfMth\circ\cfMot = \cfMth\otCby\cfMot.
\eee

\subsubsection{\tKmf s}

A Koszul complex corresponding to a list of polynomials
$\bap=\ap_1,\ldots,\ap_{\xkk}\in\ICbx$ is the tensor product of complexes
\ee
\label{bae1.22a} %*r
\xKobp=\bigotimes_{i=1}^{\xkk}{}%_{A}
\Big(\xymatrix{
\ICbxo \ar[r]^{\ap_i}& \ICbxz }
\Big),
\eee
where $\ICbxv{i}$ denotes a rank-1 $\ICbx$-module of
\Ztdgr\ $i$. A Koszul \xmf\ is defined similarly:
for two sequences of polynomials $\bap,\baq\in\ICbx$ of equal length
$\xkk$, a Koszul \xmf\ $\xKmfbpq$ is a tensor product of
rank-$(1,1)$ \xmfs
\ee
\label{bae1.23} %*r
\xKmfbpq =
\begin{pmatrix}
\ap_1 & \aq_1
\\
\ap_2 & \aq_2
\\
\hdotsfor{2}
\\
\ap_k & \aq_k
\end{pmatrix}%_{A}
=
\bigotimes_{i=1}^{\xkk}%_{A}
\Big(\xymatrix{
\ICbxo \ar@<0.5ex>[r]^{\ap_i}& \ICbxz \ar@<0.5ex>[l]^{\aq_i}
}
\Big).
\eee
%
%where $\xRv{i}$ denotes a rank-1 $\xR$-module of \Ztdgr\ $i$.
Obviously, $\xKmfbpq\in\xcMFvv{\bax}{\pdq}$, where
%%
%\ee
%\label{bae1.24}
%\xcW = \xcd{\bap}{\baq}
%\eee
%%
%and
we use the notation
\ee
\xlabel{eab1.24a1}
\pdq = \sum_{i=1}^{\xkk} \ap_i\aq_i.
\eee

Suppose that an ideal $(\bap)\subset\ICbx$ is generated by a
regular sequence $\bap$ and $\xcW\in (\bap)$. Then the polynomial
$\xcW$ has a presentation $\xcW=\pdq$, where $\baq\in\ICbx$.
%
%Now suppose that $\bap\in\ICbx$ is a regular sequence and
%$I\subset\ICbx$ is an ideal generated by them: $I = (\bap)$. If
%$\xcW\in I$ then this polynomial has a presentation\rx{bae1.24}.
It is easy to check that for fixed $\xcW$, the isomorphism class of the
\xmf\rx{bae1.23} does not depend on the choice of the polynomials
$\baq$, so we can use an abbreviated notation
\ee
\label{e.koza} %*r
\xKpvv{\xcW}{\bap} = \xKmfbpq
\eee
%
%$\xKpWttbp$
for the \xmf\ $\xKmfbpq$ of \ex{bae1.23}.

%thus omitting a particular choice of $\baq$.
%Neither does it depend on a choice of a
%regular sequence $\bap$ which generates $I$. In other words, the
%\xmf\rx{bae1.23} is determined by the ideal $I$ and the polynomial
%$\xcW$.

The identity endofunctor of a \xmf\ category $\xcMFbxW$ can
be presented in the form\rx{bae1.12a1} with the help of a \tKmf.
%
%A \tKmf\ presents the identity functor of a \xmf\ category
%$\xcMFbxW$ in the form\rx{bae1.12a1}.
%
For a list of
variables $\bax$, consider another list $\baxp$ of the same
length. The difference $\xcW(\baxp)-\xcW(\bax)$ belongs to the ideal
$(\baxp-\bax)\subset\ICv{\bax,\baxp}$, and the corresponding
\tKmf\
\ylee{e.koza1}
\IdbxW\edfn\xKpvv{\xcW(\baxp)-\xcW(\bax)}{\baxp-\bax}
\yeee
determines the
functor\rx{bae1.12a1}
\ee
\xlabel{e.kid}
\xymatrix
@C=2cm
%@C=3.5cm
{\xcMFbxW
%\ar[r]^-{\xPhv{\xKpvv{\xcW(\baxp)-\xcW(\bax)}{\baxp-\bax}}}
\ar[r]^-{\xPhv{\IdbxW}}
&\xcMFbxpW},
\eee
which becomes the identity functor, if we identify the categories
$\xcMFbxW$ and $\xcMFbxpW$ by identifying the variables $\bax$ and
$\baxp$.

\subsubsection{\Knrp\ and the translation 2-functor}
\label{ss.ttran}

If the list of variables $\bax$ is empty and the  curving $\xcW$
is zero, then the corresponding category $\xcMFnz$
is equivalent to the category of \Ztgrdd\ vector spaces. The
latter has only two indecomposable objects: the 1-dimensional vector space
$\IC$ in degree 0 and its translation $\IC[\hat 1]$.

The category $\xcMFyot$ also has only two indecomposable objects:
the \tKmf\
\ylee{bag2.1a}
\cfMyot\edfn \xKpvv{\ayots}{\ayo-\sqrt{-1}\,\ayt}
\yeee
and its translation
$\cfMyot[\hat 1] = \xKpvv{\ayots}{\ayo+\sqrt{-1}\,\ayt}$.

The \Knrp\ theorem states that there is an equivalence of
categories
\ylee{bag2.1}
\xcMFnz \cong \xcMFyot,
\yeee
established by the functor
\ee
\xlabel{bag2.2}
\xymatrix@C=2.5cm{\xcMFnz
\ar[r]^-{\xPhv{\cfMyot}}
&\xcMFv{y_1,y_2;y_1^2 + y_2^2}}.
\eee
More generally, for any \crvngp\ $\xcW\in\ICbx$ there is an
equivalence of categories
\ee
\xlabel{bag2.3}
\xcMFbxyWt\cong\xcMFbxW,
\eee
established by the functor
\ee
\label{bag2.4}
\xymatrix@C=3.5cm{\xcMFbxW
\ar[r]^-{\xPhv{\IdbxW\otimes\cfMyot}}
&\xcMFbxpyWt}.
\eee

The translation 2-functor $\btrno\colon\xcMFd\rightarrow\xcMFd$ acts on a
matrix factorization category $\xcMFbxW$ by adding an extra variable
to the list $\bax$ and adding its square to the \crvngp\ $\xcW$:
\ee
\label{e.trn1}
\xcMFbxW\btrno \edfn \xcMFvv{\bax,\ay}{\xcW(\bax) + \ay^2}.
\eee
The action of $\btrno$ on categories of morphisms is provided by the
\Knrp\ functors\rx{bag2.4}. \Knrp\ also implies that the square of the translation
2-functor is equivalent to the identity 2-functor:
$$\btrnt\edfn \btrno\circ\btrno \simeq
\xIdv{\xcMFd}.$$

%******************
%
%
%The \Knrp\ theorem states that for any \crvng\ polynomial
%$\xcW\in\ICbx$, the following categories of matrix factorizations
%are equivalent:
%%
%\ee
%\xlabel{e.knp1a}
%\xcMFbxW\ytrnt=\xcMFbxyWt\cong\xcMFbxW.
%\eee
%%
%The equivalence
%
%
%A \ytran\ $\ytrno$ is a transformation of a set of \xmf\
%categories, which adds an extra variable to the list $\bax$ and adds its
%square to the \crvng\ $\xcW$:
%%
%\ee
%\label{e.trn1x}
%\xcMFbxW\ytrno \edfn \xcMFvv{\bax,\ay}{\xcW(\bax) + \ay^2}.
%\eee
%%
%\emph{\Knrp} states that a \ytran\ by 2 units creates an equivalent
%category:
%%
%\ee
%\xlabel{e.knp1}
%\xcMFbxW\ytrnt=\xcMFbxyWt\cong\xcMFbxW,
%\eee
%%
%%
%%\emph{\Knrp} states that adding two new variables $\ayo,\ayt$
%%to the list $\bax$ and the sum of their squares $\ayots$ to $\xcW(\bax)$
%%produces an equivalent category:
%%%
%%\ee
%%\xlabel{e.knp1x}
%%\xcMFbxyWt\cong\xcMFbxW,
%%\eee
%%%
%and the equivalence is established by the following functor
%%
%\ee
%\xlabel{e.knp2}
%\xymatrix@C=7.5cm{\xcMFbxW
%\ar[r]^-{\xPhv{\xKpvv{\xcW(\baxp)-\xcW(\bax)}{\baxp-\bax}\otimes\xKpvv{\ayots}{\ayo-\sqrt{-1}\,\ayt}}}
%&\xcMFbxpyWt}.
%\eee
%%

%The notion\rx{e.trn1}
%%\rx{bae1.27a5}
%of a translation
%can be applied also in

The definition\rx{e.trn1} of the translation 2-functor $\btrno$ can be
adapted to
the general setting of the \Ztpdcat\ of a \cdga\ (see
subsection\rw{ztpdcat}). Consider a special \CDGA\
$\xdlAx=(\IC[\shx],0,\shx^2)$, so that according to \ex{bae1.11}
$\rDprf(\xdlAx)=\xcMFvv{\shx}{\shx^2}$. For a \CDGA\ $\xdlA$ we
define the translation of its \Ztpdcat\ as
\ee
\label{bal1.1}
\rDprf(\xdlA)\btrno = \rDprf(\xdlA\otimes\xdlAx).
\eee
\Knrp\ establishes an equivalence of categories
\aee
\xlabel{e.knp3}
\rDprf(\xdlA)\btrnt \cong \rDprf(\xdlA).
\aeee

For a list of variables $\shbx=\lvar{\ay}{n}$
%of\rx{bae1.11p}
define a \CDGA\ $\xdlAbx=(\IC[\shbx],0,\shbx^2)$, so that by our definition
$\rDprfcA\btrnn=\rDprfcAAbx$.
% Note that the
%
The category $\rDprfcAAbx$ has an equivalent `intrinsic'
description in terms of objects and morphisms of $\rDprfcA$.
Namely, consider a category $\rDprfcAdn$, whose objects are pairs
$(\cP,\bafcP)$, where $\cP=\xdgmP$ is a \xper\ module of $\rDprfcA$,
while
$$\bafcP = \afcPv{1},\ldots\afcPv{n}\in\Exttod(\cP,\cP)$$
satisfies the property
$$\{\afcPv{i},\afcPv{j}\}= \delta_{ij}\,\xId_{\cP}.$$
Morphisms between two objects $(\cP,\bafcP)$ and
$(\cPp,\bafcPp)$ are $\rDprfcA$-morphisms
$\ag\in\Extub(\cP,\cPp)$ which
intertwine the lists $\bafcP$ and $\bafcPp$: $\ag\,\bafcP=\bafcPp\ag$.
An equivalence functor between the categories
$\rDprfcAdn$ and $\rDprfcAAbx$ maps an object $(\cP,\bafcP)$ of $\rDprfcAdn$
into an object $\big(\xmP\otimes\IC[\shbx],\nbbv{\xmP}+\xcd{\shbx}{\bafcP}
\big)$ of $\rDprfcAAbx$.

\subsection{The  \tcrca}

Fix a finite set of variables $\bax$ of length $n$. The 2-category of
relative curved differential graded (\yCDG) polynomial algebras
$\xcMFdbx$ is
a result of `fibering' the 2-category $\xcMFd$ over the algebra
$\ICbx$. One should regard this 2-category as a 2-category of boundary conditions for the RW model whose target is $\Ts \CC^n$. (These are not the most general boundary conditions: more general ones will be described in the next section.) Objects of $\xcMFdbx$
%For a finite set of variables $\bax$ we define a 2-category
%$\drDfbx$. Its \zobj s
are pairs $\cWy$,
%=\xvprvv{\bay}{\xcW}$,
where
$\bay=\lvar{\ay}{k}$ is a list of `extra' variables of arbitrary length
and the \crvng\ $\xcW$ is an element of the algebra $\ICbxy$ over the ring $\ICbx$.
The category of morphisms between two \xzobj s is defined as
%The morphisms between two \xzobj s form a \cxmf
the \cxmf\ of the difference of \crvng s:
\ee
\label{bae1.13} %*r
\Hom_{\xcMFdbx}\brb{\cWyo,\cWzt} \edfn
\xcMFvv{\bax,\bay,\baz}{\xcWt(\bax,\baz)-\xcWo(\bax,\bay)}
\eee
(\cf \ex{bae1.12a3}).
The composition of morphisms between \xzobj s is given by the tensor
product functor\rx{bae1.12a}: for
$\cfMot\in\Hmo{\cWyo,\cWzt}$ and $\cfMth\in\Hmo{\cWzt,\cWuh}$
%%
%\ee
%\label{bae1.14}
%\xmMot\in\HmB{(\bay,\xcWo),(\baz,\xcWt)},\qquad
%\xmMth\in\HmB{(\baz,\xcWt),(\bau,\xcWh)},
%\eee
%%
we define
\ee
\xlabel{bae1.15}
\cfMth\circ\cfMot =
\cfMth\otCbxz\cfMot\in\xcMFvv{\bax,\bay,\bau}{\xcWh-\xcWo}
\eee
(\cf \ex{bae1.12a2}).
%Note that since $\xcWt$ cancels from the curvature of the tensor
%product\rx{bae1.15}, we can forget its structure as a
%$\ICby$-module.

The simplest objects of $\xcMFdbx$ are the ones without extra
variables, and we denote them as $\xcW \edfn \cWe$, where
$\xcW\in\ICbx$. The category of morphisms between such objects is
\ylee{bag3.0a}
\Hom_{\xcMFdbx}(\xcWo,\xcWt) = \xcMFvv{\bax}{\xcWt-\xcWo}.
\yeee

An object $\cWzot\in\xcMFdbxy$ determines a
\cfun
\ee
\label{bae1.16} %*r
\xymatrix@C=1.5cm{ \xcMFdbx \ar[r]^-{\dPhWzot} & \xcMFdby },
\eee
which turns an object $\cWu\in\xcMFdbx$ into an
object
\ee
\label{bae1.17}
\dPhWzot \cWu  = \xvprvv{\bax,\bau,\baz}{\xcW+\xcWot}.
\eee

The action of the \cfun\rx{bae1.16} on categories of morphisms
between \xzobj s is defined with the help of \tKmf s. For
\xzobj s $\cWuo,\cWwt\in\xcMFdbx$  we have to define the functor
%
%Further, it acts on morphisms between
%two objects
%$\cWuo,\cWwt\in\xcMFdbx$ % there is a
%by the following functor (we denote it by the same symbol as\rx{bae1.16})
%
\ee
\xlabel{bae1.18}
\xymatrix@C=2cm{
\Hmo{ \cWuo,\cWwt }
%\ar[r]^-{\dPhWzot} &
\ar[r]^-{\xPhuotWzot} &
\HmB{ \dPhWzot\cWuo,\dPhWzot\cWwt }
},
\eee
%
%that is, there is a functor between the categories of matrix
%factorizations
or more explicitly, according to \eex{bae1.13} and\rx{bae1.17},
\ee
\label{bae1.19} %*r
\xymatrix@C=2cm{
\xcMFvv{\bax,\bau,\baw}{\xcWt(\bax,\baw)-\xcWo(\bax,\bau)}
\ar[r]^-{
%\xPhWzot
\xPhuotWzot
} &
\xcMFvv{\bay,\baxp,\baup,\bazp,\baxpp,\bawpp,\bazpp}
%{W(\bay,\baxp,\baup,\bazp,\baxpp,\bawpp,\bazpp)}
{\xcWttt}
},
\eee
where
\ee
\xlabel{bae1.20}
%W(\bay,\baxp,\baup,\bazp,\baxpp,\bawpp,\bazpp) =
\xcWttt =
%\\
\xcWt(\baxpp,\bawpp)+ \xcWot(\baxpp,\bay,\bazpp)
-
\xcWo(\baxp,\baup) - \xcWot(\baxp,\bay,\bazp)
\eee
and for some lists of variables we used primed and double-primed
lists of the same length in order to change the names of
variables. The functor\rx{bae1.19} can be written in the form\rx{bae1.12a2}
\ylee{bae1.21}
\xPhuotWzot = \xPhv{\cfMot}
\yeee
%
%
%\ee
%\label{bae1.21}
for a certain matrix factorization
$\cfMot\in\xcMFvv{\xRtt}{\xcWttot},$
%\eee
%
where
%
%\ee
%\label{bae1.21a}
$$\xRtt =
\ICv{\bax,\bau,\baw,\bay,\baxp,\baup,\bazp,\baxpp,\bawpp,\bazpp}$$
%\eee
%
and
\begin{multline}
\xlabel{bae1.22}
\xcWttot = \Big(\xcWt(\baxpp,\bawpp) - \xcWt(\bax,\baw)\Big)
-\Big(\xcWo(\baxp,\baup)-\xcWo(\bax,\bau)\Big)
\\
\nonumber
+ \Big(
\xcWot(\baxpp,\bay,\bazpp) - \xcWot(\baxp,\bay,\bazp)
\Big).
\end{multline}
The full formula for the \xmf\ $\cfMot$ is rather bulky, so we
describe it indirectly in terms of the \tKmf\rx{e.koza}.
Denote $$\bap =
(\baxpp-\bax,\bawpp-\baw,\baxp-\bax,\baup-\bau,\bazpp-\bazp).$$
Then $\xcWttot\in(\bap)$, and we set
\ee
\xlabel{bae1.27}
%\xmMot
\cfMot=\xKpWttbp.
\eee

When the lists $\bax$ and $\bay$ have equal number of variables,
there exist two important equivalence \xxtf s of
type\rx{bae1.16}. The first one is the \xxtf
\xlee{bag3.1}
\Idba=\dPhv{\baa;\baa\cdot(\bay-\bax)}.
\xeee
If we identify $\xcMFdbx$ with
$\xcMFdby$ by identifying the variables $\bax$ and $\bay$, then
$\Idba$ becomes equivalent to the identity \xxtf. We leave the details of the
proof to the reader, while illustrating this statement by the
following example: after the identification of $\bax$ and $\bay$,
the action of $\Idba$ on an object $\ocW\in\xcMFdbx$ becomes
$\Idba(\ocW) = \xvprvv{\baxp,\baa}{\tcW}$, where $\tcW(\bax,\baxp,\baa)=\xcW(\baxp)+
\baa\cdot(\bax-\baxp)$, and the isomorphism between $\ocW$ and
$\Idba(\ocW)$ is established by the \tKmf\
$\xKpvv{\tcW-\xcW}{\baxp-\bax}$.

The second equivalence \xxtf\ is the `\Ldrt' and it has two
versions:
\xlee{bag3.1a}
\dLtp = \dPhv{\varnothing;\bax\cdot\bay},\qquad
\dLtm=\dPhv{\varnothing;-\bax\cdot\bay}.
\xeee
%
%, where $\bax\cdot\bay\edfn x_1 y_1 + \cdots x_n y_n$.
Again, we leave the verification of their
equivalence nature  to the reader. Note that the equivalence of morphism
categories $\Hom_{\xcMFdbx}\brb{\ocWo,\ocWt}$
and $\Hom_{\xcMFdby}\brb{\dLtpm(\ocWo),\dLtpm(\ocWt)}$ is a
corollary of the \Knrp.

The composition of two \Ldrt s with opposite signs is equivalent to the identity 2-functor:
\ylee{bag3.2}
\dLtp\circ\dLtm \simeq \dLtm\circ\dLtp \simeq \dId.
\yeee

The translation 2-functor is an equivalence 2-functor
$\ttrno\colon\xcMFdbx\rightarrow\xcMFdbx$, which acts on an object
$\cWy$ by adding a new variable $\aa$ to the extra variable list
and adding its square to the polynomial $\xcW$:
\xlee{bag2.1b}
\cWy\ttrno \edfn \xvprvv{\bay,\aa}{\xcW+\aa^2}.
\xeee
An equivalence
between the morphism categories
$\HmMFbx\brb{\xvprvv{\bay}{\xcWo},\xvprvv{\baz}{\xcWt}}$ and\\
$\HmMFbx\brb{\xvprvv{\bay}{\xcWo}\ttrno,\xvprvv{\baz}{\xcWt}\ttrno}$
is established by the \Knrp\
functor in view of an obvious
equivalence of categories
\ee
\label{bae1.27a6}
\Hom_{\xcMFdbx}\brB{
\xvprbvv{\bay}{\xcWo}\ttrno,\xvprbvv{\baz}{\xcWt}} =
\Hom_{\xcMFdbx}\brB{
\xvprbvv{\bay}{\xcWo},\xvprbvv{\baz}{\xcWt}}\btrno.
\eee
The \xtran\ by 2 is isomorphic to the identity endofunctor:
$\ttrnt\cong \xIdv{\xcMFdbx}$.

\subsection{The 3-category $\xcMFdd$ of polynomial algebras}

The 2-categories described in the previous subsection can be combined into a single 3-category. This 3-category should be thought of as the 3-category of RW models whose target spaces have the form $\Ts \CC^n$ for some nonnegative integer $n$.

An \xzobj\ of the 3-category $\xcMFdd$ is a list of variables
$\xlobx$.
%(brackets indicate that $\bax$ is considered as a \zobj\ of $\xcMFdd$)
The 2-category of morphisms between two \xzobj s $\xlobx,\xloby\in\xcMFdd$
is the correspondence 2-category $\xcMFdbxy$:
%The morphisms between two \zobj s $\bax,\bay\in\trDa$ are
%correspondences:
%
\ee
\xlabel{bae1.28}
\Hmb{\xlobx,\xloby} = \xcMFdbxy.
\eee
Each correspondence determines a 2-functor \rx{bae1.16}, and the
composition of correspondences as morphisms of $\xcMFdd$ is defined
to agree with the composition of the corresponding 2-functors. Namely, the
composition of two correspondences
%
%\ee
%\label{bae1.29}
$\xvprvv{\bau}{\xcWot}\in\xcMFdbxy$ and %,\qquad
$\xvprvv{\baw}{\xcWth}\in\xcMFdbyz$
%\eee
%
is the correspondence
\ee
\xlabel{bae1.30}
\xvprvv{\baw}{\xcWth}\circ\xvprvv{\bau}{\xcWot} =
\xvprvv{\bau,\baw,\bay}{\xcWot+\xcWth}\in\xcMFdbxz.
\eee

The identity endomorphism of an \xzobj\ $\xlobx$ can be represented by
the correspondence
\ee
\label{bae1.31} %*r
\xIdbx \simeq \xvprbvv{ \baa}{\xcd{\baa}{(\baxp-\bax)} %\sum_i\au_i(\baxp_i-\bax_i)
}\in\xcMFdv{\bax,\baxp} = \End(\xlobx),
\eee
the lists $\baxp$ and $\baa$ having the same length as $\bax$
(\cf
\ex{bag3.1}).
%1.27a2}).

%Let $\bax$, $\baxp$ and $\bau$ be three list of variables of equal
%length. If we identify the variables $\bax$ and $\baxp$, then the
%correspondence
%%
%\ee
%\label{bae1.31}
%\lrbc{\bau,\sum_i\au_i(\baxp_i-\bax_i) }\in\xcMFdv{\bax,\baxp}
%\eee
%%
%represents the identity endomorphism of the object $\xlobx$ (\cf
%\ex{bae1.27a2}).

The 3-category $\xcMFdd$ has a symmetric monoidal structure corresponding to
\ex{bag1.1}. The
product of \xzobj s corresponds to the concatenation of lists
\ee
\xlabel{bae1.31a}
\xlobx
%\xtnz
\times
\xloby = \xlobxy
%,\qquad \zOnthz = \xlobe.
\eee
and the
%\xempt\ \xzobj\
unit element
is the empty list $\xlobe$. The duality endofunctor $\hve$ of
\ex{bag1.3} acts on $\xcMFdd$ as the identity.

\section{The 3-category of complex manifolds}
\label{s.sct4}
\label{ss.3catlv}

\subsection{The 2-periodic derived category of a curved complex manifold}
\label{tZtgdcs}

The \cxmf\rx{bae1.11} is based on a polynomial algebra $\ICbx$.
One can define a similar `analytic' category $\xcMFAbxW$ based on the algebra of holomorphic
functions on $\IC^n$.
The 2-periodic derived category of a curved complex manifold $\rDsprfUW$, which is the \Ztpdcat\ of its curved \Dlb\ algebra,
%The \tZtgdcs\
generalizes the category
$\xcMFAbxW$ from $\IC^n$ to a general complex manifold $\xcA$, so
that if $\xcA=\IC^n$, then there is an equivalence of categories
\ee
\xlabel{bae1.35e}
\rDsprfvv{\IC^n}{\xcW} \simeq \xcMFAbxW.
\eee
{}From the physical point of view $\rDsprfUW$ is the category of boundary conditions for the B-model with target $\xcA$ deformed by a curving $W$. In the special case when $W$ is a holomorphic function on $\xcA$, this theory is the topological Landau-Ginzburg model with target $\xcA$ and the superpotential $W$.

Following the general construction of subsection\rw{ztpdcat},
let us define the curved \Dlb\ algebra of
a complex manifold $\xcA$. The \Zgrdd\ \CDGA\ $\xdlA$ of \ex{bae1.1a}
is  the algebra of \xahfs\ $\xbOmbU$, the differential $\nbb$
is $\dlb$, and
%
%\ee
%\label{bae1.36}
%\nbb = \dlb.
%\eee
%%
the \crvng\ $\xcW$ is a
$\dlb$-closed even form:
%%
%\ee
%\label{bae1.38}
%\xcW\in\xbOmeU,\qquad\dlb\xcW=0.
%\eee
%%
%
\ee
%\label{bae1.38}
\label{bae1.39} %*4
\nbb = \dlb,\qquad\xcW\in\xbOmeU,\qquad \dlb\xcW=0.
\eee

Let $\xE$ be a smooth (not necessarily holomorphic) \Ztgrdd\ vector
bundle over $\xcA$, and let $\xbOmbE$ be the space of \xahfs\
with values in $\xE$:
\ee
\xlabel{bae1.40}
\xbOmbE = \Gamma\big(\xE\otimes\wedge^\bullet\bTs\xX \big)
= \Gamma(\xE)\otimes_{\xbOmz(\xcA)}\xbOmbU,
\eee
where $\Gamma(\xE)$ is the space of sections of $\xE$.
The space $\xbOmbE$ is a \Ztgrdd\ module over $\xbOmbU$ of the
form\rx{bae1.10}.
%, $\xmP=\Gamma(\xE)$ being the space of sections of $\xE$.
%
A \emph{\cqhlmvb} is a pair $(\xE,\nbbE)$, where
$\nbbE$ is a curved \ydah-differential acting on $\xbOmbE$, that is,
$\nbbE$ is a $\IC$-linear operator % of \Ztdgr\ one
\ee
\xlabel{bae1.41}
\nbbE\in\End_{\IC}\big(\xbOmbE\big),
\eee
which satisfies the following properties:
\begin{gather}
\label{bae1.42} %*r
%\dgZt
\zdgt{\nbbE} = \yod,\\
\label{bae1.42a} %*r
\nbbE\,(\smu\wedge\sigma) = (\dlb\smu)\wedge \sigma +
%(-1)^{\dgZt\smu}
(-1)^{\zdgt{\smu}}
\smu\wedge (\nbbE\,\sigma),
\\
\label{bae1.42a1} %*r
\nbbE^2\,\sigma = \xcW\wedge\sigma,
\end{gather}
where $\smu\in\xbOmbU$ and $\sigma\in\xbOmbE$.

We call the pair $(\xE,\nbbE)$ \qhlm, because even if $\xcW=0$, the bundle $\xE$ is not necessarily
holomorphic: if we split the differential $\nbbE$ according to the
Dolbeault degree:
\ylee{baj6.1}
\nbbE = \sum_{i=0}^{\dim\xcA} \nbbEai,\qquad
\nbbEai\colon \xbOmbE\longrightarrow\xbOmv{\bullet + i}(\xE),
\yeee
then $\nbbEao^2=-\{\nbbEaz,\nbbEat\}$ rather than $\nbbEao^2=0$,
so in general $\nbbEao$ does not determine a holomorphic structure on $\xE$.

If $(\xE,\nbbE)$ is a \cqhlmvb,
then the pair
\ee
\label{bae1.44} %*r
\cE=\xdgmE
\eee
is a \xper\ \ZtDGM\ over the \CDGA\ $\ycdgaBU$. In fact,
all \xper\ \ZtDGM s over this \CDGA\ originate in this way from vector bundles
with curved differentials.

The pair $\cmfUW$ will be called a \emph{\ccm}. We
define its 2-periodic derived category $\rDsprfUW$ as the
\Ztpdcat\ % $\rDprf\ycdgaBU$
of its curved \Dlb\ algebra: % $\ycdgaBU$:
\ee
\label{bae1.45} %*r
%\rDX
\rDsprfUW= \rDprf\ycdgaBU,
\eee
%
%Hence we define a
%2-periodic perfect derived category $\rDprfUW$ of a \cDA\ $\ycdgaBU$
%%\tZtgdcs\ $\ycdgaBU$
%as $\rDprf\ycdgaBU$,
its \xper\ objects being the pairs\rx{bae1.44} and
morphisms defined according to the general formula\rx{bae1.7}.
The monoidal structure and the action of the duality functor
$\dulv{}$ also follow the general definitions\rx{bae1.8}
and\rx{bae1.9}.

For a \xper\ object\rx{bae1.44} we use an abbreviated
notation
\ee
\label{bae1.45b}
\cE=\adgmE.
\eee
If $\xcW=0$, then we will abbreviate the notation\rx{bae1.45} down
to
\ee
\label{bae1.45a} %*r
\rDprfU \edfn \rDprf(\xcA,0).
\eee
The latter category contains a full subcategory
which is equivalent to the bounded derived category of coherent
sheaves $\rDbU$:
% as its full subcategory:
%
\ee
\label{bae1.45cm1} %*r
\rDbU\hookrightarrow\rDprfU.
\eee
An object of $\xrDX$, represented by
a chain complex of holomorphic vector bundles
\ee
\label{bae1.45c}
\xymatrix{\xEv{0} \ar[r]^-{\sigma_0} &\xEv{1} \ar[r]^-{\sigma_1}
&\ldots \ar[r]^-{\sigma_{k-1}} & \xEv{k} },
\eee
corresponds to an object $(\xE,\dlb+\sigma)$ of $\rDprfU$, where $\xE$ is
the total
\Ztgrdd\ vector bundle
\ee
\xlabel{bae1.45c1}
\xE= \bigoplus_{i-\rm{even}} \xEv{i} \oplus\bigoplus_{i-\rm{odd}}
\xEv{i},
\eee
while $\dlb$ is the \ydah\ differential for holomorphic vector bundles and
$\sigma$ is the combined differential of the complex\rx{bae1.45c}:
%
%\ee
%\label{bae1.45c2}
$\sigma = \sum_{i=1}^k \sigma_i.$
%\eee
%

The category $\rDsprfUW$ admits a certain `\Dlbf'. Consider
two \xper\ \ZtDGM s, which share the same vector bundle $\xE$:
$ \cE=\adgmE$ and $\cE\p=\adgmpE$.
%The difference of their
%connections in the element of $\xbOmbE$:
%%
%\ee
%\label{bae1.45ax1}
%\nbbE\p-\nbbE\in\xbOmbE.
%\eee
%%
We say
that the objects $\cE$ and $\cE\p$
are isomorphic up to \zord\ $k$, if the difference between their
connections is of higher degree as an element of $\xbOmbE$:
\ee
\xlabel{bae1.45a1} %*r
\nbbE\p-\nbbE=\dOk,\qquad\dOk\in\bigoplus_{i>k}\xbOmv{i}(\xE).
\eee

The tensor product of \xper\ \ZtDGM s over $\xbOmbU$ corresponds
to the tensor product of vector bundles:
\ee
\label{bae1.45a2} %*r
\cEo
%\otOmU
\otimes
\cEt=\brb{
%\xbOmbv{\xEo\otimes\xEt}
\xEo\otimes\xEt
,\nbbv{\xEo}+\nbbv{\xEt}}.
\eee
Since
\ee
\xlabel{bae1.45a3}
(\nbbv{\xEo} + \nbbv{\xEt})^2 = (\xcWo+\xcWt)\,\xIdv{\cEo\otimes\cEt},
\eee
the tensor product\rx{bae1.45a2} gives rise to a functor
\ee
\xlabel{bae1.45a4}
\rDsprfUWao \times \rDsprfUWat \xrarv{\otimes}
\rDsprfvv{\xcA}{\xcWo+\xcWt}.
\eee

For a holomorphic map $\mF\colon \xcAp\longrightarrow\xcA$ and for
a \xahf\ $\xcW$ of\rx{bae1.39} let $\mF^\ast(\xcW)$ denote its \pb\ to
$\xcAp$. We introduce a `derived' \pb\ functor $\mF^\ast$ and a
\pf\ functor $\mF_\ast$:
\ee
\xlabel{bae1.46}
\xymatrix@C=1cm{ \rDsprfBvv{\xcAp}{\mF^\ast(\xcW)}
\ar@<0.5ex>[r]^-{\mF_\ast}
& \;\rDsprfUW
%\ar[l]_-{\mF^\ast} }
\ar@<0.5ex>[l]^-{F^\ast}
}.
\eee
The \pb\ functor $\mF^\ast$ acts on \xper\ objects\rx{bae1.44} by pulling back
%the vector bundles
%\cqhlmvb s.
quasi-holomorphic vector bundles. The definition of the \pf\ functor $\mF_\ast$ is a bit
tricky. There is a \pb\ homomorphism of \CDGA s
\ee
\xlabel{bae1.47}
\xymatrix@C=1cm{
\Big(\xbOmbUp,\dlb,\mF^\ast(\xcW) \Big)& \ar[l]\ycdgaBU},
\eee
which turns a \xper\ \ZtDGM\ $\cE\p$ over the \CDGA\ $\Big(\xbOmbUp,\dlb,\mF^\ast(\xcW)
\Big)$ into a \ZtDGM\ over the \CDGA\ $\ycdgaBU$. We conjecture that the latter
\ZtDGM\ is \xad\ (\cf \ex{bae1.10a}) and use it as the definition of $\mF_\ast(\cE\p)$.

\subsection{The 2-category of \cfbs}
\label{ssctU}
The 2-category of \cfbs\ $\rDDprf$ is a generalization of the
2-category of \xanl\ matrix factorizations $\xcMFAd$. From the physical viewpoint, it is the 2-category of curved B-models, or equivalently the 2-category of boundary conditions for the RW model whose target is a point.
An object of $\rDDprf$ is a \ccm\ $\cmfUW$. A morphism between two
\ccm s $\cmfUWo$ and $\cmfUWt$ is a \xper\ (or an \xad) \ZtDGM
\ee
\xlabel{bae1.48}
\ycF\in \rDprf\cmfBv{ \xcAo\times\xcAt }{\;\yepiut(\xcWt)-\yepiuo(\xcWo)},
\eee
where $\yepio$ and $\yepit$ are projections
\ee
\label{bae1.49} %*r
\xymatrix@C=0.5pc{
& \xcAo\times\xcAt \ar[dl]_{\yepio} \ar[dr]^{\yepit}
\\
\xcAo &&\xcAt
}
\eee
Such an object $\ycF$ determines
a \FMtr\ functor
\ee
\xlabel{bae1.50}
\xymatrix@C=1.5cm{
\rDprfUWo \ar[r]^-{\xPhcF} & \rDprfUWt,
}
%\xPhcF\colon\rDprfUWo\longrightarrow\cMFvv{\xcWp}{\xcAp},
\eee
with the standard action on objects $\ycE\in\rDprfUWo$:
\ee
\label{bae1.51} %*r
\xPhcF(\ycE) = \yepidt\Big(\ycF\otimes\yepiuo(\ycE) \Big).
\eee
We define the composition of morphisms $\ycF$ as the
composition of the corresponding functors, so
%a composition of two morphisms
for
%
%\ee
%\label{bae1.52}
$\ycFot\in\HmB{\cmfUWo,\cmfUWt}$ and
%,\qquad
$\ycFth\in\HmB{\cmfUWt,\cmfUWh}$
%\eee
%
the composition is
\ee
\label{bae1.53} %*r
\ycFth\circ\ycFot = \yepidoh \Big(
\yepiuot(\ycFot)\otimes\yepiuth(\ycFth)
\Big),
\eee
where the maps $\yepiij$ are the projections
\ee
\label{bae1.54} %*r
\xymatrix{ & \xcAo\times\xcAt\times\xcAh
\ar[dl]_{\yepiot}\ar[dr]^{\yepioh} \ar[d]^{\yepith}
\\
\xcAo\times\xcAt & \xcAt\times\xcAh &\xcAo\times\xcAh
}
\eee

\subsection{The 2-category of \xcfbs}
\label{xcfbs}
\def\xcA{ U }
%%%%%%%%%%%%%%%%%%%%%%
%\def\xcV{ \mathcal{U} }
%\def\Vfb{ V }
%\def\Vfbv#1{ \Vfb_{#1} }
%%%%%%%%%%%%%%%%%%%%%%%

Let us fix a complex manifold $\xcA$. A \hlfb\ over $\xcA$ is a
smooth \fbl
\ee
\label{bae1.55} %*r
\xfbrv{\Vfb}{\xcU}{\xcA}{\xcp}
\eee
where $\xcU$ is a complex manifold and the projection $\xcp$ is
holomorphic. The bundle\rx{bae1.55} does not have to be
holomorphic, so the complex structure of the fiber $\Vfb$ may be
different over different points of the base $\xcA$.

Let $\ycapU$ denote the fibered product of two fibrations with the
same base $\xcA$.
%
%We define a fiberwise product $\ycapU$ of two fibrations with the same base as a
%fibration, whose base is the still the same and whose fibers are
%products of fibers
%%
%\ee
%\label{bae1.56}
%\lrbc{\vcenter{\xfbrv{\Vfbo}{\xcUo}{\xcA}{\xcp}}}\ycapU
%\lrbc{\vcenter{\xfbrv{\Vfbt}{\xcUt}{\xcA}{\xcp}}} =
%\vcenter{\xfbrv{\Vfbo\times\Vfbt}{\xcUo\ycapU\xcUt}{\xcA}{\xcp}}
%\eee
%%
%The fiberwise product
It has natural projections
\ee
\label{bae1.57} %*r
\xymatrix@C=0.5pc{
& \xcUo\ycapU\xcUt \ar[dl]_{\yepio} \ar[dr]^{\yepit}
\\
\xcUo &&\xcUt
}
\eee

The 2-category of \xcfbs\ $\rDDprfU$ is obtained by `fibering' the 2-category $\rDDprf$
over $\xcA$, the fibers $\Vfb$ being the analogs of the
manifolds $\xcA$
%$\Vfb$
appearing in subsection\rw{ssctU}.
This category is also a generalization of the \xanl\ 2-category $\xcMFAdbx$:
\ee
\xlabel{bae1.57e}
\rDDprfv{\IC^n} = \xcMFAdbx.
\eee
{}From the physical viewpoint, $\rDDprfU$ is the 2-category of boundary conditions for the RW model with target $\sT U$.

An \xzobj\ of  %a 2-category of \xcfbs\
$\rDDprfU$ is a \xcfb\
$\cmfcUW$ where, according to the definition\rx{bae1.39},
$\xcW\in\xbOmecU$ and $\dlb\xcW=0$. The category of morphisms
between two \xzobj s is the \tpd\ derived category of their curved fibered
product:
\ee
\label{bae1.58} %*r
\HmDDU\brB{\cmfcUWo,\cmfcUWt} = \rDprf\cmfBv{ \xcUo\ycapU\xcUt
}{\;\yepiut(\xcWt)-\yepiuo(\xcWo)}.
\eee
An object $\ycF$ of the category\rx{bae1.58} generates a \FMtr\
functor
\ee
\xlabel{bae1.59}
\xymatrix@C=1.5cm{
\rDprfcUWo \ar[r]^-{\xPhcF} & \rDprfcUWt
}
\eee
which acts on an object $\ycEo\in\rDprfcUWo$ according to the
formula\rx{bae1.51}.
%Again, we define the composition of morphisms
The composition of morphisms
%$\ycF$ as a
%composition of corresponding functors, so a composition of two morphisms
%
%\ee
%\label{bae1.60}
$\ycFot\in\HmB{\cmfcUWo,\cmfcUWt}$ and
%,\qquad
$\ycFth\in\HmB{\cmfcUWt,\cmfcUWh}$
%\eee
%
is defined again
by the formula\rx{bae1.53},
where the maps $\yepiij$ are projections
\ee
\xlabel{bae1.61}
\xymatrix{ & \xcUo\ycapU\xcUt\ycapU\xcUh
\ar[dl]_{\yepiot}\ar[dr]^{\yepioh} \ar[d]^{\yepith}
\\
\xcUo\ycapU\xcUt & \xcUt\ycapU\xcUh &\xcUo\ycapU\xcUh
}
\eee
so that this composition agrees with the composition
of \FM\ functors.

%For a complex manifold $\yY$, let $\TrfY$ denote a \emph{\opfib}
%over $\yY$, that is, $\TrfY$ is a trivial fibration whose fiber is
%a  single point.

The simplest \xzobj s of $\rDDprfU$ are \copfib s $\TrfUW$, where
\ee
\label{bae1.61b} %*r
\TrfU = %(\DlX\rightarrow\DlX),
\vcenter{
\xymatrix@C=1.5pc@R=1.5pc{
\{\text{1-point}\}\ar[r] & \xcA \ar[d]
\\
&\xcA
%\ar@{}[r]|<<<{\subset}
%&
%\xX\times\xX
}
}
\eee
is a \opfib, that is, a fibration whose fiber is a single point,
and $\xcW$ is a holomorphic function on $\xcA$. We denote the
objects $\TrfUW$ simply as $\xcW$.
%
%
%.
%For a complex manifold $\yY$, let $\TrfY$ denote a \emph{\opfib}
%over $\yY$, that is, $\TrfY$ is a trivial fibration whose fiber is
%a  single point. A \copfib\ \zobj\ is a pair $\TrfUW$, where
%$\TrfU$ is a \opfib\ over $\xcA$ and $\xcW$ is a holomorphic
%function on $\xcA$. We denote such object simply as $\xcW$.
According to the general definition\rx{bae1.58}, the category of
morphisms between two \copfib s is the curved 2-periodic derived category
\ee
\label{bae1.63ap}
\Hom(\xcWo,\xcWt) = \rDsprfUWtmo,
\eee
and, in particular, the
endomorphisms of a \copfib\ $\xcW$ form the \tpd\ derived category of
$\xcA$:
\ee
\label{bae1.63a} %*r
\End(\xcW) = \rDprfU
\eee
containing the bounded derived category of coherent sheaves
% sheaf category
 $\xrDX$ as a subcategory. The
composition of endomorphisms in $\End(\xcW)$ corresponds to the
tensor product\rx{bae1.45a2}, in other words, to the standard monoidal structure on $\rDprf(U)$.

The 2-category $\rDDprfU$ has a `\psmon' structure. For two objects
$\cmfcUWo$ and $\cmfcUWt$ we define
\ee
\xlabel{bae1.62}
\cmfcUWo\zcapU\cmfcUWt =
\Big(\xcUo\ycapU\xcUt,\yepiuo(\xcWo)+\yepiuo(\xcWt) \Big).
\eee
%
%where $\xcVo\ycapU\xcVt$ is the fiberwise product of
%fibrations\rx{ec1.19}, while
%%where $\yepio$ and $\yepit$ are the projections\rx{bae1.57}.
%%For a complex manifold $\yY$, let $\TrfY$ denote a \emph{\opfib}
%%over $\yY$, that is, $\TrfY$ is a trivial fibration whose fiber is
%%a  single point.
%Let $\TrfU$ denote a \emph{\opfib}
%over $\xcA$, that is, $\TrfU$ is a trivial fibration whose fiber is
%a  single point.
The
main property of the \psmon\ structure is that for the object $\TrfUZ$
consisting of the \opfib\ over $\xcA$ and $\xcW=0$, the following isomorphism holds:
\ee
\xlabel{bae1.63}
\HmB{\cmfcUWo\zcapU\cmfcUmWt,(\TrfU,0)} \cong
\HmB{\cmfcUWo,\cmfcUWt}.
\eee
%

%According to the general definition\rx{bae1.58}, the category of
%endomorphisms of the special \zobj\ $\TrfUZ$ is that of
%\ex{bae1.45a}:
%%
%\ee
%\label{bae1.63ax} %*r
%\End\TrfUZ = \rDprfU,
%\eee
%%
%so it contains the sheaf category $\xrDX$ as a subcategory. The
%composition of endomorphisms in $\End\TrfUZ$ corresponds to the
%tensor product\rx{bae1.45a2}.

The 2-category $\rDDprfU$ has a translation endo-2-functor
similar to\rx{bag2.1b}, which acts on a \zobj\ $\cmfcUW$ by
adding a space $\ICa$ (which is $\IC$ with the standard coordinate $\aa$) to
fibers of $\xcU$ and adding $\aa^2$ to the \crvng\ $\xcW$:
\ee
\xlabel{bae1.63a1}
\cmfcUW \ttrno = \cmfv{\xcU\times\ICa}{\xcW+\aa^2}.
\eee
%
%where $\ICx$ is a space $\IC$ with a coordinate $\ax$.

For a holomorphic map $\mF\colon \xcAp\longrightarrow\xcA$
there exists a pull-back \xxtf
\ee
\xlabel{bae1.64}
\xymatrix@C=1cm{
\rDDprfUp & \rDDprfU \ar[l]-_{\dFua}
}
\eee
which acts on an object $\cmfcUW\in\rDDprfU$ by pulling back
the fibration $\xcU$ and the Dolbeault cohomology class $\xcW$:
\ee
\xlabel{bae1.65}
\dFua\cmfcUW =
\cmfBv{\mFua(\xcU)}{\mFua(\xcW)}
%\big( \mFua(\xcVt),\mFua(\xcWt) \big).
\eee
If $\xcAp$ is a \hlfb\ over $\xcA$ and the map $\mF$ is its projection,
then there is also a push-forward \xxtf
\ee
\xlabel{bae1.66}
\xymatrix@C=1.5cm{
\rDDprfUp \ar[r]^{\dFda}& \rDDprfU
}.
\eee
Indeed, for an object $\cmfcUWp\in\rDDprfUp$, a fibration
$\xcUp\xrarv{\xcp\p}\xcAp$ can be pushed forward
to a fibration $\mFda(\xcUp)$ over $\xcA$:
%
%\ee
%\label{ec1.26a3}
$\xcUp\xrarv{\mF\circ\xcp\p}\xcA$,
%\eee
%
so we can keep
%the cohomology class $\xcWp\in\HDl(\xcUp)$
the form $\xcWp\in\xbOmecUp$
and  define
\ee
\xlabel{bae1.67}
\dFda\cmfcUWp = \Big( \mFda(\xcUp), \xcWp \Big).
\eee

\subsection{The 3-category of complex manifolds}
The 2-categories $\rDDprfU$ can be assembled  into a 3-category
$\rDDDprf$. It should be thought of as the 3-category of RW models with target-spaces of the form $\Ts\xcA$, for all complex manifolds $\xcA$.
Objects of $\rDDDprf$ are complex manifolds $\xcA$ (or,
equivalently, categories $\rDDprfU$). The category of morphisms
between two \xzobj s is
\ee
\label{bae1.68} %*r
\Hom(\xcAo,\xcAt) = \rDDprf(\xcAo\times\xcAt).
\eee
Objects of the category\rx{bae1.68} are called \emph{\corrs}. A
\corr\ $\cmfcUWot\in\Hom(\xcAo,\xcAt)$ determines a \xxtf
\ee
\label{bae1.69}
\xymatrix@C=2cm{
\rDDprf(\xcAo) \ar[r]^-{\xPhUWot} & \rDDprf(\xcAt)
}
\eee
acting on an object $\cmfcUWo\in\rDDprf(\xcAo)$ according to the
formula
\ee
\xlabel{bae1.70}
\xPhUWot\cmfcUWo = \yepidt\Big(\cmfcUWot\;\zcapv{(\xcAo\times\xcAt)}\;\yepiuo\cmfcUWo \Big),
\eee
where $\yepio$ and $\yepit$ are the projections\rx{bae1.49}.
The composition of correspondences
$\cmfcUWot\in\Hom(\xcAo,\xcAt)$ and
$\cmfcUWth\in\Hom(\xcAt,\xcAh)$
%%
%\ee
%\label{bae1.71}
%\cmfcUWot\in\Hom(\xcAo,\xcAt),\qquad\cmfcUWth\in\Hom(\xcAt,\xcAh)
%\eee
%%
as morphisms in the 3-category $\rDDDprf$ is defined to
agree with the composition of their functors:
\ee
\xlabel{bae1.72}
\xcUth\circ\xcUot =
\yepidoh \Big(
\yepiuot\cmfcUWot\;\zcapv{(\xcAo\times\xcAt\times\xcAh)}\;\yepiuth\cmfcUWth
\Big),
\eee
where the maps $\yepiij$ are the projections\rx{bae1.54}.

The 3-category $\rDDDprf$ has a symmetric monoidal structure corresponding to that of
\ex{bag1.1}. The
product of \xzobj s corresponds to the product of the underlying complex manifolds
$\xcAo\times\xcAt$
%%
%\ee
%\xlabel{bae1.73}
%\xcAo\otimes\xcAt = \xcAo\times\xcAt
%\eee
%%
and the unit
%\xempt\
\xzobj\
%$\zOnthz$
is the complex manifold $\xcAopt$ consisting of a single point.
%
%In case of $\rDDDprf$
The duality endofunctor $\hve$ of \ex{bag1.3}
acts on $\rDDDprf$ as the identity.

\subsection{\xAug\ categories}
\label{ss.aug}

%Our goal is to relate the 2-category $\rDDprfU$ of a complex
%manifold $\xcA$ to its cotangent bundle $\TsU$. This relation
%requires a slight modification of $\rDDprfU$.
{}The curving $W$ enters the path-integral formulation of the RW model with boundaries only through it derivative $\partial W$. This suggests that one should define the 2-category of boundary conditions in such a way that $W$ is defined only up to addition of a locally constant function. Below we describe such a modification of the 2-category $\rDDprfaU$. It is also necessary for a geometric interpretation of $\rDDprfaU$ in terms of the cotangent bundle $\Ts U$, as we will see in the next section.

For an element
$\xcW\in\xbOmecU$, such that $\dlb\xcW=0$, define an \xaug\
category $\rDsprfaUW$ as a formal union over all locally constant
functions $\xcWlc$:
\ee
\xlabel{bae1.81a3}
\rDsprfaUW = \bigcup_{\del\xcWlc=0}%_{\xcWlc\in\xLCU}
\rDsprfvv{\xcA}{\xcW+\xcWlc}.
%\xcMFvv{\bax}{\xcW+\shC},
\eee
We define the 2-category $\rDDprfaU$ in exactly the same way as
$\rDDprfU$, except that in the definition of morphisms\rx{bae1.58}
we replace the 2-periodic category $\rDprf$ with its \xaug\
version $\rDprfa$:
\ee
\label{bae1.58b} %*r
\HmDDaU\brB{\cmfcUWo,\cmfcUWt} = \rDprfa\cmfBv{ \xcUo\ycapU\xcUt
}{\;\yepiut(\xcWt)-\yepiuo(\xcWo)}.
\eee
Two objects $\cmfv{\xcU}{\xcWo}$ and $\cmfv{\xcU}{\xcWt}$ are
isomorphic within $\rDDprfaU$ if the difference $\xcWt-\xcWo$ is
locally constant on $\xcU$, that is, if $\del\xcWo=\del\xcWt$.

The \xaug\ 3-category $\rDDDprfa$ is defined in the same way as
$\rDDDprf$, except that we replace the categories $\rDDprf$
appearing in its definition with \xaug\ categories $\rDDprfa$.

The \xaug\ matrix factorization category is defined as the formal
union of categories
\ee
\xlabel{bae1.81a}
\xcMFabxW = \bigcup_{\shC\in\IC} \xcMFvv{\bax}{\xcW+\shC},
\eee
and the \xaug\ categories $\xcMFadbx$ and $\xcMFad$ are defined
similar to $\rDDprfaU$ and $\rDDDprfa$.

\def\xcWz{ \xcWv{0} }
%%\section{A 2-category of a general \hlsmm}
%\section{A geometric description of the 2-category $\ctLLTsU$}
%\label{s.sct6}

\section{The 2-category $\ctLLTsU$: a geometric description and a
relation to $\rDDprfaU$}
\label{s.sct6}

%
%\subsection{A 2-category of a \hlsmm}
%\subsection{A category of fibrations with lagrangian bases}
\subsection{A geometric description of the 2-category
%$\ctLL$ of a cotangent bundle}
$\ctLLTsU$}
\label{ss.tcthlsmm}

Let $(X,\omega)$ be a holomorphic symplectic manifold. The RW model associates to $(X,\omega)$ a 2-category of boundary conditions $\ctLLXsom$.
Path integral arguments suggest that a certain part of $\ctLLXsom$ can be described in geometric terms.
In this subsection we consider the geometric description
when $\Xsom$ is the cotangent bundle of a complex manifold $\xcA$:
$\xX= \TsU$. The description uses only the \hlsm\
structure of $\xX$, and in our definitions we never refer to the cotangent
bundle structure. Conjecturally, this property should hold also for the whole 2-category $\ctLLTsU$, in the sense that it should be acted upon by the group of symplectic automorphisms of $\TsU$. This is far from obvious from the algebraic definition of $\ctLLTsU$ as the 2-category $\rDDprfaU$ given in the previous section.

%Let $\Xsom$ be a \hlsmm. We are going to describe some features of
%the 2-category $\ctLLXsom$, which is associated to $\Xsom$ by a
%\thsgmd, and which is expected to satisfy the
%property\rx{bae1.82}.
%We will describe the simplest objects of $\ctLLXsom$, which can be
%later deformed into more general objects.

\subsubsection{$\OnC$ bundles and matrix factorizations}

A holomorphic $\OnC$ vector bundle $\bnB$ over a complex manifold $\xcA$
determines  a `$\bnB$-twisted' version $\rDsprfUW\ytrnB$ of the
%\ytran\rx{e.trn1}.
category $\rDsprfUW$.
%
%A quadratic \crvng s can also be added to a 2-periodic category
%$\rDsprfUW$ of a complex manifold $\xcA$. Let $\vcL$ be a
%holomorphic $\OnC$ vector bundle over $\xcA$.
The $\OnC$
structure determines, up to a non-zero constant factor,  a
holomorphic function $\xcWq$ on the total space of $\bnB$, which
is quadratic along the fibers. The
%map $\ytrnB$
$\bnB$-twisting
replaces $\xcA$
with the total space of the bundle $\bnB$ and adds $\xcWq$ to the
\crvng:
\ee
\label{bae1.86a} %*r
\rDsprfUW\ytrnB = \rDsprfvv{\bnB}{\xcW+\xcWq}.
\eee
It is easy to see that the composition of \ytran s corresponds to
the sum of vector bundles:
$\ytrnv{\bnB_1}\ytrnv{\bnB_2}=\ytrnv{\bnB_1\oplus\bnB_2}$, and if
the bundle $\bnB$ is trivial, then
$
\rDsprfUW\ytrnB = \rDsprfUW\btrnv{\rnk\bnB},
$
where
%$\btrnv{\xblnk}$
$\btrno$
is the translation 2-functor\rx{bal1.1}.

A line bundle $\vcL\rightarrow\xcA$ has an $\rmO(1,\IC)$ structure
if and only if it is \sd, that is, if it is isomorphic to its
dual: $\vcL\cong\dulv{\vcL}$.
The top exterior power $\wdtpB$ of an $\OnC$ bundle $\bnB$ is \sd,
%
%Let $\wdtpB$ denote the line bundle which is
%the highest exterior power of the bundle $\bnB$. The self-duality of
%$\bnB$ implies the self-duality of $\wdtpB$
and there is a canonical equivalence of categories\footnote{We thank
M.~Kontsevich for pointing this out.}
\xlee{baf1.2}
\rDsprfUW\ytrnB = \rDsprfUW\ytrnv{\wdtpB}{\btrnv{\rnk\bnB-1}}.
\xeee
%
%where $n$ is the rank of $\bnB$.

%************** [down]

Similar to the untwisted \ytran\ discussed in
subsection\rw{ss.ttran}, the category
%$\rDsprfUW\ytrnL$
\rx{bae1.86a}
has an
alternative `intrinsic' description in terms of objects of $\rDsprfUW$. Consider the case when $\bnB$
is a  line bundle $\vcL$.
%The bundle $\vcL$ also determines another category $\rDsprfmUWL$.
Let
$\ycL=\big( \xbOmbv{\vcL},\dlb\big)$ be the \xper\ \ZtDGM\
corresponding to $\vcL$.
%The bilinear function $\tcWq$
The self-duality of $\vcL$
determines an isomorphism $\fvL$ between $\ycL\otimes\ycL$
and the structure sheaf $\strsU$:
\ee
\xlabel{bae1.87}
\fvL\in\Exttev\brb{\ycL\otimes\ycL,\strsU}.
\eee
An object of
$\rDsprfUW\ztrnL$
%$\rDsprfmUWL$
is a pair $\yobcE$, where $\ycE=\adgmE$
is a \xper\ \ZtDGM\rx{bae1.44} and the extension
$\fvcE\in\Exttod(\ycE,\ycE\otimes\ycL)$ satisfies the property
\ee
\xlabel{bae1.88}
\fvL\circ\fvcE\circ\fvcE = \xIdv{\ycE}.
\eee
Morphisms between two objects $\yobcE,\yobcEp\in\rDsprfmUWL$
are morphisms $g\in\Extub(\ycE,\ycEp)$ which intertwine the
extensions: $g\fvcE = \fvcEp g$.
We conjecture that the categories
%$\rDsprfLWQ$
$\rDsprfUW\ytrnL$
and $\rDsprfUW\ztrnL$
are equivalent.
%
%Let $p\colon\vcL\rightarrow\xcA$ be a projection
%of the bundle $\vcL$ onto its base. For a vector bundle $\xE$ on
%$\xcA$ let $p^\ast(\xE)$ denote its pull-back onto the total space
%of $\vcL$. The bilinear function $\tcWq$ determines a linear
%function $\ttcWq$ on the fibers of $p^\ast(\xL)$ by pairing an
%element of the fiber of $p^\ast(\xL)$ with the $\xL$ fiber part of the
%base point. Now to an object $\yobcE$, where $\ycE=\adgmE$,
%the equivalence functor associates a
%\xper\ \ZtDGM\ represented by a pair
%$\brB{p^\ast(\xE), p^\ast(\nbbE)+\ttcWq\brb{p^\ast(\fvcE)}}$.

%************************ [up]

\subsubsection{A holomorphic fibration with a lagrangian base as
a geometric object}

Let $\cnKU$ denote the canonical line bundle of a complex manifold
$\xcA$: $\cnKU\edfn \wdtpv{\TsU}$.
Let $\yY\subset\xX$ be a \hlgrsm, that is, $\yY$ is a
holomorphic submanifold of $\xX$, such that $\dim_{\IC} \yY =
\shlf\dim_{\IC}\xX$ and $\xom|_{\yY}=0$.
We are going to consider `geometric' objects of $\ctLLXsom$ which are pairs $\goYL$, where $\ycY$ is
a fibration
\ee
\label{bae1.90}
\xymatrix@C=1.5pc@R=1.5pc{
\yZ \ar[r] & \ycY \ar[d]^{\ypY} &
\\
& \yY \ar@{}[r]|<<<{\subset} & \xX
}
%\xfbrv{\yZ}{\ycY}{\yY}{\ypY}
\eee
with a \hlgr\ base $\yY$, and $\vcLY$ is a holomorphic line bundle
on $\ycY$, whose square is the pull-back of the canonical bundle of $\yY$:
$\vcLY^{\otimes 2} = \ypY^\ast\cnKY.$% \edfn \wdtpv{\Ts\yY}$.

A particularly simple type of a holomorphic fibration\rx{bae1.90}
is a \opfib
\ee
\label{bae1.90b} %*r
\TrfY = %(\DlX\rightarrow\DlX),
\vcenter{
\xymatrix@C=1.5pc@R=1.5pc{
\yY \ar[d]
\\
\yY \ar@{}[r]|<<<{\subset}
&
\xX
}
}
\eee
Unless there is a danger of confusion, we denote such a fibration
simply as $\yY$. The pairs $\gYL$, where $\vLY\rightarrow\yY$ is a line bundle
such that $\yY\ott=\cnKY$,
are the simplest objects of the
type $\goYL$.

%objects $\goYL$ is the

\subsubsection{Morphisms between geometric objects}

We say that two holomorphic submanifolds $\yYo,\yYt\subset\xX$
have a \emph{\gdint}, if any point $x\in\yYo\cap\yYt$ has an open
neighborhood $U_x$ which is isomorphic to a neighborhood of the
origin of the tangent space $\Tngx\xX$ in such a way that $\yYo$
and $\yYt$ correspond to their tangent spaces
$\Tngx\yYo,\Tngx\yYt\subset\Tngx\xX$.  This condition
guarantees that the intersection $\yYoct\edfn\yYo\cap\yYt$ is a disjoint union of holomorphic
submanifolds of $\xX$.

Define the $\xX$-product of two fibrations $\ycYo$ and $\ycYt$ as
a fibration over the intersection of their bases:
%\goYLi$, $i=1,2$ as
%
\ee
\xlabel{bae1.91}
\ycYo \ycapX \ycYt \edfn \ycYo|_{\yYoct} \ycapv{\yYoct}
\ycYt|_{\yYoct},\qquad
\xumap{\ycYo\ycapX \ycYt}{\xcpot}{\yYoct}.
\eee
There are obvious projections
\ylee{bae1.19b} %*r
\xymatrix@C=0.5pc{
& \ycYo \ycapX \ycYt \ar[dl]_{\pi_1} \ar[dr]^{\pi_2}
\\
\ycYo|_{\yYoct} &&\ycYt|_{\yYoct}
}
\yeee

Suppose that
the lagrangian bases $\yYo,\yYt\subset\xX$ of the fibrations of
two objects $\goYLo$, $\goYLt$
%$\ycYo$ and $\ycYt$
%the lagrangian bases $\yYi$  of two objects $\goYLi$, $i=1,2$
have a clean intersection, so the intersection $\yYot$ is a
complex manifold. The line bundle
\xlee{bae1.19b1}
\vcLot \edfn \pi_1^\ast(\vcLYo|_{\yYoct})\otimes
\pi_2^\ast(\vcLYt|_{\yYoct})
\otimes\xcpot^\ast(\cnKYot^{-1})
\xeee
is \sd. Indeed, for $i=1,2$ we have $\xcpot= \ypYi\circ\pi_i$, and since
$\vcLYi^2 = \ypYi^\ast\cnKYi$, the square of $\vcLot$ can be
presented as the pull-back of the product of canonical bundles:
\xlee{bae1.19b2}
\vcLot^2 =
\xcpot^\ast\brb{\cnKYo|_{\yYoct}\otimes\cnKYt|_{\yYoct}\otimes\cnKYot^{-2}}.
\xeee
Now consider
%The line bundle
%$\vcLot \edfn\vcLYo|_{\yYoct}\otimes \vcLYt
%|_{\yYoct}\otimes\cnKYot^{-1}$ is \sd. Indeed, consider
the quotient bundles
\xlee{bae1.90a} %*r
%\vcLi = (\Tng\yYi|_{\yYoct})/\Tng(\yYoct).
\vccLi = \Tng\yYi|_{\yYoct}/\Tng\yYoct,\qquad i = 1,2.
\xeee
The \hlsm\ form $\som$ produces a non-degenerate pairing between
$\vccLo$ and $\vccLt$, so their top exterior powers are dual to
each other and, as a result, the tensor product
$\wdtpv{\vccLo}\otimes\wdtpv{\vccLt}$ is a trivial line bundle. At the same
time,
$\cnKYi|_{\yYoct}=\wdtpv{\vccLdi}\otimes\cnKYot$, so
$\cnKYo|_{\yYoct}\otimes\cnKYt|_{\yYoct} = \cnKYot^2$ and the
tensor square\rx{bae1.19b2} is trivial, that is, $\vcLot$ is \sd.

Having established the self-duality of $\vcLot$, we propose
that the category of morphisms between the objects, whose
lagrangian bases have a \gdint, is the shifted 2-periodic derived category
of the $\xX$-product $\ycYo \ycapX \ycYt$:
\ee
\label{bae1.92} %*r
\!\!\!\!\!\!\!\!\!\Hom_{\ctLLX}\brB{\goYLo,\goYLt} =
\rDsprfcYoct\,\ytrnLot\,\btrnv{\shlf\dim \xX - \dim\yYot-1}.
%\rDsprfmv{\ycYo \ycapX \ycYt}{\vcLot}.%\vcLYo\oplus\vcLYt\oplus\vcLo},
%\rDprf\brb{\ycYo \ycapX \ycYt },
\eee
Roughly speaking, the category of morphisms
$\Hom_{\ctLLX}\brB{\goYLo,\goYLt}$ is the 2-perio\-dic derived category of
coherent sheaves on the product $\ycYoct$. The origin of the shifts
in the \rhs of \ex{bae1.92} will become clear when we compare
\eex{bae1.92} and\rx{bae1.58}.

In the special case when $\ycYo$ and $\ycYt$ are \opfib
s\rx{bae1.90} with the same base $\yYo=\yYt=\yY$ and the accompanying line bundles
are the same, the formula\rx{bae1.92} becomes
\xlee{baj3.1}
\End_{\ctLLX}(\yY,\vcLY) = \rDprfY.
\xeee

\subsubsection{Composition of morphisms}

We describe the geometric composition of morphisms under the
simplifying assumption that
the lagrangian bases $\yYi$ of the fibrations $\ycYi$, $i=1,2,3$ are
\fCY\ and the accompanying line bundles $\vcLYi$ are trivial. This
implies that a clean intersection of two lagrangian submanifolds $\yYij=\yYi\cap\yYj$
is `semi-\fCY', that is, $\cnKYij\ott$ is trivial. Let us make a
stronger assumption that $\yYij$ is \fCY.
%Finally, we assume that $\tdfbtYi=0$ for all $i=1,2,3$.
Then the general
formula\rx{bae1.92} simplifies:
\ee
\label{bae1.93}
\Hom(\ycYi,\ycYj) =
%\rDsprfmv{\ycYo \ycapX \ycYt}{\vcLYo\oplus\vcLYt\oplus\vcLo},
\rDprf\brb{\ycYi \ycapX \ycYj }
\eee
and we suggest that the composition of morphisms is a combination
of derived pull-backs, push-forwards and the tensor product: for two
morphisms $\cE_{12}\in\xMor(\ycY_1,\ycY_2)$ and
$\cE_{23}\in\xMor(\ycY_2,\ycY_3)$
%consider a third morphism $\cE_{13}\in\xMor(\ycY_1,\ycY_3)$ defined by the formula
their composition is
\ee
\label{bae1.97}
\cE_{23}\circ\cE_{12}=
%\cE_{13}=
(\yepi_{13})_\ast\Big(
\yepi_{12}^\ast(\cE_{12})\otimes \yepi_{23}^\ast(\cE_{23})
\Big),
\eee
where $\yepiv{ij}$ are the \xeps\
\ylee{bae1.98}
\xymatrix{ & \ycY_1\ycapX\ycY_2\ycapX\ycY_3
\ar[dl]_{\yepi_{12}}\ar[dr]^{\yepi_{23}} \ar[d]^{\yepi_{13}}
\\
\ycY_1\ycapX\ycY_2 &\ycY_1\ycapX\ycY_3& \ycY_2\ycapX\ycY_3
}
\yeee
%
%If $\tdfbtYi\neq 0$ for some $i$, then the morphism
%categories\rx{bae1.93} have to be deformed as suggested by
%\ex{bae1.92}, and the composition formula\rx{bae1.97} will also
%have to be deformed.

In the special case when all $\ycYi$ are \opfib s\rx{bae1.90} with
the same base
$\yYo=\yYt=\yYh=\yY$, their categories of morphisms are given by
\ex{baj3.1} and the composition rule\rx{bae1.97} reduces to the
tensor product within $\rDprfY$: for $\cE,\cE\p\in\rDprfY$
\xlee{baj3.2}
\cE\circ\cE\p = \cE\otimes\cE\p.
\xeee

\subsection{Holomorphic lagrangian correspondences and
the  3-category of \hlsmm s}

In this subsection we describe part of the 3-category of RW models in geometric terms.
Throughout this subsection we will ignore the line bundles $\vcLY$
in the definition of the objects $\goYL$ of $\ctLLXsom$. To be
more precise, we may assume that all complex manifolds appearing
here are \fCY\ and all these bundles are trivial. Moreover, we
assume that all intersections are \xgd.

The statements of this subsection apply when $\Xsom$
are cotangent bundles: $\xX=\TsU$, but in the case of \opfib s the
statements apply to general \hlsmm s $\Xsom$.

\subsubsection{Lagrangian correspondence 2-functors}

Let us forget for a moment that the complex manifold $\xX$ has a
\hlsm\ structure and that the base $\yY$ of a fibration $\ycY$ must
be lagrangian. Then we may  define pull-back and
push-forward functors associated with a holomorphic map
$\mF\colon \xX\longrightarrow\xXp$. A pull-back of a fibration
$\ycYp\rightarrow\yYp\subset\xXp$ is a fibration $\mFua(\ycYp)\rightarrow \mF^{-1}(\yYp)$
constructed by pulling back $\ycYp$ by the restriction
$\mF|_{\mF^{-1}(\yYp)}$. In order to define a push-forward of a
fibration $\ycY\rightarrow\yY\subset\xX$ we assume that the
restricted map $\mF|_{\yY}\colon\yY\rightarrow\mF(\yY)$ represents a
holomorphic fibration. Then we define the fibration $\mFda(\ycY)$
as $\ycY\rightarrow\mF(\yY)$, whose projection is the
composition of projections
$\ycY\rightarrow\yY\rightarrow\mF(\yY)$.

A holomorphic fibration $\ycYaot\in\ctLLcrot$ determines a lagrangian correspondence 2-functor
\ee
\label{bae1.107}
%\xPhYot\colon\Fibv{X_1}\rightarrow\Fibv{\xX_2}
\xymatrix@C=1.5cm{
\ctLLXsomo \ar[r]^-{\dPhYot} & \ctLLXsomt
}
\eee
defined by the formula
\ee
\label{bae1.108} %*r
\dPhYot= (\yepi_2)_\ast\Big(\ycYaot\ycapXott
\yepi_1^\ast
\Big),
\eee
where $\yepi_1$ and $\yepi_2$ are projections onto the factors
of the product $\xX_1\times\xX_2$:
\ee
\label{bae1.109} %*r
\xymatrix@C=0.5pc{
& \xX_1\times\xX_2 \ar[dl]_{\pi_1} \ar[dr]^{\pi_2}
\\
\xX_1 &&\xX_2
}
\eee
The action of the 2-functor\rx{bae1.108} on the bases of holomorphic
fibrations is described by a simple set-theoretic formula: if
$\ycYt=\xPhYot(\ycYo)$, then
\xlee{bah10.1}
\yYt = \yepit\brb{\yYot\cap\yepio^{-1}(\yYo)},
\xeee
where $\yYo$ and $\yYt$ are the bases of the corresponding
fibrations.

Although the operations $\yepi_1^\ast$ and
$(\yepi_2)_\ast$ do not correspond to well-defined 2-functors for 2-categories $\ctLL$,
their composition\rx{bae1.108} is \xwd, because if the base $\yYo$
of the fibration $\ycYo$ is lagrangian then so is the base $\yYt$
of its image determined by \ex{bah10.1}.

%. In particular, if the base
%$\yYo\subset\xXo$ of a fibration $\ycYo\rightarrow\yYo$ is
%lagrangian, then so is the base
%%
%\xlee{bah10.1x}
%\yYt = \yepit\brb{\yYot\cap\yepio^{-1}(\yYo)}
%\xeee
%%
%of the fibration
%$\ycYt=\xPhYot(\ycYo)$.

The \opfib\
\ee
\label{bae1.109b} %*r
\TrfDlX = %(\DlX\rightarrow\DlX),
\vcenter{
\xymatrix@C=1.5pc@R=1.5pc{
\DlX \ar[d]
\\
\DlX \ar@{}[r]|<<<{\subset}
&
\xX\times\xX
}
}
\eee
over the diagonal
$\DlX\subset\Xsomm\times\Xsom$ determines the identity endo-2-functor
$\dPhv{\TrfDlX}$ of the category $\ctLLXsom$.
%%
%\ee
%\label{bae1.109a}
%\xPhv{\TrfDlX}\in\End\DlX.
%\eee
%%

\subsubsection{A geometric description of the 3-category $\ctLLL$}

As we mentioned in subsection\rw{ss.hcatintr}, objects of the
3-category $\ctLLL$ are \hlsmm s $\Xsom$. The duality functor $\hve$
switches the sign of the symplectic form: $\zdlv{\Xsom} = \Xsomm$,
and we define the 2-category of morphisms between two objects in
accordance with the general formula\rx{bag1.4}:
\ee
\xlabel{bae1.110}
\Hom_{\ctLLL}\brB{\Xsomo,\Xsomt } = \ctLLcrot.
\eee

The composition of morphisms represented by holomorphic fibrations
with lagrangian bases
%Let us discuss the properties of morphisms represented by
%holomorphic fibrations with lagrangian bases.
%
%%\subsubsection{A 3-category of \hlsmm s}
%The 2-categories $\ctLLXsom$ can be combined into a single
%3-category of \hlsmm s $\ctLLL$. A \zobj\ of $\ctLLL$ is a
%\hlsmm\ $\Xsom$ or, equivalently, its category $\ctLLXsom$.
%Lagrangian correspondences form a 2-category of morphisms between
%\zobj s:
%%
%\ee
%\xlabel{bae1.110}
%\Hom\brB{\Xsomo,\Xsomt } = \ctLLcrot.
%\eee
%%
%The composition of morphisms
is defined in such a way that it
would agree with the composition of their correspondence
functors\rx{bae1.108}: the composition of two morphisms
\ee
\xlabel{bae1.111}
\ycYaot\in\Hom\brB{\Xsomo,\Xsomt}
,\qquad\ycYath\in\Hom\brB{\Xsomt,\Xsomh}
\eee
is
\ee
\xlabel{bae1.112}
\ycYath\circ\ycYaot = (\yepi_{13})_\ast\Big(
\yepi_{12}^\ast(\ycYaot)\ycapv{\xXo\times\xXt\times\xXh} \yepi_{23}^\ast(\ycYath)
\Big),
\eee
where $\yepi_{ij}$ are the projections
\ee
\xlabel{bae1.113}
\xymatrix{ & \xXo\times\xXt\times\xXh
\ar[dl]_{\yepi_{12}}\ar[dr]^{\yepi_{23}} \ar[d]^{\yepi_{13}}
\\
\xXo\times\xXt &\xXo\times\xXh& \xXt\times\xXh
}
\eee

The identity endomorphism $\xIdXsom\in\End\Xsom$ is
the \opfib\rx{bae1.109b}:
\ee
\label{bae1.113a} %*r
\xIdXsom = \TrfDlX.
\eee

As we mentioned in subsection\rw{ss.intr},
the 3-category $\ctLLL$ has a symmetric monoidal structure:
%. The
%product of \zobj s comes from the product of the \hlsmm s and the
%sum of symplectic forms
%
\ee
\xlabel{bae1.114}
\Xsomo\times\Xsomt = \brb{\xXo\times\xXt,\yepiuo(\somo) +
\yepiut(\somt) },
\eee
where $\yepio$ and $\yepit$ are projections of the
diagram\rx{bae1.109}. The
unit object is the \hlsmm\ $\xXopt$ consisting of a single point.
%The duality functor\rx{bag1.3} acts by switching the sign of the
%symplectic form: $\zdlv{\Xsom} = \Xsomm$.

%\xempt\ \zobj\ $\zOnthz$ is the \hlsmm\ $\xXopt$ consisting of
%a single point:
%%
%\ee
%\xlabel{bae1.115}
%\zOnthz = \xXopt.
%\eee
%%

\subsection{A relation between the geometric and the algebraic
descriptions}
\label{ss.reltan}

We have already outlined the equivalence of categories\rx{bah5.1}
%
%%\subsubsection{Cotangent bundle}
%Let us explain the equivalence of categories
%%
%\xlee{eeq3}
%\rDDprfUSk = \ctLLTUSk.
%\xeee
%
%%
%\xlee{eeq2}
%\rDDprfaU \cong \ctLLTsU
%\xeee
%%
in subsection\rw{ss.relout}.
Here we provide a more detailed description of the equivalence
\xxtf\ $\etfe$.
%account of this relation.
%
%
%
%First of all,

\subsubsection{Localization}
\label{ss.lcl}
Let us split the \crvng\ $\xcW$, which determines the category
$\rDsprfUW$, into zero degree and positive degree parts:
\ee
\label{bae1.89a} %*r
\xcW=\xcWz + \xcWpl,\qquad\xcWz\in\xbOmz(\xcA),\qquad
\xcWpl\in\bigoplus_{i\geq 1}\xbOmv{i}(\xcA).
\eee
We say that the set of critical points of $\xcWz$
\ee
\xlabel{bae1.89a1}
\xcAWz = \{ x\in\xcA\,|\, d\xcWz(x)=0\}
\eee
is \emph{\ygd} if it is
%form
a smooth holomorphic submanifold of $\xcA$ and
the quadratic form induced by
the Hessian of
$\xcWz$ (that is, by the quadratic form of its second derivatives)
on the normal bundle $\Nrm\xcAWz$ is non-degenerate.
The non-degenerate Hessian gives rise to an $\OnC$ structure on
$\Nrm\xcAWz$,
%This non-degeneracy implies that the normal bundle $\Nrm\xcAWz$ is \sd\
and we conjecture the following equivalence of categories:
\ee
\label{bae.89a2}
\rDsprfUW =
\rDsprfmvv{ \xcAWz }{ \xcW|_{\xcAWz} }{\Nrm\xcAWz}.
\eee
In other words, we expect that the category $\rDsprfUW$
localizes to a small tubular neighborhood of $\xcAWz$ and that if
the dominant terms in the expansion of $\xcWz$ in the normal
directions to $\xcAWz$ are quadratic then lower degree terms do
not matter.

Since $\xcWz$ is locally constant on its critical set, the
relation\rx{bae.89a2} implies the following category equivalence:
\ee
\label{bae.89a2x}
\rDsprfaUW =
\rDsprfmvv{ \xcAWz }{ \xcWpl|_{\xcAWz} }{\Nrm\xcAWz}.
\eee
Note that we did not have to \xau\ the \rhs\ category, because
the connected parts of the critical set $\xcAWz$ will contribute
to it only when the constant value of $\xcW$ on them is zero.

In view of \ex{baf1.2}, the equivalence\rx{bae.89a2x} can be
simplified:
\xlee{bae.89a2y}
\rDsprfaUW = \rDsprfmvv{\xcAWz}{ \xcWpl|_{\xcAWz} }{
\wdtpv{\Nrm\xcAWz} }
\btrnv{\rnk\Nrm\xcAWz-1}.
\xeee
Finally, if $\xcWpl=0$, that is, if $\xcW$ is just a holomorphic function on $\xcA$,
then the equivalence takes the form
\ee
\label{bae.89a2x1} %*r
\rDsprfaUW =
\rDprf\lrbc{\xcAW}\ytrnv{\wdtpv{\Nrm\xcAW}}\btrnv{\rnk\Nrm\xcAW-1}.
%,\qquad
%\rDsprfaUW = \rDprf(\xcAWz),
\eee

\subsubsection{The relation between objects}
\label{ss.relobj}
Choose a
line bundle $\vcLz\rightarrow \xcA$ such that $\vcLz\ott=\cnKU$.
%In order to explain the relation\rx{bae1.82} between two
%2-categories, we set $\xX=\TsU$.
For simplicity we will consider only curved fibrations
$\cmfcUW\in\rDDprfaU$, for which $\xcWpl=0$ (see \ex{bae1.89a}),
that is, $\xcW$ is just a holomorphic function on $\xcU$. Recall that $\xcU$ is a
fibration\rx{bae1.55}: $\xcp\colon\xcU\rightarrow\xcA$. For a point $\xfu\in\xcU$ let
$\xfVu\subset\xcU$ be the fiber to which $\xfu$ belongs: $\xfVu=\xcp^{-1}(\xcp(\xfu))$.
Let
$\crcUW$ be the set of `fiber-critical' points of
$\xcW$: $\crcUW = \{\xfu\in\xcU\,\,|\,\,\,\del\xcW|_{\Tng_{\xfu}\xfVu} = 0
\}$. In other words, there is an exact sequence
$\Ts_{\xcp(\xfu)}\xcA\xrightarrow{a}\Ts_{\xfu}\xcU\xrightarrow{b}\Ts_{\xfu}\xfVu$
and $\xfu\in\crcUW$ if $b\big(\del\xcW(\xfu)\big) = 0$. The latter
condition means that
$\del\xcW(\xfu)$ is in the image of $a$. %, hence
% an element $a^{-1}\big(\del\xcW(\xfu))$ exists.
Hence
there is a map $\mpf\colon\crcUW\rightarrow\TsU$
such that $\mpf(\xfu) =
a^{-1}\big(\del\xcW(\xfu))\in\Ts_{\xcp(\xfu)}$.
The \emph{\spprt} of the object $\cmfcUW$ is defined to be the
image of $\mpf$, and we denote it suggestively
as $\yYUW$:
\xlee{bah8.2}
\yYUW=\xsupp\cmfcUW \edfn \mpf(\crcUW)\subset\TsU.
\xeee
%
%and we denote it suggestively as $\yYW$: $\yYW\edfn\xsupp\cmfcUW$.

%Let $\yYW\subset\TsU$ denote the image of
%$\mpf$: $\yYW \edfn \mpf(\crcUW)$. We call $\yYW$ the \emph{\spprt}
%of the object $\cmfcUW$: $\xsupp\cmfcUW \edfn \yYW$.

Generally, $\yYUW$ is an isotropic submanifold with respect to the
symplectic structure $ \TsU$. Assume that for all $x\in\xcA$, the critical locus of the
function $\xcW|_{\xcp^{-1}(x)}$
%restricted to the fiber $\xcp^{-1}(x)$
is \xgd\ and that
$\crcUW$ is a complex manifold. Then $\yYUW\subset\TsU$ is a
lagrangian submanifold and the map $\mpf\colon\crcUW\rightarrow\yYUW$ is a
fibration. Let $\bnB_x$ denote the  normal bundle to the
critical set of $\xcW$ restricted to the fiber $\xcp^{-1}(x)$. This bundle
has an $\OnC$ structure given by the Hessian of $\xcW|_{\xcp^{-1}(x)}$, and all
these bundles together form a holomorphic $\OnC$ bundle $\bnB$ over
$\crcUW$.
Thus we define the action of the equivalence \xxtf\rx{bah5.1} on
the curved fibration $\cmfcUW$ as follows:
\ylee{bah5.2}
\etfe\cmfcUW = (\crcUW, (\xcp^\ast\vcLz)|_{\crcUW}\otimes\wdtpv{\bnB})\ttrnv{\rnk \bnB
-1},
\yeee
that is, the pair in the \rhs of this equation
%
%The pair
%$$(\crcUW, (\xcp^\ast\vcLz)|_{\crcUW}\otimes\wdtpv{\bnB})\ttrnv{\rnk \bnB -1}$$
is the object of $\ctLLTsU$ corresponding to the object
$\cmfcUW\in\rDDprfaU$.

%Fix a fiber $\xcp^{-1}(x)$ and let $\bnB_x$ denote the
%normal bundle to the critical set of
%
%The normal spaces (within the fibers of
%$\xcp\colon\xcU\rightarrow\xcA$) to the points of the critical set

% There is a map $\mpf\colon\crcUW\rightarrow\TsU$ which maps a
%point $\xfu\in\crcUW$ into an element of $\Ts_{\xcp(\xfu)}\xcA $
%determined by $\del\xcW(\xfu)$.
%%
%%For $\xfu\in\crcUW$, $\del\xcW(\xfu)$ determines an element
%%in $\Ts_{\xcp(\xfu)}\xcA$, so there is a map $\mpf\colon \crcUW\rightarrow \TsU$
%%such that $\mpf(\xfu) =

The geometric object corresponding to the pair $\cmfcUW$ is particularly simple, if $\xcU$ is a \opfib\ over $\xcA$. Then
$\xcW$ is just a holomorphic function on $\xcA$ and the object of
$\ctLLTsU$ corresponding to $\xcW$ is the pair $(\yYW,\zzf^\ast\vcLz)$, where $\yYW$ is the
graph of $\del\xcW$:
%$\yYW = \{ (x,p)\in\TsU\,|\, p = \del\xcW|_x \}$
%
\xlee{bah1.1}
\yYW = \{ p\in\Ts_x\xcA\,|\,x\in\xcA,\; p = \del\xcW|_x \}
\xeee
and
$\zzf\colon\yYW\rightarrow\xcA$ is the restriction of the
projection $\TsU\rightarrow\xcA$ to $\yYW$ (it establishes the
isomorphism between $\yYW$ and $\xcA$ as complex manifolds).

\subsubsection{The relation between categories of morphisms}

We will compare the categories of morphisms within 2-categories $\rDDprfaU$ and
$\ctLLTsU$ only for the simplest objects. Let $\xcWo$ and
$\xcWt$ be holomorphic functions on $\xcA$ such that their
difference $\xcWot=\xcWt-\xcWo$ has a \xgd\ set of critical
points. This is equivalent to saying that $\yYWo$ and $\yYWt$ have a \gdint.

According to the definition\rx{bae1.63ap} and the
equivalence\rx{bae.89a2x1}, the category of morphisms within
$\rDDprfaU$ is
\xlee{baf2.1}
\Hom_{\rDDprfaU}(\xcWo,\xcWt)   = %\rDsprfUWtmo
\rDsprfaUWot =
\rDprf\lrbc{\xcAWot}\ytrnv{\Nrm\xcAWot}\btrnv{\rnk\Nrm\xcAWot-1}.
\xeee
At the same time, according to \ex{bae1.92},
\xlee{baf2.2}
\Hom_{\ctLLTsU}\brB{\yobYWo,\yobYWt} = \rDsprfYot\ytrnLot
\btrnv{\dim\xcA - \dim\yYot - 1},
\xeee
where $\yYot \edfn \yYWo\cap\yYWt$, the maps
$\zzfi\colon\yYWi\rightarrow\xcA$ are the restrictions of the
projection $\TsU\rightarrow\xcA$ and
\xlee{baf2.3}
\vcLot = (\zzf_1^\ast\vcLz)|_{\yYot}\otimes(\zzf_2^\ast\vcLz)|_{\yYot}\otimes
\cnKYot^{-1}.
\xeee
%
%The deformation parameter $\xdfmot$ is absent,
%because the class $\tdfbtY$ is zero for a lagrangian submanifold
%$\yYW\subset\TsU$.

The maps $\zzfo$ and $\zzft$, as well another projection restriction
$\zzfot\colon \yYot\rightarrow \xcAWot$, establish isomorphisms
between the corresponding complex manifolds.
Therefore, there is an equivalence of categories
$$\rDsprfYot\ytrnLot =
\rDsprfv{\xcAWot}\ytrnv{\zzfv{12,\ast}\vcLot},$$ and according to
\ex{baf2.3}, the push-forward of the line bundle $\vcLot$ is
$$\zzfv{12,\ast}\vcLot = \vcLz^2\otimes\cnKYot^{-1}= \cnKU\otimes\cnKYot^{-1}
=\wdtpv{\Nrm\xcAWot}.$$
Since
$$ \rnk\Nrm\xcAWot = \dim\xcA - \dim\xcAWot,\qquad \xcAWot=  \yYot,$$
we established the equivalence of categories\rx{baf2.1}
and\rx{baf2.2} provided by the \xxtf\ $\etfe$.

\subsubsection{The relation between 2-functors}

Let $\som$ be the canonical symplectic form of the cotangent
bundle $\TsU$. The symplectomorphism $\xtsm\colon
(\TsU,\som)\rightarrow(\TsU,-\som)=\zdlv{(\TsU)}$ reverses the cotangent
vectors: for $p\in\Ts_{q}\xcA$ we define $\xtsm(p)\edfn-p$.
For
two cotangent spaces we define the symplectomorphism
$\xtsmo\colon(\TsUo,\somo)\times(\TsUt,\somt)\rightarrow
(\TsUo,-\somo)\times(\TsUt,\somt)$ which reverses the orientation
of the first cotangent vector: $\xtsmo\edfn\xtsm\times \xId$.

The symplectomorphism $\xtsm$ acts as a \xxtf\
$\ddxtsm\colon\ctLLTsU\rightarrow\ctLL\zdlv{(\TsU)}$ by transforming the bases
of holomorphic fibrations $\ycY\rightarrow\yY\subset\TsU$.
%on holomorphic fibrations of $\ctLLTsU$ by
%transforming their bases as subspaces of $\TsU$.
Similarly, $\xtsmo$ acts as a \xxtf\ $\ddxtsmo\colon\ctLL(\TsUo\times\TsUt)\rightarrow
\ctLL\brb{\zdlv{(\TsUo)}\times\TsUt}$. Now consider a composition
of 2-functors
\ylee{bah7.1}
\ddxtsmo\circ\etfeot\colon\rDDprfav{\TsUo\times\TsUt}\longrightarrow
\ctLL\brb{\zdlv{(\TsUo)}\times\TsUt}.
\yeee
We leave it for the reader to check that the 2-functors\rx{bae1.69} and\rx{bae1.107}
coming from the objects related by\rx{bah7.1} are intertwined by
\xxtf s $\etfe$, that is, the diagram
\ylee{bah7.2}
\xymatrix@C=3cm@R=1.2cm{
\rDDprf(\xcAo) \ar[d]^{\etfeo} \ar[r]^-{\xPhUWot} & \rDDprf(\xcAt)
\ar[d]^{\etfet}
\\
\ctLLv{(\TsUo)} \ar[r]^-{\dPhYot} & \ctLLv{(\TsUt)}
}
\yeee
is commutative, if $\ycYaot = \ddxtsmo\circ\etfeot \cmfcUW$. The
easiest part of this commutativity is the verification that
the support of a \xcfb  from $\rDDprf(\xcAo)$ is transformed as in
\ex{bah10.1}:
\xlee{bah7.3}
\yYv{\xPhUWot\cmfcUWo} =
\yepit\brb{
\yYv{\cmfcUWot}\cap
\yepio^{-1}\brb{\yYv{\cmfcUWo}}
},
\xeee
where $\yepio$ and $\yepit$ are the projections
\ylee{bah7.4}
\xymatrix@C=0.5pc{
& \TsUo\times\TsUt \ar[dl]_{\pi_1} \ar[dr]^{\pi_2}
\\
\TsUo &&\TsUt
}
\yeee
%

%\section{Algebraic and geometric descriptions of the 3-category of cotangent bundles}

%\subsection{Deformed local version}
%\label{ss.dlv}
%\label{ss.3catdfrm}

\section{The 2-category of a \zdctngb}
\label{s.sct5}
\label{ss.dlv}
\label{ss.3catdfrm}

\subsection{Outline}
\label{ss.outline}

The geometric description of the 2-category $\ctLLXsom$ provided
in subsection\rw{ss.tcthlsmm} is completely correct only when
$\Xsom$ is a cotangent bundle: $\xX=\TsU$. However, path-integral
considerations\cx{KRS1} suggest that the main feature of that
description holds true for a general \hlsmm\ $\Xsom$: a pair $\gYL$, where
$\yY\subset \xX$ is a lagrangian submanifold and
%a holomorphic fibration with a lagrangian base\rx{bae1.90} and
the
line bundle $\vLY$ is a square root
%of the pull-back
of the
canonical bundle of $\yY$, always represents an object in
$\ctLLXsom$.

Path integral arguments\cx{KRS1} also suggest that the 2-category
$\ctLLXsom$ is local. From the physics perspective, locality
means that there are no instanton corrections to the path
integral, that is, there are no 3-dimensional analogs of
A-model holomorphic disks which lie at the heart of the Floer homology and the
Fukaya category. From the mathematical perspective, locality means
that the category of morphisms between two objects of $\ctLLXsom$
is determined by the structure of $\xX$ in the formal neighborhood of
the intersection of their \spprt s, that is,
$\Hom_{\ctLLXsom}\brb{\gYLo,\gYLt}$ is determined by the formal neighborhood of $\yYo\cap\yYt\subset \xX$ (when $\xX$ is
a cotangent bundle and the intersection of supports is \xgd,
the category of morphisms is determined just by the intersection itself, as
suggested by \ex{bae1.92}). The composition of
morphisms is determined by the structure of $\xX$ near the triple
intersection of supports (\cf \ex{bae1.97}).

Let $\yY\subset\xX$ be a lagrangian submanifold of a general
\hlsmm\ $\Xsom$. Locality of $\ctLLXsom$ means that if we knew the
structure of the 2-category $\ctLL$ of the formal neighborhood of
$\yY$, we would know exactly all categories of morphisms involving
the objects
$\gYL$
%$(\ycY\rightarrow\yY,\vcLY)$
as well as the compositions of
such morphisms.

In real symplectic geometry, a sufficiently small tubular
neighborhood of a lagrangian submanifold $\yY$
%has Darboux coordinates, that is, such neighborhood
is symplectomorphic to a tubular neighborhood
of the zero section of the cotangent bundle $\TsY$. However, in
holomorphic case this is no longer true: there may be non-trivial
deformations of the holomorphic symplectic structure of (the formal
neighborhood of the zero section) of the cotangent bundle $\TsY$ and the
formal neighborhood of $\yY\subset\xX$ may be isomorphic to a
deformed formal neighborhood of the zero section of $\TsY$. Therefore, in
order to gain information about the morphisms involving the
object
$\gYL$
%$(\ycY\rightarrow\yY,\vcLY)$
of $\ctLLXsom$
%with a given lagrangian base $\yY$,
we have to explore the 2-category
corresponding to a deformed cotangent bundle $\mtdfTYSk$, where $\hdf$ is a
deformation parameter of the holomorphic symplectic structure of $\TsY$.

To understand the 2-category $\mtdfTYSk$ we follow the
algebraic approach:
%We study the 2-category of $\mtdfTYSk$ by algebraic approach:
for a complex manifold $\xcA$, which plays the role of $\yY$,
we construct a deformation $\rDDprfUSk$ of the 2-category
$\rDDprfU$. The construction of this deformation is based on two
path-integral-motivated assumptions: the deformation parameter $\hdf$ of the 2-category is
the same as the deformation parameter of the holomorphic symplectic structure
of $\TsU$ and the simplest objects of the deformed category
$\rDDprfUSk$ (that is, the objects corresponding to \opfib s and
described by holomorphic functions $\xcW$ on $\xcA$ in the
category $\rDDprfU$)
should be related to the lagrangian submanifolds of the deformed
cotangent bundle $\mtdfTUSk$ in the same way as in the undeformed
case discussed in subsection\rw{ss.reltan}.

\subsection{Differential Lie-Gerstenhaber algebras and the \CMe}

Let us review some well-known facts about algebras
governing the deformations of objects and categories appearing in
this paper.
% appearing in this paper.

\subsubsection{General definitions}

A \DPGa\ $\xdlL$ is a \Ztgrdd\  vector space endowed with
a differential $\pgD$ and a compatible graded Lie bracket $\pgbdd$  which may be even or odd. Let $\pgdg$ be the \Ztdgr\ of the Lie
bracket. If $\pgdg=\yev$ then $\xdlL$ is called a differential Lie algebra. If in addition $\xdlL$ has a supercommutative associative product compatible both with $\pgD$ and the Lie bracket, then $\xdlL$ is called a differential Poisson algebra. If $\pgdg=\yod$ and $\xdlL$ has a supercommutative associative product compatible both with $\pgD$ and $\pgbdd$, then $\xdlL$ is called a \DGa.
%If the \Ztdgr\ of the Lie bracket is odd then $\xdlL$ is called a
%\DPa\ and if it is even then $\xdlL$ is called a \DGa.

The graded Lie bracket of $\xdlL$ descends to its
$\pgD$-homology $\rmH_{\pgD}(\xdlL)$.

If an element $\pga\in\xdlL$ of \Ztdgr\ $\pgdg+1$ satisfies the
\CMe\
\xlee{bcea2.1}
\pgD\pga + \shlf\,\pgbvv{\pga}{\pga} = 0,
\xeee
then the operator
\xlee{bcea2.1a}
\pgDa = \pgD + \pgbvv{\pga}{\cdot}
\xeee
is also a
differential and it determines a deformed \DPGa\ $\xdlLa$. If two
\CMlt s are related by a `gauge transformation' $\pgb$, whose
infinitesimal form is
\xlee{bcea2.2}
\delta\pga = \pgDa\pgb,
%\qquad\text{where $\pgb\in\xdlL$ and  $\dgZt\pgb= \pgdg$,}
\qquad \pgb\in\xdlL,\quad\dgZt\pgb=\pgdg,
\xeee
then the corresponding deformed algebras are isomorphic.

If a \CMlt\ is presented as a formal power series in a parameter
$\dfe$:
\ylee{bcea2.2a}
\pgae =
\sum_{i=1}^\infty \pgai\,\dfe^i,
\yeee
then the leading coefficient
$\pgao$ is $\pgD$-closed and, due to the gauge symmetry,
the corresponding deformation of $\xdlL$ is
determined by its class $\cpgao\in\rmH_{\pgD}(\xdlL)$. The $\dfe^2$ part
of the \CMe\rx{bcea2.1} says that
\ylee{bcea2.3}
\pgD\pgat + \shlf\,\pgbvv{\pgao}{\pgao} = 0,
\yeee
so $\cpgao$ satisfies the condition
\xlee{bcea2.4}
\pgbvv{\cpgao}{\cpgao}= 0. %\qquad\text{in $\rmH_{\pgD}(\xdlL)$}.
\xeee

The importance of \DPGa s stems from the fact that they
appear as deformation complexes of objects in categories and Hochschild complexes of categories and 2-categories.
Equivalence classes of \CMlt s parameterize deformations
of those objects, categories, and 2-categories.
We need three particular examples
:
the differential Lie algebra$\zdrcEnb$ which governs deformations of a \cqhlmvb\ $\adgmE$,
the \DGa\
%$\zdgalWU$
$\zdgaU$ governing \tAinf\ deformations of the
2-periodic derived category
$\rDprfU$ of a complex manifold $\xcA$, and
the differential Poisson algebra $\zdpaXsom$ which, according to
path integral considerations\cx{KRS1}, governs deformations of the
2-category $\ctLLXsom$.

%, a \DGa\
%%$\zdgalWU$
%$\zdgaU$ governing the \tAinf\ deformations of the
%2-periodic category
%$\rDprfU$ of a complex manifold $\xcA$ and

%!!!!!!!!!!!!!!!!!!!!!!!
%$\rDsprfUW$ of a \ccm.
%\tZtgdcs\ $\rDsprfUW$.

\subsubsection{The \DLa\ of a \cqhlmvb}

Let the pair $\adgmE$, where $\nbbE^2=\xcW\xIdv{\xE}$, be a \cqhlmvb\
%(that is, $\nbbE^2=\xcW\xIdv{\xE}$)
defined in subsection\rw{tZtgdcs}.
The corresponding \darc\ $\zdrcEnb$
%\DLa\ of \rlcn s
%Its \DLa\
is the algebra of Dolbeault forms $\xbOmbzeE$ (with total grading),
the differential is the covariant Dolbeault differential $\nbbE$, and the
\xLie\ bracket is the supercommutator:
$\Lbrv{\zdfmo,\zdfmt} \edfn \zdfmo\zdfmt - (-1)^{\zdgt{\zdfmo}\zdgt{\zdfmt} }
\zdfmt\zdfmo$.

A \CMlt\ $\zdfm$ determines a deformed \qhlmvb\
$$\admgEzdfm\edfn(\xE,\nbbE+\zdfm)$$ (in the \rhs of this formula $\zdfm$
denotes an odd bundle map
$\zdfm\colon\xbOmbE\rightarrow\xbOmbE$).

\subsubsection{The \DGa\ of a complex manifold}
The well-known \DGa\ $\zdgaU$ of a complex
manifold $\xcA$ is the algebra $\xbOmbUUWT$, the
\Ztgrdng\ coming from the total degree of forms and
wedge-powers, with the differential $\dlb$ and the \NSb\
$\Lbrv{\cdot,\cdot}$. In fact, $\zdgaU$ has \Zgrdng, and its
degree 2 \CMlt s parameterize deformations of the derived category of
coherent sheaves $\rDbU$. More generally, even \CMlt s of $\zdgaU$
parameterize \tAinf\ deformations of $\rDprfU$.

A holomorphic function $\xcW\in\xbOmz(\xcA)$, $\dlb\xcW=0$
obviously satisfies the \CMe\ and hence determines a deformation of $\zdgaU$
which we denote as $\zdgalWU$. It has a new differential $\dlbW=\dlb+\Lbrv{\xcW,\cdot}$
(\cf \ex{bcea2.1a}). The corresponding deformation of $\rDprfU$ is
the curved category $\rDsprfUW$.

Define a \emph{\xrl} grading on $\zdgaU$ as
\ylee{beq2.1a1}
\dgrl\xbOmv{\ydgn}(\xcA,\wedge^{\ydgm}\TU) = \ydgn-\ydgm,
\yeee
then, obviously,
\ylee{beq2.1a2}
\dgrl \dlb = \dgrl [\cdot,\cdot ] = 1.
\yeee
We say that a \CMlt\ $\xdfm$ is \emph{\xrlnn} if $\dgrl\xdfm\geq
0$ and \emph{\xrlb} if $\dgrl\xdfm=0$.

\subsubsection{The \DPa\ of a \hlsmm}
\label{ss.dpahlsmm}
The
\DPa\ $\zdpaXsom$ of a \hlsmm\ $\Xsom$
%Define the corresponding \DGa\ $\zdpaXsom$
is defined as the algebra $\xbOmbX$
of $(0,\bullet)$-forms on $\xX$ with the differential $\dlb$ and
with the Poisson bracket coming from $\som$. If $\xX=\TsU$, where $\xcA$ is a complex manifold,
then we may consider a simpler version of this algebra, which we
denote as $\zdpaTU$:
it is the algebra of $\SbTU$-valued $(0,\bullet)$ forms
%its space is
$\xbOmbUUST$, where $\Sb$ is the symmetric algebra, and its
differential is $\dlb$. There is a natural injection
\xlee{bcea2.5}
\xbOmbUUST \hookrightarrow \xbOmbUST,
\xeee
which turns an element of $\xbOmbUUST$ into a $(0,\bullet)$ differential form
on $\TsU$ having a polynomial dependence on fiber coordinates and
restricting to zero on all fibers.
The bracket of $\zdpaTU$ is a well-defined restriction of the
Poisson bracket of $\xbOmbUST$, so
%The Poisson bracket of $\xbOmbUST$ has a well-defined restriction on
%$\xbOmbUUST$ which defines the bracket for $\zdpaTU$, and
the injection\rx{bcea2.5} becomes an injection of \DPa s
\xlee{bcea2.6}
\zdpaTU\hookrightarrow \zdpaTsU.
\xeee

\subsubsection{\Ainfalg\ and its modules}

Let us recall the main definitions. An \Ainfalg\ is a \Ztgrdd\
vector space $\aA$ endowed with a series of $\xmn$-linear maps (\xmnmlt
s) $\xmmbf=\xmmz,\xmmo,\ldots$,
\ylee{beq2.1b}
\xmmn\colon \aA^{\otimes \xmn}\rightarrow \aA,\qquad \dgZt\xmmn = \xmn,%\ypn,
\qquad \xmn=0,1,2,\ldots,
\xeee
satisfying the relations
\xlee{beq2.2b}
\sum_{\substack{m_1,m_2,\xmn\geq 0 \\ m_1+m_2+\xmn=N}}
(-1)^{m_1 + \xmn m_2}\,
\xmmv{m_1+m_2+1}(\xId^{\otimes m_1}\otimes\xmmn\otimes\xId^{\otimes
m_2})=0
\xeee
for all $N\geq 0$.
Among other things, these relations indicate
that if $\xmmz=0$, then $\xmmo$ is a differential ($\xmmo^2=0$) and it satisfies the
usual Leibnitz rule with respect to the multiplication $\xmmt$.

If $\xmmz\neq 0$ then the \Ainfalg\ $\aA$ is called \emph{\xwk} or \emph{\xacrv}.

A module over an \Ainfalg\ $\aA$ is a vector space $\xmM$ endowed
with a series of $\xmn$-linear maps (\xmnact s) $\zmmbfM=\zmmMo,\zmmMt,\ldots$
\ylee{beq2.3b}
\zmmMn\colon \xmM\otimes\aA^{\otimes (\xmn-1)}\rightarrow \aA,\qquad \dgZt\zmmMn = \xmn,%\ypn,
\qquad \xmn=1,2,\ldots,
\xeee
satisfying the relations
\begin{multline}
\label{beq2.4b}
\sum_{\substack{m_2,\xmn\geq 0 \\ m_2+\xmn=N}}
(-1)^{\xmn m_2}\,\zmmMv{m_2+1}(\zmmMn\otimes\xId^{\otimes
m_2})
\\
+
\sum_{\substack{m_1\geq 1\\m_2,\xmn\geq 0 \\ m_1+m_2+\xmn=N}}
(-1)^{m_1 + \xmn m_2}\,
\zmmMv{m_1+m_2+1}(\xId^{\otimes m_1}\otimes\xmmn\otimes\xId^{\otimes
m_2})=0
\end{multline}
for all $N\geq 0$.

For two $\aA$-modules $\xmMo$ and $\xmMt$, the vector space
$\tHom(\xmMo,\xmMt)$ is formed by sequences $\zmbf=(\zmfo,\zmft,\ldots) $,
$\zmfn$ being $\xmn$-linear maps
\ylee{beq2.5b}
\zmfn\colon \xmMo\otimes \aA^{\otimes(\xmn-1)}\rightarrow\xmMt.
\yeee
Define a differential $d$ acting on $\tHom(\xmMo,\xmMt)$ by the
following
formula for each term in $d\zmbf$:
\begin{multline}
\label{beq2.6b}
(d\zmbf)_{N} = \sum_{\substack{m_2,\xmn\geq 0 \\ m_2+\xmn=N}}
(-1)^{\xmn m_2}\,
\zmmuMt_{m_2+1} (\zmfn\otimes\xId^{\otimes m_2} )
\\
- \sum_{\substack{m_1\geq 1\\m_2,\xmn\geq 0 \\ m_1+m_2+\xmn=N}}
(-1)^{m_1 + \xmn m_2}\,
\zmfv{m_1+m_2+1}(\xId^{\otimes m_1}\otimes\xmmn\otimes\xId^{\otimes
m_2})
\\
-
\sum_{\substack{m_2,\xmn\geq 0 \\ m_2+\xmn=N}}
(-1)^{\xmn m_2}\,
\zmfv{m_2+1} (\zmmuMo_{\xmn}\otimes\xId^{\otimes m_2} ).
\end{multline}

\subsubsection{A \fDAinfalg\ and the perfect homotopy category of its modules}
Now we adapt the general definitions to the Dolbeault setting.

First of all, note that the Dolbeault algebra
$(\xbOmbU,\dlb)$ associated with a complex manifold $\xcA$ has a
canonical sequence of \xmnmlt s
%$\xmmuzn$
$\xmmbfuz$
which turn it into an
\Ainfalg: $\xmmuzo=\dlb$, $\xmmuzt$ is the wedge-product, and $\xmmuzn=0$ for $\xmn\neq 1,2$.

We define \emph{a \fDAinfalg} (\FDAi) on a complex manifold $\xcA$ as
the space $\xbOmbU$ endowed with \xmnmlt s $\xmmbf$ satisfying the
relation\rx{beq2.2b} and the following restriction on Dolbeault
degrees of their deviation from the canonical \xmnmlt s:
\xlee{beq2.7b}
\dgDlbv{(\xmmn-\xmmuzn)}\geq\xmn,\qquad \any\xmn\geq 0.
\xeee

If $\xdfm\in\xbOmbUUWT$ is a \xrlnn\ \CMlt, then the corresponding
deformation
%$(\xbOmbU,\dlb,\xdfm)$
$\ycdgaBUl$
of the Dolbeault algebra is a
\fDAinfalg, the relation between the components of $\xdfm$ and
\xmnmlt s being fairly complicated. If the deformation parameter
$\xdfm$ is \xrlb, that is,
\ylee{beq2.8b}
\xdfm = \sum_{\xmn=0}^\infty\xdfm_{\xmn},\qquad
\xdfm_{\xmn}\in\xbOmv{\xmn}(\xcA,\wedge^{\xmn}\TU),
\yeee
then the dominant part of the deviations $\xmmn-\xmmuzn$ is
determined by the formula
\xlee{beq2.9b}
(\xmmn-\xmmuzn)(\smu_1,\ldots,\smu_{\xmn})
= \xdfm_{\xmn}\spsmb(\del\smu_1,\ldots,\del\smu_{\xmn}) + \cdots,
\yeee
where $\smu_1,\ldots,\smu_{\xmn}\in\xbOmbU$ and the correction
terms have higher Dolbeault degree than the first term in the
\rhs of this equation.

A simple example of a \xrlb\ deformation of the Dolbeault algebra
$(\xbOmbU,\dlb)$ is $\xdfm = \xcW$, where $\xcW$ is a holomorphic
function on $\xcA$. Then the formula\rx{beq2.9b} has no correction terms and
the deformation results in the \xacrv\
Dolbeault algebra
%$(\xbOmbU,\dlb,\xcA)$
$\ycdgaBU$ discussed already in subsection\rw{tZtgdcs}.

Consider again the Dolbeault algebra $(\xbOmbU,\dlb)$ as a
\fDAinfalg\ with \xmnmlt s $\xmmbfuz$. Its
\xper\ \ZtDGM\rx{bae1.44} with $\nbbE^2=0$ has a canonical
structure of an \tAinf-module if we endow it with
\xmnact s $\xmmbfuzE$ such that $\xmmuzEo=\nbbE$, $\xmmuzEt$ is the
standard multiplication and $\xmmuzEn=0$ for $n>2$. We define a
\xprmi\ over a \fDAinfalg\ as the vector space $\xbOmbE$ endowed
with a sequence of \xmnact s $\zmmbfE$ satisfying the
relations\rx{beq2.4b} and the restriction
\ylee{beq2.10b}
\dgDlbv{(\zmmEn-\xmmuzEn)}\geq\xmn,\qquad \any\xmn\geq 1.
\yeee

Finally, we define the 2-periodic perfect derived category of a \fDAinfalg\ $(\xbOmbU;\xmmbf)$:
its objects are \xper\ \tAinf-modules $(\xbOmbE;\zmmbfE)$
and morphisms between two modules are homologies of the
differential\rx{beq2.6b}.

If $\xdfm\in\xbOmbUUWT$ is a \xrlnn\ \CMlt, then $\rDsprfUd$
denotes the 2-periodic perfect derived category corresponding to the deformed \tAinf-algebra
$\ycdgaBUl$. In other words, $\rDsprfUd$ is the result of deforming $\rDsprfU$
with $\xdfm$. In particular, if $\xdfm=\xcW$, where $\xcW$ is a
holomorphic function on $\xcA$, then $\rDsprfUW$ is the
category\rx{bae1.45}.

\subsection{Deformations of \hlsms s}

\subsubsection{The general case}

Let $\Xsom$ be a \hlsmm. Deformations of its \hlsms\
which preserve the \dR\ cohomology class of $\som$ are
parameterized up to gauge equivalence by \CMlt s of the \DPa\
$\zdpaXsom$ defined in subsection\rw{ss.dpahlsmm}. Namely, if an
element
\ylee{beq2.11b}
\hdf\in\xbOmov{\xX}\subset\zdgav{\xX}
\yeee
satisfies the \CMe
\xlee{bbea2.3}
\label{bbea2.2a}
\dlb\hdf + \shlf\,\Pbrv{\hdf,\hdf} = 0,
\xeee
where $\Pbrv{\ ,\ }$ is the Poisson bracket corresponding
to the symplectic form $\som$, then the corresponding deformation
of the complex structure of $\xX$ is described by the
Beltrami differential
\xlee{bbea2.2a1}
\bdmu = \som^{-1}(\del\hdf),
\xeee
that is, the $(0,1)$ part of the deformed Dolbeault differential is
\xlee{beq2.12b}
\dlb\p = \dlb + \som^{-1}(\del\hdf)\spsmb\del,
%\bdmua{\del}.
\xeee
while the symplectic form $\som$ is replaced by
\ylee{beq2.13b}
\som\p=\som +
d\hdf,
\yeee
so that it remains of type $(2,0)$ relative to the new
complex structure. In the formula\rx{beq2.12b}
we defined
%$\som^{-1}\colon\Omega^{1,0}(\xX)\rightarrow \Omega^{0,0}(\Tng\xX)$
$\som^{-1}\colon\Gamma(\Ts\xX)\rightarrow\Gamma(\Tng\xX)$
%by the relation $\iota_{\som^{-1}(\alpha)\som = \alpha
as the inverse of $\iota_{\xblnk}\,\som$.

%we used the
%following notation: if $V$ is a vector space and $\som\in\wedge^2
%V^\ast$ is non-degenerate, then $\som^{-1}\colon V^\ast\rightarrow
%V$ is a linear map defined by the property that
%$\som\brb{v,\som^{-1}(w)} = w\spsmb v$ for any $v\in V$ and $w\in
%V^\ast$.

If the deformation of $\Xsom$ is perturbative, that is, if $\hdf$ is a formal power series
\xlee{bbea2.3a1}
\hdfe =
\sum_{i=1}^\infty \hdfi\,\dfe^i,
\xeee
then the relation\rx{bbea2.2a} says that the leading coefficient $\hdfo$
must be $\dlb$-closed, and its gauge equivalence class is
determined by its Dolbeault cohomology class
$\hcdfo\in\Hdlbo(\xX)$, while the relation
\ylee{bbea2.3a1}
\dlb\hdftw + \shlf\,\Pbrv{\hdfo,\hdfo} = 0
\yeee
implies the  `integrability condition' for $\hcdfo$:
%
%then with a slight abuse of
%notation, the first order deformation $\hdfo$ is (determined by) an element of
%$\Hdlbo(\xX)$ satisfying the condition
%
\ylee{bbea2.4}
\Pbrv{\hcdfo,\hcdfo} = 0.
\yeee

\subsubsection{Deformations of the \hlsm\ structure of a cotangent bundle}

If $\xX=\TsU$, then we restrict ourselves to \CMlt s $\hdf$ belonging to the subalgebra
$\zdpaTU\hookrightarrow \zdpaTsU$ defined in
subsection\rw{ss.dpahlsmm}.
Moreover, we consider only the deformations which do
not deform the complex structure of the zero section $\xcAz\subset\TsU$,
so we impose the condition $\bdmu|_{\xcAz}=0$ on the Beltrami
differential\rx{bbea2.2a1}.
This condition is satisfied if $\hdf$ is at least quadratic as a
function of holomorphic coordinates on fibers of $\TsU$ :
\xlee{bbea2.5} %*r
\label{bbea2.4a}
\hdf = \hdf_2 + \hdf_3 +\cdots,
\qquad
\hdf_i\in \xbOmov{\xcA,\Stai\TU}.
%
%\hdf\in
%\xbOmov{\xcA,\;\bigoplus\nolimits_{i=2}^\infty\Stai\TU}\subset\zdpaTU,
\xeee
%%
%which means in particular that
%in case of a perturbative deformation\rx{bbea2.3a1},
%%
%\ee
%\label{bae2.0a2} %*r
%\hcdfo\in
%\Hdlbo\lrbc{\xcA,\;\bigoplus\nolimits_{i=2}^\infty\Stai\TU}. %\subset\Hdlbo\brb{\xcA,\SbTU}.
%\eee
%%
The first two terms in this sum play a particularly important role in what follows and we give
them special names:
\xlee{bbea2.5b}
\hdf_2 = \hdfbt,\qquad\hdf_3 = \hdfgm.
\xeee
According to \ex{bbea2.2a}, they satisfy the equations
\xlee{beq2.25b}
\dlb\hdfbt=0,\qquad\dlb\hdfgm+ \shlf\Pbrv{\hdfbt,\hdfbt} = 0.
\xeee
Thus $\hdfbt$ is $\dlb$-closed, and its gauge equivalence class is
determined by the Dolbeault cohomology class that it represents:
\ylee{beq2.25b1}
\cdfbt\in\Hdlbo(\xcA,\Stat\TU),\qquad \Pbrv{\cdfbt,\cdfbt}=0.
\yeee

The class $\cdfbt$ has a simple geometric interpretation. For a
holomorphic submanifold $\yY\subset\xX$ of a complex manifold
$\xX$, the exact sequence of vector bundles on $\yY$
\xlee{beq2.25b2}
\TY\longrightarrow\Tng\xX|_{\yY} \longrightarrow\NY
\xeee
determines an extension class $\tdfbtY\in\Ext^1(\NY,\TY)$.
\ftnt{If $\xX$ is \Khl, then the class $\tdfbtY$ may be
represented by the anti-holomorphic part of the second fundamental
form of $\yY$ contracted with the \Khl\ metric and with its
inverse in order to turn two anti-holomorphic indices on the second fundamental form
into the holomorphic ones.}
If
$\yY$ is a lagrangian submanifold of a \hlsmm\ $\xX$ then the symplectic form $\som$
establishes an isomorphism $\NY\simeq\TsY$, so
$\tdfbtY\in\Ext^1(\TsY,\TY)$. The zero-section $\xcA\subset\TsU$
is a lagrangian submanifold and its exact sequence\rx{beq2.25b2}
splits, so in this case $\tdfbtU=0$. However, if we consider the
zero-section of the deformed bundle $\TsUhdf$ then the
sequence\rx{beq2.25b2} does not have to split. The injection
$\Hdlbo(\xcA,\Stat\TU)\hookrightarrow\Ext^1(\TsU,\TU)$ turns
$\cdfbt$ into an extension class within $\Ext^1(\TsU,\TU)$ and, in
fact,
\xlee{beq2.25b3}
\cdfbt = \tdfbtU.
\xeee
In other words, the leading coefficient $\cdfbt$ in the
expansion\rx{bbea2.4a} of $\hdf$ reflects the fact that the
sequence\rx{beq2.25b2} for the zero-section of the deformed
cotangent bundle $\TsUhdf$ does not split.

%The injection
%$\Hdlbo(\xcA,\Stat\TU)\hookrightarrow\Ext^1(\TsY,\TY)$ turns
%$\cdfbt$ also into an extension class within $\Ext^1(\TsY,\TY)$.
%In fact, it turns out that both classes are the same:

%Let us introduce two notations related to $\hdf$. First of all,

The injection\rx{bcea2.5} turns
an element $\hdf\in\xbOmbUUST$
into a $\bar{\Tng}^\vee\xcA$-valued function
(or, rather, a formal power series)
%$\hdft$
on the total space of $\TsU$. We denote this function by the same letter
$\hdf$.
The evaluation of $\hdf$ on a section of $\TsU$ gives a map
%on sections of $\TsU$, we will denote it also as
%
\xlee{bae2.0c1}
\hdf\colon\Gamma(\TsU)\rightarrow \xbOmoU.
\xeee

The restriction of the $(1,0)$ part of the differential $\del\hdf$
of an element $\hdf\in\xbOmbUST$
to the fibers of
$\TsU$ determines a vertical holomorphic differential map
\xlee{bae2.0c2}
\dfib\hdf\colon\Gamma(\TsU)\rightarrow \xbOmv{1}(\xcA,\TU).
\xeee

Recall that if $\xcW$ is a function on $\xcA$, then
$\yYW\subset\TsU$ denotes the graph of $\del\xcW$ defined by \ex{bah1.1}.
If $\xcW$ is holomorphic, then $\yYW$ is a \hlgrsm.
Let $\mtdfTUh$ denote the total space of the cotangent bundle
$\TsU$ whose \hlsm\ structure is deformed by
the \CMlt\rx{bbea2.4a}.
%the \Hdf\rx{bbea2.4a} satisfying the \CMe\rx{bbea2.3}.
The deformed cotangent bundle $\mtdfTUh$ is canonically
diffeomorphic to $\TsU$, so for an arbitrary function $\xcW$, the graph
$\yYW$ is still a submanifold in $\mtdfTUh$, but this time $\yYW$ is
lagrangian if $\xcW$ satisfies the equation
%If a function $\xcW\in\xbOmz(\xcA)$ is holomorphic ($\dlb\xcW=0$),
%then the graph of its holomorphic differential $\del\xcW$
%is a \hlgrsm\ of the cotangent bundle $\TsU$.
%The graph of $\del\xcW$ is a \hlgrsm\ of the deformed cotangent
%bundle $\mtdfTUh$, if the function $\xcW$ satisfies the equation
%
\xlee{bae2.0a3}
\dlb\xcW = \hdf(\del\xcW),
\xeee
where $\hdf(\cdot)$ is the map\rx{bae2.0c1}.
%denotes a $\xbOmov{\xcA}$-valued formal power series
%evaluated on sections of $\TsU$.
%corresponding to the element\rx{bbea2.4a}.
If we consider a perturbative deformation\rx{bbea2.3a1}, then the
generating function becomes a formal power series
\xlee{bae2.0b1}
\xcWe = \xcWzz + \sum_{i=1}^{\infty}\xcWi\,\dfe^i.
\xeee
The leading term $\xcWzz$ is a holomorphic function describing
a \hlgrsm\  $\yY\subset\xX$ and it has a special property
\ylee{bae2.0b2}
\hcdfo(\del\xcWzz) = 0,
\yeee
which guarantees that $\yY$ can be deformed to the first order in $\dfe$, while the first order perturbation $\xcWo$
satisfies the equation
\ylee{bae2.0b3}
\dlb\xcWo - \hdfo(\del\xcWzz) = 0.
\yeee

The complex structure of the \hlgrsm\ $\yYW$ determined by the
function $\xcW$ satisfying the condition\rx{bae2.0a3} can be
described by saying that the bundle projection of $\TsU$
establishes an isomorphism between $\yYW$ and the base $\xcA$,
whose complex structure is deformed by the Beltrami differential
\xlee{bae2.0b4}
\mu_{\xcW} = -\dfib\hdf(\del\xcW),
\xeee
where $\dfib\hdf$ is the map\rx{bae2.0c2}.
%Note that according to
%\ex{bae2.0a3}, the function $\xcW$ is holomorphic with respect to
%that complex structure on $\xcA$.

\subsection{Deformation of the 2-category of \copfib s:
deformation of the category of morphisms}
\label{ss.dfmtcmr}

\subsubsection{Objects of the deformed category}
\label{ss.odc}

Following the outline of subsection\rw{ss.outline},
we conjecture that the \CMlt\  $\hdf$ parameterizing the
deformations of the \hlsms\ of $\TsU$, parameterizes also the
deformations of the 2-category $\rDDprfU$. We are going to discuss the
structure of the deformed category $\rDDprfUSk$, but we will limit
ourselves to simplest \xzobj s in it, which are the analogs of \copfib\
objects $\TrfUW$ denoted simply as $\xcW$.

%\subsubsection{Deformation of the category of morphisms}

Recall that
in the undeformed case the function $\xcW$ labeling an object of
$\rDDprfU$ is holomorphic,
%a holomorphic function on $\xcA$,
so that the graph of its holomorphic
differential\rx{bah1.1} is a lagrangian submanifold of $\TsU$.
%%
%\ylee{baf1.1}
%\yYW = \{(x,p)\in\TsU\,|\,x\in\xcA,\;p\in\Ts_{x}\xcA,\;p = \del\xcW(x)
%\}.
%\yeee
%%
%
We conjecture that in the case of $\rDDprfUSk$, a
similar  object
%analogs of \copfib\ \zobj s
is (parameterized by) a function $\xcW$ on $\xcA$, which satisfies
the equation\rx{bae2.0a3}, because then the graph of its
holomorphic differential $\yYW$ defined by the same
equation\rx{bah1.1} is again a lagrangian submanifold of the
deformed cotangent bundle $\mtdfTUSk$, and this is in line with
our conjecture that lagrangian submanifolds represent the objects
of $\ctLLXsom$ not only when $\Xsom$ is an undeformed cotangent bundle,
but also when it is a general \hlsmm.

%even if $\Xsom$ is a general \hlsmm\ and not just
%
%this allows us to relate
%$\xcW$ again to a lagrangian submanifold $\yYW\subset\mtdfTUSk$ of
%\ex{bah1.1}
%$\rDDprfUSk$ to $\ctLLTUSk$ by identifying the \zobj s $\xcW$ with
%the \hlgrsm s $\yYW\subset\mtdfTUSk$, corresponding to graphs of
%$\del\xcW$ in the total space of $\TsU$.

\subsubsection{The universal \CMlt}

Consider the tensor algebra over $\IC$ of \emph{Dolbeault tensor fields}
\ylee{bah3.1}
\xTflU \edfn \bigoplus_{\ixk,\ixl=1}^\infty \TflklU,\qquad
\TflklU \edfn \xbOmb(\Tng^{\ixk}\xcA\otimes\Tng^{\vee,\ixl}\xcA).
\yeee
Fix a $\del$-connection
$\nabla\colon\TflbbU\rightarrow\Tflvv{\bullet}{\bullet+1}(\xcA)$.
For a tensor field $\tnsct\in\xTflU$, let $\bnbl\tnsct$ denote a sequence of
multiple covariant derivatives:
$\bnbl\tnsct\edfn\tnsct,\nabla\tnsct,\nabla^2\tnsct,\ldots$. For
tensor fields $\tnsct_1,\ldots,\tnsct_k$, let
$\Tflnv{\tnsct_1,\ldots,\tnsct_k}$ denote a subalgebra of $\xTflU$
generated by all tensor fields
$\bnbl\tnsct_1,\ldots,\bnbl\tnsct_k$ and by all possible contractions
within their tensor products.

Let $\hdf_i$ denote an element of $\xbOmov{\xcA,\Stai\TU}$. The Poisson-Schouten
bracket $\Pbrv{\hdf_i,\hdf_j}$ of two such elements
%For two tensor fields $\hdf_i\in\xbOmov{\xcA,\Stai\TU}$ and
%$\hdf_j\in\xbOmov{\xcA,\Staj\TU}$, their Poisson-Schouten bracket
%$\Pbrv{\hdf_i,\hdf_j}$
is an element of $\Tflnv{\hdf_i,\hdf_j}$
and it is universal in the sense that the coefficients of its expression in terms
of the appropriate contractions of $\hdf_i\otimes\nabla\hdf_j$ and
$\hdf_j\otimes\nabla\hdf_i$ are universal constants. The same
holds true for elements
$\xdfm_{ij}\in\xbOmv{\xmi}(\xcA,\wedge^{\xmj}\TU)$ and their
\NSb\ $[\xblnk,\xblnk]$.

For a \CMlt\ $\hdf\in\xbOmv{1}(\xcA,\Sb\TU)$ satisfying
\ex{bbea2.2a} and for two functions $\xcWo$, $\xcWt$ on
$\xcA$ satisfying \ex{bae2.0a3}, denote
\xlee{bah4.1}
\Tflnbot\edfn\Tflnv{\del\xcWo,\del\xcWt,\crR,\hdf_2,\hdf_3,\ldots},
\xeee
where $\hdf_i$ are the components\rx{bbea2.4a}, while
$\crR\edfn[\dlb,\nabla]\in\xbOmov{\rmS^2\TsU\otimes\TU}$
is the curvature of the tangent bundle $\TU$ corresponding
to the connection $\nabla$. The Dolbeault differential $\dlb$ acts
universally on the elements of $\Tflnbot$. Indeed, its action on
the components of $\hdf$ is prescribed by the \CMe\rx{bbea2.2a},
its action on $\del\xcWi$ is prescribed by \ex{bae2.0a3}, $\dlb\crR=0$, and
a permutation of $\dlb$ and $\nabla$ generates the curvature tensor $\crR$.

\begin{conjecture}
\label{cnj.univ}
For a \CMlt\ $\hdf\in\xbOmv{1}(\xcA,\Sb\TU)$ satisfying
\ex{bbea2.2a} and for two functions $\xcWo$, $\xcWt$ on
$\xcA$ satisfying \ex{bae2.0a3} there exists a universal \xrlb\ \CMlt\
\begin{gather}
\label{beq2.15b}
\xdfmot = \sum_{\xmi=0}^\infty \xdfmoti,\qquad
%\xdfm_{\xmi}
\xdfmoti\in\xbOmv{\xmi}(\xcA,\wedge^{\xmi}\TU),
\\
\label{beq2.17b}
\dlb\xdfmot + \shlf\,[\xdfmot,\xdfmot] = 0
\end{gather}
%
%
%$\xdfmot = \sum_{i=0}^\infty \xdfmoti$,
%$\xdfmot = \sum_{i,j = 0}^\infty\xdfmotv{ij}$,
such that
\xlee{beq2.16b}
%\xdfmot = \sum_{i,j = 0}^\infty\xdfmotv{ij}\qquad
\xdfmotz=\xcWot\edfn
\xcWt-\xcWo
\xeee
and $\xdfmoti\in\Tflnbot$ for $i\geq 1$,
%%
%\ylee{bah2.3}
%\xdfmoti\in
%\xbOmv{i}(\xcA,\wedge^{j}\TU)\cap
%\Tflnbot,
%%\Tflnv{\del\xcWo,\del\xcWt,\crR,\hdf_2,\hdf_3,\ldots}
%\qquad \text{for $i\geq 1$},
%\yeee
%%
where $\hdf_i$ are the components\rx{bbea2.4a}.
%, while $\crR$ the curvature of the tangent bundle $\TU$ corresponding
%to the connection $\nabla$.
The universality of $\xdfmot$ means
that the coefficients in the expression of its components
$\xdfmoti$ in terms of the tensor fields and their derivatives
are constants that do not depend on $\xcA$. The universal element
$\xdfmot$ is unique up to gauge equivalence\rx{bcea2.1a}, and different choices of
$\nabla$ also lead to gauge equivalent elements $\xdfmot$.
\end{conjecture}

Simply put, if two functions $\xcWo$, $\xcWt$
satisfy \ex{bbea2.2a} then their difference is not necessarily
holomorphic and hence it cannot serve as a deformation parameter
of the category $\rDprfU$. However, we conjecture that there is a
special unique correction to $\xcWot$ which turns it into a
\CMlt\ suitable for deforming  $\rDprfU$.
Hence we conjecture that the category of morphisms between the
objects of $\rDDprfUSk$ represented by $\xcWo$ and $\xcWt$ is the
deformed category $\rDprfU$:
%
%Thus the simplest objects of $\rDDprfUSk$ are functions $\xcW$ on $\xcA$
%satisfying \ex{bae2.0a3}. We conjecture further that a category of
%morphisms between two such functions is the deformed category
%$\rDprfU$:
%
\xlee{beq2.14b}
\Hom_{\rDDprfUSk}(\xcWo,\xcWt) =
 \rDsprfUdot,
\xeee
where $\xdfmot$ is the unique universal deformation parameter of
Conjecture\rw{cnj.univ}. Note that for a fixed manifold $\xcA$,
the sum in \ex{beq2.15b} is effectively finite, the highest value
of $i$ being the complex dimension of $\xcA$.

\subsubsection{Perturbative construction of the universal \CMlt}
\label{ss.pcucmlt}

We construct the universal expression for $\xdfmot$ perturbatively
in the Dolbeault degree $\dgDlb$ defined in an obvious way:
\xlee{e.dgdlb}
%\dgDlb\dlb=1,
\qquad\dgDlb \xcWi=0,\qquad \dgDlb\crR = 1,\qquad
\dgDlb\hdfi=1,\qquad\dgDlb\nb=0.
\xeee
We substitute the
expansion\rx{beq2.15b} into the \CMe\rx{beq2.17b} and find the
equation determining $\xdfmotn$ in terms of $\xdfmoti$ with
$\xmi<\xmn$:
\xlee{beq2.18b}
[\xcWot,\xdfmotn] = - \dlb\xdfmotnmo - \hlf\sum_{i=1}^{\xmn-1}
[\xdfmoti,\xdfmotv{\xmn-i}].
\xeee

We will find expressions for the first three corrections in the
expansion\rx{beq2.15b}. We introduce a special notation for them:
\ylee{beq2.18ba1}
\xdfmoto = \bdmu,\qquad
\xdfmott = \bdnu,\qquad
\xdfmoth = \bdxi.
\yeee
(the indices $12$ at $\bdmu$, $\bdnu$ and $\bdxi$ are dropped temporarily in this subsection).
We are particularly interested in the third correction, because,
as we will see shortly, this is the dominant (that is, the lowest Dolbeault degree)
term in $\xdfmot$
when $\xcWo=\xcWt=0$, that is, from the
$\ctLLTUSk$ perspective
it describes the leading deformation
of the category of endomorphisms of the zero-section of $\TsU$.

According to \ex{beq2.18b}, the first three corrections are determined by
the equations
\begin{align}
\label{beq2.18ba2}
[\xcWot,\bdmu] & = - \dlb\xcWot,
\\
\label{beq2.18ba3}
[\xcWot,\bdnu] & = -\dlb \bdmu - \shlf\,[\bdmu,\bdmu],
\\
\label{beq2.18ba4}
[\xcWot,\bdxi] & = -\dlb \bdnu - [\bdmu,\bdnu].
\end{align}
%
%
%According to \ex{beq2.18b}, the degree 1 component of $\xdfmot$,
%which we denote as $\bdmuot=\xdfmoto$, must satisfy the equation

First of all, we find an exact universal solution for $\bdmu$.
Equation\rx{beq2.18ba2} can be rewritten simply as
\xlee{beq2.19b}
\dlb\xcWot + \bdmu \spsmb \del \xcWot = 0.
\xeee
%
%In order to describe its universal solution,
In order to solve it,
we
have to introduce \ddf\ notations. Let $\xaV$ and $\xbV$ be  vector
spaces. An element $\xaa\in\Sb\xaV\,\otimes\,\xbV$ determines a
polynomial function (or a formal power series) $\xaa\colon
\Vd\rightarrow\xbV$. The first \ddf\ of $\xaa$ is
a symmetric map $\fsd\xaa\colon \Vd \times\Vd\rightarrow\xaV\otimes\xbV$
defined by the property
\ylee{beq2.20b}
\fsd\xaa(\xavo,\xavt) \spsmb (\xavt - \xavo) = \xaa(\xavt) -
\xaa(\xavo),\qquad\forall\,\xavo,\xavt\in\xaV.
\yeee
The second \ddf\ of $\xaa$ is a totally symmetric
map $\ssd\xaa\colon\xaV\times\xaV\times\xaV\rightarrow
\Stav{2}\xaV\,\otimes\,\xbV$ defined by the property
\ylee{beq2.21b}
\ssd\xaa(\xavo,\xavt,\xavh) \spsmb (\xavh-\xavt)=
\fsd\xaa(\xavo,\xavh) - \fsd\xaa(\xavo,\xavt)
,\qquad\forall\,\xavo,\xavt,\xavh\in\xaV.
\yeee
In application to an element
$\hdf\in\xbOmov{\xcA,\;\Sb\TU}$ \ddfs\ produce symmetric maps
\ylee{beq2.22b}
\fsd\hdf\colon \Gamma(\TsU)^{\times 2}\rightarrow
\xbOmov{\xcA,\TU},\qquad
\ssd\hdf\colon \Gamma(\TsU)^{\times 3}\rightarrow
\xbOmov{\xcA,\Stav{2}\TU}.
\yeee
According to the condition\rx{bae2.0a3},
\ylee{beq2.24b}
\dlb\xcWot = \hdf(\del\xcWt) - \hdf(\del\xcWo) =
\fsd\hdf(\del\xcWo,\del\xcWt)\spsmb \del\xcWot.
\yeee
Hence the universal solution to \ex{beq2.19b} is the \ddf\
of $\hdf$ evaluated on the differentials $\del\xcWo$ and
$\del\xcWt$:
\xlee{beq2.23b}
\bdmu = - \fsd\hdf(\del\xcWo,\del\xcWt).
\xeee
%
%Indeed, since $\xcWo$ and $\xcWt$ satisfy the
%condition\rx{bae2.0a3},

We will find the universal expressions for $\bdnuot$  and
$\bdxiot$ only approximately, up to certain powers of $\xcWo$ and
$\xcWt$:
\xlee{beq2.23c}
\bdmu = \bdmupr + \OWh,\qquad
\bdnu = \bdnupr + \OWt,\qquad
\bdxi = \bdxipr + \OWo,
\xeee
where $\OWn$ denotes an expression which is at least of
combined degree $n$ in $\xcWo$ and $\xcWt$. According to
\ex{beq2.23b},
\ylee{beq2.24c}
\bdmupr = - \fsd\hdfbt(\del\xcWo,\del\xcWt)
%= \OWo,
- \fsd\hdfgm(\del\xcWo,\del\xcWt),
\yeee
where $\hdfbt$
and $\hdfgm$ are
defined by \ex{bbea2.5b}.
%, so $\bdmupr$ is of order $\OWo$.
By
substituting \eex{beq2.23c} into \eex{beq2.18ba3}
and\rx{beq2.18ba4} we find
\begin{align}
%\label{beq2.18bb2}
%[\xcWot,\bdmu] & = - \dlb\xcWot,
%\\
\label{beq2.18bb3}
[\xcWot,\bdnupr] & = -\dlb \bdmupr - \shlf\,[\bdmupr,\bdmupr], % + \OWh,
\\
\label{beq2.18bb4}
[\xcWot,\bdxipr] & = -\dlb \bdnupr,
%- [\bdmupr,\bdnupr]
%+ \OWt,
\end{align}
%
%so these approximations are sufficient to find
which determine the universal expressions for
$\bdnupr$ and
$\bdxipr$.

In order to simplify the calculations required to derive the
universal formula for $\bdnupr$ and $\bdxipr$, we present the
formula\rx{beq2.23b} in a different form. Consider the relation
between the Lie bracket on a manifold, and
the Poisson bracket on the total space of its cotangent bundle.
For a function $\xcW$ on a complex manifold $\xcA$, let $\tcW$
denote its pull-back to the total space of the cotangent bundle
$\TsU$. For an element $\mu\in\xbOmbv{\xcA,\TU}$ let $\tmu$ denote
the corresponding $(0,\bullet)$-form on the total space of $\TsU$ which is
linear along the fibers. In our previous notations,
$\mu=\dfib\tmu$. The Lie bracket on $\xcA$ and the Poisson bracket
on $\TsU$ are related as follows:
\ylee{beq2.28b}
\widetilde{\Lbrv{\mu,\xcW} } = -\Pbrv{\tmu,\tcW},\qquad
\widetilde{\Lbrv{\mu,\mu\p}} = -\Pbrv{\tmu,\tmu\p}.
\yeee

For a function $\xcW$ on $\xcA$, let $\hcW=\Pbrv{\tcW,\cdot}$ be a
linear operator acting on functions on the total space of $\TsU$.
The operators $\hcW$ commute with each other:
%The family of operators $\hcW$ is commutative:
for two functions $\xcWo$ and $\xcWt$
\ylee{beq2.29b}
[\hcWo,\hcWt] = \widehat{ \Pbrv{\xcWo,\xcWt}   } = 0.
\yeee

For $\hdf\in\xbOmbUUST$ let $\hdf\xcmi$ denote its component in
$\xbOmbv{\xcA,\Stai\TU}$.
It is easy to see that in our new notations the \rhs of the
equations\rx{bae2.0a3} and\rx{beq2.23b} can be presented as
\begin{gather}
\label{beq2.30b}
\dlb\xcW = \hdf(\del\xcW) = \left.\lrbc{e^{\hcW}\hdf}\right|_{0},%\xcmz,
\\
\bdmuot = - \fsd\hdf(\del\xcWo,\del\xcWt)
=-\dfib\left.\lrbc{\frac{e^{\hcWt}-e^{\hcWo}}{\hcWt-\hcWo}\,\hdf}\right|_1
\end{gather}
(in the \rhs of these formulas $\hdf$ is considered to be a
$\bar{\Tng}^\vee\xcA$-valued function
% $\hdft$
on the total space of $\TsU$).
According to the first formula,
\xlee{beq2.30c}
\dlb\xcW = \hdfbt(\dlb\xcW) + \OWh = \shlf\,\hcW^2\hdfbt + \OWh.
\xeee
According to the second formula,
\ylee{beq2.31c}
\bdmupr = -\dfib\Big(\shlf\, (\hcWo + \hcWt)\hdfbt
+ \tfrac{1}{6}\,(\hcWo^2 + \hcWo\hcWt + \hcWt^2)\,\hdfgm
\Big).
\yeee

Both sides of \ex{beq2.18bb3} are elements of
%$\xbOmbv{\xcA,\TU}$
$\xbOmv{2}(\xcA,\TU)$,
so applying $\;\tilde{ }\;$ to them (that is, turning them into
$(0,2)$-forms on the total space of $\TsU$, which are linear along
fibers), we find
%%
%\begin{equation}
%\label{beq2.35bx}
%\begin{split}
%\widetilde{\Lbrv{\xcWotz,\bdnuott} }
%& =
%- \dlb\tdmuott + \shlf\,\Pbrv{\tdmuoto,\tdmuoto}
%\\
%& = \shlf\,\Pbrv{(\dlb\xcWoo+\dlb\xcWto),\hdfbt}
%+\tfrac{1}{6}\,(\hcWoz^2 + \hcWoz\hcWtz + \hcWtz^2)\,\dlb\hdfgm
%\\
%&\qquad + \tfrac{1}{8}\,
%\Pbrv{ (\hcWoz + \hcWtz)\hdfbt,(\hcWoz + \hcWtz)\hdfbt }
%\\
%& = - \tfrac{1}{8}\,\Pbrv{\hcWotz\hdfbt,\hcWotz\hdfbt}
%+ \tfrac{1}{24}\,\hcWotz^2\Pbrv{\hdfbt,\hdfbt}.
%\end{split}
%\end{equation}
%%
%
\begin{equation}
\label{beq2.35b}
\begin{split}
\widetilde{\Lbrv{\xcWot,\bdnupr} }
& =
- \dlb\tdmupr + \shlf\,\Pbrv{\tdmupr,\tdmupr} + \OWh
\\
& = \shlf\,\Pbrv{(\dlb\xcWo+\dlb\xcWt),\hdfbt}
+\tfrac{1}{6}\,(\hcWo^2 + \hcWo\hcWt + \hcWt^2)\,\dlb\hdfgm
\\
&\qquad + \tfrac{1}{8}\,
\Pbrv{ (\hcWo + \hcWt)\hdfbt,(\hcWo + \hcWt)\hdfbt } + \OWh
\\
& = - \tfrac{1}{8}\,\Pbrv{\hcWot\hdfbt,\hcWot\hdfbt}
+ \tfrac{1}{24}\,\hcWot^2\Pbrv{\hdfbt,\hdfbt}.
\end{split}
\end{equation}
We used the formulas\rx{beq2.25b}
%and\rx{beq2.30ba}
and\rx{beq2.30c}
as well as the
Jacobi identity for the Poisson bracket in order to derive
the last line in this equation.

In order to solve the equation\rx{beq2.35b} for $\bdnupr$, we
express its \rhs in terms of a torsionless covariant $(1,0)$-differential on
the tangent bundle $\TU$:
$$\nb\colon\Gamma(\TU)\rightarrow\Omega^{1,0}(\xcA,\TU).$$
We use index notations: let $\vrxI$,
$\inI=1,\ldots,\dim_{\IC}\xcA$ be local holomorphic coordinates on
$\xcA$. The corresponding frames in $\TU$ and $\TsU$ are formed by
$\dlI$ and $d\vrxI$.
In our formulas we assume summation over repeated indices
appearing on opposite levels (this corresponds to applying
contraction to tensor products). Anti-holomorphic indices are
hidden. Thus in our notations
%
%\ylee{beq2.36b}
\begin{gather}
%\label{beq2.36b}
\nonumber
\hdfbt = \hdfbt^{\inI\inJ}\,\dlI\dlJ,
\qquad \hdfbt^{\inI\inJ} = \hdfbt^{\inJ\inI},
%\qquad \del\xcW = (\dlI\xcW)\,d\vrxI,
\qquad
\bdnu = \bdnu^{\inI\inJ}\,\dlI\wedge\dlJ,
\qquad
\bdnu^{\inI\inJ} = - \bdnu^{\inJ\inI}
\\
\label{beq2.37b}
\Pbrv{\xcW,\hdfbt} =
2\hdfbt^{\inI\inJ}\,(\dlI\xcW)\,\dlJ,\qquad
\Pbrv{\hdfbt,\hdfbt} =
-4\hdfbt^{\inI\inL}\,(\nbcL\hdfbt^{\inJ\inK})\,\dlI\dlJ\dlK,
\qquad
\Lbrv{\xcW,\bdnu} = 2\bdnu^{\inI\inJ}(\dlJ\xcW)\,\dlI.
\end{gather}
%\yeee
%
A straightforward computation shows that if we lift the tilde in
\ex{beq2.35b} by applying $\dfib$ to both sides, then its
%the \rhs of \ex{beq2.35b}
\rhs can be rewritten as
\begin{multline}
\nonumber
\Lbrv{\xcWot,\bdnupr}
=
\\
\Bigg(
\hdfbt^{\inJ\inK}\hdfbt^{\inI\inL}\,(\dlJ\xcWot)\,(\nbcK\dlL\xcWot)
-\tfrac{1}{3}\Big( \hdfbt^{\inI\inL}\,(\nbcL\hdfbt^{\inJ\inK})
-
\hdfbt^{\inJ\inL}\,(\nbcL\hdfbt^{\inI\inK})\Big)(\dlJ\xcWot)(\dlK\xcWot)
\Bigg)\dlI.
\end{multline}
Comparing this expression with the last equation of\rx{beq2.37b}
we find a formula for the leading term in $\bdnu$:
%%
%\ylee{beq2.38bx}
%\bdnuot = \dfe^2 \;\frac{1}{2}\;
%\Bigg(
%\hdfbt^{\inJ\inK}\hdfbt^{\inI\inL}\,(\nbcK\dlL\xcWotz)
%-\tfrac{1}{3}\Big( \hdfbt^{\inI\inL}\,(\nbcL\hdfbt^{\inJ\inK})
%-
%\hdfbt^{\inJ\inL}\,(\nbcL\hdfbt^{\inI\inK})\Big)(\dlK\xcWotz)
%\Bigg)\dlI\wedge \dlJ.
%\yeee
%%
%
\ylee{beq2.38b}
\bdnupr = \frac{1}{2}\;
\Bigg(
\hdfbt^{\inJ\inK}\hdfbt^{\inI\inL}\,(\nbcK\dlL\xcWot)
-\frac{2}{3}\; \hdfbt^{\inI\inL}\,(\nbcL\hdfbt^{\inJ\inK})
%- \hdfbt^{\inJ\inL}\,(\nbcL\hdfbt^{\inI\inK})\Big)
(\dlK\xcWot)
\Bigg)\dlI\wedge \dlJ .  % + \OWt.
\yeee

Finally, we solve \ex{beq2.18bb4} for $\bdxipr$.
In order to compute its \rhs, we introduce the
Riemann curvature tensor $\crR\in\xbOmov{\xcA,\Stav{2}\TsU\otimes\TU}$
of the connection $\nb$ by the
formula $\crR = [\dlb,\nb]$. In index notations
\ylee{beq2.40b}
\crR = \crR^{\inI}_{\inJ\inK}\, d\vrxJ d\vrxK\,\dlI,\qquad
\crR^{\inI}_{\inJ\inK}=\crR^{\inI}_{\inK\inJ}.
\yeee
Then a straightforward computation shows that
\ylee{beq2.41b}
\dlb \bdnupr
 = \frac{1}{3}
\Big(
\hdfbt^{\inI\inL}\hdfbt^{\inJ\inM}\crR^{\inK}_{\inL\inM}
+
\hdfbt^{\inJ\inL}\hdfbt^{\inK\inM}\crR^{\inI}_{\inL\inM}
+
\hdfbt^{\inK\inL}\hdfbt^{\inI\inM}\crR^{\inJ}_{\inL\inM}
\Big)(\dlK\xcWot)\,\dlI\wedge\dlJ.
\yeee
Since in index notations
\ylee{beq2.42b}
\bdxi = \bdxi^{\inI\inJ\inK}\,\dlI\wedge\dlJ\wedge\dlK,\qquad
\Lbrv{\xcW,\bdxi} = 3\,
\bdxi^{\inI\inJ\inK}(\dlK\xcW)\,\dlI\wedge\dlJ,
\yeee
we find that
\xlee{beq2.42b}
\bdxipr = \frac{1}{3}\,\hdfbt^{\inI\inL}\hdfbt^{\inJ\inM}\crR^{\inK}_{\inL\inM}
\;\dlI\wedge\dlJ\wedge\dlK.
\xeee

\subsubsection{Semi-classical grading}
\label{ss.scgr}

The algebra $\Tflnbot$ of \ex{bah4.1} has an important
\emph{\smcl} grading defined as follows:
\xlee{bah4.2}
\spdg \xcW_i = -2,\qquad\spdg\crR=0 ,\qquad\spdg
\hdf_i=i-3,\qquad\spdg\nabla=1.
\xeee
The defining relation $\crR=[\dlb,\nabla]$ implies $\spdg\dlb=-1$.
If we set
\xlee{bah4.4}
\spdg\xdfmot=-2,
\xeee
then all three defining
relations of our deformation construction:
\ylee{bah4.3}
\dlb\hdf + \shlf\,\Pbrv{\hdf,\hdf} = 0,
\qquad
\dlb\xcW = \hdf(\del\xcW),
\qquad
\dlb\xdfmot + \shlf\,[\xdfmot,\xdfmot] = 0
\yeee
respect the \smclgrd. Hence the universal \CMlt, determined recursively
by \ex{beq2.18b}, must have the \smcldgr\ -2: its zeroth term
$\xdfmz=\xcWot$ has degree -2 in virtue of $\spdg\xcWot=-2$ and the
degree of higher terms is expressed through \ex{beq2.18b} in
terms of the degrees of the lower terms.

Conjecture\rw{cnj.univ} together with the relation\rx{bah4.4} has
three easy corollaries.

Consider a new grading on the algebra
$\Tflnbot$:
\xlee{bah2.4}
\dgdW\bnbl\hdf_i=\dgdW\bnbl\crR=0,\qquad
\dgdW \nabla^k \del \xcW =
\begin{cases}
1, & \text{if $k=0$,}
\\
0, & \text{if $k\geq 1$.}
\end{cases}
\xeee
\begin{corollary}
\label{crl1}
If $\hdfbt=0$, then $\dgdW\xdfmi\geq 2$ for $i\geq 1$.
\end{corollary}
\begin{corollary}
\label{crl2}
If $\hdfbt=0$ and $\xcWo=\xcWt=0$, then $\xdfm=0$.
\end{corollary}
Let $\ccrR\in\Ext^1(\rmS^2\TU,\TU)$ denote the Atiyah class of the
tangent bundle $\TU$ (it is the class represented by the curvature tensor
$\crR$).
\begin{corollary}
\label{crl3}
If $\ccrR=0$ and $\xcWo=\xcWt=0$, then $\xdfm = 0$.
\end{corollary}
%%
%\begin{corollary}
%\label{crl2}
%
%\end{corollary}
%It is easy to find the \smcldgr\ of the Poisson-Schouten bracket:
%$\spdg\Pbrv{\xblnk,\xblnk}=$.

%To prove Corollary\rw{crl1},
Observe that among all generators
$\bnbl\hdf_i$, $\bnbl\crR$ and $\bnbl\del\xcWi$ of $\Tflnbot$ only
$\hdfbt$ and
$\del\xcW_i$ have negative \smcldgr s:
$\spdg\hdfbt=\spdg\del\xcW_i=-1$. At the same time, $\spdg\xdfm=-2$, so if $\hdfbt=0$,
then each
term in the expression of $\xdfmi$, $i\geq 1$ must have at least
two powers of $\del\xcW$. This proves Corollary\rw{crl1}. If $\xcWo=\xcWt=0$, then, obviously,
$\xdfmi=0$ for $i\geq 1$ and, at the same time,
$\xdfmz=\xcWot=0$. This proves Corollary\rw{crl2}.

In order to prove Corollary\rw{crl3}, we introduce two more
gradings on the algebra
%$\Tflnbot/(\bnbl\del\xcWo,\bnbl\del\xcWt,\bnbl\crR)$
$\Tflnbot$. The first grading reflects the difference between the
total numbers of upper and lower indices in a tensor field:
\ylee{bah6.2}
\dgbl\nabla=-1,\qquad\dgbl\hdf_i=i-1,\qquad\dgbl\del\xcWi=-1,
\qquad\dgbl\crR=-2.
\yeee
Since the universal deformation parameter $\xdfm$ is \xrlb, its
degree must be zero: $\dgbl\xdfm=0$.

The second degree is the sum: %between $\spdg$ and $\dgbl$:
$\dgdf=\spdg + \dgbl$, so
\ylee{bah6.3}
\dgdf\nabla=0,\qquad\dgdf\hdf_i=2i-4,\qquad\dgdf\del\xcWi=-2,
\qquad\dgdf\crR=-2.
\yeee
Since $\spdg\xdfm=-2$ and $\dgbl\xdfm=0$, then $\dgdf\xdfm=-2$.
However, $\dgdf\hdf_i\geq 0$ for $i\geq 2$, so
if we set $\xcWo=\xcWt=0$ and $\crR=0$, then all remaining tensor
fields in $\Tflnbot$ must have non-negative degree, so $\xdfm$
must be zero. This proves Corollary\rw{crl3}.

% simply
%counts the powers of tensor fields $\hdf_i$ and of their
%holomorphic derivatives:
%%
%\ylee{bah6.1}
%\dgkp\bnbl\hdf_i=1,\qquad\dgkp\bnbl\del\xcWi=\dgkp\bnbl\crR=0.
%\yeee
%%

\subsection{Deformation of the 2-category of \copfib s:
deformation of the composition of morphisms}

Describing the deformation of the composition functor requires the
following steps: first, for a pair of \xzobj s $\xcWo$, $\xcWt$ we
have to describe the \fDAinfalg\ $\brb{\xbOmbU,\dlb,\xdfmot}$
corresponding to the \CMlt\ $\xdfmot$ found in
subsection\rw{ss.dfmtcmr},
%up to the terms of combined Dolbeault and $\xcW$ degree 4,
by working out the expressions for its \xmnmlt s.
Second, we have to describe the objects and morphisms of the
\tpdprc\ of $\brb{\xbOmbU,\dlb,\xdfmot}$. After that we can
define the composition between objects of
$\Hom(\xcWo,\xcWt)$ and $\Hom(\xcWt,\xcWh)$.

\subsubsection{A first order deformation of the \hlsm\ structure}

The procedure is simplified if we stay within the realm of
categories $ \rDsprfUWtmo$ which are very similar to derived
categories of coherent sheaves. We achieve this by considering only infinitesimal
deformations of $\rDDprfU$ by $\hdf$. In other words, we introduce the algebra
$\ICv{\dfe}/(\dfe^2)$ and consider a deformation of the \hlsms\ of
$\TsU$ by an element $\dfe\hdf$, where $\hdf$ is of the
form\rx{bbea2.4a}. The quadratic term in the \CMe\rx{bbea2.2a}
drops out, hence $\hdf$ must be holomorphic:
\ylee{\bbf1}
\dlb\hdf = 0.
\yeee
A function $\xcW$ describing an \xzobj\ of the deformed category has
the form
\ylee{bbf2}
\xcW=\acWz+\dfe\acWo.
\yeee
It must satisfy \ex{bae2.0a3}, which reduces to two
relations
\xlee{bbf3}
\dlb\acWz = 0,\qquad\dlb\acWo = \hdf(\del\acWz).
\yeee

Let us introduce two shortcut notations related to
functions $\xcW_i$ satisfying the conditions\rx{bbf3}:
\xlee{bbf3a}
\shdfij \edfn \fsd\hdf(\del\xcWvz{i},\del\xcWvz{j}),\qquad
\shdfijk \edfn \ssd\hdf(\del\xcWvz{i},\del\xcWvz{j},\del\xcWvz{k}).
\xeee

For two functions $\xcWo$, $\xcWt$ satisfying the
conditions\rx{bbf3} and thus defining \xzobj s of
the deformed category $\rDDprfUhae$, the corresponding \CMlt\ has
the form
\xlee{bbf4}
\xdfmot = \xcWotz + \dfe(\xcWoto + \bdmuot),
\yeee
where, according to \ex{beq2.23b},
\ylee{bbf5}
\bdmuot =
%- \fsd\hdf(\del\xcWoz,\del\xcWtz)
-\shdfot
= - \fsd\hdfbt(\del\xcWoz+ \del\xcWtz)
+ \OWt.
\yeee
This \xpBd\ satisfies the equations
\ylee{bbf6}
\dlb\bdmuot = 0,\qquad\dlb\xcWoto + \bdmuot \spsmb \del\xcWotz=0,
\yeee
but generally does not satisfy the integrability condition
$\Lbrv{\bdmuot,\bdmuot}=0$.

\subsubsection{Deformation of the \tZtgdcs}
\def\xcWz{ \xcW }
According to \ex{bbf4}, the category of morphisms between two \xzobj s $\xcWo$, $\xcWt$
of the deformed 2-category $\rDDprfUSek$ is
\xlee{bbf7}
\Hom_{\rDDprfUSek}(\xcWo,\xcWt) = \rDsprfBvv{\xcA}{\xcWotz + \dfe(\xcWoto +
\bdmuot)}
= \rDsprfUeWot,
%\rDsprfvv{\xcAek}{\xcWot},
\yeee
where $\xcAekot$ denotes the complex manifold $\xcA$ whose complex
structure is deformed by the Beltrami differential $\dfe\bdmuot$.
Hence the category\rx{bbf7} is defined along the lines of
subsection\rw{tZtgdcs}. However, before we go into specifics, let
us recall the definition of the \cAc.

Recall that a \xper\ object\rx{bae1.44} of a 2-periodic category $\rDsprfUW$
defined in
subsection\rw{tZtgdcs} is determined by a pair $\cE=(\xE,\nbbE)$, where
$\xE$ is a
\cqhlmvb\
%smooth vector bundle
%
over a complex manifold $\xcA$,
while $\nbbE$ is its curved differential satisfying the
properties\rx{bae1.42}--(\ref{bae1.42a1}).
Let us endow $\xE$ also with a (possibly curved) \ydh\ covariant
differential
\ee
\xlabel{bae2.2b7}
\xbOmxvb{i}(\xE) \xrarv{\nbE} \xbOmxvb{i+1}(\xE),\qquad
%\dgZt
\zdgt{\nbE} = \yev,
\eee
which satisfies the analog of \ex{bae1.42a}:
\ee
\label{bae2.2b8}
\nbE\,(\smu\wedge\sigma) = (\del\smu)\wedge \sigma +
(-1)^{\dgZt\smu}\smu\wedge (\nbE\,\sigma).
\eee
The choice of $\nbE$ is not unique, and the difference of two differentials
$\nbE$,  $\nbEp$ satisfying\rx{bae2.2b8} is a
%(multiplication by) a
differential form
\ee
\label{bae2.2b9} %*r
\nbEp = \nbE + \sea,\qquad\sea
\in
%\xbOmoeE.
\xbOmxob(\End\xE).
%,\qquad\zdgt{\sea}=\yod.
\eee
In other words, all possible differentials $\nbE$ form an affine
space based on a vector space $\xbOmoeE$. The commutator of
\ydh\ and \ydah\ differentials is the $(1,\hat 1)$-curvature
\ylee{bbf8}
[\nbbE,\nbE ]=\acFE\in
\xbOmxob(\End\xE),
%\xbOmooE,
\yeee
which satisfies the curved Bianchi identity
\xlee{bbf9}
\nbbE\acFE = -(\del\xcW)\,\xIdv{\xE}.
\yeee
The curvature $\acFE$ is determined by the object $\cE$ (that is, by
$\nbbE$) up to a $\nbbE$-exact element: if we replace $\nbE$ by
$\nbEp$ of \ex{bae2.2b9}, then $\acFE$ is replaced by
$\acFpE= \acFE + \nbbE\sea$. Hence $\cE$ determines the
\emph{\cAc}
\ylee{bbf10}
\hacFE\in
%\xbOmooE
\xbOmxob(\End\xE)
\Big/ \nbbE \lrbc{
%\xbOmoeE
\xbOmxob(\End\xE)
}.
\yeee
If $\xcW=0$, then the Bianchi identity\rx{bbf9} implies that
\ylee{bbf11}
\hacFE\in\HnbbE(\End\xE) = \Ext(\xcE,\xcE).
\yeee

A perfect object of the deformed category $\rDsprfUeWot$ is a pair
\ylee{bbf12}
\cE=\adgmE,\qquad
\nbbE = \nbbzE + \dfe\nbboE,
%\qquad \nbboE = \ydfmotv{\nbE} + \seb,
\yeee
(\cf
\ex{bae1.45b}), such that the pair $\xcEbz=(\xE,\nbbzE)$ is an
object of the undeformed category $\rDsprfUWot$, while
\xlee{bae2.2a10}
\nbboE = \ydfmotv{\nbE} + \seb,\qquad
\seb\in\xbOmbv{\End\xE}
\xeee
and $\seb$ satisfies the condition
\ee
\label{bae2.2a9} %*r
\nbbzE\,\seb = \xcWoto\,\xIdE - \ydfmotv{\acFE}.
\eee
Here $\acFE$ is the $(1,1)$-curvature of $\xcEbz$, so $\nbbE$
satisfies the condition\rx{bae1.42a1}: $\nbbE^2 =
\xcWot\xIdv{\xE}$.
A change\rx{bae2.2b9} in the choice of $\nbE$ is
compensated by the corresponding replacement of $\seb$ by
$\seb\p=\seb - \ydfmotv{\sea}$.

The space of morphisms %$\Hom\brB{\adgmeEo,\adgmeEt}$
$\Hom_{\rDsprfUeWot}(\xcEo,\xcEt)$
between two \xper\ objects is defined by means of an obvious deformation of the general
formula\rx{bae1.7}. A morphism between two objects
\wlee{bae2.2a11}
%\xlabel{bae2.2a11}
%\msgmotz\in\Hom\brB{\adgmzEo,\adgmzEt}
%\msgmotz
\msgm
%=\smgmotz + \dfe\smgmoto
\in\Hmo{\cEo,\cEt}
\weee
is represented by a $\nbbE$-closed sum
%can be deformed into a deformed morphism
%
\begin{gather}
\nonumber
\msgm  = \msgmbz + \dfe\msgmbo,\qquad
\msgmbz,\msgmbo\in\xbOmbv{\xEt\otimes\xEo^\ast},
%\Hom\brB{\adgmeEo,\adgmeEt}
%\Hmo{\acEoe,\acEte}
\\
%\label{bae2.2a12}
\label{bae2.2a14}
\nbbzcxE\,\msgmbz = 0,\qquad
\nbbzcxE\,\msgmbo = - \nbbocxE\,\msgmbz
\end{gather}
up to $\nbbE$-exact elements. Note that the dominant component $\msgmbz$
defines a morphism between the undeformed objects $\xcEobz$ and
$\xcEtbz$.

%\subsubsection{Atiyah class}
%
%Recall that a \xper\ object\rx{bae1.44} of a 2-periodic category $\rDsprfUW$
%defined in
%subsection\rw{tZtgdcs} is determined by a pair $(\xE,\nbbE)$, where
%$\xE$ is a smooth vector bundle over a complex manifold $\xcA$,
%while $\nbbE$ is a \ydah\ differential satisfying the
%properties\rx{bae1.42}--(\ref{bae1.42a1}).
%Let us endow $\xE$ also with a (possibly curved) \ydh\ covariant differential
%%
%\ee
%\xlabel{bae2.2b7}
%\xbOmxvb{i}(\xE) \xrarv{\nbE} \xbOmxvb{i+1}(\xE),\qquad
%\dgZt \nbE = \yev,
%\eee
%%
%which satisfies the analog of \ex{bae1.42a}:
%%
%\ee
%\label{bae2.2b8}
%\nbE\,(\smu\wedge\sigma) = (\del\smu)\wedge \sigma +
%(-1)^{\dgZt\smu}\smu\wedge (\nbE\,\sigma).
%\eee
%%
%The choice of $\nbE$ is not unique, and a difference of two differentials
%$\nbE$,  $\nbEp$ satisfying\rx{bae2.2b8} is a (multiplication by) a differential form
%%
%\ee
%\label{bae2.2b9} %*r
%\nbEp = \nbE + \sea,\qquad\sea
%\in\xbOmoeE.
%\eee
%%
%In other words, all possible differentials $\nbE$ form an affine
%space based on a vector space $\xbOmoeE$. The commutator of
%\ydh\ and \ydah\ differentials is the \emph{\cAc} $\acFE$ of the pair $(\xE,\nbbE)$:
%%\Ac\ of $\xdgmE$ is defined as an extension $\acFE\in\Ext^{\yod}(\xE,\xE\otimes\TsU)$
%%%\Dlb\ cohomology class $\acFE\in$
%%represented by an anti-commutator
%%
%\begin{gather}
%\xlabel{bae2.2b10}
%[\nbbE,\nbE ]=\acFE\in \xbOmooE \Big/ \nbbE \lrbc{\xbOmoeE},
%\\
%\xlabel{bae2.2b11}
%\nbbE\acFE = -(\del\xcW)\,\xIdv{\xE}.
%%\qquad \Ext^{\yod}(\xE,\xE\otimes\TsU).
%\end{gather}
%%

\subsubsection{Deformation of the composition of morphisms}
\label{ss.dfmcmpf}
The composition of morphisms between three \xzobj s
$\xcWo$, $\xcWt$ and $\xcWh$ of the deformed
2-category $\rDDprfUSek$
 is a
bi-functor
\ee
\label{bae2.8} %*r
%\rDsprfUmWot \times \rDsprfUmWth \longrightarrow \rDsprfUmWoh.
\rDsprfUeWot \times \rDsprfUeWth \longrightarrow \rDsprfUeWoh.
\eee
The composition of two morphisms $\xcEot\in\rDsprfUeWot$ and
$\xcEth\in\rDsprfUeWth$
is the appropriately deformed tensor product:
% $\otS$:
%
\xlee{bae2.10} %*r
%%%%\adgmeEth\circ\adgmeEot
\xcEth\circ\xcEot=
\adgmeEth
%\otS
\circ
\adgmeEot =
%\adgmeEoh,
(\xEth\otimes\xEot,\nbbv{\xEot\otimes\xEth} + \dfe\nboth),
\xeee
where the deformation term is
\xlee{bae2.11}
\begin{split}
\nboth & =
%\ssd \hdf(\del\xcWoz,\del\xcWtz,\del\xcWhz)
\shdfoth
\spsmb
\brB{(\del\xcWotz)\nbEth - (\del\xcWthz)\nbEot+ \acFEot\acFEth}
=
%\hdfbt\spsmb (\acFEot\acFEth) + \OWo
\\
& =
\hdfbt\spsmb (\acFEot\acFEth) + \OWo,
\end{split}
\xeee
and $\shdfoth$ is a shortcut notation defined by \ex{bbf3a}.
%%
%\ee
%\begin{split}
%\nboth &=
%\xsfS\brB{\dWothzb\,(\del\xcWthz)\,\nbEot} -
%\xsfS\brB{\dWothzb\,(\del\xcWotz)\,\nbEth}
%\\
%&\hspace{3.2in}
%-\xsfS\brB{\dWothzb\,\acFEot\acFEth}
%\\
%&=\xsfSt(\del\xcWthz\,\nbEot) -
%\xsfSt(\del\xcWotz\,\nbEth)
%-\xsfSt(\acFEot\acFEth) + \cdots,
%\label{bae2.11xx} %*r
%\end{split}
%\eee
%%
%and we used a shortcut notation
%%
%\ee
%\xlabel{bae2.11a}
%\dWothzb = \msmbv{\del\xcWoz,\del\xcWtz,\del\xcWhz}.
%\eee
%%

The first two terms in the \rhs of this equation are related to the fact
that each of three categories in\rx{bae2.8}
has its own deforming \xpBd. Hence we had
to add  correction terms to $\nbbv{\xEoh}$
so that it
would satisfy the
condition\rx{bae2.2b8}
%\rx{bae2.2a5}
or, more precisely, the
condition\rx{bae2.2a10}, that is, that the difference
%%
%\ee
%\label{bae2.12x}
%\nbboEoh-\bdmuoh(\nbEoh),\quad\text{where $\nbEoh=\nbEot+\nbEth$,}
%\eee
%%
%
\ee
\xlabel{bae2.12}
%\nbbev{\xEot}+\nbbev{\xEth} + \dfe\nboth
\nbbv{\xEoh}
-\dfe\bdmuoha{\nabla_{E_{12}\otimes E_{23}}}
%,\quad\text{where $\nbEoh=\nbEot+\nbEth$,}
\eee
must be just an odd element of $\xbOmbv{\End\xEoh}$. The third correction term
in \ex{bae2.11} is required to comply with the
condition\rx{bae2.2a9}.

The composition of morphisms $\xcEth\circ\xcEot$ defined by \ex{bae2.10}
is independent (up to an isomorphism) of the choice of \ydh\
differentials $\nbEot$ and $\nbEth$. Indeed, if we replace
$\nbEot$ with $\nbEot\p=\nbEot+\sea$ as in \ex{bae2.2b9},
then $\nboth$ is replaced by
%%
%\begin{equation}
%\label{bae2.13x}
%\begin{split}
%\nboth\p& = \nboth + \hlf\brB{ \sfSvv{\del\xcWthz}{\sea} -
%\sfSvv{\nbbzEot\;\sea}{\acFEth} }
%\\
%&= \nbboEoh - \hlf\,\nbbzEoh \,\sfSvv{\sea}{\acFEth},
%\end{split}
%\end{equation}
%%
%
%
\begin{equation}
\nonumber
\begin{split}
\nboth\p
&=
\nboth -
\shdfoth
\spsmb
\brB{  (\del\xcWthz)\sea - (\nbbzEot\;\sea)\acFEth}
\\
&=
\nboth + \nbbzEot\, \brB{ \shdfoth \spsmb (\sea\,\acFEth) },
\end{split}
\end{equation}
%
%
%\begin{equation}
%\xlabel{bae2.13}
%\begin{split}
%\nboth\p& = \nboth +
%\xsfS\brB{\dWothzb\,(\del\xcWthz)\,\sea}
%-\xsfS\brB{\dWothzb\,(\nbbzEot\;\sea)\acFEth}
%\\
%&= \nboth - \nbbzEoh \,\xsfS\brB{\dWothzb\,\sea\,\acFEth},
%\end{split}
%\end{equation}
%
so it changes by a $\nbbzEoh$-exact term.

The bi-functorial nature of the map\rx{bae2.8} means that an
object
%
%\ee
%\label{bae2.13z}
%\adgmeEth
$\xcEth\in\rDsprfUeWth$
%\eee
%
determines a functor
\ee
\xlabel{bae2.14}
\xymatrix@C=2cm{
\rDsprfUeWot
\ar[r]^-{\xPhEth} &
\rDsprfUeWoh
}.
\eee
Its action on objects is defined by the composition\rx{bae2.10}.
Consider now its action on morphisms. For
$\msgmot\in\Hom(\xcEot,\xcEotp)$, where
$\xcEot,\xcEotp\in\rDsprfUeWot$,
%be a morphism
%between two objects
%%
%\ee
%\label{bae2.15}
%\adgmeEot,\adgmeEpot\in\rDsprfUmWot.
%\eee
%%
%Then
we define
\ee
\xlabel{bae2.16}
\xPhEth(\msgmot) = \msgmot\otimes\xIdv{23}
%- \dfe \xsfS\brB{\dWothzb\,(\nbEot\msgmotz)\,\acFEth }
+ \dfe\, (\shdfoth) \spsmb \brb{\nbEot\msgmotz\otimes\acFEth}
.
%\frac{\dfe}{2}\,\sfSvv{\nbEot\msgmotz}{\acFEth}.
\eee
The correction term
%$\sfSvv{\nbEot\msgmotz}{\acFEth}$
is required to satisfy the
relation\rx{bae2.2a14}.
Similarly, an object $\xcEot\in\rDsprfUmWot$
%%
%\ee
%\label{bae2.14a}
%\adgmeEot\in\rDsprfUmWot
%\eee
%%
determines a functor
\ee
\xlabel{baez2.17}
%\xymatrix@C=2cm{
%\rDsprfUmWth
%\ar[r]^-{\xPhEot} &
%\rDsprfUmWoh
%},
\xymatrix@C=2cm{
\rDsprfUeWth
\ar[r]^-{\xPhEot} &
\rDsprfUeWoh
},
\eee
which maps a morphism $\msgmth\in\Hom(\xcEth,\xcEth\p)$,
where $\xcEth,\xcEth\p\in\rDsprfUeWth$
%between the objects
%%
%\ee
%\label{bae2.18}
%\adgmeEth,\adgmeEpth\in\rDsprfUmWth.
%\eee
%%
into a morphism
\ee
\xlabel{bae2.19}
\xPhEot(\msgmth) = \xIdv{12}\otimes\msgmth
%\frac{\dfe}{2}\,\sfSvv{\acFEot}{\nbEth\msgmthz}.%{\acFEth}.
%\dfe \xsfS\brB{\dWothzb\,\acFEot\,(\nbEth\msgmth) }.
- \dfe\, (\shdfoth) \spsmb \brb{\acFEot\otimes\nbEth\msgmthz}.
\eee

The images of morphisms $\msgmote$ and $\msgmthe$ commute in the
following sense:
\begin{multline}
\xlabel{bae2.20}
\xPhEpot(\msgmth)\circ\xPhEth(\msgmot) -
\xPhEpth(\msgmot)\circ\xPhEot(\msgmth)
%-\xPhEpot(\msgmth)\circ\xPhEth(\msgmot)
=0
\\
\text{in}\quad
\Hom\brB{
%\adgmeEth\otimes\adgmeEot,\adgmeEpth\otimes\adgmeEpot
\xcEth
%\otS
\circ
\xcEot,\xcEth\p
%\otS
\circ
\xcEot\p
},
\end{multline}
because
\begin{multline}
\xlabel{bae2.21}
\xPhEpot(\msgmth)\circ\xPhEth(\msgmot) -
\xPhEpth(\msgmot)\circ\xPhEot(\msgmth)
%- \xPhEpot(\msgmth)\circ\xPhEth(\msgmot)
\\
=
%\frac{\dfe}{2}\,\nbbv{\xEoh}\sfSvv{\nbEot\msgmotz}{\nbEth\msgmthz}
%\dfe\,\nbbv{\xEoh}\xsfS\brB{\dWothzb\,(\nbEot\msgmotz)\,(\nbEth\msgmthz)}.
 \dfe\, \nbbv{\xEoh|0}
\brB{(\shdfoth) \spsmb \brb{\nbEot\msgmotz\otimes\nbEth\msgmthz}}.
\end{multline}

In the special case $\xcWo=\xcWt=\xcWh=0$, the
formula\rx{bae2.11} says the following. Let $0$ denote the object of $\rDDprfUSek$ corresponding to the trivial fibration over $U$ with $W=0$. The endomorphism category
$\End_{\rDDprfUSek}(0)$ is a monoidal category which is equivalent to $\rDprfU$ as a category, but with a monoidal structure given by the deformed tensor product
\xlee{baj4.1}
\adgmeE
\circ
%\otS
\adgmeEp =
\brb{\xE\otimes\xEp,\nbbv{\xE\otimes\xEp} + \dfe\hdfbt\spsmb (\acFE\acFv{\xEp})},
\xeee

\subsubsection{Deformation of the monoidal structure beyond the
first order}
\label{ss.dmsbfo}

Now we return to the 2-category $\rDDprfUSk$ for a general \CMlt\
$\hdf$ and consider the category
$\xEndsp$ of endomorphisms of the \opfib\ with $\xcW=0$ denoted here
as $\zsU$. This
category has a monoidal structure corresponding to the composition
of endomorphisms. According to the general formula\rx{beq2.14b},
the endomorphism category itself is a deformation of the category
$\rDDprfU$: $\xEndsp=\rDsprfUd$, where
$\xdfm = \xdfmv{3}+\xdfmv{4}+\cdots$ and
$\xdfmi\in\xbOmv{\xmi}\brb{\xcA,\wedge^{\xmi}\TU}$, while its
monoidal structure is a deformation of the monoidal structure of
$\rDDprfU$, the latter being the tensor product\rx{bae1.45a2}.

Let us assume that the Atiyah class $\ccrR$ of the tangent bundle
$\TU$ is zero. Then, according to Corollary\rw{crl3}, the deformation
parameter is zero: $\xdfm=0$,
so we have an equivalence of categories
\xlee{bai5.1}
\xEndsp\simeq\rDprfU.
\xeee
However, as the study of the first order perturbation in subsection\rw{ss.dfmcmpf}
demonstrated, the monoidal structure of $\xEndsp$ is still a
non-trivial deformation of the tensor product monoidal structure
of $\rDprfU$. The relatively simple nature of the
category\rx{bai5.1} allows us to discuss the properties of this
deformation without invoking \Ainfalg s and their modules.

A deformation of the monoidal structure of the category $\rDprfU$
is described by two sets of data. First, for every pair of \qhlmvb s
$\adgmeEo$, $\adgmeEt$ there is a \CMlt\
$\zdfmot\in\xbOmbz\brb{\End(\xEo\otimes\xEt)}$,
\xlee{bai5.1a}
\dlb\zdfmot + \shlf\,[\zdfmot,\zdfmot]=0,
\xeee
which determines the deformed monoidal bifunctor of the composition within the endomorphism
category $\xEndsp$:
\xlee{bai5.2}
\adgmeEo\ycrc\adgmeEt = (\xEo\otimes\xEt,\bar\nabla_{E_1\otimes E_2}  +
\zdfmot).
\xeee
Second, for every triple of \qhlmvb s $\adgmeEo$, $\adgmeEt$,
$\adgmeEh$ there is an associator
$\zasoth\in\xbOmbz\brb{\End(\xEo\otimes\xEt\otimes\xEh)}$ which
establishes the associativity isomorphism
\ylee{bai5.3}
\zasoth\colon \brB{\adgmeEo\ycrc\adgmeEt}\ycrc\adgmeEh
\xmapta{\cong}
\adgmeEo\ycrc\brB{\adgmeEt\ycrc\adgmeEh}.
\yeee
If both sides of the associativity isomorphism have the presentation
\begin{equation}
\label{bai5.4}
\begin{split}
\brB{\adgmeEo\ycrc\adgmeEt}\ycrc\adgmeEh
&=
(\xEo\otimes\xEt\otimes\xEh,\nbboth+\zdfmoth),
\\
\adgmeEo\ycrc\brB{\adgmeEt\ycrc\adgmeEh}
&=
(\xEo\otimes\xEt\otimes\xEh,\nbboth+\zdfmothp),
\end{split}
\end{equation}
where $\nbboth \edfn \bar\nabla_{E_1\otimes E_2\otimes E_3}$,
then $\zasoth$ is an invertible element satisfying the equation
\xlee{bai5.5}
\nbboth\zasoth + \zdfmothp\,\zasoth - \zasoth \, \zdfmoth = 0.
\xeee

We conjecture that there exist unique universal formulas for the element
$\zdfmot$ and for the associator $\zasoth$ related to the deformation of the tensor product
monoidal structure of $\rDprfU$ into the monoidal structure
of the endomorphism category $\xEndsp$. These universal formulas
express
$\zdfmot$ and $\zasoth$ in terms of the deformation parameter
$\hdf$, \ooc s $\acFEi$ and their holomorphic covariant
derivatives:
\xlee{bai5.6}
\zdfmot\in
\Tflnv{\acFEo,\acFEt,\hdf_2,\hdf_3,\ldots}
 , \qquad
\zasoth\in
\Tflnv{\acFEo,\acFEt,\acFEh,\hdf_2,\hdf_3,\ldots}.
\xeee

We propose to derive the universal formulas perturbatively. Define
the Dolbeault degree by \ex{e.dgdlb} and by the additional formula
$\dgDlb\acFEi=1$ (note that generally $\acFEi\in\xbOmb(\End\xEi)$
and its Dolbeault degree coincides with $j$ of $\xbOmv{j}$ only when
$\xEi$ is a holomorphic vector bundle with the operator $\nbbv{\xEi}$ not containing forms of degree other than $1$).
%Also define the \smcl\ grading by \ex{bah4.2} and by the formula
%$\spdg \acFEi=1$.
We present the deformation parameter $\zdfmot$
and the associator $\zasoth$ as the sums
\ylee{bai5.7}
\zdfmot = \sum_{i=1}^{\infty}\zdfmoti,\qquad
\zasoth = \xIdEm + \sum_{i=2}^\infty\zasothi,
\yeee
where
\ylee{bai5.8}
\dgDlb\zdfmoti = 2i + 1,\qquad \dgDlb\zasothi = 2i
%,\qquad \spdg\zdfmoti = -1,\qquad \spdg\zasothi = 0.
\yeee
(the reason for assuming that the Dolbeault degree of $\zdfmot$ is
even and the Dolbeault degree of $\zasoth$ is odd will become
clear shortly).
The \CMe\rx{bai5.1a} splits:
\xlee{bai5.8}
\dlb\zdfmotn + \sum_{i=1}^{\xmn-1} \zdfmoti\,\zdfmotv{\xmn-i}=0,
\xeee
and the associativity equation\rx{bai5.5} splits:
\xlee{bai5.9}
\dlb\zasothn + \zdfmothnp - \zdfmothn +
\sum_{i=2}^{\xmn-1} (\zdfmothv{\xmn-i}\p\,\zasothi -
\zasothi\,\zdfmothv{\xmn-i})=0.
\xeee
The action of the Dolbeault differential $\dlb$ on the
elements of the algebras\rx{bai5.6} follows from its action on the
elementary tensor fields prescribed by the \CMe\rx{bbea2.2a} and
by the Bianchi identity $\nbbEi\acFEi=0$
\ftnt{We assume for simplicity that the curvatures of the $\partial$-connections $\nbv{\xEi}$
are zero.}, and from the defining
equation of the curvature tensor $\acFEi = [\nbbEi,\nbv{\xEi}]$.

We introduce the notation $\zdfmdv{\xblnk,\xblnk}$ to emphasize
the dependence of the universal deformation parameter $\zdfmot$ on
curvatures: $\zdfmot = \zdfmdv{\acFEo,\acFEt}$. The
parameters $\zdfmoth$ and $\zdfmothp$ of \ex{bai5.4} can be
expressed in terms of $\zdfmdv{\xblnk,\xblnk}$:
\begin{equation}
\xlabel{bai5.10}
\begin{split}
\zdfmoth &= \zdfmdv{\acFEo,\acFEt} + \zdfmdv{\acFEo+\acFEt+\nbv{\xEo\otimes\xEt}
\zdfmdv{\acFEo,\acFEt},\acFEh},
\\
\zdfmothp &=\zdfmdv{\acFEt,\acFEh} + \zdfmdv{\acFEo,\acFEt+\acFEh
+\nbv{\xEo\otimes\xEt}\zdfmdv{\acFEt,\acFEh} }.
\end{split}
\end{equation}
These formulas allow us to present the difference $\zdfmothnp -
\zdfmothn$ appearing in \ex{bai5.9} in the form
\xlee{bai5.11}
\zdfmothnp - \zdfmothn = \adltothzn + \tzdfmoth,
\xeee
where
\ylee{bai5.12}
\zdfmothn = \zdfmdnv{\acFEt,\acFEh} +
\zdfmdnv{\acFEo,\acFEt+\acFEh} -
\zdfmdnv{\acFEo,\acFEt} -\zdfmdnv{\acFEo+\acFEt,\acFEh}
\yeee
and the expression $\tzdfmoth$ contains the deformation parameter
components $\zdfmi$ only with $i<n$.

After the substitution\rx{bai5.11}, the associativity
equation\rx{bai5.9} becomes
\xlee{bai5.13}
\dlb\zasothn + \adltothzn + \tzdfmoth +
\sum_{i=2}^{\xmn-1} (\zdfmothv{\xmn-i}\p\,\zasothi -
\zasothi\,\zdfmothv{\xmn-i})=0.
\xeee

We conjecture that the \CMe\rx{bai5.8} together with the associativity
equation\rx{bai5.13} can be solved perturbatively
over the Dolbeault degree, thus producing the unique universal solutions
$\zdfmot$ and $\zasoth$ if, following \ex{baj4.1}, we set
\xlee{bai5.14}
\zdfmoto=\hdfbt\spsmb (\acFEo\acFEt)=
\hdfbt^{\inI\inJ} \acFv{\xEo,\inI}\acFv{\xEt,\inJ},
\xeee
where
we used the index notations explained at the end of
subsection\rw{ss.pcucmlt}, as well as the notation $\acFE =
\acFv{\xE,\inI}\,dx^{\inI}$ for the \ooc\ tensor components.
The parity of Dolbeault degrees of $\zdfmoti$ and $\zasothi$ is
dictated by these equations.

%A stronger conjecture is that $\zdfmot$ and $\zasoth$ are determined
%just by the associativity equation\rx{bai5.13}, and $\zdfmot$
%satisfies the \CMe\rx{bai5.8} automatically.

It is easy to verify that the expression\rx{bai5.14} satisfies
\eex{bai5.8} and\rx{bai5.13} for $\xmn=1$.
We leave it for the reader to verify that the following
expressions
%%
%\ylee{bai5.15}
%\zdfmott =
%\tfrac{1}{3}\,\hdfbt^{\inJ\inL}(\nb_{\inL}\hdfbt^{\inI\inK})
%\acFv{\xEo,\inI}(\acFv{\xEo,\inJ}-\acFv{\xEt,\inJ})\acFv{\xEh,\inK}
%,\qquad\zasotht =
%\yeee
%%
%
%\begin{equation}
\begin{align}
\xlabel{bai5.15}
%\begin{split}
\zdfmott &=
\tfrac{1}{3}\,
%\Big(
\hdfbt^{\inJ\inL}(\nb_{\inL}\hdfbt^{\inI\inK})
\acFv{\xEo,\inI}(\acFv{\xEo,\inJ}-\acFv{\xEt,\inJ})\acFv{\xEh,\inK}
\\
&\qquad\qquad
+\shlf\,\hdfbt^{\inI\inJ}\hdfbt^{\inK\inL}\brB{
(\nb_{\inI}\acFv{\xEo,\inK})\, \acFv{\xEt,\inJ} \acFv{\xEt,\inL}
+
(\nb_{\inI}\acFv{\xEt,\inK})\, \acFv{\xEo,\inJ} \acFv{\xEo,\inL}
},
%\\
%& \qquad\qquad\qquad +
%\hdfbt^{\inI\inJ}\hdfbt^{\inK\inL}(\nb_{\inL}\acFv{\xEo,\inI})
%(\acFv{\xEo,\inJ}-\acFv{\xEt,\inJ})
%\acFv{\xEh,\inK}.
%\\
%& \qquad\qquad\qquad
%+
%\hdfbt^{\inJ\inK}\hdfbt^{\inI\inL}\acFv{\xEo,\inI}
%(\acFv{\xEo,\inJ}-\acFv{\xEt,\inJ})
%(\nb_{\inL}\acFv{\xEh,\inK})
%\Big),
\\
\label{bai5.15a}
\zasotht &= \tfrac{2}{3}\,\hdfgm^{\inI\inJ\inK}\acFv{\xEo,\inI}\acFv{\xEt,\inJ}\acFv{\xEh,\inK}.
%\end{split}
%\end{equation}
\end{align}
satisfy these equations
%\eex{bai5.8} and\rx{bai5.13}
for $\xmn=2$.

We extend the \smcl\ grading of subsection\rw{ss.scgr} to the
algebras\rx{bai5.6} by setting $\spdg \acFEi=0$.
Solving the system of equations\rx{bai5.8} and\rx{bai5.9}
recursively over $\zn$ with the initial condition\rx{bai5.14}
determines the \smcl\ degrees of the deformation parameter
$\zdfmot$ and of the associator $\zasoth$:
\ylee{bai5.16}
\qquad \spdg\zdfmot = -1,\qquad \spdg\zasoth = 0.
\yeee

Let us consider what happens if we set $\hdfbt=0$ in the universal formulas for $\zdfmot$ and
$\zasoth$. Since $\hdfbt$ is the only generator of the
algebras\rx{bai5.6} with negative \smcl\ grading, then
$\spdg\zdfmot=-1$ implies
\ylee{bai5.17}
\zdfmot|_{\hdfbt=0}=0,
\yeee
that is, the composition part\rx{bai5.2} of the monoidal structure
remains undeformed.
However, the equation\rx{bai5.15a} indicates that if $\hdfgm\neq 0$, then
$\zasoth|_{\hdfbt=0}\neq \xId$. This means that if for a complex
manifold $\xcA$ there exists a non-trivial \CMlt\
$\hdf\in\xbOmbUUST$ such that $\hdfbt\equiv\hdf_2=0$ while
$\hdfgm\equiv\hdf_3\neq 0$, then the tensor product monoidal structure
of the category $\rDprfU$ has a non-trivial associator
$\zasoth\neq \xId$ in addition to the standard one. This situation is realized, for example, when $U$ is a holomorphic symplectic manifold $X$ \cx{Kapranov}.
The element $\hdf$ in this case describes the formal neighborhood of the diagonal in $X\times X$. It follows that the $\ZZ_2$-graded derived category of any holomorphic symplectic manifold admits a non-trivial monoidal structure with a deformed associator. This provides an underlying reason  for the results of J. Roberts and S. Willerton \cx{RobWil}.

If $\hdfbt=0$, then all remaining generators of the
algebras\rx{bai5.6} have non-negative \smcl\ degrees. Among them, only
$\acFEi$ and $\hdfgm$ have zero degrees, and all others, including
holomorphic derivatives, have positive degrees. Since
$\spdg\zasoth=0$, this means that $\zasoth$ belongs to
the algebra generated by $\hdfgm$ and $\acFEi$:
\ylee{bai5.18}
\zasoth|_{\hdfbt=0}\in\Tfl[\hdfgm,\acFEo,\acFEt,\acFEh].
\yeee
In fact, we conjecture that if $\hdfbt=0$, then the associator is a pure exponential:
\ylee{bai5.19}
\zasoth|_{\hdfbt=0} =
\exp\brb{\tfrac{2}{3}\,\hdfgm^{\inI\inJ\inK}\acFv{\xEo,\inI}\acFv{\xEt,\inJ}\acFv{\xEh,\inK}}.
\yeee

\subsection{A geometric description of the 2-category $\ctLLXsom$}
\label{ss.gdgs}

Following the outline of subsection\rw{ss.outline}, we apply the
results of the previous subsection to formulate conjectures about
a geometric description of the category $\ctLLXsom$, where $\Xsom$
is a general \hlsmm. Our goal is to
 explain how the statements
of subsection\rw{ss.tcthlsmm} referring to the case of $\xX=\TsU$, should be modified for
a general $\Xsom$.

The pairs $\gYL$, where $\yY\subset\xX$ is a lagrangian submanifold
and $\vLY\rightarrow \yY$ is a line bundle such that
$\vLY\ott=\cnKY$, are still objects of the 2-category
$\ctLLXsom$. We conjecture that the analogs of holomorphic
fibration
objects $\goYL$ also appear, but this time $\ycY\rightarrow\yY$ is
not a holomorphic fibration, as in the case of $\xX=\TsU$, but
rather a special `non-holomorphic' deformation of a holomorphic fibration. The
reason for this deformation is similar to the non-holomorphicity
of the functions $\xcW$ which solve the equation\rx{bae2.0a3}, but
we will not explore this subject further.

Suppose that two lagrangian submanifolds $\yYo,\yYt\subset\xX$
have a \gdint. We conjecture that the category of morphisms between them
is the deformed and shifted 2-periodic category of their
intersection:
\ee
\label{bae1.92z} %*r
\!\!\!\!\!\!\!\!\!\Hom_{\ctLLXsom}\brB{\gYLo,\gYLt} =
\rDsprfYadoct\ytrnLot\btrnv{\shlf\dim \xX - \dim\yYot-1},
%\rDsprfmv{\ycYo \ycapX \ycYt}{\vcLot}.%\vcLYo\oplus\vcLYt\oplus\vcLo},
%\rDprf\brb{\ycYo \ycapX \ycYt },
\eee
where $\yYoct\edfn\yYocct$, the line bundle $\vcLot\rightarrow\yYoct$ is defined by
\ex{bae1.19b1} adapted to the case of \opfib s:
\ylee{bae1.92z1}
\vcLot \edfn
\vLYo|_{\yYoct}\otimes
\vLYt|_{\yYoct}
\otimes\cnKYot^{-1}
\yeee
and
 $\adfmot\in\xbOmb\bro{\xtbw^{\bullet} \Tng\yYoct}$
is a special \CMlt\ which determines the \tAinf-deformation
of the category $\rDsprfv{\yYoct}$.

Based on the results of the previous subsection, we make
the following conjectures about $\adfmot$:
\begin{enumerate}

\item

The \CMlt\ $\adfmot$ is \xrlb\ and $\deg \adfmot\geq 2$:
%its degree is greater than 1:
%
\ylee{bah2.1a}
\adfmot = \sum_{\xmi=2}^\infty \adfmoti,\qquad
%\xdfm_{\xmi}
\adfmoti\in\xbOmv{\xmi}(\xcA,\wedge^{\xmi}\yYoct).
\yeee

\item

If at least one of the classes $\tdfbtYo$, $\tdfbtYt$ determined by the exact
sequences\rx{beq2.25b2} is zero
and  the other lagrangian submanifold has a presentation of
subsection\rw{ss.odc} as the
graph of a differential $\del\xcW$, where $\xcW$ is a function on the first
lagrangian surface, then $\xdfmot=0$.

\item

If $\yYo=\yYt=\yY$, then
$\adfmotv{2}=0$
%$\adfmot = \adfmotv{3}+\adfmotv{4}+\cdots$
%$\deg \adfmot\geq 3$
and $\adfmotv{3}$
is given by the formula\rx{beq2.42b}, where $\hdfbt$ represents
$\tdfbtY$ and $\crR$ is the curvature of the tangent bundle $\Tng\yY$.

\item If $\yYo=\yYt=\yY$ and the Atiyah class $\ccrR$ of the
tangent bundle $\TY$ is zero, then $\adfmot=0$.

\end{enumerate}

In order to derive these conjectures from \ex{beq2.18b},
we consider a tubular neighborhood of
$\yYo$ (or $\yYt$) as a tubular neighborhood of the zero section of
a deformed cotangent bundle $\mtdfTYoSk$ with an appropriate
deformation parameter $\hdf$. Then the object $\yYo$ corresponds to
the zero section and hence it is represented by the holomorphic function
$\xcWo = 0$. We assume that within $\mtdfTYoSk$ the second object $\yYt$
is of the form $\yYW$ for an appropriate function $\xcW$ on
$\yYo$. Generally, this is not true, but we expect that conjecture 1 holds
true independently of whether such a presentation
exists, while conjectures 3 and 4 correspond to the case $\xcW=0$.
Finally, we assume that the line bundle $\vLYt$ is the
pull-back of the deformation of $\vLYo$ under the projection of $\yYt$ onto the
zero-section of $\TsY_1$
(recall that the complex structure of the projection of $\yYt$
onto the base $\yYo$ of the cotangent bundle has a complex
structure corresponding to the Beltrami differential\rx{bae2.0b4},
hence the bundle $\vLYo$ has to be deformed in order to be
holomorphic with respect to it). Under these assumptions
\xlee{bah2.1}
\Hom_{\ctLLXsom}\brB{\gYLo,\gYLt} =
\Hom_{\rDDprfYoSk}(0,\xcW) = \rDsprfYoxdoct
\xeee
(\cf \ex{beq2.14b}), where $\xdfmot$ is the deformation parameter
determined by \ex{beq2.18b}, in which we set $\xcWo=0$,
$\xcWt=\xcW$ and, consequently, $\xcWot=\xcW$. Hence $\xdfmot$ has
the expansion\rx{beq2.15b}:
$\xdfmot = \sum_{\xmi=0}^\infty \xdfmoti$,
$\xdfmoti\in\xbOmv{\xmi}(\xcA,\wedge^{\xmi}\TY_1)$. Here $\xdfmotz
= \xcW$, while all other terms $\xdfmoti$ depend on $\xcW$ by
being polynomials in $\del\xcW$ and its covariant holomorphic
differentials $\nabla^k\del\xcW$, $k\geq 1$.

Since we assumed that
$\yYo$ and $\yYt$ have a \gdint, it follows that $\xcW$ has a
\xgd\ critical locus $\CrW$ which is isomorphic to the intersection
$\yYoct$. We conjecture that the category $\rDsprfYoxdoct$
localizes to $\CrW$:
\ylee{bah2.3}
\rDsprfYoxdoct = \rDsprfYadoct\ytrnLot\btrnv{\shlf\dim \xX -
\dim\yYot-1},
\yeee
and the deformation parameter $\adfmot$ is determined somehow by the
restriction $\xdfmot\rCrW$. We do not understand this relation
precisely, but we can still make conjectures about $\adfmot$ based
on the properties of $\xdfmot\rCrW$.

Consider the degree $\dgdW$ defined by \ex{bah2.4}.
%
%The components $\xdfmoti$ of $\xdfmot$ with $i\neq 0$ depend on
%$\xcW$ by being polynomials in multiple covariant differentials
%$\nabla^k \del\xcW$. Let $\dgdW$ denote the degree of such a
%polynomial with respect to the simple differential $\del\xcW$:
%%
%\ylee{bah2.4x}
%\dgdW \nabla^k \del \xcW =
%\begin{cases}
%1 & \text{if $k=0$,}
%\\
%0 & \text{if $k\geq 1$.}
%\end{cases}
%\yeee
%%
The critical locus $\CrW$ is determined by
the condition $\del\xcW=0$, hence if $\dgdW\xdfmoti\geq 1$, then
$\xdfmoti\rCrW=0$.
%
%only the terms which do not
%contain the simple differential $\del\xcW$ (that is, the ones that
%contain only $\nabla^k \del\xcW$ with $k>0$) remain non-zero at $\CrW$.
Then explicit formula\rx{beq2.23b} for $\xdfmoto=\mu$
implies that $\xdfmo\rCrW=0$. Since $\xcW$ is locally constant at
$\CrW$, we may also assume that $\xdfmotz=0$. Thus our first
conjecture is that $\adfmotz=\adfmoto=0$. We also conjecture that
$\adfmot$ is \xrlb, because the same is true for $\xdfmot$.

The formula\rx{beq2.25b3} states that $\cdfbt = \tdfbtYo$, so if
$\tdfbtYo=0$, then $\cdfbt=0$ and we can use the gauge transformation of the
\CMlt\ $\hdf$
%$\xdfmoto$
in order to set $\hdfbt=0$. Now
Corollary\rw{crl1} says that $\dgdW\xdfmi\geq 2$ for $i\geq 2$, so
$\xdfmoti\rCrW=0$ for all $i$. Hence we conjecture that if
$\tdfbtYo=0$ and $\yYt$ has a presentation as the graph of $\del\xcW$, then $\adfmot=0$.

If $\hdfbt=0$ and $\yYt$ is presented as the graph of $\del\xcW$
for a function $\xcW$ satisfying the equation\rx{bae2.0a3}, then
the normal bundle $\vccLo\edfn\Tng\yYo|_{\yYoct}/\Tng(\yYoct)$
appearing in \ex{bae1.90a} admits an $\OnC$ structure. Indeed,
$\yYoct=\CrW$, so $\del\xcW|_{\yYoct}=0$ and there is a
well-defined Hessian
%$\hsnW\in\Gamma(\rmS^2\Ts\yYo|_{\yYoct})$.
$\hsnW\in\Gamma(\rmS^2\dulv{\vccLo})$. This Hessian is
non-degenerate, because we assumed that $\yYo$ and $\yYt$ have a
\gdint. Equation\rx{bae2.0a3} implies that generally it satisfies
the equation
\ylee{bai2.1}
\dlb\hsnWb = \hdfbt\spsmb \brb{\hsnWb\,\hsnWb},
\yeee
but since we assumed that $\hdfbt=0$ we find that the Hessian is
holomorphic: $\dlb\hsnWb=0$. The holomorphic non-degenerate
Hessian provides the $\OnC$ structure for the bundle $\vccLo$.
In fact, we suspect that the converse is also true: if $\hdfbt=0$
and the bundle $\vccLo$ has an $\OnC$ structure then the
lagrangian submanifold $\yYt$ has a presentation as the graph of
$\del\xcW$ at least in a tubular neighborhood of $\yYocct$.

If $\yYo=\yYt=\yY$, then $\yYoct=\yY$ and $\xcW=0$, so
\ex{bah2.1} says that $\adfmot=\xdfmot$. Hence
$\deg \adfmot\geq 3$ and $\adfmotv{3}$
is given by the formula\rx{beq2.42b}. Also, if $\tdfbtYo=0$, then
$\adfmot=0$ follows directly from Corollary\rw{crl1} without any
further conjectures regarding the localization properties of the
deformed category $\rDsprfYoxdoct$.

Finally, if $\yYo=\yYt=\yY$ and the Atiyah class $\ccrR$ of $\TY$
is zero, then Corollary\rw{crl3} says that $\xdfmot=0$. Since in
this case $\adfmot=\xdfmot$, then $\adfmot=0$.

We cannot say much about the deformation of the
composition\rx{bae1.93} except that when all
$\ycYi$ are \opfib s with the same base $\yY=\yYo=\yYt=\yYh$, and the Atiyah class
$\ccrR$ of $\TY$ is zero, then the deformation of
the composition rule\rx{baj3.2} is described by the formulas
of subsection\rw{ss.dmsbfo} in which we replace $\xcA$ with $\yY$.

\section{\Mlc\ definition of the 2-category $\ctLLXsom$}
\label{s.mcl}

\subsection{Symplectic rectangles}

Let $\xcAbx$ denote an $\zn$-dimensional Stein complex manifold $\xcA$ equipped with
holomorphic coordinate functions $\bax=\ax_1,\ldots\ax_n$. The
functions $\bax$ determine an embedding
$\xcAbx\hookrightarrow\ICnbax$, where $\ICnbax$ is the affine
space $\IC^{\zn}$ equipped with the standard coordinates $\bax$.
In other words, $\xcAbx$ is just an open subspace of $\ICnbax$
with inherited coordinates.

A \emph{\smrc} is a product $\UxVy$
with the \hlsms\ determined by the 2-form
%. It has a natural \hlsms\ determined by the symplectic form
$\som = \sum_{i=1}^{\zn} d\ay_i\wedge d\ax_i$.
The identity map establishes an isomorphism between $\UxVy$ and
$\UmxVmy$. The permutation map
$\prms\colon\xcA\times\xcV\rightarrow \xcV\times\xcA$ establishes
the isomorphisms $\UxVy\rightarrow\VmyUx$ and
$\UxVy\rightarrow\VyUmx$.

A \smrc\
$\UxVy$ has a pair of transversal lagrangian fibrations:
a \qfib\ $\au\times\xcVby$ for $\au\in\xcA$ and a \pfib\
$\xcAbx\times\av$ for $\av\in\xcV$.
A \emph{\qemb}
\xlee{bah8.0}
\sfq\colon\UxVypp\xemq\UxVy
\xeee
is a
symplectic embedding such that there exists an embedding
$\rfq\colon\xcAp\hookrightarrow\xcA$ for which the diagram
\ylee{bah8.1}
\xymatrix{
\sUVpp \ar@{^{(}->}[r]^{\sfq}_>{q} \ar[d]& \sUV \ar[d]
\\
\xcAp \ar@{^{(}->}[r]^{\rfq} & \xcA
}
\yeee
is commutative. In other words, a \qemb\ must preserve the \qfib.
A composition of \qemb s %is a \qemb:
\xlee{bah8.1b}
\xymatrix@C=1.2cm{
\UxVyh
\ar@{^{(}->}[r]^-{\sfqth}_>{q}
%\ar@/^2.5pc/@{^{(}->}[rr]^-{\sfqth}^>{q}
&
\UxVyt
\ar@{^{(}->}[r]^-{\sfqot}_>{q}
&
\UxVyo
}
\xeee
is a \qemb.

The cotangent bundle $\TsUbx$ has a canonical structure of a \smrc,
because
the holomorphic differentials $\del\bax$ form a frame of the
cotangent bundle $\TsUbx$ thus providing an isomorphism
\xlee{bah8.1a}
\TsUbx\xmapta{\cong}\xcAbx\times\ICnbay.
\xeee
Moreover, an embedding
$\xcVby\hookrightarrow\ICnbay$ generates an embedding
$\UxVy\hookrightarrow\TsUbx$, which preserves the symplectic
structure as well as both lagrangian fibrations.

\subsection{2-categories of \smrc s and their functors}
\label{ss.tctf}

\hyphenation{ca-te-go-ry}
We define the 2-cate\-go\-ry $\rDDprfaUxVy$ as a full subcategory of
$\rDDprfaU$. A \xcfb\ $\cmfcUW\in\rDDprfaU$ is an object of
$\rDDprfaUxVy$ if its support\rx{bah8.2} lies within
$\UxVy$ as embedded into $\Ts\xcAbx$:
\ylee{bah8.3}
\yYUW\subset\UxVy\subset\Ts\xcAbx.
\yeee
The isomorphism\rx{bah8.1a} implies the equivalence of categories
\xlee{bah8.3a}
%\rDDprfav
\ctLLbv{\xcAbx\times\ICnbay} \simeq \rDDprfU.
\xeee

A curved fibration $\cmfcUWot\in\rDDprfav{\xcAo\times\xcAt}$
determines the \xxtf\ $\xPhUWot$ of \ex{bae1.69}. The formula\rx{bah7.3} describing
the transformation of the support of a curved fibration under the
action of $\xPhUWot$ implies that if the support of $\cmfcUWot$
fits within the product of \smrc s:
\xlee{bah8.4}
\yYv{\cmfcUWot}\subset
(\UxVyo)\times(\xcAxt\times\xcVyt),
\xeee
then the \xxtf\ $\xPhUWot$ restricts to the \xxtf
\xlee{bah8.5}
%\xPhUWot\colon\rDDprfav{\UxVyo}\longrightarrow\rDDprfav{\UxVyt}.
\xPhUWot\colon\ctLLbv{\UxVyo}\longrightarrow\ctLLbv{\UxVyt}.
\yeee

A particular example of the \xxtf\rx{bah8.5} is the analog of
\Ldrt s\rx{bag3.1a}. This time the \Ldrt s are the \xxtf s
\xlee{bah8.6}
\dLtp\colon\rDDprfaUxVy \longrightarrow \rDDprfaVyUmx,
\qquad
\dLtm\colon\rDDprfaUxVy\longrightarrow\rDDprfaVmyUx
\xeee
determined through \ex{bae1.69} by the \opfib\ and the curving
$\bax\cdot\bay\edfn  \sum_{i=1}^\zn\ax_i\ay_i$:
\wlee{bah8.7}
\dLt_{\pm} \edfn \dPhv{\pm\bax\cdot\bay}.
\weee
It is easy to see that the curvings $\pm\bax\cdot\bay$ satisfy the
condition\rx{bah8.4} and, moreover, the \Ltf s essentially do not change the
supports of objects: for a \xcfb\ $\cmfcUW\in\rDDprfaUxVy$
\ylee{bah8.8}
\yYv{\dLt_{\pm}\cmfcUW} = \prms\brb{ \yYUW }.
\yeee

We conjecture that the composition of \Ltf s yields the identity
\xxtf:
\wlee{bah8.9}
\dLtp\circ\dLtm \simeq \dLtm\circ\dLtp \simeq \xIdv{\rDDprfaUxVy},
\weee
so the \Ltf s themselves establish equivalences of 2-categories in
\ex{bah8.6}.

An important class of \xxtf s related to 2-categories
$\rDDprfaUxVy$ are restrictions. Suppose that $\xcAp$ is a
submanifold of $\xcA$ of the same dimension. Then there is the restriction
\xxtf\ $\rsf\colon\rDDprfaU\rightarrow\rDDprfaUp$, which acts on
\xcfb s and their morphisms by restricting them from $\xcA$ to
$\xcAp$. The \xxtf\ $\rsf$ can be restricted to the subcategory:
\xlee{bah8.10}
\rsf\colon \rDDprfaUxVy\longrightarrow \rDDprfaUp.
\xeee
If the subset $\xcAp\subset\xcA$ inherits the coordinates $\bax$
then the image of the \xxtf\rx{bah8.10} lies within
%$\rDDprfav{\UpxVy}$
$\ctLLbv{\UpxVy}$:
\xlee{bah8.11}
\rsf\colon \rDDprfaUxVy\longrightarrow \rDDprfaUpxVy.
\xeee

For a \qemb\rx{bah8.0} we define the restriction \xxtf
\ylee{bah8.12}
\rsfe\colon\rDDprfaUxVy\longrightarrow\rDDprfaUxVypp
\yeee
as the composition of five \xxtf s:
%
%\vspace*{1cm}
\ylee{bah8.13}
\xymatrix@C=1.5cm@R=1.5cm{
%\rDDprfaUxVy
%\ar[r]^-{\rsfo}
%\ar@/^3pc/[rrr]^-{\rsfe}
%&
%\rDDprfaUp
%\ar[r]^-{\xId}
%&
%\rDDprfav{\xcApbxp\times\IC_{\bayp}}
%\ar[d]^-{\dLtp}
%&
%\rDDprfaUxVypp
%\\
%& &
%\rDDprfav{\IC_{\bayp}\times\xcApmbxp}
%\ar[r]^-{\rsft}
%&
%\rDDprfav{\xcVpbyp\times\xcApmbxp}
%\ar[u]^-{\dLtm}
%
\rDDprfaUxVy
\ar[r]^-{\rsfo}
\ar@/^3pc/[rrr]^-{\rsfe}
&
\rDDprfaUp
\ar[r]^-{\xId}
&
\ctLLbv{\xcApbxp\times\IC_{\bayp}}
\ar[d]^-{\dLtp}
&
\rDDprfaUxVypp
\\
& &
\ctLLbv{\IC_{\bayp}\times\xcApmbxp}
\ar[r]^-{\rsft}
&
\ctLLbv{\xcVpbyp\times\xcApmbxp}
\ar[u]^-{\dLtm}
}
\yeee
In this diagram the restriction \xxtf\ $\rsfo$ is of the
type\rx{bah8.10}, the restriction \xxtf\ $\rsft$ is of the
type\rx{bah8.11} and the equivalence $\xId$ is of the
type\rx{bah8.3a}.

We conjecture that the restriction \xxtf\ of the composition of
\qemb s is isomorphic to the composition of individual restrictions, that is,
for a chain of \qemb s\rx{bah8.1b}
\xlee{bah8.13}
\rsfv{\sfqot\circ\sfqth}\simeq\rsfv{\sfqot}\circ\rsfv{\sfqth}.
\xeee

%%
%\ylee{bah8.13x}
%\UxVyo\xemq\UxVyt\xemq\UxVyh
%\yeee
%%

\subsection{A \prshf\ of 2-categories}

A \emph{\rcch} in a \hlsmm\ $\Xsom$ is a symplectic map
\xlee{bah8.14}
\mch\colon\UxVy\rightarrow\xX
\xeee
To every \rcch\ we associate the
2-category $\rDDprfaUxVy$. The relation\rx{bah8.13} suggests that these
chart 2-categories form a \prshf\ $\prsfXsom$. An object $\zzO$ of the category
$\ctLLXsom$ is defined to be a global section of this \prshf: to
every \rcch\rx{bah8.14} we associate an object
$\zzOf\in\rDDprfaUxVy$ with two conditions:
for any commutative triangle
\ylee{bah8.15}
\xymatrix{
\UxVy \ar[rr]^-{\prms}
\ar[rd]_-{\mch}
&&
\VyUmx
\ar[ld]^-{\mch\p}
\\
&\xX
}
\yeee
there is a relation $\zzOfp \cong \dLtp \zzOf$, and for any
commutative triangle
\ylee{bah8.16}
\xymatrix{
\UxVypp
\ar[rd]_-{\mch\p}
\ar@{^{(}->}[rr]^-{\sfq}_>{q}
&&
\UxVy
\ar[ld]^-{\mch}
\\
&\xX
}
\yeee
there should be a relation $\zzOfp \cong \rsfe\zzOf$.

Two global sections $\zzOo,\zzOt\in\ctLLXsom$ determine a \prshf\
of categories $\mathcal{H}om(\zzOo,\zzOt)$: to every
\rcch\rx{bah8.14} we associate the category
$\Hom(\zzOfo,\zzOft)$ and we define
$\Hom_{\ctLLXsom}(\zzOo,\zzOt)$ as the category of global sections
of this \prshf.

\section{Categorified algebraic geometry and the RW model}\label{sec:sheaves}

\subsection{RW model of a graded cotangent bundle}

In the case when $X=\Ts Y$ one can promote the RW model from a
$\ZZ_2$-graded \TQFT\ to a $\ZZ$-graded one, as explained in \cite{KRS1}.
To this end, one assigns cohomological degree $2$ to linear coordinates on the
fiber of the projection $\Ts Y\raa Y$. From the physical viewpoint,
the degree is the weight with respect to a $U(1)$ ghost number
symmetry. We will call the resulting graded manifold $\Ts Y[2]$. The
sheaf of holomorphic functions on $\Ts Y[2]$ is a quasicoherent
sheaf of graded algebras on $Y$:
$$
\cO_X=\oplus_p \Sym^p {\rm T} Y.
$$
The RW model with the target $\Ts Y[2]$ has $U(1)$ ghost number
symmetry, and it is natural to consider boundary conditions and
topological defects which preserve this symmetry. This gives a
$\ZZ$-graded version of the model.

The 2-category of boundary conditions supported on $Y$ has a
distinguished object: the zero section of $\Ts Y[2]$. It is easy to
see that this boundary condition is invariant with respect to the
$U(1)$ ghost number symmetry. The corresponding endomorphism
category is $\sDb(Y)$, the bounded derived category of coherent
sheaves on $Y$. From the physical viewpoint, it arises as the
homotopy category of a DG-category $\Dperf(Y)$. Objects of
$\Dperf(Y)$ are perfect DG-modules over the $\ZZ$-graded Dolbeault
DG-algebra $(\obul(Y),\bpartial)$, with morphisms being the usual
morphisms of DG-modules. $\sDb(Y)$ is a symmetric monoidal
DG-category; as discussed in \cite{KRS1}, the monoidal structure is
the standard one (this is easy to see on the classical level, but it
takes some work to show that there are no quantum corrections). The
algebra of boundary local operators for the distinguished boundary
condition (i.e. the endomorphism algebra of the unit object in the
endomorphism category) is isomorphic to $H^*(\cO_Y)$.

In the $\ZZ$-graded case, infinitesimal deformations of a boundary
condition correspond to degree-2 elements in the algebra of local
boundary operators. Thus infinitesimal deformations of the
distinguished boundary condition are parameterized by $H^2(\cO_Y)$.
If $Y$ is compact and K\"ahler, such deformations are unobstructed.
Indeed, we can choose a harmonic representative $B$ of a class in
$H^2(\cO_Y)$, and then the deformation of the boundary action is
simply
$$
\int_{\partial M} \phi^*B,
$$
where $\phi$ is a map from the space-time $M$ to the target $X$.
Since the form $B$ is closed, such deformation is obviously
BRST-invariant and does not affect BRST-transformations of any
fields. We will call such a deformation a B-field deformation, by
analogy with the 2d sigma-models.

Let $(Y,B)$ denote the distinguished boundary condition deformed by
$B$. The category of morphisms from $(Y,B_1)$ to $(Y,B_2)$ is the
bounded derived category of twisted coherent sheaves on $Y$, where
the twist is given by the class of $B_2-B_1$. We will denote this
category $\sD(Y,B_2-B_1)$. The composition of morphisms is the
obvious one (tensor product of twisted coherent sheaves).
Physically, $\sDb(Y,B)$ arises as the homotopy category of a certain
DG-category which we denote $\Dperf(Y,B)$. It is the category of
perfect CDG-modules over the CDGA $(\obul(Y),\bpartial,B)$.

More complicated boundary conditions can be obtained by considering
complex fibrations $\Z\raa Y$ equipped with a B-field $B\in
H^2(\cO_\Z)$. The category of morphisms from $(\Z_1,B_1)$ to
$(\Z_2,B_2)$ is the bounded derived category of twisted coherent
sheaves on $\Z_1\times_Y \Z_2$ with the twist given by
$\pi_2^*B_2-\pi_1^* B_1$, where $\pi_s$ is the projection from
$\Z_1\times_Y \Z_2$ to $\Z_s$, $s=1,2$.

To understand the resulting 2-category better, note that an object
of $\sDb(\Z_1\times_Y\Z_2,B_2-B_1)$ defines a functor from
$\sDb(\Z_1,B_1)$ to $\sDb(\Z_2,B_2)$. Composition of morphisms in
the 2-category of boundary conditions is simply the composition of
functors. Moreover, this functor intertwines the natural action of
$\sDb(Y)$, regarded as a monoidal category, on $\sDb(\Z_1,B_1)$ and
$\sDb(\Z_2,B_2)$. That is, if we regard the categories
$\sDb(\Z_s,B_s),$ $s=1,2$ as modules over the monoidal category
$\sDb(Y)$, then this functor defines a morphism in the 2-category of
modules.

Note that for a $\CC$-linear (or DG) monoidal category $\cC$ there
are two very different notions of a module: a module over $\cC$
regarded simply as a $\CC$-linear (or DG) category, and a module
over $\cC$ regarded as a monoidal $\CC$-linear (or monoidal DG)
category. The former is a functor from $\cC$ to the category of
complex vector spaces $\Vect_\CC$ (or the category of differential
graded complex vector spaces); the latter is a $\CC$-linear (or DG)
category which is acted upon by $\cC$. To avoid confusion, we will
call the latter notion a 2-module over $\cC$. This terminology is
not standard,\footnote{The more standard name for a 2-module is a
module category.} but natural, if we think about a monoidal category
as a 2-algebra, i.e. a categorification of an algebra. 2-modules
over a monoidal category $\cC$ form a 2-category.

One could hope that any morphism in the 2-category of 2-modules is
represented by an object of $\sDb(\Z_1\times_Y\Z_2,B_2-B_1)$. Then
the 2-category of boundary conditions in the RW model would be a
full sub-2-category of the 2-category of 2-modules. This statement
is incorrect as formulated, however, apparently it does become
correct if we replace the derived category of (twisted) coherent
sheaves with its enhancement. Recall that an enhancement of a
triangulated category $\cC$ is a DG-category $\frC$ whose homotopy
category $H^0(\frC)$ is triangulated and an equivalence of
$H^0(\frC)$ and $\cC$. From the physical viewpoint, a natural
enhancement of $\sDb(Y)$ is the DG-category $\Dperf(Y)$. Similarly,
a natural enhancement of $\sDb(Y,B)$ is the DG-category
$\Dperf(Y,B)$ of perfect CDG-modules over the CDGA
$(\obul(Y),\bpartial,B)$. The category $\Dperf(Y)$ is a monoidal
DG-category which acts by DG-functors on the DG-category
$\Dperf(\Z,B)$. Any object of $\Dperf (\Z_1\times_Y\Z_2,B_2-B_1)$
determines a DG-functor from $\Dperf (\Z_1,B_1)$ to
$\Dperf(\Z_2,B_2)$ which intertwines the action of $\Dperf(Y)$. The
improved version of the conjecture is that any such DG-functor is
represented by an object of $\Dperf (Y)$.

In \cite{Toen} this conjecture was proved for $Y$ being a point. In
\cite{BFN} the proof was extended to the case when $Y$ is a more
general scheme.

The conclusion is that the 2-category of boundary conditions in the
$\ZZ$-graded version of the RW model with target $\Ts Y[2]$ is the
homotopy category of a full sub-2-category in the 2-category of
2-modules over the monoidal DG-category $\Dperf (Y)$.

\subsection{Derived categorical sheaves}

Complex fibrations over $Y$ play a role in the RW model similar to
that played by holomorphic vector bundles in the B-model with target
$Y$. But it is well-known that more general coherent sheaves also
arise as B-branes, and it is natural to ask if boundary conditions
in the RW model can be similarly generalized.

It is convenient to take a more algebraic viewpoint and replace
complex fibrations over $Y$ with families of algebras or DG-algebras
over $Y$. Likely this entails no essential loss of generality. For
example, it is known that for any sufficiently nice (quasi-compact
and quasi-separated) scheme $Z$ the derived category of complexes of
sheaves on $Z$ with quasicoherent cohomology is equivalent to the
derived category of modules over some DG-algebra with bounded
cohomology \cite{BvB}. Thus we will replace the fibration $\Z$ with
a sheaf of DG-algebras over $Y$. One may conjecture that any sheaf
of DG-algebras over $Y$ can be interpreted as a boundary condition
in the RW model.

To test this conjecture, we need to have a reasonable definition of
the category of morphisms between sheaves of DG-algebras. A natural
definition has been sketched by B.~Toen and G.~Vezzosi \cite{ToVe}. They
work with more general objects called derived categorical sheaves
over $Y$. A derived categorical sheaf is a sheaf of DG-categories
over $Y$. This means that to any affine open subscheme $\Spec\,
A=U\subset Y$ one attaches a DG-category $\frC(U)$ over $A$, to any
inclusion of affine open subschemes $U'\subset U$ one attaches a
morphism of DG-categories $r_{U'U}:\frC(U)\raa\frC(U')$, and to any
inclusion of affine open subschemes $U''\subset U'\subset U$ one
attaches an invertible 2-morphism from $r_{U''U'}\circ r_{U'U}$ to
$r_{U''U}$. These data must satisfy a number of conditions which are
spelled out, for example, in \cite{catsheaf}. A sheaf of DG-algebras
can be thought of as a special case of this, with the DG-category
$\frC(U)$ having a single object for any $U$.

The category of morphisms from the derived categorical sheaf $\sT_1$
to the derived categorical sheaf $\sT_2$ is defined as follows.
First of all, one can define the derived tensor product of two
derived categorical sheaves which is again a sheaf of DG-categories. In \cite{ToVe} it is denoted $T_1\otimes^{\mathbb L} T_2$.
The category of morphisms from $\sT_1$ to $\sT_2$ is defined to be
the derived category of modules over the sheaf of DG-categories $\sT_1^{op}\otimes^{\mathbb L} T_2$, where $T_1^{op}$ is the opposite of $T_1$. In
this way one gets a 2-category of derived categorical sheaves over
$Y$. There are versions of this definition which depend on which
modules precisely one considers.

The simplest object in this 2-category is the structure sheaf
$\cO_Y$ regarded as a sheaf of DG-algebras with zero differential.
It corresponds to the distinguished boundary condition in the
2-category of boundary conditions for the RW model with target $\Ts
Y[2]$. Its endomorphism category is $\sDb(Y)$; this agrees with the
endomorphism category of the distinguished boundary condition in the
RW model. Given any other derived categorical sheaf $\sT$ over $Y$,
the category of morphisms from the distinguished object to $\sT$ is
a 2-module over the monoidal category $\sDb(Y)$. Thus the 2-category
of derived categorical sheaves over $Y$ is embedded into the
2-category of 2-modules over $\sDb(Y)$.

A simple but interesting example of a derived categorical sheaf is a
skyscraper sheaf, i.e. a sheaf of DG-categories such that $\frC(U)$
is quasi-equivalent to the trivial category if $U$ does not contain
a point $p\in Y$, and is quasi-equivalent to a fixed DG-category
$\frC_0$ otherwise. We may call this a skyscraper sheaf with stalk
$\frC_0$. We now explain how to construct the corresponding boundary
condition in the RW model by allowing the fibration $\Z$ over $Y$ to
carry a nontrivial curving $W\in H^0(\cO_\Z)$. Such boundary
conditions should be regarded as 3d analogues of 0-branes. For
simplicity we will assume that the DG-category $\frC_0$ is simply a
DG-algebra $\wA$, and is moreover of a geometric origin, i.e. its
derived category of modules $\sD(\wA)$ is equivalent to the derived
category of coherent sheaves on some complex manifold $V$.

First we note that in order for a curving $W$ to preserve the ghost
number symmetry, we have to allow the fiber $\Z$ to be a graded
manifold with a nontrivial $\CC^*$ action. Then the space
$H^0(\cO_Z)$ is also graded, and $W$ must sit in its degree-2
component. We will call such $W$ a superpotential. A graded
fibration $\Z\raa Y$ equipped with a superpotential $W$ of degree 2
defines a boundary condition for the RW model.

The category of morphisms from the distinguished boundary condition
to the boundary condition $(\Z,W)$ is $H^\xbulx(\Dperf
(\Z,W))=\sDb(\Z,W)$. Note that this category is equivalent to a
trivial one if $W$ has no critical points \cite{KLi,Orlov:MF}. This
is a local statement: given an open set $U\subset Y$ we may consider
the restriction $(\Z_U,W_U)$ of $(\Z,W)$ to $U$ and the category
$\sDb(\Z_U,W_U)$; this category is trivial if $\Z_U$ does not
contain critical points of $W$. Therefore a natural candidate for an
analogue of a skyscraper sheaf is a pair $(\Z,W)$ such that all
critical points of $W$ are contained in the fiber over a point $p\in
Y$.

To be concrete, let us consider the case when $Y$ is the
$n$-dimensional affine space $\AA_n$ with coordinates
$y^1,\ldots,y^n$. We will describe a boundary condition in the RW
model with target $\Ts Y[2]$ which corresponds to a skyscraper sheaf
over $Y$ with the stalk at $y=0$ being a DG-algebra $\wA$ of a
geometric origin. Let $\Z=\AA_n[2]\times Y\times V$, where
$\AA_n[2]$ denotes the affine space with linear coordinates
$a_1,\ldots,a_n$ of cohomological degree 2. The graded manifold $\Z$ is a trivial fibration over $Y$.
The superpotential will be
$$
W=\sum_i y^i a_i,
$$
We may think of $y^i$ and $a_i$ as coordinates on $\Ts Y[2]$.

The category of morphisms from the distinguished boundary condition
to $(\Z,W)$ is
\\
$\sDb(\Z,W)$. By Kn\"orrer periodicity, it is
equivalent to $\sDb(V)$. Furthermore, $W$ has a single critical
point $y=a=0$, so the category $\sDb(\Z_U,W_U)$ is equivalent to a
trivial one if $U$ does not contain the point $y=0$.

We propose that this boundary condition corresponds to the
skyscraper sheaf with the stalk $\wA$ at $y=0$. By definition, the
category of morphisms from the derived categorical sheaf $\cO_Y$ to
this skyscraper is the category $\sDb(\wA)\simeq\sDb(V)$, which
agrees with the RW model.

To check this proposal further, let us compare the endomorphism
category of $(\Z,W)$ regarded as a boundary condition in the RW
model and the endomorphism category of the skyscraper sheaf. The
former is the category $\sDb(\AA_n[2]\times \AA_n[2]\times V\times
V\times Y,\tilde W)$. The superpotential $\tilde W$ is given by
$$
\tilde W= y^i(a_i-\tilde a_i),
$$
where $\tilde a_i$ denote the coordinates on the second copy of
$\AA_n[2]$. By Kn\"orrer periodicity, this category is equivalent to
$\sDb(\AA_n[2]\times V\times V)$.

To compute the endomorphism category of a skyscraper sheaf, we first
need to compute its derived tensor product with itself. Since the
base is an affine space $\AA_n$, we can think about sheaves of
DG-algebras in algebraic terms, i.e. as DG-algebras over the ring
$\CC[y^1,\ldots,y^n]$. From this point of view, the skyscraper sheaf
with a stalk $\wA$ is simply the DG-algebra $\wA$ made into a
$\CC[y^1,\ldots,y^n]$-module by letting all $y^i$ act trivially.
Equivalently, it is a tensor product over $\CC$ of the DG-algebra
$\wA$ over $\CC$ and $\CC$ regarded as a DG-algebra over
$\CC[y^1,\ldots,y^n]$ with a trivial action of $y^i$ for all $i$ and
a trivial differential.

Since such a module is not flat over $\CC[y^1,\ldots,y^n]$, to
compute its derived tensor product with itself we need a flat
resolution for it.\footnote{We are grateful to Dima Orlov for
explaining to us the content of this paragraph.} Consider a
DG-algebra
$$
\KK_n=\left(\CC[y^1,\ldots,y^n\vert
\theta^1,\ldots,\theta^n],Q\right),
$$
where $\theta^1,\ldots,\theta^n$ are anticommuting odd variables of degree $-1$,
and the differential $Q$ is the Koszul differential
$$
Q=y^i \frac{\partial}{\partial \theta^i}
$$
It is quasi-isomorphic to $\CC$ regarded as a DG-algebra over
$\CC[y^1,\ldots,y^n]$. Hence we can obtain the desired flat
resolution by tensoring over $\CC$ the DG-algebra $\wA$ with
$\KK_n$. The derived tensor product is now computed by tensoring
with $\wA$ over $\CC[y^1,\ldots,y^n]$. The result is a DG-algebra
\begin{equation}\label{wA}
\wA_\theta=\wA\otimes_\CC\wA\otimes_\CC
\CC[\theta^1,\ldots,\theta^n],
\end{equation}
with a trivial action of the variables $y^i$. By definition, the
endomorphism category of the skyscraper is a suitable version of the
derived category of modules over this DG-algebra.

The algebra $\CC[\theta^1,\ldots,\theta^n]$ is Koszul-dual to the
algebra
\begin{equation}\label{CCalg}
\CC[a_1,\ldots,a_n],
\end{equation}
where the variables $a_i$ are even and have degree $2$, and the
differential is zero. Consequently, suitably defined derived
categories of the DG-algebra (\ref{wA}) and the DG-algebra
\begin{equation}\label{wwA}
\wA_a=\wA\otimes_\CC\wA\otimes_\CC \CC[a_1,\ldots,a_n].
\end{equation}
are equivalent. This agrees with what we got from the RW model and
Kn\"orrer periodicity.

Note that the resolution of the skyscraper categorical sheaf used
above is in some sense Koszul-dual to the trivial fibration $\Z=\AA_n[2]\times Y\times V$;
the role of the Koszul differential is played by the superpotential
$W$.

Let us consider a slightly more complicated example: a sheaf of
algebras over $\AA_1=\Spec\, \CC[y]$ which in algebraic terms is the
algebra $\CC[y]/y^k$ over the ring $\CC[y]$. For $k=1$, this is a
special case of the previous example (with $n=1$). We will argue
that there exists a boundary condition in the RW model equivalent to
such a sheaf of algebras.

Note first that the above sheaf of DG-algebras can be deformed into
a collection of $k$ skyscrapers by replacing $y^k$ with $P_k(y)$,
where $P_k$ is a degree-$k$ polynomial without multiple roots. This
corresponds to the following boundary condition in the RW model:
$\ZZ=\AA_1\times \CC[2]$, $W=a P_k(y)$. In the limit when $P_k(y)$
degenerates to $y^k$, we get $W=a y^k$. Therefore we propose that
the boundary condition with $\ZZ=\AA_1\times \CC[2]$, $W=a y^k,$
corresponds to the DG-algebra $\CC[y]/y^k$ over $\CC[y]$.

The category of morphisms from the distinguished boundary condition
to this one is the category of $\CC^*$-equivariant matrix
factorizations of $W=a y^k$. If the proposal is correct, then this
category must be equivalent to the derived category of DG-modules
over $\CC[y]/y^k$. The equivalence presumably arises from the
following matrix factorization:
\begin{align}
D=\begin{pmatrix} 0 & a \\ y^k & 0\end{pmatrix}
\end{align}
Its endomorphism algebra is a DG-algebra quasi-isomorphic to
$\CC[y]/y^k$ with the zero differential. Thus we get a bimodule
which defines a functor from the category of equivariant matrix
factorizations to the derived category of DG-modules over
$\CC[y]/y^k$. With suitable definitions, this functor should be an
equivalence of categories \cite{Orlov:priv}.

We note in passing that this construction allows one to think about
the derived category of modules over $\CC[y]/y^k$ as a category of
B-branes in some physical theory (namely, the Landau-Ginzburg model
on $\CC\times\CC[2]$ with the superpotential $W=a y^k$). In other
words, the Landau-Ginzburg model whose target is a graded manifold allows one to give meaning to such
a singular-looking theory as a sigma-model with target
$\Spec(\CC[y]/y^k)$. Similarly many other graded Landau-Ginzburg models can
be thought of as representing sigma-models whose targets are
singular schemes or even DG-schemes.

\end{document}